\documentclass[12pt,a4paper,twoside]{article}

\usepackage[margin=1.8cm]{geometry}
\usepackage[T1]{fontenc}
\usepackage{amsmath,amsfonts,amssymb,amsthm}
\usepackage[numbers]{natbib}
\usepackage[colorlinks=true,citecolor=blue,urlcolor=blue,linkcolor=black]{hyperref}
\usepackage[nameinlink,noabbrev]{cleveref}
\usepackage{graphicx}
\usepackage{multirow}
\usepackage{authblk}
\usepackage{caption}
\usepackage{array}
\usepackage{tikz}
\usepackage{booktabs,makecell}
\usetikzlibrary{decorations.pathreplacing}

\allowdisplaybreaks

\theoremstyle{plain}
\newtheorem{theorem}{Theorem}
\newtheorem{lemma}{Lemma}
\newtheorem{proposition}{Proposition}
\newtheorem{corollary}{Corollary}

\theoremstyle{definition}

\newtheorem{condition}{Condition}
\newtheorem{definition}{Definition}
\newtheorem{remark}{Remark}

\DeclareMathOperator*{\argmin}{arg\,min}

\def\bbR{\mathbb{R}}
\def\bbE{\mathbb{E}}
\def\bfI{{\mathbf{I}}}
\def\bbP{\mathbb{P}}

\def\bbB{\mathbb{B}}

\def\caN{\mathcal{N}}
\def\supp{{\rm supp}}
\def\CI{{\rm CI}}
\def\al{\alpha}
\def\ds{\displaystyle}
\def\ind{\mathbf{1}}

\def\Var{{\rm Var}}
\def\rmd{\,{\rm d}\,}

\def\bfz{{\mathbf{0}}}
\def\bfA{{\mathbf{A}}}

\def\alphab{\boldsymbol{\delta}}
\def\sign{\mathrm{sign}}

\crefname{theorem}{theorem}{theorems}
\Crefname{theorem}{Theorem}{Theorems}
\crefname{lemma}{lemma}{lemmas}
\Crefname{lemma}{Lemma}{Lemmas}
\crefname{proposition}{proposition}{propositions}
\Crefname{proposition}{Proposition}{Propositions}
\crefname{corollary}{corollary}{corollaries}
\Crefname{corollary}{Corollary}{Corollaries}
\crefname{definition}{definition}{definitions}
\Crefname{definition}{Definition}{Definitions}
\crefname{assumption}{assumption}{assumptions}
\Crefname{assumption}{Assumption}{Assumptions}
\crefname{condition}{condition}{conditions}
\Crefname{condition}{Condition}{Conditions}
\crefname{remark}{remark}{remarks}
\Crefname{remark}{Remark}{Remarks}
\crefname{example}{example}{examples}
\Crefname{example}{Example}{Examples}

\def\THETITLE{Linear Functional Testing with General Loadings in Sparse Regression: Separation Rates and Computational Barriers}
\title{\THETITLE}
\author{Jie Xie}
\author{Dongming Huang}
\affil{Department of Statistics and Data Science, National University of Singapore}
\date{}

\begin{document}

\maketitle

\begin{abstract}
We study the problem of testing $H_0: \xi^\top\beta=t_0$ in high-dimensional sparse linear regression with Gaussian random design and unknown design covariance.
The loading vector $\xi$ is arbitrary, and the exact sparsity level $k$ is unknown but bounded by a known value $k_u$.
Tests are required to control Type I error uniformly over the $k_u$-sparse null, while power is evaluated against $k$-sparse alternatives.
We construct a computationally efficient mixed test that gives an upper bound on the adaptive separation distance and establish an information-theoretic lower bound calibrated to the magnitude profile of $\xi$.
In the ultra-sparse regime $k_u\lesssim \sqrt n/\log p$, these bounds characterize the adaptive separation rate up to logarithmic factors for arbitrary $\xi$.
In the moderately sparse regime $\sqrt n/\log p\ll k_u\lesssim n/\log p$, these bounds match for several classes of loading vectors but may differ in general.
In this regime, we further prove a low-degree lower bound that matches the upper bound up to logarithmic factors.
This provides evidence that improving on the rate of the mixed test, if statistically possible, may be computationally hard.
For flat sparse loadings, we complement this evidence with a polynomial-time reduction from sparse CCA.
Finally, we examine how information about the design covariance affects the adaptive separation rate in two settings.
Under a sparse signed-spiked covariance model, the information-theoretic lower bound is attainable up to logarithmic factors by a computationally inefficient procedure, while the low-degree lower bound and sparse-CCA reduction continue to apply, providing evidence for a statistical-computational gap.
When the design covariance is known and diagonal, the adaptive separation rate takes the same form as in the ultra-sparse regime.
\end{abstract}

\section{Introduction}

High-dimensional regression is commonly studied in the regime where the number of covariates $p$ exceeds the sample size $n$. Consider the linear regression model
\begin{equation}\label{eq: linear regression}
    Y = X\beta + \varepsilon, \qquad \varepsilon \sim {\caN}(0, \sigma^2 \mathbf{I}_n),
\end{equation}
where $Y \in \mathbb{R}^n$, $X \in \mathbb{R}^{n \times p}$, and $\beta \in \mathbb{R}^p$. Under suitable conditions on $X$, regularized estimators such as the Lasso \citep{tibshirani1996regression} can attain the minimax-optimal rate \(k\log p/n\) for the squared estimation error over $k$-sparse vectors with $k \le  c n/\log p$ for some constant $c>0$; see, for example, \cite{bickel2009simultaneous,raskutti2011minimax}.

Beyond estimation, uncertainty quantification and hypothesis testing are also important tasks in high-dimensional regression.
In this paper, we study testing problems of the form
\begin{equation}\label{eq: hypothesis test}
    H_0: \xi^\top \beta = t_0,
\end{equation}
where \(t_0 \in \mathbb{R}\) is fixed and \(\xi \in \mathbb{R}^p\) is a loading vector.
Through test inversion, this testing problem is related to confidence interval construction for the linear functional $L(\beta)=\xi^\top \beta$. This formulation includes coordinate-wise
inference as the special case where \(\xi\) is a standard basis vector, and also
covers prediction-type targets, where \(\xi\) represents a test-point covariate
vector and \(\xi^\top\beta\) is the corresponding conditional mean.

For coordinate-wise inference, Zhang and Zhang \cite{zhang2014confidence}, van de Geer et al.\ \cite{vandeGeer2014optimal}, and Javanmard and Montanari \cite{javanmard2014confidence} develop debiasing Lasso methods that yield asymptotically valid confidence intervals. However, these procedures are primarily designed for the ultra-sparse regime \(k \lesssim \sqrt{n}/\log p\), which is substantially more restrictive than the condition \(k \lesssim n/\log p\) required for consistent estimation.

Inference for general linear functionals is more delicate, and existing optimality theory is largely tied to structured loading vectors.
Cai and Guo \cite{cai2017confidence} study confidence
intervals for \(\xi^\top\beta\) under the regular loading condition
\begin{equation}\label{eq: tony loading vector}
    \frac{\max_{i \in \supp(\xi)} |\xi_i|}{\min_{i \in \supp(\xi)} |\xi_i|}
    \le  \bar{c},
    \qquad \text{for some constant } \bar{c} \ge 1,
\end{equation}
and distinguish two settings: the \emph{sparse-loading}  setting where \(\|\xi\|_0\lesssim k\) and the \emph{dense-loading} setting where $k \asymp p^{\gamma}$, $\|\xi\|_0 \asymp p^{\gamma_\xi}$, and  \(\gamma_\xi>2\gamma\).
For both settings, they derive the minimax expected length of confidence intervals for $\xi^\top \beta$.
Cai, Cai and
Guo~\cite{cai2019optimal} develop an inference procedure that is valid for arbitrary $\xi$, but their lower-bound theory is
established for loadings satisfying \eqref{eq: tony loading vector}, and for a specific class of polynomially decaying loadings.

We study the problem \eqref{eq: hypothesis test} in an adaptive testing framework, where the true sparsity level \(k\) is unknown and only an upper bound \(k_u\) is available.
This question lies at the core of adaptive inference, a central theme in high-dimensional statistics \cite{cai2017confidence,cai2018accuracy,cai2019optimal,bradic2022testability}.
Following the adaptive testing formulation of \cite{cai2019optimal}, we require the test to control Type I error uniformly over the enlarged parameter space with sparsity \(k_u\), while its power is evaluated over alternatives with the true sparsity level \(k\le k_u\).
An \textit{adaptive separation distance} quantifies the smallest signal size needed to distinguish the null from the local alternative when the test is only allowed to use $k_u$.

Within this adaptive framework, our goal is to characterize the separation distance for general loading vectors. As discussed above, existing optimality theory is largely tied to structured loading classes. In particular, prior results mainly cover loadings that satisfy the regular loading condition \eqref{eq: tony loading vector} and whose support size falls into either the sparse-loading regime or the dense-loading regime. The intermediate regime between these two cases remains open even under regular loading. More generally, existing results do not describe the adaptive separation distance for heterogeneous loading vectors. Such heterogeneity is natural in prediction problems, where \(\xi\) may be a test-point covariate vector. For instance, if the coordinates of \(\xi\) are i.i.d.\ Gaussian, then the nonzero coordinates need not be of the same order, so \(\xi\) violates the regular loading condition \eqref{eq: tony loading vector} with high probability and does not belong to the polynomial-decay classes. This motivates the central question studied in this paper:

\begin{center}
\fbox{
\begin{minipage}{0.8\linewidth}
\centering
\textit{Given only $k_u$, what is the separation scale for testing $H_0: \xi^\top \beta=t_0$ against $k$-sparse alternatives with any loading vector $\xi$?}
\end{minipage}
}
\end{center}

\subsection{Major Contributions}
This paper develops a loading-dependent theory to characterize the adaptive separation distance for testing a linear functional in high-dimensional sparse regression with Gaussian design.
Notably, we identify a potential computational barrier that arises in the moderately sparse regime
\(
\sqrt{n}/\log p \ll k_u \lesssim n/\log p
\)
when the design covariance matrix $\Sigma$ is unknown.

Our theory has three parts: adaptive separation distances for general loadings in the ultra-sparse regime, computational evidence for the gaps in the moderately sparse regime, and refined analysis to identify the source of the gaps.
We next describe the main results in more detail.

\textit{Adaptive separation distances for general loading vectors in the ultra-sparse regime.}
In the ultra-sparse regime $k_u \lesssim \sqrt{n}/\log p$,
we characterize the adaptive separation rate, up to logarithmic factors, for testing~\eqref{eq: hypothesis test} with an arbitrary loading vector \(\xi \in \mathbb{R}^p\). This extends prior work \cite{cai2017confidence,cai2019optimal}, which mainly treats loading vectors satisfying the regular loading condition~\eqref{eq: tony loading vector} or a polynomial-decay condition.
Our upper bound is obtained by decomposing $\xi$ into large and small coordinates and mixing debiased and plug-in tests; the lower bound uses least favorable priors calibrated to the magnitude profile of $\xi$.
Extending into the moderately sparse regime, the upper and lower bounds continue to match for some classes of loading vectors, but may differ in general.

\textit{Evidence for computational barriers in the moderately sparse regime.}
In the moderately sparse regime, we prove a low-degree lower bound showing that no degree-$D$ polynomial weakly separates the null and alternative at and below a certain separation level.
For $D=O(\log p)$, the low-degree lower bound matches the computationally feasible upper bound up to logarithmic factors.
Under the standard low-degree heuristic \citep{hopkins2018statistical,kunisky2019notes}, this result provides evidence for a computational barrier.
In addition, for flat sparse loadings, we derive an explicit polynomial-time reduction to demonstrate that the testing problem~\eqref{eq: hypothesis test} is at least as hard as sparse canonical correlation analysis (CCA). This connection is non-obvious: testing a linear functional in sparse linear regression has no apparent structural resemblance to sparse CCA or sparse PCA, and the existing literature on linear functional inference gives little indication that such a connection should arise.
Since sparse CCA is widely believed to exhibit computational barriers \cite{gao2017sparse,laha2023support}, our reduction provides complementary evidence for the low-degree barrier.

To the best of our knowledge, this is the first work to provide low-degree lower-bound evidence and sparse-CCA reduction evidence for computational barriers in testing linear functionals in high-dimensional sparse linear regression. Existing computational lower bounds for sparse regression mainly concern recovery, estimation, or detection of the sparse signal itself \cite{reeves2019all,bresler2018sparse}; by contrast, our problem concerns inference for a linear functional, where the difficulty arises from the interaction between the unknown covariance structure and sparsity level.

\textit{Refined analysis with knowledge of covariance.}
We calibrate the scope of the lower-bound analysis through two benchmarks with additional information about the design covariance.
When the design covariance is known and diagonal, the upper bound is refined to match the lower bound up to logarithmic factors.
When $\Sigma$ remains unknown but satisfies a sparse signed-spiked assumption, the lower bound is again sharp up to logarithmic factors, but now the attaining optimal test is computationally inefficient. In this setting, the low-degree lower bound and sparse-CCA reduction continue to apply, thereby providing evidence for a statistical--computational gap.

\subsection{Related literature}
We review several lines of research related to our setting.

\textit{Minimax linear functional estimation.}
The minimax theory for estimating linear functionals has been extensively developed under various statistical models; see, for example, \cite{ibragimov1985nonparametric,cai2004minimax,cai2005adaptive,collier2017minimax,xie2025minimax}. The work closest to ours is \cite{xie2025minimax}, which establishes minimax rates for estimating a general linear functional \(\xi^\top\mu\), with an arbitrary loading vector \(\xi\), in the Gaussian sequence model with a sparse mean vector \(\mu\).
Our analysis differs from theirs in three respects: the observations come from a random-design regression model, the quantity of interest is an adaptive testing boundary rather than an estimation risk, and the unknown design covariance plays an important role in the moderately sparse regime.

\textit{Linear hypothesis testing in high-dimensional linear models.}
Beyond the debiased Lasso approach discussed earlier, linear hypothesis testing in high-dimensional regression has also been studied in \cite{javanmard2020flexible,zhu2018linear,bradic2022testability,zhao2024estimation}. Javanmard and Lee \cite{javanmard2020flexible} reduce testing a general linear functional to testing its projections onto a selected orthogonal basis, and then construct debiased estimators for the corresponding basis coefficients. Their approach, however, relies critically on sparsity of the loading vector \(\xi\), and therefore does not extend to general loadings.
Zhao et al.~\cite{zhao2024estimation} propose an estimator that adapts to the sparsity of \(\beta\) and to the strength of correlations among predictors, but their objectives and assumptions differ substantially from those considered here.

Zhu and Bradic \cite{zhu2018linear} construct confidence intervals for linear functionals under the assumption that
\[
\mathbb{E}\!\left[\xi^\top X_i \,\middle|\, \omega_1^\top X_i,\ldots,\omega_{p-1}^\top X_i\right]
\]
is a sparse linear combination of \(\omega_1^\top X_i,\ldots,\omega_{p-1}^\top X_i\), where \(\{\omega_j\}_{1\le j\le p-1}\) is an orthonormal basis of the subspace orthogonal to \(\xi\). When this conditional sparsity level is sufficiently small, their procedure attains the parametric rate \(n^{-1/2}\) for estimating the linear functional. Bradic et al.~\cite{bradic2022testability} develop related ideas for inference on a single coefficient \(\beta_1\), assuming that \(X_1\mid X_{-1}\) admits a sparse linear representation in terms of \(X_{-1}\), and establish corresponding minimax rates.
Although these structural conditions lead to elegant inferential procedures, they need not hold in our setting, where \(\xi\) is arbitrary and \(\Sigma\) is only assumed to have bounded spectrum.

\textit{Computational barriers in high-dimensional problems.}
Several lines of work address computational barriers in high-dimensional statistical problems.
One common strategy is to establish polynomial-time reductions from conjecturally hard problems to the statistical task of interest.
Such reductions have been used to provide evidence of computational barriers in a variety of high-dimensional statistical problems, including sparse PCA \citep{berthet2013complexity}, submatrix detection \citep{ma2015computational}, and sparse CCA \citep{gao2017sparse}, typically under the assumption that certain instances of the planted clique problem cannot be solved by randomized polynomial-time algorithms \citep{rossman2010average,feldman2017statistical}.
However, constructing these reductions is often delicate: the resulting transformations can be fragile, may require additional structural assumptions, and can limit the scope of the resulting hardness conclusions.

A complementary line of work establishes lower bounds for broad classes of algorithms, including sum-of-squares algorithms \cite{deshpande2015improved,potechin2022sub,ma2015sum}, statistical query algorithms \cite{feldman2017statistical}, and low-degree polynomial algorithms \cite{hopkins2018statistical,kunisky2019notes}. Although these classes do not capture all polynomial-time algorithms, they are sufficiently expressive to model many practical procedures and often provide strong evidence of computational hardness. Among them, the low-degree polynomial framework has proven particularly powerful, as the behavior of such algorithms can be analyzed directly and sharply. This framework has been successfully applied to a wide range of high-dimensional testing problems, including sparse PCA \cite{hopkins2018statistical,kunisky2019notes}, tensor PCA \cite{hopkins2018statistical}, sparse CCA \cite{laha2023support}, graphon estimation \cite{luo2024computational}, and independent component analysis \cite{auddy2025large}.
In this work, our computational results rely on these two complementary forms of evidence: a direct low-degree lower bound for general loading vectors and a polynomial-time reduction from sparse CCA for flat sparse loading vectors.

\subsection{Organization of the paper}
The remainder of the paper is organized as follows.
\Cref{sec: preliminary and preview} introduces the framework for analyzing the minimax properties of linear hypothesis testing and provides a preview of the main results.
\Cref{sec: minimax rate} develops the general upper and lower bounds for the adaptive separation distance under arbitrary loading vectors and identifies conditions under which these bounds match.
\Cref{sec: computational barrier} provides evidence for computational barriers in the moderately sparse regime \( \sqrt{n}/\log p \ll k_u \lesssim n/\log p \).
\Cref{sec: extensions} examines two cases with prior knowledge about design covariance, including sparse signed-spiked covariance and known design covariance.
\Cref{sec: discussion} concludes with further discussions and open directions.

\subsection{Notation}

For a positive integer \(p\), let \([p]=\{1,\ldots,p\}\). For a set \(S\), denote its cardinality by \(|S|\). For \(x,y\in\mathbb{R}\), let \(x\vee y=\max\{x,y\}\), \(x\wedge y=\min\{x,y\}\), and \(x_+=x\vee 0\). For \(x\in\mathbb{R}\), let \(\lfloor x\rfloor\) denote the greatest integer less than or equal to \(x\).
For a vector \(x\in\bbR^p\) and a subset \(J\subset[p]\), \(x_J\) denotes the subvector of \(x\) indexed by \(J\), and \(x_{-J}\) denotes the subvector indexed by \(J^c\). We write \(\supp(x)\) for the support of \(x\). For \(q>0\), define \(\|x\|_q=\left(\sum_{i=1}^p |x_i|^q\right)^{1/q}\).
We also use the conventions
\(\|x\|_0=|\supp(x)|\) and \(\|x\|_\infty=\max_{1\le j\le p}|x_j|\).
Let \(e_i\) denote the \(i\)th standard basis vector in \(\bbR^p\).
For a matrix \(X\in\bbR^{n\times p}\), \(X_{i\cdot}\), \(X_{\cdot j}\), and \(X_{ij}\) denote, respectively, its \(i\)th row, \(j\)th column, and \((i,j)\) entry. The notation \(X_{i,-j}\) denotes the \(i\)th row of \(X\) with the \(j\)th coordinate removed, and \(X_{-j}\) denotes the submatrix of \(X\) obtained by removing its \(j\)th column. For \(J\subset[p]\), \(X_{\cdot J}\) denotes the submatrix of \(X\) formed by the columns indexed by \(J\). For a symmetric matrix \(A\), \(\lambda_{\min}(A)\) and \(\lambda_{\max}(A)\) denote its smallest and largest eigenvalues, respectively.
We use \(c\) and \(C\) to denote generic positive constants whose values may vary from line to line. For two positive sequences \(a_n\) and \(b_n\), we write \(a_n\lesssim b_n\) if there exists a constant \(C>0\) such that \(a_n\le Cb_n\) for all \(n\). We write \(a_n\gtrsim b_n\) if \(b_n\lesssim a_n\), and \(a_n\asymp b_n\) if both \(a_n\lesssim b_n\) and \(b_n\lesssim a_n\). We write \(a_n\ll b_n\) if \(a_n/b_n\to 0\), and \(a_n\gg b_n\) if \(b_n\ll a_n\). The notation \(a_n\asymp_{\log} b_n\) means that \(a_n\lesssim b_n\) and \(b_n\lesssim a_n\) hold up to logarithmic factors in \(p\).
For a measure \(\pi\), we denote by \(\pi^{\otimes 2}\) the product measure \(\pi\otimes\pi\).

\section{Preliminaries and Preview}\label{sec: preliminary and preview}
In this section, we introduce the framework for linear hypothesis testing in high-dimensional linear regression and state several technical preliminaries used in the subsequent analysis. We also provide an informal summary of the main results in \Cref{sec: main results}, which serves as a theorem map for the rest of the paper.

\subsection{Problem setup}\label{sec: problem setup}
Throughout the article, we focus on the high-dimensional linear model \eqref{eq: linear regression} with random design, where the rows of $X$ satisfy $X_{i\cdot}\stackrel{\text{i.i.d.}}{\sim} \caN(0,\Sigma),i=1,\cdots,n$, and are independent of $\varepsilon$. Both $\Sigma$ and the noise level $\sigma$ are treated as unknown.
The observed data are ${\cal Z}=\{Z_1,\cdots,Z_n\}$, where $Z_i=(Y_i,X_{i \cdot})\in \bbR^{p+1}$ for $i=1,\cdots,n$.
The distribution of the data is now indexed by the parameter
$$\theta=(\beta,\Sigma,\sigma),$$
which consists of the signal $\beta$, the covariance matrix $\Sigma=\bbE[X_{i \cdot}X_{i \cdot}^\top]$ for the random design, and the noise level $\sigma$.
We consider the following collection of parameter spaces:
\begin{equation}\label{eq: parameter space}
    \begin{aligned}
    \Theta(k)=&\left\{\theta=(\beta,\Sigma,\sigma):\left\|\beta\right\|_0\le  k,\frac{1}{M_1}\le  \lambda_{\min}(\Sigma)\le  \lambda_{\max}(\Sigma)\le  M_1, 0<\sigma\le  M_2\right\},
    \end{aligned}
\end{equation}
where $M_1>1$ and $M_2>0$ are some positive constants.
The eigenvalue bound and noise-variance bound
are two mild regularity conditions \cite{cai2017confidence, cai2018accuracy,cai2020semi}. We use $\bbP_\theta^n$ and $\bbE_\theta$ to denote the probability and expectation with respect to the distribution of the data ${\cal Z}$ indexed by $\theta$.

\subsection{Optimality framework for hypothesis testing: minimaxity and adaptivity}\label{sec: framework}
For $0<\alpha<1$ and a given parameter space $\Theta$, the class of tests of nominal level $\alpha$ for testing the null hypothesis $\theta\in\Theta$ is defined as
\begin{equation}\label{eq: null test}
\Psi_{\alpha}(\Theta)
=
\left\{
\psi : \mathcal{Z} \mapsto [0,1]
:\;
\sup_{\theta \in \Theta} \mathbb{E}_{\theta} \psi \le  \alpha
\right\},
\end{equation}
see, for example, Lehmann and Romano \cite{lehmann2005testing}. Here, we allow for both randomized and non-randomized tests.
We consider the linear hypothesis testing problem \eqref{eq: hypothesis test}. The corresponding null parameter space is given by
\begin{equation}\label{eq: null space}
\Theta(k;\xi,t_0)
=
\left\{
\theta=(\beta,\Sigma,\sigma)\in \Theta(k)
:\;
\xi^\top \beta = t_0
\right\}.
\end{equation}
To characterize local alternatives, we define
\begin{equation}\label{eq: alternative space}
\Theta_{\pm\tau}(k;\xi,t_0)
=
\left\{
\theta=(\beta,\Sigma,\sigma)\in \Theta(k)
:\;
\lvert \xi^\top \beta - t_0 \rvert \ge \tau
\right\},
\end{equation}
where $\tau>0$ is a given separation level. Larger values of \(\tau\) correspond to alternatives that are farther from the null and are therefore easier to distinguish.

For fixed \(0<\alpha,\eta<1\), we define the minimax separation distance by
\begin{equation}\label{eq: minimax separation}
\tau_{\text{mini}}(k\,;\,\xi)
=
\inf\left\{
\tau
:\;
\sup_{\psi\in\Psi_\alpha(\Theta(k;\xi,t_0))}
\inf_{\theta\in\Theta_{\pm\tau}(k;\xi,t_0)}
\mathbb{E}_\theta \psi
\ge 1-\eta
\right\}.
\end{equation}
Thus, \(\tau_{\mathrm{mini}}(k\,;\,\xi)\) is the smallest separation level at which there exists a test with Type I error at most \(\alpha\) uniformly over \(\Theta(k;\xi,t_0)\) and power at least \(1-\eta\) uniformly over \(\Theta_{\pm\tau}(k;\xi,t_0)\).
The separation distance depends on $\xi,k,t_0,\alpha$ and $\eta$. Throughout the paper, we suppress the dependence on $t_0$, $\alpha$, and $\eta$ in the notation $\tau_{\text{mini}}(k\,;\,\xi)$, as these quantities do not affect the rate of the separation distance. The constants $\alpha$ and $\eta$ can be chosen arbitrarily small.

A test is said to be \emph{valid} if its Type~I error over the null parameter space does not exceed $\alpha$, and \emph{powerful} if its power over the local alternative parameter space is at least $1-\eta$. We say that a test $\psi$ is \emph{minimax optimal} if
\begin{equation}\label{eq: optimal}
\psi \in \Psi_\alpha(\Theta(k;\xi,t_0))
\quad\text{and}\quad
\inf_{\theta\in\Theta_{\pm\tau}(k;\xi,t_0)}
\mathbb{E}_\theta \psi
\ge 1-\eta
\quad\text{for}\quad
\tau \asymp \tau_{\text{mini}}(k\,;\,\xi).
\end{equation}

The minimax separation distance in~\eqref{eq: minimax separation} is defined for a fixed sparsity level \(k\). To formalize adaptivity, we follow \cite{cai2019optimal} and consider two sparsity levels \(k\le  k_u\), where \(k\) is the unknown true sparsity level and \(k_u\) is a known upper bound. When prior information about sparsity is limited, \(k_u\) may be substantially larger than \(k\). The test must therefore control its Type I error uniformly over the larger null parameter space \(\Theta(k_u;\xi,t_0)\), while its power is evaluated over alternatives with sparsity level \(k\). Accordingly, we define the \emph{adaptive separation distance} by
\begin{equation}\label{eq: adaptive separation}
    \tau_{\text{adap}}(k_u, k; \xi) = \inf
    \left\{
    \tau : \sup_{\psi\in \Psi_\alpha(\Theta(k_u;\xi,t_0))} \inf_{\theta\in \Theta_{\pm \tau}(k;\xi, t_0)} \bbE_\theta \psi\ge  1 - \eta
    \right\}.
\end{equation}
This work focuses on the asymptotic regime in which $n$, $p$, and $k_u$ diverge.
For simplicity, we use the notation $\lim$, $\asymp$, $\to$, $o(\cdot)$, and $O(\cdot)$ to denote limits and asymptotic relations.

Comparing \eqref{eq: adaptive separation} with \eqref{eq: minimax separation}, we observe that the lack of precise knowledge about the sparsity level affects only the size control over a larger parameter space, while the power functions in \eqref{eq: minimax separation} and \eqref{eq: adaptive separation} are evaluated over the same parameter space. It is evident that \( \tau_{\text{mini}}(k\,;\,\xi) = \tau_{\text{adap}}(k,k; \xi) \). Analogous to \eqref{eq: optimal}, a test \( \psi \) is defined as \textit{adaptively optimal} if it satisfies:
\begin{equation}\label{eq: adaptively optimal}
    \psi\in \Psi_\alpha(\Theta(k_u;\xi,t_0))\quad \text{and}\quad \inf_{\theta\in \Theta_{\pm \tau}(k;\xi, t_0)}\bbE_\theta \psi \ge  1 - \eta \quad \text{for}\quad \tau \asymp \tau_{\text{adap}}(k_u, k; \xi).
\end{equation}

The quantities $\tau_{\text{mini}}(k\,;\,\xi)$ and $\tau_{\text{adap}}(k_u,k\,;\,\xi)$ do not depend on a specific testing procedure but instead reflect the intrinsic difficulty of the testing problem~\eqref{eq: hypothesis test}, which is determined by the parameter space and the loading vector $\xi$.

We can tell whether one pays a statistical price for not knowing $k$ by comparing $\tau_{\text{mini}}(k\,;\,\xi)$ and $\tau_{\text{adap}}(k_u,k\,;\,\xi)$. We focus on the
nontrivial case \(k\ll k_u\). When \(k\asymp k_u\), the upper bound \(k_u\)
already localizes the sparsity level sufficiently well, and one typically has
\(\tau_{\text{mini}}(k;\xi)\asymp \tau_{\text{adap}}(k_u,k;\xi)\).
If
\[
\tau_{\text{mini}}(k\,;\,\xi) \asymp \tau_{\text{adap}}(k_u,k\,;\,\xi),
\]
then we say the hypothesis testing problem~\eqref{eq: hypothesis test} is \emph{adaptive to sparsity}; that is, even without knowing the exact sparsity level, it is possible to construct a test that achieves the same separation rate as if the sparsity level were known. In contrast, if
\[
\tau_{\text{mini}}(k\,;\,\xi) \ll \tau_{\text{adap}}(k_u,k\,;\,\xi),
\]
then the hypothesis testing problem~\eqref{eq: hypothesis test} is \emph{non-adaptive}, indicating that information about the sparsity level is essential.
In this case, the adaptive separation distance $\tau_{\text{adap}}(k_u,k\,;\,\xi)$ is of interest, as it quantifies the best achievable performance in the absence of precise sparsity information.

As pointed out in \cite{cai2019optimal}, the minimax detection boundary quantifies the intrinsic difficulty of the testing problem when the sparsity level is known, whereas the adaptive separation distance addresses a more challenging setting where the sparsity level is unknown. In practice, adaptively optimal tests that satisfy condition \eqref{eq: adaptively optimal} are more useful than minimax optimal tests, as the exact sparsity level $k$ is typically unknown in real applications.
Furthermore, the definitions immediately yield the identity $\tau_{\text{mini}}(k\,;\,\xi)=\tau_{\text{adap}}(k, k; \xi)$, so it suffices to focus on characterizing the adaptive separation distance.

In addition, we consider the setting in which the covariance matrix is known and fixed as
\(\Sigma = \Sigma_0\) for some $\Sigma_0\in \bbR^{p \times p}$ satisfying the eigenvalue condition in \eqref{eq: parameter space}.
In this case, we specify the parameter space as
\[
\Theta(k,\Sigma_0)
= \left\{ \theta = (\beta,\Sigma_0,\sigma) : \theta \in \Theta(k) \right\}.
\]
Analogous to \eqref{eq: null space}, \eqref{eq: alternative space}, \eqref{eq: minimax separation}, and
\eqref{eq: adaptive separation}, we define
\begin{subequations}
\begin{align}
    \Theta(k,\Sigma_0;\xi,t_0)
        &= \left\{ \theta = (\beta,\Sigma_0,\sigma) \in \Theta(k;\xi,t_0) \right\},\nonumber \\
    \Theta_{\pm \tau}(k,\Sigma_0;\xi, t_0)
        &= \left\{ \theta = (\beta,\Sigma_0,\sigma) \in \Theta_{\pm \tau}(k;\xi, t_0) \right\}, \nonumber \\
    \tau_{\text{mini}}(k,\Sigma_0\,;\,\xi)
        &= \inf
        \left\{
            \tau :\;
            \sup_{\psi \in \Psi_\alpha(\Theta(k,\Sigma_0;\xi,t_0))}
            \inf_{\theta \in \Theta_{\pm \tau}(k,\Sigma_0;\xi, t_0)}
            \mathbb{E}_\theta \psi \ge 1 - \eta
        \right\}, \label{eq: minimax separation design} \\
    \tau_{\text{adap}}(k_u,k,\Sigma_0\,;\,\xi)
        &= \inf
        \left\{
            \tau :\;
            \sup_{\psi \in \Psi_\alpha(\Theta(k_u,\Sigma_0;\xi,t_0))}
            \inf_{\theta \in \Theta_{\pm \tau}(k,\Sigma_0;\xi, t_0)}
            \mathbb{E}_\theta \psi \ge 1 - \eta
        \right\}. \label{eq: adaptive separation design}
\end{align}
\end{subequations}

\subsection{Estimation and sparsity conditions}
Before turning to the inference problem, we first state several conditions used throughout this article.

To develop the inference procedures, we assume the existence of estimators $\hat{\beta}$ and $\hat{\sigma}^2$ satisfying the following conditions.
\begin{condition}\label{cdt: linear estimator}
    With probability approaching 1, the estimator $\hat{\beta}$ satisfies
    \begin{equation}\label{eq: linear estimator}
        \|\hat{\beta}-\beta\|_1 \le  c_\beta \sigma k_u \sqrt{\frac{\log p}{n}}, \quad \text{and} \quad \|\hat{\beta}-\beta\|_2 \le  C_\beta \sigma\sqrt{\frac{k_u \log p}{n}},
    \end{equation}
    for some constants \( c_\beta, C_\beta > 0 \).
\end{condition}
\begin{condition}\label{cdt: linear variance}
    $\hat{\sigma}^2$ is a consistent estimator of $\sigma^2$, i.e., $|\hat{\sigma}^2/\sigma^2-1|\stackrel{p}{\to}0$.
\end{condition}

\Cref{cdt: linear estimator,cdt: linear variance} are used in constructing upper bounds on adaptive separation distances.
Various polynomial-time regularized estimators satisfy Conditions~\ref{cdt: linear estimator} and \ref{cdt: linear variance}; one example is the scaled Lasso, defined by
\[
\{\hat{\beta},\hat{\sigma}\} = \argmin_{\beta \in \bbR^p,\, \sigma \in \bbR^+} \left\{ \frac{\|Y - X\beta\|_2^2}{2n\sigma} + \frac{\sigma}{2} + \sqrt{\frac{2.01 \log p}{n}} \sum_{j=1}^p \frac{\|X_{\cdot j}\|_2}{\sqrt{n}} |\beta_j| \right\},
\]
which Sun and Zhang \cite{sun2012scaled} showed satisfies these conditions. A key assumption required for general regularized estimators to meet these conditions is the \textit{restricted eigenvalue condition} on the design matrix \( X \), originally introduced by \cite{bickel2009simultaneous}:
\begin{equation}\label{eq: restricted eigenvalue}
    \kappa(X, k_u, \alpha_0) = \min_{\substack{J_0 \subseteq \{1, \ldots, p\} \\ |J_0| \le  k_u}} \ \min_{\substack{\alphab \neq 0 \\ \|\alphab_{J_0^c}\|_1 \le  \alpha_0 \|\alphab_{J_0}\|_1}} \frac{\|X \alphab\|_2}{\sqrt{n} \|\alphab_{J_0}\|_2}.
\end{equation}

In the classical analysis of sparsity-inducing estimators
\citep{bickel2009simultaneous}, the constants \(c_\beta\) and \(C_\beta\) in
\eqref{eq: linear estimator} typically depend on the restricted eigenvalue condition;
in particular, the error bound deteriorates as \(\kappa(X,k_u,\alpha_0)\) decreases.
Furthermore, directly verifying the restricted eigenvalue condition is computationally difficult
\citep{bandeira2013certifying,tillmann2013computational}.

In our analysis, the conditions can be handled probabilistically. Specifically,
for Gaussian random designs whose covariance matrix \(\Sigma\) satisfies the eigenvalue
condition in \eqref{eq: parameter space}, we can show that, for each fixed
\(\alpha_0>0\), there exists a constant \(c_{\rm RE}>0\), depending only on \(M_1\)
and \(\alpha_0\), such that if \(k_u\log p/n \le c_{\rm RE}\), then
\(\kappa(X,k_u,\alpha_0)\) is bounded away from zero with high probability. This
high-probability lower bound can then be used to choose admissible values of
\(c_\beta\) and \(C_\beta\). A more detailed discussion is deferred to
\Cref{sec: restricted eigenvalue}.

\begin{remark}
Although the estimation error in \eqref{eq: linear estimator} is adaptive to the unknown exact sparsity level $k$, Cai and Guo \cite{cai2018accuracy} showed that the accuracy assessment of the \( \ell_q \)-loss, for \( 1 \le  q \le  2 \), is hard and generally non-adaptive. Consequently, we bound the estimation loss using the upper bound in \eqref{eq: linear estimator}, stated in terms of the sparsity upper bound $k_u$ rather than the unknown exact sparsity level $k$.
While such a bound may seem non-tight, we will later show that it still leads to an optimal test, thereby justifying the use of the $k_u$-based error bound.
\end{remark}

\begin{condition}\label{cdt: sparsity assumption}
There exists a constant \(\gamma\in[0,1/2)\) such that
\(
    k_u \lesssim  p^\gamma.
\)
\end{condition}
Condition \ref{cdt: sparsity assumption} is standard and it places $k_u$ in the sparsity regime commonly assumed when deriving minimax rates in high-dimensional settings \cite{cai2017confidence, cai2018accuracy, bradic2022testability}.

\subsection{Informal theorem map}\label{sec: main results}

We now give an informal summary of the main results. Throughout the paper, we assume without loss of generality that the loading vector \(\xi=(\xi_1,\ldots,\xi_p)^\top\in\mathbb R^p\) is ordered by decreasing magnitude,
\(
    |\xi_1|\ge |\xi_2|\ge \cdots \ge |\xi_p|.
\)
We also write \(k_\xi=\|\xi\|_0\) for the support size of \(\xi\).
For \(t>0\), define the top-\(\lceil t\rceil\) norm of \(\xi\) by
\begin{equation}\label{eq: H definition}
    H(t;\xi)
    :=
    \sqrt{\sum_{j\le \lceil t\rceil \wedge p}\xi_j^2}
    =
    \max_{\substack{A\subseteq[p]\\ |A|\le \lceil t\rceil}}
    \|\xi_A\|_2,
\end{equation}
and $H(0;\xi)=0$.
The main rate statements are summarized in \Cref{tab: informal theorem map}.

\begin{table}[bt]
\centering
\renewcommand{\arraystretch}{1}
\caption{Informal summary of the main separation rates. All comparisons are
understood up to logarithmic factors.}
\label{tab: informal theorem map}
\begin{tabular}
{
>{\centering\arraybackslash}m{0.35\linewidth}
>{\centering\arraybackslash}m{0.55\linewidth}
}
\toprule
Setting & Informal rate statement \\
\midrule
\makecell[c]{Unknown \(\Sigma\)\\
\(k_u \lesssim \sqrt n/\log p\)}
&
\(\displaystyle
    \tau_{\text{adap}}(k_u,k;\xi)
    \asymp_{\log}
    \frac{H(k_u^2\log p;\xi)}{\sqrt n}
\)
\\
\midrule
\makecell[c]{Unknown \(\Sigma\)\\
\(\sqrt n/\log p \ll k_u \lesssim n/\log p\)}
&
\(\displaystyle
    \tau_{\text{com}}(k_u,k;\xi)
    \asymp_{\log}
    H({n}/{\log p};\xi)\frac{k_u\log p}{n}
\)
\\
\midrule
\makecell[c]{
Sparse signed-spiked covariance \(\Sigma\)\\
All sparsity levels: \(k_u\lesssim n/\log p\)}
&
\(\displaystyle
    \tau^{\text{spike}}_{\text{adap}}(k_u,k;\xi)
    \asymp_{\log}
    \frac{H(k_u^2\log p;\xi)}{\sqrt n}
    +
    H(k_u;\xi)\frac{k_u\log p}{n}
\)
\\
\midrule
\makecell[c]{Known diagonal \(\Sigma=\Sigma_0\)\\
All sparsity levels: \(k_u\lesssim n/\log p\)}
&
\(\displaystyle
    \tau_{\text{adap}}(k_u,k,\Sigma_0^{\textrm{diag}};\xi)
    \asymp_{\log}
    \frac{H(k_u^2\log p;\xi)}{\sqrt n}
\)
\\
\bottomrule
\end{tabular}
\begin{minipage}{0.95\linewidth}
\footnotesize\itshape

1. \(\tau_{\text{com}}(k_u,k;\xi)\) denotes the separation rate  achieved by the computationally feasible test and matched by a low-degree lower bound.
It should be interpreted as an evidence-based computational barrier.

2. The formal definition of sparse signed-spiked covariance is given in
\Cref{sec: np hard}.

3. \(\tau_{\text{adap}}(k_u,k,\Sigma_0^{\textrm{diag}};\xi)\) denotes the adaptive
separation distance when \(\Sigma=\Sigma_0\) is known and diagonal.

\end{minipage}
\end{table}

The entries in Table~\ref{tab: informal theorem map} should be read as follows. The first, third, and fourth rows give adaptive separation rates matched by statistical lower bounds up to logarithmic factors. The second row concerns the moderately sparse unknown-covariance regime and should instead be interpreted as a computational separation distance: it is achieved by computationally feasible tests and is supported by computational lower-bound evidence.

The sparse signed-spiked row shows that, under additional covariance structure, the
adaptive separation distance can be characterized up to logarithmic factors.
This rate is no larger than the computational separation distance in the
general unknown-covariance setting. Indeed, since \(k_u\lesssim n/\log p\), we
have \(H(k_u;\xi)\lesssim H(n/\log p;\xi)\). Moreover, in the moderately sparse
regime \(\sqrt n/\log p\ll k_u\), we have
\(n/\log p\lesssim k_u^2\log p\). By the definition of \(H(t;\xi)\) and the
decreasing ordering of \(|\xi_j|\), this implies
\[
    \frac{H(k_u^2\log p;\xi)}{\sqrt{k_u^2\log p}}
    \lesssim
    \frac{H(n/\log p;\xi)}{\sqrt{n/\log p}} .
\]
Consequently, the sparse signed-spiked rate is no larger, up to logarithmic factors,
than the computational separation rate
in the moderately sparse regime $\sqrt{n}/\log p\ll k_u\lesssim n/\log p$. However, the attaining procedure relies on computationally inefficient
covariance estimation, as discussed in \Cref{sec: np hard}. Finally, when the covariance matrix \(\Sigma_0\) is known and diagonal, the separation distance can be further reduced. This is consistent with the intuition that, even under sparse signed-spiked structure, the
covariance matrix must still be estimated, whereas the known-diagonal benchmark
removes covariance uncertainty entirely.

We next relate these results to the known-sparsity minimax rate
\(\tau_{\text{mini}}(k;\xi)\) and discuss adaptivity for
testing~\eqref{eq: hypothesis test}. Consider the ultra-sparse
unknown-covariance setting, corresponding to the first row of
Table~\ref{tab: informal theorem map}. In this setting,
\[
    \tau_{\text{mini}}(k;\xi)
    \asymp_{\log}
    \frac{H(k^2\log p;\xi)}{\sqrt n},
    \qquad
    \tau_{\text{adap}}(k_u,k;\xi)
    \asymp_{\log}
    \frac{H(k_u^2\log p;\xi)}{\sqrt n}.
\]
Thus, testing~\eqref{eq: hypothesis test} is adaptive, in the sense that
unknown \(k\) incurs no additional cost up to logarithmic factors, whenever, for \(k\ll k_u\),
\[
    H(k_u^2\log p;\xi)
    \asymp_{\log}
    H(k^2\log p;\xi).
\]
If \(\xi\) satisfies the regular loading condition
\eqref{eq: tony loading vector}, this condition reduces, up to logarithmic factors, to \(k_\xi=\|\xi\|_0 \lesssim k^2\). In this case, the minimax separation rate simplifies to \(\|\xi\|_2/\sqrt n\), consistent with the adaptivity results in \cite{cai2017confidence,cai2019optimal}.

\section{Adaptive separation distance} \label{sec: minimax rate}
In this section, we investigate the adaptive separation distance for the linear hypothesis testing problem.

Since $\xi$ is nonzero, we have $k_\xi=\|\xi\|_0\ge 1$.
Let $\zeta \in \mathbb{R}$ be the solution to the equation
\begin{equation}\label{eq: equation sol}
    \frac{\sum_{j=1}^{k_\xi} |\xi_j| \exp(-\zeta / \xi_j^2)}
    {\sqrt{\sum_{j=1}^{k_\xi} \xi_j^2 \exp(-\zeta / \xi_j^2)}}
    = \frac{k_u}{2},
    \qquad \text{and set } \lambda = \sqrt{\zeta_+}.
\end{equation}
\Cref{lem: decreasing} guarantees that \eqref{eq: equation sol} admits a unique solution, since the left-hand side is continuous, strictly decreasing, and tends to $+\infty$ as $\zeta \to -\infty$ and to $0$ as $\zeta \to +\infty$.

The first key quantity in our analysis is
\begin{equation}\label{eq: nu1 definition}
    \nu_1=\nu_1(k_u;\xi)
    =
    \lambda k_u
    +
    \sqrt{\sum_{j=1}^{k_\xi}
    \xi_j^2\exp\{-\lambda^2/\xi_j^2\}} .
\end{equation}
The second key quantity is
\begin{equation}\label{eq: nu2 definition}
    \nu_2=\nu_2(k_u;\xi)
    =
    H(k_u;\xi),
\end{equation}
where \(H(t;\xi)\) is defined in \eqref{eq: H definition}.

\subsection{General upper and lower bounds}\label{sec: bounds on separation}
For arbitrary loading vectors,
\Cref{thm: hypothesis} below gives general upper and lower bounds for the adaptive separation distance of testing \eqref{eq: hypothesis test}.

\begin{theorem}\label{thm: hypothesis}
Under \Cref{cdt: sparsity assumption},
there exists some constant $c>0$ such that if $n\ge c\,k_u\log p$,
then for any \(1\le k\le k_u\), the following statements hold:
    \begin{enumerate}
        \item \textit{Upper Bound:} Let $\xi_{p+1}=0$, then
        \begin{equation}\label{eq: upper bound}
            \tau_{\text{adap}}(k_u, k; \xi) \lesssim \min_{0\le  m\le  p}\left(
             H(m\,;\,\xi)
            \left(\frac{1}{\sqrt{n}}+\frac{k_u\log p}{n}\right)
            +|\xi_{m+1}|k_u\sqrt{\frac{\log p}{n}}\right).
        \end{equation}
        \item \textit{Lower Bound:}
        \begin{equation}\label{eq: lower bound}
            \tau_{\text{adap}}(k_u, k; \xi) \gtrsim \nu_1\frac{1}{\sqrt{n}}\vee \nu_2\frac{k_u\log p}{n}.
        \end{equation}
    \end{enumerate}
\end{theorem}

We discuss the implications of these bounds and identify conditions where they match.
The following proposition rewrites the upper bound in an equivalent and more interpretable form.

\begin{proposition}\label{prop: upper bound equivalent}
We write $\xi_{p+1}=0$.
For any \(t\ge  1\), it holds that
\[
    \min_{0\le  m\le  p}
    \left\{
        H(m\,;\,\xi)
        +
        |\xi_{m+1}|\,t
    \right\}
    \asymp
    H(t^2\,;\,\xi).
\]
Consequently, the upper bound in~\eqref{eq: upper bound} is equivalent to
\begin{equation}\label{eq: upper bound equivalent}
\tau_{\text{adap}}(k_u,k\,;\,\xi)
\;\lesssim\;
\begin{cases}
\displaystyle
\frac{1}{\sqrt n}
H(k_u^2\log p\,;\,\xi),
&
\displaystyle
k_u \lesssim \frac{\sqrt n}{\log p},
\\[1.2em]
\displaystyle
\frac{k_u\log p}{n}
H(n/\log p\,;\,\xi),
&
\displaystyle
k_u \gg \frac{\sqrt n}{\log p}.
\end{cases}
\end{equation}
\end{proposition}

\Cref{prop: upper bound equivalent} separates the upper bound into two sparsity regimes:
the ultra-sparse regime \(k_u \lesssim \sqrt n/\log p\) and the moderately sparse regime
\(\sqrt n/\log p \ll k_u \lesssim n/\log p\).
We next discuss the corresponding lower bounds in these two regimes.

\medskip

\textit{Ultra-sparse regime.}
In this regime, the quantity \(\nu_1\) in the lower bound in \Cref{eq: lower bound} plays an important role.
It has a more explicit characterization that relates to the upper bound.

\begin{proposition}\label{prop: nu1-expression}
    Let \( j_1 = \max\left\{j\in[p]: |\xi_j|\ge  \lambda\right\}\)
with the convention that \(j_1=0\) if the set is empty. Then, for \(\nu_1\) defined in~\eqref{eq: nu1 definition},
\[
     \nu_1 \asymp H(j_1\,;\,\xi) + \lambda k_u .
\]
\end{proposition}

In the ultra-sparse regime \(k_u \lesssim \sqrt n/\log p\), the upper bound in \Cref{eq: upper bound} can be bounded by the choice of $m=j_1$ so that
\begin{align*}
    \tau_{\text{adap}}(k_u,k\,;\,\xi)
    &
    \;\lesssim\;
    \frac{1}{\sqrt n}
    H(j_1\,;\,\xi)
    +
    \lambda k_u\sqrt{\frac{\log p}{n}}
    \;\lesssim\;
    \nu_1\sqrt{\frac{\log p}{n}} .
\end{align*}
Therefore, the term $\nu_1/\sqrt{n}$ in the lower bound in \Cref{eq: lower bound} for the adaptive separation distance is sharp up to logarithmic factors, as summarized in the following corollary.

\begin{corollary}\label{cor: nu1-sharp-ultrasparse}
Assume the conditions in \Cref{thm: hypothesis} hold.
In the ultra-sparse regime \(k_u \lesssim \sqrt n/\log p\), it holds for all $1\le k\le  k_u$ that
$$\tau_{\text{adap}}(k_u,k\,;\,\xi)\asymp_{\log} \frac{\nu_1}{\sqrt n} \asymp_{\log} \frac{H(k_u^2\log p\,;\,\xi)}{\sqrt n}
 .$$
If further $k_u^2 \log p \lesssim j_1$, it holds that
$$
\tau_{\text{adap}}(k_u,k\,;\,\xi)\asymp \frac{\nu_1}{\sqrt n},
$$
that is, the logarithmic gap disappears.
\end{corollary}

The closest existing result is \cite[Theorem~4]{cai2019optimal}, which gives a lower bound for the adaptive separation distance in the testing problem~\eqref{eq: hypothesis test}.
However, even in the ultra-sparse regime \(k_u \lesssim \sqrt n/\log p\),
they only establish the sharpness of their lower bound for a few specific loading classes considered there.
By contrast, the lower bound based on \(\nu_1\) applies to arbitrary loading vectors and is tight up to logarithmic factors throughout the ultra-sparse regime.

\medskip

\textit{Moderately sparse regime.}
In this regime, $\nu_1/\sqrt{n}$ can still be sharp in a subclass of loading vectors while the quantity \(\nu_2\) also becomes relevant and could be the dominant term when $\xi$ is light-tailed.
\begin{corollary}
Assume the conditions in \Cref{thm: hypothesis} hold.
Suppose that $k_u \gg \sqrt {n}/{\log p}$ and that the loading vector $\xi$ satisfies that
\begin{equation}\label{eq:sharp-cond-nu1-moderatesparse}
    \nu_1  \gtrsim  \frac{k_u\log p}{\sqrt{n}} H( n/\log p \,;\,\xi).
\end{equation}
Then for all $1\le k\le  k_u$,
$$
\tau_{\text{adap}}(k_u,k\,;\,\xi)\asymp \frac{\nu_1}{\sqrt n}.
$$
\end{corollary}

For regular loading vectors, Appendix~F shows that the dense side \(k_\xi\gtrsim k_u^2\) gives
\[
    \tau_{\text{adap}}(k_u,k\,;\,\xi)
    \asymp_{\log}
    \|\xi\|_\infty k_u\sqrt{\frac{\log p}{n}} .
\]
If, more strongly, \(k_\xi/k_u^2\ge p^c\) for some constant \(c>0\), then
\Cref{eq:sharp-cond-nu1-moderatesparse} holds and the logarithmic equivalence can be strengthened to
\[
    \tau_{\text{adap}}(k_u,k\,;\,\xi)
    \asymp
    \|\xi\|_\infty k_u\sqrt{\frac{\log p}{n}} .
\]

The quantity \(\nu_2\) in the lower bound of \Cref{eq: lower bound} is relevant in the moderately sparse regime.
As revealed by the proof, this term arises from the unknown covariance among covariates.
Furthermore, \Cref{thm: hypothesis design} shows that when the design covariance is known and diagonal, the adaptive separation distance reduces to $\nu_1/\sqrt{n}$ and $\nu_2$ no longer plays a role.

\Cref{cor: sharp-nu2} characterizes an important subclass of loading vectors for which the adaptive separation distance is determined by the $\nu_2$ quantity.

\begin{corollary}\label{cor: sharp-nu2}
Assume the conditions in \Cref{thm: hypothesis} hold.
Suppose that $k_u \gg \sqrt {n}/{\log p}$ and that the loading vector $\xi$ satisfies that
\begin{equation}\label{eq:sharp-cond-nu2-moderatesparse}
    H( k_u \,;\,\xi)
    \asymp
    H( n/\log p \,;\,\xi).
\end{equation}
Then, for any \(1\le k\le  k_u\),
\[
    \tau_{\text{adap}}(k_u,k\,;\,\xi)
    \asymp
    \frac{k_u\log p}{n}\nu_2.
\]
\end{corollary}

The condition in~\Cref{eq:sharp-cond-nu2-moderatesparse} means that the top-$k_u$ subvector of $\xi$ captures nearly as much energy as the enlarged top-$(n/\log p)$ counterpart.
This condition holds, for example, when the loading vector \(\xi\) is sufficiently sparse or when its coordinates decay sufficiently fast.
This includes the sparse-loading case \(k_\xi\lesssim k_u\) considered by \cite{cai2017confidence}. For such loading vectors, the lower bound involving \(\nu_2\) is sharp in the moderately sparse regime.

\medskip

To clarify the scope of our bounds, Appendix~F in the supplementary material provides details on several loading-profile examples.
Except for the regular loading example in Appendix~F.1, existing theory in \cite{cai2017confidence,cai2019optimal} does not apply the other examples.
Dense nonregular profiles in Appendix~F.2 show that the upper and lower bounds in \Cref{thm: hypothesis} can match even when $\xi$ is not regular, whereas the multiscale example in Appendix~F.3 illustrates that a gap can exist in the moderately sparse regime.
Appendix~F.4 considers loadings with random coordinates that have
sub-Weibull tails, including Gaussian and exponential distributions.

For general loading vectors, there may be a gap between the information-theoretic lower bound \eqref{eq: lower bound} and the computationally feasible upper bound \eqref{eq: upper bound} in the moderately sparse regime \( k_u \gg {\sqrt n}/{\log p}\).
\Cref{sec: computational barrier} examines this issue from a computational perspective.

\subsection{Phase diagram for regular loading vectors}\label{sec: examples}

To provide a concrete interpretation of our results, we consider the case in which the loading vector $\xi \in \mathbb{R}^p$ satisfies \eqref{eq: tony loading vector} for some constant $\bar{c}>0$ with sparsity level
$k_\xi = \|\xi\|_0$.
We compare the resulting adaptive separation distances with existing results in \cite{cai2017confidence,cai2019optimal}.

\begin{figure}[bt]
    \centering
    \definecolor{detectgreen}{RGB}{48,140,94}
    \definecolor{gapblue}{RGB}{54,100,177}
    \definecolor{impossiblegray}{RGB}{138,143,150}
    \definecolor{guidegray}{RGB}{70,70,70}
    \tikzset{
        axisline/.style={thick,->},
        outerbox/.style={thick},
        boundary/.style={very thick},
        guideheavy/.style={densely dashed, guidegray, line width=0.55pt},
        point/.style={circle, fill=black, inner sep=1.25pt},
        ticklabel/.style={font=\scriptsize},
        regionlabel/.style={font=\scriptsize\bfseries, align=center},
        smalllabel/.style={font=\tiny, align=center, fill=white, inner sep=1.2pt}
    }
    \begin{tikzpicture}[scale=3.0]
        \begin{scope}
            \fill[detectgreen!26] (0,0) rectangle (1.5,1);
            \fill[impossiblegray!78] (0,0) -- (1,0.5) -- (1.5,0.5) -- (1.5,0) -- cycle;
            \begin{scope}
                \clip (0,0) -- (1,0.5) -- (1.5,0.5) -- (1.5,0) -- cycle;
                \foreach \x in {0.12,0.28,...,1.44} {
                    \foreach \y in {0.08,0.20,0.32,0.44} {
                        \fill[guidegray!45] (\x,\y) circle (0.006);
                    }
                }
            \end{scope}
            \draw[guideheavy] (1,0) -- (1,0.5);
            \draw[guideheavy] (0,0.5) -- (1.5,0.5);
            \node[point] at (1,0.5) {};
            \draw[axisline] (0,0) -- (1.62,0) node[right, font=\scriptsize] {$\gamma_\xi$};
            \draw[axisline] (0,0) -- (0,1.1) node[above, font=\scriptsize] {$\gamma_\tau$};
            \draw[boundary] (0,0) -- (1,0.5) -- (1.5,0.5);
            \draw[outerbox] (0,0) rectangle (1.5,1);
            \node[ticklabel,below] at (0,0) {$0$};
            \node[ticklabel,below] at (1,0) {$2\gamma_u$};
            \node[ticklabel,below] at (1.5,0) {$1$};
            \node[ticklabel,left] at (0,0.5) {$\gamma_u$};
            \node[regionlabel] at (0.53,0.73) {easy};
            \node[regionlabel,text=white] at (1.08,0.18) {statistically\\impossible};
            \node[font=\scriptsize] at (0.75,-0.17) {(a) ultra-sparse};
        \end{scope}

        \begin{scope}[xshift=2.55cm]
            \fill[detectgreen!78] (0,0.125) -- (0.5,0.375) -- (0.75,0.5) -- (1.5,0.5) -- (1.5,1) -- (0,1) -- cycle;
            \fill[gapblue!78] (0.5,0.375) -- (0.75,0.5) -- (1,0.5) -- (0.75,0.375) -- cycle;
            \begin{scope}
                \clip (0.5,0.375) -- (0.75,0.5) -- (1,0.5) -- (0.75,0.375) -- cycle;
                \foreach \x in {0.44,0.52,...,1.00} {
                    \draw[white, line width=0.45pt] (\x,0.34) -- ++(0.36,0.26);
                }
            \end{scope}
            \fill[impossiblegray!80] (0,0) -- (0,0.125) -- (0.5,0.375) -- (0.75,0.375) -- (1,0.5) -- (1.5,0.5) -- (1.5,0) -- cycle;
            \begin{scope}
                \clip (0,0) -- (0,0.125) -- (0.5,0.375) -- (0.75,0.375) -- (1,0.5) -- (1.5,0.5) -- (1.5,0) -- cycle;
                \foreach \x in {0.10,0.26,...,1.42} {
                    \foreach \y in {0.07,0.19,0.31,0.43} {
                        \fill[guidegray!45] (\x,\y) circle (0.006);
                    }
                }
            \end{scope}
            \draw[guideheavy] (0,0.125) -- (0.18,0.125);
            \draw[guideheavy] (0.5,0) -- (0.5,0.375) -- (0,0.375);
            \draw[guideheavy] (1,0) -- (1,0.5);
            \node[point] at (0.5,0.375) {};
            \node[point] at (0.75,0.5) {};
            \node[point] at (1,0.5) {};
            \draw[axisline] (0,0) -- (1.62,0) node[right, font=\scriptsize] {$\gamma_\xi$};
            \draw[axisline] (0,0) -- (0,1.1) node[above, font=\scriptsize] {$\gamma_\tau$};
            \draw[boundary] (0,0.125) -- (0.5,0.375);
            \draw[boundary] (0.5,0.375) -- (0.75,0.5);
            \draw[boundary] (0.75,0.5) -- (1,0.5);
            \draw[boundary] (0.5,0.375) -- (0.75,0.375);
            \draw[boundary] (0.75,0.375) -- (1,0.5);
            \draw[boundary] (1,0.5) -- (1.5,0.5);
            \draw[outerbox] (0,0) rectangle (1.5,1);
            \node[ticklabel,below] at (0,0) {$0$};
            \node[ticklabel,below] at (0.5,0) {$\gamma_u$};
            \node[ticklabel,below] at (1,0) {$2\gamma_u$};
            \node[ticklabel,below] at (1.5,0) {$1$};
            \node[ticklabel,left] at (0,0.125) {$\gamma_u-\gamma_n/2$};
            \node[ticklabel,left] at (0,0.375) {$(3\gamma_u-\gamma_n)/2$};
            \node[ticklabel,left] at (0,1) {$1$};
            \node[regionlabel,text=white] at (0.58,0.78) {easy};
            \node[smalllabel,above left=1pt] at (0.75,0.5) {$(\gamma_n,\gamma_u)$};
            \node[regionlabel,text=white] at (0.77,0.17) {statistically\\impossible};
            \node[font=\scriptsize] at (0.75,-0.17) {(b) moderately sparse};
        \end{scope}
    \end{tikzpicture}
    \caption{Phase diagram for the loading vector in \eqref{eq: tony loading vector} with sparsity level $k_\xi$.
    We parameterize $k_u = p^{\gamma_u}$, $k_\xi = p^{\gamma_\xi}$, $n = p^{\gamma_n}$, and
    {
    the rescaled alternative shift $\sqrt{n}\tau=\|\xi\|_\infty p^{\gamma_\tau}$.
    }
    Left panel (a): Ultra-sparse regime, $k_u\lesssim \sqrt{n}/\log p$. Right panel (b): Moderately sparse regime, $\sqrt{n}/\log p \ll k_u \lesssim n/\log p$.}
    \label{fig: tony loading}
\end{figure}
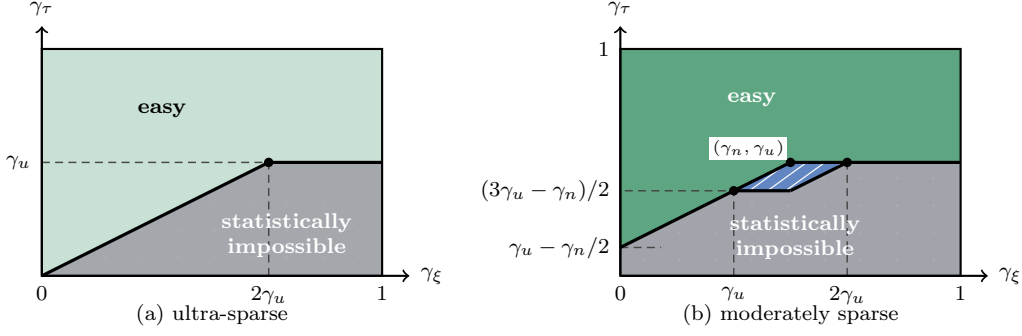

Figure~\ref{fig: tony loading}
presents the upper and lower bounds for adaptive separation distance for testing $H_0: \xi^\top \beta = t_0$ versus $ H_1: |\xi^\top \beta - t_0| = \tau$ under various $k_\xi$.
We parametrize the problem as
\[
k_\xi = p^{\gamma_\xi}, \qquad
k_u = p^{\gamma_u}, \qquad
n = p^{\gamma_n},\text{ and }\tau \asymp\|\xi\|_\infty \, p^{\gamma_\tau - \gamma_n/2}.
\]
The parametrization of $\tau$ can be interpreted as a rescaled version of the
alternative shift, where the exponent $\gamma_\tau$ characterizes the signal strength after normalizing for the common loading magnitude $\|\xi\|_\infty$ and the classical parametric rate $n^{-1/2}$.

We first explain the left panel, corresponding to the \emph{ultra-sparse regime}
$k_u \lesssim \sqrt{n}/\log p$.
In this regime, the statistical boundary is sharp and completely characterizes the phase
transition between distinguishability and indistinguishability.
Specifically, the separation distance is given by
\begin{equation}\label{eq: separation under regular and ultrasparse}
    \tau_{\text{adap}}(k_u,k\,;\,\xi)
\asymp
\begin{cases}
\|\xi\|_2/\sqrt{n}, & \gamma_\xi \le 2\gamma_u,\\[0.4em]
\|\xi\|_\infty k_u \sqrt{\log p/n}, & \gamma_\xi > 2\gamma_u;
\end{cases}
\end{equation}
at the boundary $\gamma_{\xi}=2 \gamma_u$, the expression should be read up to logarithmic factors.

Notably, when $\gamma_\xi <  2\gamma_u$, $\tau_{\text{adap}}(k_u,k\,;\,\xi)$ does not depend on $k_u$,
indicating that full adaptivity to the unknown exact sparsity is achievable.
This is in agreement with the existing results by \cite{cai2019optimal}.

The right panel depicts the \emph{moderately sparse regime}
$\sqrt{n}/\log p \ll k_u \lesssim n/\log p$, where the geometry becomes more intricate.
In this regime, the phase diagram should be interpreted through two curves:
an information-theoretic lower-bound curve, below which detection is statistically impossible,
and a computationally feasible upper-bound curve, above which there exist computationally efficient procedures that succeed.
When the two curves coincide, their common value gives the adaptive separation distance.
When they differ, the region between them is not resolved by the statistical bounds alone.
Accordingly, there are three different regimes:
\begin{itemize}
    \item When the loading sparsity satisfies $\gamma_\xi \le \gamma_u$ or $\gamma_\xi \ge 2\gamma_u$, the two curves coincide, and the adaptive separation distance can be explicitly characterized as
\begin{equation}\label{eq: separation under regular and moderatesparse}
\tau_{\text{adap}}(k_u,k\,;\,\xi)
\asymp
\begin{cases}
\|\xi\|_2 \, k_u \log p / n, & \gamma_\xi \le \gamma_u,\\[0.4em]
\|\xi\|_\infty k_u \sqrt{\log p/n}, & \gamma_\xi \ge 2\gamma_u.
\end{cases}
\end{equation}
In these loading settings, the boundary is achievable by computationally efficient tests.

Comparing with \Cref{eq: separation under regular and ultrasparse}, the only difference lies in the sparse-loading regime $\gamma_\xi\le  \gamma_u$, where the rate is increased by a factor of $k_u\log p /\sqrt{n} \gg 1$. This is driven by the $\nu_2$ quantity resulting from the unknown design covariance.

Accordingly, the statistical boundary now has a nonzero intercept at
$(0,\gamma_u-\gamma_n/2)$.
This intercept corresponds to a separation distance of order $k_u\log p/n$ when the loading
vector $\xi$ has only a constant number of nonzero entries.
Thus, even for very sparse $\xi$, the parametric rate $n^{-1/2}$ is no longer attainable,
and adaptivity with respect to $k_u$ necessarily fails.

\item When the loading sparsity satisfies $\gamma_u < \gamma_\xi < 2\gamma_u$, the available information-theoretic lower bound and the computationally feasible upper
bound do not match.
More precisely,
\begin{equation}\label{eq: gap region}
\|\xi\|_\infty \left[\frac{k_u^{3/2}\log p}{n}+\sqrt{\frac{k_\xi}{n}}\right]
\;\lesssim\;
\tau_{\text{adap}}(k_u,k\,;\,\xi)
\;\lesssim\;
\|\xi\|_\infty
\sqrt{\frac{n}{\log p}\wedge k_\xi}\,
\frac{k_u\log p}{n}.
\end{equation}
The middle shaded region in the right panel corresponds to this non-negligible gap between these two curves.
\end{itemize}

The characterization summarized in \Cref{fig: tony loading} is consistent with the results of \cite{cai2017confidence}, while the intermediate loading-sparsity setting on the right panel remains unresolved in the existing literature.
In \Cref{sec: computational barrier}, we show that the upper curve is matched by a low-degree lower bound up to logarithmic factors.
Accordingly, under the low-degree heuristic, the shaded region is interpreted as a
computationally hard region.

\subsection{Computationally feasible upper bound via loading decomposition}\label{sec: upper}
The upper bound in \Cref{thm: hypothesis} is proven by inverting confidence intervals with desired lengths. Suppose that
\[
    \CI_{\alpha_1}(Z)
    =
    [\hat L(Z\,;\,\xi)-r(Z\,;\,\xi),\,\hat L(Z\,;\,\xi)+r(Z\,;\,\xi)]
\]
is a \((1-\alpha_1)\)-level confidence interval for
\(L(\beta\,;\,\xi)=\xi^\top\beta\), and that
$r(Z\,;\,\xi)\le  \tilde r(\xi)$
with probability at least \(1-\alpha_2\). Then the test
\begin{equation}\label{eq: inverted-ci-test}
    \psi(Z)
    =
    \ind\{|t_0-\hat L(Z\,;\,\xi)|>\tilde r(\xi)\}
\end{equation}
belongs to \(\Psi_{\alpha_1+\alpha_2}(\Theta(k_u;\xi,t_0))\) and has power at least \(1-(\alpha_1+\alpha_2)\) over \(\Theta_{\pm 2\tilde r(\xi)}(k;\xi,t_0)\).

Taking \(\alpha'=\min\{\alpha,\eta\}\), it is sufficient to construct a
\((1-\alpha'/2)\)-level confidence interval for \(L(\beta\,;\,\xi)=\xi^\top\beta\) whose radius is bounded by \(\tilde r(\xi)\) with probability at least \(1-\alpha'/2\). The resulting inverted test has Type I error at most \(\alpha\) and power at least \(1-\eta\) over alternatives separated from the null by \(2\tilde r(\xi)\).

To obtain a confidence interval that is effectively tailored to the particular loading profile of any $\xi$, we propose a mixed construction that generalizes two existing approaches:
\begin{itemize}
    \item \textit{Plug-in confidence interval.}
    By \Cref{cdt: sparsity assumption} and the discussion following \Cref{cdt: linear estimator,cdt: linear variance}, we can construct an interval centered at \(
    \hat L_{\mathrm{pi}}(Z\,;\,\xi)=\xi^\top\hat\beta\) with radius bounded with probability tending to 1 by
    \begin{equation}\label{eq: plug in length}
    \tilde r_{\mathrm{pi}}(\xi)
    \asymp
    \sigma \|\xi\|_\infty  k_u\sqrt{\frac{\log p}{n}} .
    \end{equation}

    \item \textit{Debiased confidence interval.}
The debiased interval is centered at
\begin{equation}\label{eq: debiased center}
    \hat L_{\mathrm{db}}(Z\,;\,\xi)
    =
    \xi^\top\hat\beta
    +
    \hat u^\top \frac{1}{n}X^\top(Y-X\hat\beta),
\end{equation}
where
\begin{equation}\label{eq: projection vector}
    \hat u
    =
    \argmin_{u\in\mathbb R^p}
    \left\{
        u^\top\hat\Sigma u:
        \|\hat\Sigma u-\xi\|_\infty
        \le
        C_\xi\|\xi\|_2\sqrt{\frac{\log p}{n}}
    \right\},
\end{equation}
\(\hat\Sigma=n^{-1}X^\top X\), and \(C_\xi>0\) is sufficiently large. If the feasible set in~\eqref{eq: projection vector} is empty, we set \(\hat u=0\); by \Cref{lem: feasibility of optimization}, this event has vanishing probability for some large enough $C_\xi$. The radius of the debiased interval admits the high-probability deterministic bound
\begin{equation}\label{eq: debiased length}
    \tilde r_{\mathrm{db}}(\xi)
    \asymp
   \sigma   \|\xi\|_2
    \left(
        \frac{1}{\sqrt n}
        +
        k_u\frac{\log p}{n}
    \right).
\end{equation}
\end{itemize}

Comparing \eqref{eq: plug in length} and \eqref{eq: debiased length}, the plug-in confidence interval is effective when \(\|\xi\|_\infty\) is small, whereas the debiased confidence interval is effective when \(\|\xi\|_2\) is small.
To combine these complementary advantages, we decompose the loading vector according to coordinate magnitude. For any \(0\le  m\le  p\), write \(\xi=\xi^{(1)}+\xi^{(2)}\), where
\[
    \xi^{(1)}
    =
    (\xi_1,\ldots,\xi_m,0,\ldots,0)^\top,
    \qquad
    \xi^{(2)}
    =
    (0,\ldots,0,\xi_{m+1},\ldots,\xi_p)^\top .
\]
At the level of \(1-\alpha'/4\), we construct the debiased confidence interval for \((\xi^{(1)})^\top\beta\), which corresponds to the \(m\) largest coordinates of \(\xi\) in absolute value, and the plug-in confidence interval for \((\xi^{(2)})^\top\beta\), which corresponds to the remaining coordinates.

We call the Minkowski sum of these two intervals the \textit{mixed confidence interval}, which has level \(1-\alpha'/2\), and the corresponding inverted test the \textit{mixed test}.

By \Cref{eq: plug in length,eq: debiased length}, the length of the mixed confidence interval is at the scale of
\begin{equation}\label{eq:upper-length-mixed}
        \sigma\sqrt{\sum_{j\le  m}\xi_j^2}
    \left(
        \frac{1}{\sqrt n}
        +
        k_u\frac{\log p}{n}
    \right)
    +
   \sigma |\xi_{m+1}|k_u\sqrt{\frac{\log p}{n}}.
\end{equation}

The choices \(m=0\) and \(m=p\) recover the plug-in confidence interval and the debiased confidence interval, respectively.
Optimizing \Cref{eq:upper-length-mixed} over \(m\) yields the upper bound in \Cref{thm: hypothesis}.
Moreover, \Cref{prop: upper bound equivalent} shows that the rate-optimal cutoff can be taken as
\begin{equation}\label{eq: optimal cutoff}
m_*
=
\begin{cases}
    \lceil k_u^2\log p\rceil ~ \wedge ~p ,
    & k_u\lesssim \sqrt n/\log p ,\\
    \lceil n/\log p\rceil ~ \wedge ~p ,
    & k_u\gg \sqrt n/\log p.
\end{cases}
\end{equation}

\subsection{Proof idea of lower bound}\label{sec: lower}

The lower bound in \Cref{thm: hypothesis} is obtained from two distinct
least-favorable prior constructions. Both constructions are governed by the magnitude profile of \(\xi\), but they exploit different components of the model to turn this profile into
a testing difficulty.
\begin{itemize}
    \item The first construction yields the term $\nu_1/\sqrt{n}$.
    This construction uses the full magnitude profile of $\xi$ to calibrate the local separation that remains hard even when the design covariance is fixed at $\Sigma=\mathbf{I}_p$.
\item
The second construction yields the term
\(\nu_2 k_u\log p/n\).
Unlike the first term, this term exploits the unknown-covariance component of the model, because the construction couples the leading coordinates of \(\xi\) with a sparse perturbation of \(\Sigma\).

\end{itemize}

Both constructions are analyzed through Le Cam's method \citep{tsybakov2009introduction}.
Since Type I error must be controlled over the enlarged null space \(\Theta(k_u;\xi,t_0)\), it is natural to use a mixture-over-null versus point-alternative comparison; see, for example, \cite{cai2019optimal}.
Concretely, we construct a prior \(\pi_1\) supported on \(\Theta(k_u;\xi,t_0)\) and compare the induced mixture distribution
\(
    \mathbb P^n_{\pi_1}
    =
    \int \mathbb P^n_\theta\,\pi_1({\rm d}\theta)
\)
with the distribution \(\mathbb P^n_{\theta_\star}\) generated by a fixed alternative point
\(\theta_\star \in \Theta_{\pm\tau}(k;\xi,t_0)\).
If their total variation
\(
    {\rm TV}\bigl(\mathbb P^n_{\pi_1},\mathbb P^n_{\theta_\star}\bigr)
\)
is small, then every test with uniform Type I error control over
$\Theta(k_u;\xi,t_0)$ has limited power at $\theta_\star$.

The main technical step is the construction of $\pi_1$ for arbitrary
loading vectors.
The prior must
simultaneously satisfy the sparsity constraint, the exact null constraint, and the eigenvalue restrictions on \(\Sigma\), while keeping ${\rm TV}\bigl(\mathbb P^n_{\pi_1},\mathbb P^n_{\theta_\star}\bigr)$ small.

Below we outline the key ideas behind our two constructions.
For exposition, we describe the constructions after a deterministic translation that cancels out $t_0$. Accordingly, the fixed alternative has
\(\beta_\star=0\) and the null prior is supported on the space with $L(\beta\,;\,\xi)=\tau$. This translation preserves the relevant total variation and \(\chi^2\)-divergence, so the argument applies to a general \(t_0\).

\medskip

\textit{Lower bound via the quantity \(\nu_1\).} To capture the term involving \(\nu_1\), we use a random-sparsity prior, following the ideas of \cite{chhor2024sparse,xie2025minimax}.
    Under such a prior, the coordinates are generated independently: the \(j\)th coordinate is assigned the value \(\gamma_j\) with probability \(q_j\), and is set to zero with probability \(1-q_j\). The sequences \(\{q_j\}\) and \(\{\gamma_j\}\) are calibrated according to the magnitude profile of the loading vector \(\xi\).

    This calibration is essential for general loading vectors. The classical least favorable prior construction used in \cite{cai2017confidence,cai2019optimal} selects a support uniformly and assigns equal magnitudes on the selected coordinates. It is sharp for sufficiently homogeneous loading vectors, such as those satisfying \eqref{eq: tony loading vector}, but can be suboptimal for heterogeneous \(\xi\).
    The random-sparsity prior instead adapts \(\{q_j\}\) and \(\{\gamma_j\}\) to the magnitude profile of \(\xi\), which yields the lower-bound term involving \(\nu_1\).

    A further adjustment is needed because the vanilla random-sparsity construction induces a random value of the target functional \(L(\beta\,;\,\xi)\).
    To obtain a valid null prior that satisfies \(L(\beta\,;\,\xi)=\tau\), we modify the construction through a scalar parameter to ensure that every draw satisfies the linear constraint exactly,
    while preserving the statistical closeness between the induced null and alternative distributions.
    This adjustment step is specific to the current testing problem and does not follow directly from the constructions in \cite{chhor2024sparse,xie2025minimax}.

\medskip
    \textit{Lower bound associated with \(\nu_2\).}
    To obtain the term involving \(\nu_2\), we exploit the uncertainty in the design covariance to construct random \(\Sigma\) in the null prior.
The idea is to balance two goals:
(1) align the directions of $\beta$ and $\xi$ to produce a large shift in $L(\beta\,;\,\xi)$;
(2) adjust $\Sigma$ accordingly so that the induced mixture distribution of $(Y,X)$ stays close to the alternative distribution with \(\Sigma=\mathbf I_p\) and $\beta=\mathbf{0}$.

Let \(p_1=\lfloor k_u/4\rfloor\).
The null prior constructs a block-structured covariance matrix
\begin{equation}\label{eq: covariance construction}
    \Sigma
    =
    \begin{pmatrix}
    \mathbf I_{p_1\times p_1}
    &
    \alphab_1\alphab_2^\top
    \\
    \alphab_2\alphab_1^\top
    &
    \mathbf I_{(p-p_1)\times(p-p_1)}
    \end{pmatrix},
\end{equation}
where the vector \(\alphab_1\in\mathbb R^{p_1}\) is deterministic with entries
\[
    (\alphab_1)_j
    =
    -\frac{\xi_j}{\sqrt{\sum_{i=1}^{p_1}\xi_i^2}},
    \qquad j\in[p_1],
\]
and \(\alphab_2\in\mathbb R^{p-p_1}\) is sampled with a uniform random support $S$ of size \(p_1\) and nonzero entries
\[
    (\alphab_2)_j
    \asymp
    \operatorname{sign}(\xi_{p_1+j})\sqrt{\frac{\log p}{n}},
    \qquad j\in S.
\]
The prior also chooses $\beta=\Sigma^{-1}(0, \kappa \alphab_2^\top)^\top$ where \(\kappa=\kappa(b)\in[0,1]\) is chosen to enforce the translated null constraint $L(\beta \,;\,\xi)=\tau$.
The leading term of $L(\beta \,;\,\xi)$ is of order
\[
    - \kappa \xi_{[p_1]}^\top \alphab_1\,\|\alphab_2\|_2^2
    = \kappa H(p_1;\xi)\|\alphab_2\|_2^2,
\]
whereas $\xi_{p_1+S}^\top \beta_{S}$ is nonnegative.
Since \(p_1\asymp k_u\), we have \(H(p_1;\xi)\asymp H(k_u;\xi)\) by the monotone ordering of the coordinates.
Furthermore, since \(\|\alphab_2\|_2^2\asymp k_u\log p/n\), the null constraint can be satisfied if $\tau$ is at the scale of
\[
    H(k_u;\xi)\frac{k_u\log p}{n}
    =
    \nu_2(k_u;\xi)\frac{k_u\log p}{n}.
\]
Meanwhile, the random support of \(\alphab_2\) keeps the mixture statistically close to the fixed alternative,  as the resulting \(\chi^2\)-divergence is controlled by a hypergeometric overlap bound.

The particular forms of $\Sigma$ and $\beta$ are crucial for this tractable analysis.
Furthermore, the control on the $\chi^2$-divergence reveals an interesting connection under this prior construction:
the resulting functional testing comparison is as hard as detecting the sparse rank-two covariance perturbation, because the $\chi^2$-divergence in the former problem can be bounded by the one associated with the latter problem.
This connection also foreshadows our computational-hardness argument in \Cref{sec: computational barrier}, where a closely related sparse covariance perturbation is used to connect the testing problem with computationally hard sparse covariance detection problems.

\section{Computational Barriers in the Moderately Sparse Regime}
\label{sec: computational barrier}

In this section, we study computational barriers for testing~\eqref{eq: hypothesis test} in the moderately sparse regime \({\sqrt n}/{\log p} \ll k_u \lesssim {n}/{\log p} \).
The discussion following \Cref{cor: sharp-nu2} shows that, in this regime, the upper and lower bounds in Theorem~\ref{thm: hypothesis} need not coincide for general loading vectors.
Our results provide two forms of computational lower-bound evidence that match the upper bound: a direct low-degree lower bound for general loadings and a reduction from sparse CCA for flat sparse loadings.

\subsection{Frameworks for computational-barrier evidence}
\label{sec: computational preliminaries}

We introduce two frameworks for establishing evidence for computational barriers.
The first is the low-degree polynomial method, which yields lower bounds against low-degree polynomial algorithms. The second is polynomial-time reduction, which transfers hardness from a conjecturally hard source problem to the testing problem~\eqref{eq: hypothesis test}.

\textit{Low-degree polynomial framework.}
The low-degree framework has been successfully applied to a wide range of high-dimensional testing problems, including sparse PCA, tensor PCA, sparse CCA, graphon estimation, and independent component analysis; see \cite{hopkins2018statistical,kunisky2019notes,laha2023support,luo2024computational,auddy2025large} and references therein.

Let \(\mathbb R[\mathcal Z]_{\le  D}\) denote the space of multivariate polynomials in the observed data \(\mathcal Z\) of degree at most \(D\). Following \cite[Definition~1.8]{bandeira2022franz}, we use the following notion of weak separation.

\begin{definition}\label{def: separation}
We say that \(f\in\mathbb R[\mathcal Z]_{\le  D}\) \emph{weakly separates} two distributions \(\mathbb Q_1\) and \(\mathbb Q_2\) if, as \(n\to\infty\),
\[
\sqrt{
\max\left\{
\Var_{\mathbb Q_1}(f(\mathcal Z)),
\Var_{\mathbb Q_2}(f(\mathcal Z))
\right\}
}
=
O\left(
\left|
\mathbb E_{\mathbb Q_1}f(\mathcal Z)
-
\mathbb E_{\mathbb Q_2}f(\mathcal Z)
\right|
\right).
\]
\end{definition}

Accordingly, no degree-\(D\) polynomial weakly separates two parameter spaces \(\Theta_1\) and \(\Theta_2\) if there exist priors \(\pi_1\) and \(\pi_2\), supported on \(\Theta_1\) and \(\Theta_2\), respectively, such that no \(f\in\mathbb R[\mathcal Z]_{\le  D}\) weakly separates the induced mixture distributions
\[
    \mathbb P_{\pi_1}^n
    =
    \int \mathbb P_\theta^n\,\pi_1({\rm d}\theta),
    \qquad
    \mathbb P_{\pi_2}^n
    =
    \int \mathbb P_\theta^n\,\pi_2({\rm d}\theta).
\]

The following standard criterion reduces low-degree hardness to bounding the low-degree likelihood-ratio norm.

\begin{proposition}[{\cite[Proposition~6.2]{bandeira2022franz}}]
\label{prop:low-degree-hardness}
Let \(\mathbb Q_1\) and \(\mathbb Q_2\) be probability measures with \(\mathbb Q_1\ll \mathbb Q_2\). Define
\[
    L=\frac{{\rm d}\mathbb Q_1}{{\rm d}\mathbb Q_2},
    \qquad
    \mathrm{LD}(D)=\|L^{\le  D}\|_{L^2(\mathbb Q_2)}^2,
\]
where \(L^{\le  D}\) is the orthogonal projection of \(L\) onto \(\mathbb R[\mathcal Z]_{\le  D}\) in \(L^2(\mathbb Q_2)\). If
\[
    \mathrm{LD}(D)=1+o(1),
\]
then no degree-\(D\) polynomial weakly separates \(\mathbb Q_1\) and \(\mathbb Q_2\).
\end{proposition}

When \(D=\infty\), we have
\(
    \mathrm{LD}(\infty)
    =
    \mathbb E_{\mathbb Q_2}[L^2]
    =
    1+\chi^2(\mathbb Q_1\|\mathbb Q_2)
\).
Thus, \(\mathrm{LD}(D)=1+o(1)\) is a computational analogue of the usual \(\chi^2\)-based information-theoretic indistinguishability condition. In particular, it may hold even when the full \(\chi^2\)-divergence diverges. It has been conjectured that, for many average-case high-dimensional testing problems, low-degree polynomials of degree \(D=O(\log p)\) capture the limits of polynomial-time computation \cite{hopkins2018statistical,kunisky2019notes}. Although this conjecture is not universal \cite{buhai2025quasi,wein2025computational}, low-degree lower bounds provide strong evidence of computational hardness.

\textit{Reduction from sparse CCA.}
The second approach is based on polynomial-time reductions. In this section, we use sparse canonical correlation analysis as the source problem, since sparse CCA is widely conjectured to exhibit intrinsic computational barriers \cite{gao2015minimax,gao2017sparse,laha2023support}. The role of the reduction is to show that, for a special class of loading vectors, an efficient algorithm for~\eqref{eq: hypothesis test} would imply an efficient algorithm for an appropriately parameterized sparse CCA detection problem. The precise sparse CCA model and the reduction are given in \Cref{sec: reduction from scca}.

\subsection{Low-degree lower bound}
\label{sec: low-degree}

We first establish a direct low-degree lower bound for general loading vectors. The construction is motivated by the information-theoretic lower bound associated with \(\nu_2\), where the difficulty comes from uncertainty in the design covariance matrix \(\Sigma\). Here, the same covariance-perturbation mechanism in \eqref{eq: covariance construction} is calibrated to the polynomial degree \(D\), leading to the following low-degree lower bound.

\begin{theorem}\label{thm: computational lower bound}
Suppose \Cref{cdt: sparsity assumption} holds,
\({\sqrt n}/{\log p} \ll k_u \lesssim {n}/{\log p} \),
and \(D\lesssim p\).
Let
\[
k_{\mathrm{eff}}
=
\left\lfloor
\frac{n}{\log p}
\wedge
\frac{k_u^2}{D\log p}
\right\rfloor,
\qquad
\nu_3
=
H(k_{\mathrm{eff}};\xi).
\]
Then for any $1\le k\leq k_u$, no degree-\(D\) polynomial weakly separates
\(\Theta(k_u;\xi,t_0)\) and \(\Theta_{\pm\tau}(k;\xi,t_0)\) when
\[
    \tau
    =
    c\,\nu_3\,\frac{k_u\log p}{n},
\]
where \(c>0\) is a sufficiently small constant.
\end{theorem}

Theorem~\ref{thm: computational lower bound} can be viewed as the computational analogue of the lower bound associated with \(\nu_2\). Taking \(D=O(\log p)\), we have
\[
    k_{\mathrm{eff}}
    \asymp_{\log}
    \frac{n}{\log p}
    \wedge
    k_u^2.
\]
In the moderately sparse regime \(k_u\gg \sqrt n/\log p\), the lower bound in Theorem~\ref{thm: computational lower bound} matches the computationally feasible upper bound in~\eqref{eq: upper bound equivalent} up to logarithmic factors. Moreover, when \(k_u\gtrsim \sqrt{n\log p}\), the lower bound in Theorem~\ref{thm: computational lower bound} matches the upper bound in Theorem~\ref{thm: hypothesis} based on \Cref{prop: upper bound equivalent}.

To prove \Cref{thm: computational lower bound}, we construct a hidden covariance perturbation analogously to the one in \eqref{eq: covariance construction} for proving the information-theoretic lower bound based on \(\nu_2\).
We briefly outline the argument for bounding \(\mathrm{LD}(D)\), with full details deferred to \Cref{sec: proof computational lower bound}.
Let \(\pi_2\) be a point mass at a carefully chosen alternative \(\theta^*\), and let \(\pi_1\) be a prior supported on the null space with the covariance perturbation.
For
\(
    L_\theta
    =
    {{\rm d}\mathbb P_\theta^n}/{{\rm d}\mathbb P_{\theta^*}^n}
\),
linearity of the low-degree projection gives
\[
\mathrm{LD}(D)
=
\left\|
\left(\mathbb E_{\theta\sim\pi_1}L_\theta\right)^{\le  D}
\right\|_{L^2(\mathbb P_{\theta^*}^n)}^2
=
\mathbb E_{(\theta_1,\theta_2)\sim\pi_1^{\otimes2}}
\mathbb E_{\mathbb P_{\theta^*}^n}
\left[
    L_{\theta_1}^{\le  D}
    L_{\theta_2}^{\le  D}
\right].
\]

The key step is to decompose the last expectation according to a high-probability event \(A\) under \(\pi_1^{\otimes2}\).
On \(A\), the two hidden covariance perturbations have a tractable alignment structure, which allows us to extend the proof of Proposition~3.6 in \cite{bandeira2022franz} and to show that
\[
\mathbb E
\left[
    L_{\theta_1}^{\le  D}
    L_{\theta_2}^{\le  D}
    \mathbf 1_A
\right]
\le
\mathbb E
\left[
    L_{\theta_1}
    L_{\theta_2}
    \mathbf 1_A
\right];
\]
see \Cref{lem: low degree integral}.
On \(A^c\), we combine a uniform bound on
\(\|L_\theta^{\le  D}\|_{L^2(\mathbb P_{\theta^*}^n)}\) with the small probability of \(A^c\). These two bounds imply
\[
    \mathrm{LD}(D)=1+o(1),
\]
which proves Theorem~\ref{thm: computational lower bound} by Proposition~\ref{prop:low-degree-hardness}.

This argument is not a standard low-degree calculation. The null prior must satisfy both the sparsity constraint and the scalar constraint \(\xi^\top\beta=t_0\). Moreover, the proof establishes \(\mathrm{LD}(D)=1+o(1)\) even though the full \(\chi^2\)-divergence may diverge. The event decomposition isolates the covariance-perturbation pairs that contribute to the low-degree likelihood-ratio norm and is the key step behind the weak-separation lower bound.

\subsection{Reduction from sparse CCA}
\label{sec: reduction from scca}

We next give a complementary hardness result through a polynomial-time reduction. This result applies to a special class of loading vectors and gives additional evidence that the low-degree lower bound in \Cref{sec: low-degree} reflects a genuine computational limitation.

We restrict attention to loading vectors \(\xi\) with sparsity
\(k_\xi=\|\xi\|_0\) and unit nonzero entries. After relabeling coordinates, we may take \(\xi_j=1\) for \(1\le j\le k_\xi\) and \(\xi_j=0\) for \(j>k_\xi\).
Thus, \eqref{eq: hypothesis test} reduces to
\begin{equation}\label{eq: falt testing problem}
    H_0:\ \sum_{j=1}^{k_\xi}\beta_j=t_0.
\end{equation}
This restriction is natural for the reduction because sparse CCA hardness is typically formulated for equal-magnitude sparse vectors \cite{gao2017sparse,laha2023support}. It is also the setting where the remaining gap between the statistical lower and upper bounds is most transparent.

As shown in \Cref{sec: examples}, the adaptive separation distance is already characterized outside the intermediate regime
\(
    \sqrt{n}/\log p\ll k_u \ll k_\xi \ll k_u^2.
\)
For the equal-magnitude loading vectors considered here, the computational rate suggested by the feasible upper bound is
\begin{equation}\label{eq: computational upper tony}
    H(\frac{n}{\log p};\xi)\frac{k_u\log p}{n}=\sqrt{\frac{n}{\log p}\wedge k_\xi}
    \frac{k_u\log p}{n},
\end{equation}
which is increasing in \(k_\xi\) until it saturates at
\(k_\xi\asymp n/\log p\).
Here we focus on the range of $k_\xi$ before the saturation:
\[
   \sqrt n /\log p\ll k_u \ll k_\xi \lesssim {n}/{\log p}.
\]

In the following, we define the source and target problems used in the reduction. Let
\({\rm LT}(n,k_u,k,k_\xi,p,\tau)\) denote the linear testing problem
\[
    H_0:\theta\in\Theta(k_u;\xi,t_0)
    \qquad\text{vs.}\qquad
    H_1:\theta\in\Theta_{\pm\tau}(k;\xi,t_0),
\]
where \(\xi\) has \(k_\xi\) nonzero entries, all equal to one.
The total error probability of a test is defined as the sum of its size and its worst-case Type II error probability.
Let
\({\rm SCCA}(n,s,p_1,p_2,\lambda)\) denote the sparse CCA detection problem
\begin{equation}\label{eq: sparse cca detection}
H_0:
\begin{pmatrix}
U_1\\
U_2
\end{pmatrix}
\sim
\mathcal N(0,\mathbf I_{p_1+p_2})^{\otimes n}
\quad
\text{vs.}
\quad
H_1:
\begin{pmatrix}
U_1\\
U_2
\end{pmatrix}
\sim
\mathcal N\left(
0,
\begin{pmatrix}
\mathbf I_{p_1}
&
\lambda\alphab_1\alphab_2^\top
\\
\lambda\alphab_2\alphab_1^\top
&
\mathbf I_{p_2}
\end{pmatrix}
\right)^{\otimes n},
\end{equation}
where \(\alphab_1\) and \(\alphab_2\) are drawn uniformly from
\[
    \mathcal V_{s,p}
    =
    \left\{
    v\in\mathbb R^p:
    \|v\|_0=s,\
    v_i=s^{-1/2}\ \text{for all } i\in\supp(v)
    \right\}.
\]

\begin{theorem}\label{thm: reduction}
Suppose \Cref{cdt: sparsity assumption} holds, and
\(
    \sqrt{n}/\log p\ll k_u \ll k_\xi \lesssim {n}/{\log p}
\).
Assume there exists a polynomial-time algorithm that solves
\({\rm LT}(n,k_u,k,k_\xi,p,\tau)\) for some $1\le k\le k_u$ with
total error probability at most
\(\alpha+\eta\), where
\[
    \tau
    =
    c\,\rho_n^2\,\frac{k_u}{\sqrt{k_\xi}},
\]
for a sufficiently small constant \(c>0\) and some \(\rho_n<1/2\). Then there exists a polynomial-time algorithm that solves
\[
    {\rm SCCA}\left(
    2n,\,
    \lfloor k_u/4\rfloor,\,
    k_\xi,\,
    p-k_\xi,\,
    \rho_n
    \right)
\]
with total error probability at most \(\alpha+\eta\).
\end{theorem}

If we take
\(
    \rho_n
    \asymp
    \sqrt{{k_\xi\log p}/{n}},
\)
Theorem~\ref{thm: reduction} implies that under the conjectured polynomial-time hardness of the following asymmetric sparse CCA instance
\begin{equation}\label{eq: relevant SCCA}
        {\rm SCCA}\left(
    2n,\,
    \lfloor k_u/4\rfloor,\,
    k_\xi,\,
    p-k_\xi,\,
    \sqrt{\frac{k_\xi\log p}{n}}
    \right),
\end{equation}
\({\rm LT}(n,k_u,k,k_\xi,p,\tau)\) cannot be solved efficiently at the separation scale
\[
    \tau \asymp \frac{k_u\sqrt{k_\xi}\log p}{n}.
\]
This rate matches the computationally feasible upper bound in~\eqref{eq: computational upper tony} when \(k_\xi\lesssim n/\log p\).
Therefore, Theorem~\ref{thm: reduction} provides conditional evidence for the computational barrier in the flat sparse loading subclass.

\textit{Connection with known computational barriers for sparse CCA.}
The sparse CCA instance in \Cref{eq: relevant SCCA} is asymmetric, with dimensions \(k_\xi\) and \(p-k_\xi\). This asymmetry is important because many existing sparse CCA reductions are formulated for balanced dimensions and hence
do not directly apply to the instance induced by \Cref{thm: reduction}. The closest related result is \cite{laha2023support}, which establishes low-degree
barriers for asymmetric sparse CCA. However, a direct combination of \cite{laha2023support} with \Cref{thm: reduction} would only rule out strong separation for the linear testing problem, at the feasible upper bound in \eqref{eq: computational upper tony} up to logarithmic factors. By contrast, Theorem~\ref{thm: computational lower bound} rules out weak separation by low-degree polynomials. Thus, the direct low-degree analysis in \Cref{sec: low-degree} is not a straightforward consequence of known sparse CCA results.

\textit{Algorithmic evidence for computational barriers in sparse CCA.}
We also connect the above reduction to standard algorithmic evidence for computational barriers in sparse CCA detection. In the sparse CCA formulation, the empirical cross-covariance matrix contains the planted cross-correlation signal, so it is natural to compare procedures that threshold different functionals of this matrix. Exhaustive scan statistics can detect signals at the
information-theoretic boundary \(\rho_n\asymp \sqrt{k_u\log p/n}\) given in \cite{10.3150/12-BEJ470}, but they require a combinatorial search over
\(k_u\times k_u\) submatrices. By contrast, simple polynomial-time procedures, such as max-column statistics, require stronger signals, of order \(\rho_n\asymp \sqrt{k_\xi\log p/n}\) up to logarithmic factors, in the parameter regime used in \Cref{thm: reduction}. Thus, the algorithmic behavior of natural
sparse CCA tests is consistent with the computational barrier suggested by the reduction. A detailed comparison of these test statistics is deferred to \Cref{sec: test statistics}.

\section{Benchmarks with structured or known covariance}
\label{sec: extensions}

The lower bound in \Cref{sec: minimax rate} is sharp up to logarithmic factors in the ultra-sparse regime for arbitrary loading vectors and in the moderately sparse regime for certain loading subclasses.
For arbitrary loading vectors in the moderately sparse regime, it is of interest to assess the sharpness of the lower bound.
This section calibrates the scope of the lower-bound analysis through two covariance benchmarks.
In both benchmarks, we consider any sparsity level $k_u\lesssim n/\log p$ and any loading vector $\xi$.

First, when \(\Sigma\) is unknown but has a sparse signed-spiked structure, the adaptive testing boundary is shown to be $\nu_1/\sqrt n+\nu_2 k_u\log p/n$ up to logarithmic factors.
Thus the information-theoretic lower bound in \Cref{sec: minimax rate} is nearly attainable in this structured unknown-covariance model.
Meanwhile, as the computational results in \Cref{sec: computational barrier} continue to apply, we obtain evidence for a statistical--computational gap.

Second, when \(\Sigma\) is known, covariance uncertainty is removed from the testing problem and we obtain a sharper
adaptive upper bound than in the unknown-covariance setting.
When \(\Sigma_0\) is diagonal, this upper bound matches the lower bound $\nu_1/\sqrt{n}$ up to logarithmic factors.

\subsection{$\Sigma$ with sparse signed-spiked structures}\label{sec: np hard}
We consider a structured covariance class under which the lower bound in
Theorem~\ref{thm: hypothesis} is statistically attainable, up to logarithmic
factors, even in the moderately sparse regime
\(\sqrt{n}/\log p \ll k_u \lesssim n/\log p\). This structural assumption is
imposed to isolate an information-theoretic phenomenon: for sparse signed-spiked
covariance matrices, \(\Sigma\) can be estimated at a rate sufficiently fast to
attain the lower bound. The restriction is also natural from the perspective of
high-dimensional covariance modeling. Sparse signed-spiked covariance models are
standard in sparse PCA, principal subspace estimation, and principal component
regression; see, for example,
\cite{johnstone2009consistency,cai2015optimal,green2025high,wucalibrated}.

To formulate the sparse signed-spiked structure, we first introduce some notation.
For a matrix $V \in \mathbb{R}^{p \times r}$, let $V_{j*}$ denote its $j$th row.
The row support of $V$ is defined by
\begin{equation}\label{eq:suppV}
\supp(V)
=
\{ j \in [p] : V_{j*} \neq 0 \},
\end{equation}
and its cardinality is denoted by $|\supp(V)|$.
Let
\[
\mathbb{O}(p,r)
=
\{ V \in \mathbb{R}^{p \times r} : V^\top V = \mathbf{I}_r \}
\]
denote the set of $p \times r$ matrices with orthonormal columns.
For the covariance matrix $\Sigma$, we define the following sparse signed-spiked parameter space:
\begin{equation}\label{eq: sparse spiked}
\begin{aligned}
\Pi_0(k,p)
=
\Bigl\{
\Sigma = &V \operatorname{diag}(\lambda_1,\ldots,\lambda_r) V^\top + \mathbf{I}_p :
\text{ for some } r \le k, \\
&\frac{1}{M_1}-1 \le \lambda_r \le \cdots \le \lambda_1 \le M_1-1,\;
V \in \mathbb{O}(p,r),\;
|\operatorname{supp}(V)| \le k
\Bigr\}.
\end{aligned}
\end{equation}
It is immediate that any $\Sigma \in \Pi_0(k,p)$ satisfies the eigenvalue condition required in \eqref{eq: parameter space} that $1/M_1 \le \lambda_{\min}(\Sigma) \le \lambda_{\max}(\Sigma) \le M_1$.

We now define the null and alternative spaces, as well as the corresponding adaptive separation
distance, under the additional structural assumption $\Sigma \in \Pi_0(k,p)$.
Analogously to \eqref{eq: null space}, \eqref{eq: alternative space},
\eqref{eq: minimax separation}, and \eqref{eq: adaptive separation}, we define
\begin{equation*}
\left\{
\begin{aligned}
\Theta^{\text{spike}}(k;\xi,t_0)
&=
\Bigl\{
\theta=(\beta,\Sigma,\sigma)\in\Theta(k;\xi,t_0):
\Sigma\in \Pi_0(k,p)
\Bigr\},\\
\Theta^{\text{spike}}_{\pm \tau}(k;\xi,t_0)
&=
\Bigl\{
\theta=(\beta,\Sigma,\sigma)\in\Theta_{\pm \tau}(k;\xi,t_0):
\Sigma\in \Pi_0(k,p)
\Bigr\},\\
\tau^{\text{spike}}_{\text{adap}}(k_u,k\,;\,\xi)
&=
\inf\Bigl\{
\tau :
\sup_{\psi \in \Psi_\alpha(\Theta^{\text{spike}}(k_u;\xi,t_0))}
\inf_{\theta \in \Theta^{\text{spike}}_{\pm \tau}(k;\xi,t_0)}
\mathbb{E}_\theta \psi \ge 1 - \eta
\Bigr\}.
\end{aligned}
\right.
\end{equation*}
As before, the rank \(r\) and the true sparsity level \(k\) are not assumed to be known to the test.

\begin{theorem}\label{thm: sparse spiked}
Under \Cref{cdt: sparsity assumption}, there exists some constant $c>0$ such that if $n\ge c\,k_u\log p$, then for any $1\le k\le k_u$, we have
\[
\tau^{\text{spike}}_{\text{adap}}(k_u,k\,;\,\xi)
\;\asymp_{\log}\;
\frac{\nu_1}{\sqrt{n}} + \nu_2\,\frac{k_u \log p}{n}.
\]
\end{theorem}

\Cref{thm: sparse spiked} gives the optimal adaptive separation rate under the sparse signed-spiked covariance assumption $\Sigma \in \Pi_0(k,p)$, up to logarithmic factors.
The lower-bound part follows directly from the constructions used for $\nu_1$
and $\nu_2$: in those constructions, $\Sigma$ is either the identity matrix or
has the form in \eqref{eq: covariance construction}, and therefore belongs to
the sparse signed-spiked class in \eqref{eq: sparse spiked}.

The upper bound is more delicate and differs substantially from the upper-bound argument in \Cref{thm: hypothesis}. It is attained by a statistically optimal but computationally inefficient procedure. The key ingredient of this procedure is covariance estimation over the parameter space $\Pi_0(k_u,p)$.
In the proof, we construct an estimator of \(\Sigma\) based on exhaustive search over sparse supports.
This estimator is specific to the sparse signed-spiked class considered here, and its operator-norm error is of order $\sqrt{(k_u\log p)/n}$, which is minimax rate-optimal because the standard sparse-spiked submodel yields a matching lower bound \citep{cai2015optimal}.
Consequently, under the sparse signed-spiked structure, the information-theoretic limits governed by $\nu_1$ and $\nu_2$ are statistically
attainable up to logarithmic factors, although the resulting procedure is not computationally efficient.

\textit{Computational barrier.}
The computational evidence in \Cref{sec: computational barrier} is compatible with the sparse signed-spiked covariance class in \eqref{eq: sparse spiked}.
Since all covariance constructions used in the proofs of
\Cref{thm: computational lower bound,thm: reduction} are of the form \eqref{eq: covariance construction}, the proofs of
\Cref{thm: computational lower bound,thm: reduction} remain valid when we replace the used parameter spaces by their sparse signed-spiked counterparts.
Therefore, the low-degree lower bounds for general loadings and the
sparse CCA reduction for flat sparse loadings continue to apply without change under the sparse signed-spiked
covariance restriction.

Accordingly, \Cref{thm: sparse spiked} should be interpreted together with the evidence for computational barriers in the moderately sparse regime $\sqrt{n}/\log p\ll k_u \lesssim n/\log p$:
although the sparse signed-spiked adaptive separation distance is
statistically attainable up to logarithmic factors, achieving this rate appears to require computationally intractable procedures.
In particular, when the low-degree lower bound \( H(n/\log p; \xi) k_u\log p/n\) is much larger than $\nu_1/\sqrt{n}+\nu_2 k_u \log p /n$,
the testing problem is predicted to exhibit a statistical--computational gap under the standard low-degree heuristic.

\subsection{$\Sigma$ known}\label{sec: known design}
Thus far, our analysis has focused on the practically relevant setting in which the covariance matrix $\Sigma$ is unknown. We now consider the idealized setting in which \(\Sigma=\Sigma_0\) is known a priori. Although such knowledge is rarely available in practice, as noted by \cite{cai2020semi,cai2023statistical}, this setting can be viewed as an extreme case of the semi-supervised framework, corresponding to an infinite amount of unlabeled data for estimating \(\Sigma\). Comparing this benchmark with the unknown-design case helps isolate how uncertainty in the covariance matrix \(\Sigma\) affects the testing
problem~\eqref{eq: hypothesis test}.

The following theorem gives a general upper bound for arbitrary known \(\Sigma_0\) satisfying the eigenvalue condition.
When \(\Sigma_0\) is diagonal, including the identity covariance as a special case, the theorem further characterizes the minimax adaptive separation distance up to logarithmic factors.

\begin{theorem}\label{thm: hypothesis design}
Under \Cref{cdt: sparsity assumption}, there exists some constant $c>0$ such that if $n\ge c\,k_u\log p$, then
for any \(1\le k\le k_u\) and any \(\Sigma_0\in\mathbb R^{p\times p}\) satisfying
$1/M_1 \le \lambda_{\min}(\Sigma_0) \le \lambda_{\max}(\Sigma_0) \le M_1$, we have
\[
\tau_{\text{adap}}(k_u,k,\Sigma_0\,;\,\xi)
\;\lesssim\;
\min_{0\le m\le p}
\left(
H(m;\xi)\,\frac{1}{\sqrt{n}}
+
|\xi_{m+1}|\,k_u\sqrt{\frac{\log p}{n}}
\right).
\]
Furthermore, when $\Sigma_0$ is diagonal, we have
\[
\tau_{\text{adap}}(k_u,k,\Sigma_0\,;\,\xi)
\;\asymp_{\log}\;
\frac{\nu_1}{\sqrt{n}}.
\]
\end{theorem}

The proof of \Cref{thm: hypothesis design} follows the same
plug-in $+$ debiasing decomposition as in \Cref{sec: upper}, but with an oracle debiasing direction. For the coordinates of \(\xi\) with small magnitudes, the plug-in construction is unchanged and contributes the term
\(
    |\xi_{m+1}|\,k_u\sqrt{{\log p}/{n}} .
\)
For the leading coordinates, knowing \(\Sigma_0\) allows us to use the oracle population direction
\(
    \Omega_0\xi^{(1)}=\Sigma_0^{-1}\xi^{(1)}
\)
directly in the debiased estimator. This removes the error introduced by estimating the bias-correction direction from the sample covariance. Combining the oracle debiased interval for the leading coordinates with the plug-in interval for the remaining coordinates yields the displayed upper bound after optimizing over \(m\). For brevity, we defer the proof to \Cref{sec: proof upper design}.

This explains the improvement over the unknown-design upper bound in \Cref{thm: hypothesis}. In the moderately sparse regime
\( \sqrt{n}/\log p \ll k_u \lesssim n/\log p\), uncertainty in \(\Sigma\) creates an additional cost through the construction of the bias-correction direction. When \(\Sigma_0\) is known, this cost disappears,
and the plug-in $+$ debiasing construction attains a sharper rate. In particular, when \(\Sigma_0\) is diagonal, the minimax adaptive separation
distance is \(\nu_1/\sqrt n\), up to logarithmic factors.

The known-\(\Sigma\) benchmark also has a broader semi-supervised
interpretation: exact knowledge of \(\Sigma\) is stronger than necessary. It is enough to construct an estimator \(\widehat{\Omega}\) of
\(\Omega_0\) satisfying
\[
    \|\widehat{\Omega}-\Omega_0\|_{2\to\infty}
    \;\lesssim\;
    \frac{1}{k_u\sqrt{\log p}} .
\]
Here, for a matrix \(A\in\mathbb R^{p\times p}\),
\(\|A\|_{2\to\infty}:=\max_{1\le j\le p}\|A_{j\cdot}\|_2\) denotes the maximum
row \(\ell_2\)-norm. Such an estimator may be available, for example, when
\(\Sigma\) has additional structure or when sufficiently many unlabeled samples
of \(X\) are available. Under this condition, the same debiasing argument
applies with \(\widehat{\Omega}\) in place of \(\Omega\), and the induced error
is negligible. Consequently, the adaptive rate involving \(\nu_1\) remains
attainable up to logarithmic factors.

Related phenomena have been observed in \cite{cai2017confidence,javanmard2018debiasing}. In particular, \cite{javanmard2018debiasing} construct asymptotically normal debiased estimators for individual regression coefficients, while \cite{cai2017confidence} use data splitting to obtain confidence intervals for \(\xi^\top\beta\) with optimal length for certain structured loading vectors. However, their results do not directly imply the adaptive upper bound
above for general loading vectors \(\xi\), and our plug-in $+$ debiasing
construction is not a direct consequence of these existing methods.

\section{Discussion}\label{sec: discussion}

We investigated the minimax and adaptive separation distances for testing a
general linear functional \(\xi^\intercal\beta\) in high-dimensional linear
regression with Gaussian random design. Our results reveal a transition in both
statistical and computational behavior as the sparsity upper bound \(k_u\) moves
from the ultra-sparse regime \(k_u \lesssim \sqrt{n}/\log p\) to the moderately
sparse regime \(\sqrt{n}/\log p \ll k_u \lesssim n/\log p\).

In the ultra-sparse regime, the testing problem admits a
statistical characterization, $\nu_1(k_u;\xi)/\sqrt{n}$, that is precise up to logarithmic factors.
By contrast, the moderately
sparse regime exhibits qualitatively different behavior.
In this regime, another lower-bound component $\nu_2(k_u;\xi)k_u\log p/n$ becomes relevant.
This lower-bound construction is based on coupling the leading coordinates of $\xi$ with a hidden sparse perturbation in the design covariance.
A closely related perturbation mechanism is used in our low-degree lower bound for general loadings and our sparse CCA reduction for flat sparse loadings.

Several directions remain open. First, closing the remaining gaps in both the ultra-sparse and moderately sparse regimes remains an important direction for future work. Our results partially
address this question by identifying structured classes of loading vectors for
which the upper and lower bounds match sharply. Second, it would be of interest
to extend the theory beyond \(\ell_0\)-sparsity to more general structured
sparsity classes, such as \(\ell_q\)-sparsity \citep{raskutti2011minimax} and
sparse-group structures \citep{cai2022sparse}. Third, one may replace the sparse signed-spiked covariance structure with a sparse
precision matrix structure, under which each row of
\(\Omega=\Sigma^{-1}\) has only a small number of nonzero entries
\citep{vandeGeer2014optimal}. Since this class is generally larger than the sparse signed-spiked covariance class considered in \Cref{sec: np hard}, an important
question is whether a comparable minimax upper bound can still be attained by
statistically optimal (but possibly computationally inefficient) procedures. Finally, developing
computational lower bounds for sparse CCA with more general latent vectors would
strengthen the theoretical basis for the computational barrier identified here.
It also remains open whether analogous statistical--computational phase transitions
arise for quadratic or more general nonlinear functionals.

\bibliographystyle{amsplain}
\bibliography{ref}

\newpage
\appendix
\setcounter{equation}{0}
\renewcommand{\theequation}{S.\arabic{equation}}
\renewcommand{\theHequation}{S.\arabic{equation}}
\numberwithin{theorem}{section}
\numberwithin{lemma}{section}
\numberwithin{proposition}{section}
\numberwithin{corollary}{section}
\numberwithin{assumption}{section}
\numberwithin{condition}{section}
\numberwithin{definition}{section}
\numberwithin{remark}{section}
\numberwithin{example}{section}
\renewcommand{\theHtheorem}{\thesection.\arabic{theorem}}
\renewcommand{\theHlemma}{\thesection.\arabic{lemma}}
\renewcommand{\theHproposition}{\thesection.\arabic{proposition}}
\renewcommand{\theHcorollary}{\thesection.\arabic{corollary}}
\renewcommand{\theHassumption}{\thesection.\arabic{assumption}}
\renewcommand{\theHcondition}{\thesection.\arabic{condition}}
\renewcommand{\theHdefinition}{\thesection.\arabic{definition}}
\renewcommand{\theHremark}{\thesection.\arabic{remark}}
\renewcommand{\theHexample}{\thesection.\arabic{example}}

\bigskip
\bigskip
\begin{center}
\textbf{\Large Supplement to ``\THETITLE''}
\end{center}
This supplement collects proofs and auxiliary results deferred from the main text. We structure the material into three parts: \textbf{(i)} upper bounds, \textbf{(ii)} information-theoretic lower bounds, and \textbf{(iii)} additional technical lemmas.

\Cref{sec: proof upper} establishes the upper bounds in \Cref{thm: hypothesis,thm: hypothesis design} (as well as the related upper bound in \Cref{thm: sparse spiked}). We first compile concentration tools in \Cref{sec: aux upper}. We then derive confidence-interval constructions under unknown design in \Cref{sec: proof upper unknown design}, and under known design via data splitting in \Cref{sec: proof upper design}. Finally, \Cref{sec: proof upper np hard} adapts these arguments to the setting of \Cref{thm: sparse spiked}.

\Cref{sec: proof lower} proves the lower bounds in \Cref{thm: hypothesis,thm: computational lower bound,thm: reduction,thm: hypothesis design} via the standard ``least favorable prior + \(\chi^2\)-divergence'' method. Specifically, we construct priors supported on (or predominantly supported on) the null parameter sets and bound the resulting \(\chi^2\)-divergence using explicit Gaussian calculations; these bounds imply power limitations through \Cref{lem: power inequality}.

The technical ingredients required for the upper-bound arguments (in particular, feasibility of the projection step) are proved in \Cref{sec: proof lemma upper}. Additional tools for the lower bounds and checks of prior validity (e.g., eigenvalue/sparsity control and \(\chi^2\) integral identities) are collected in \Cref{sec: proof lemma lower}. Finally, \Cref{sec: additional results} contains complementary discussions, and \Cref{sec:loading-examples} contains selected loading-profile examples, including one random-predictor example.

Table~\ref{tab:proof-map} gives a cross-reference from each main result to its proof and selected lemmas. For \Cref{thm: sparse spiked}, the table lists the separate upper-bound proof; its lower-bound proof follows from the same lower-bound argument used for \Cref{thm: hypothesis}.

\begin{table}[ht]
\centering
\caption{Proof map: where each main result is proved and the key auxiliary lemmas it uses. Parentheses after a lemma indicate where it is proved.}
\label{tab:proof-map}
\small
\begin{tabular}{>{\centering}p{0.2\textwidth} >{\centering}p{0.15\textwidth} >{\centering\arraybackslash}p{0.6\textwidth}}
\hline
Result & Location & Key ingredients (selected) \\
\hline
\Cref{thm: hypothesis} (upper) &
\Cref{sec: proof upper unknown design} &
\Cref{lem: feasibility of optimization} (\Cref{sec: proof feasibility of optimization}). \\

\Cref{thm: hypothesis design} (upper) &
\Cref{sec: proof upper design} &
Debiased estimator under known design \eqref{eq: bias estimator known design}. \\

\Cref{thm: sparse spiked} (upper) &
\Cref{sec: proof upper np hard} &
Precision matrix estimation \Cref{lem: sparse spiked} (\Cref{sec: proof sparse spiked}) and the corresponding debiased estimator \eqref{eq: debiased center np hard}. \\

\Cref{thm: hypothesis} (lower) &
\Cref{sec: proof lower ultra sparse} &
Prior validity \Cref{lem: valid covariance} (\Cref{sec: proof valid covariance}); \(\chi^2\) integral bound \Cref{lem: chi square integral 1} (\Cref{sec: proof chi square integral}). \\

\Cref{thm: hypothesis design} (lower) &
\Cref{sec: proof design lower} &
Prior validity \Cref{lem: valid covariance 2} (\Cref{sec: proof valid covariance 2}); \(\chi^2\) integral bound \Cref{lem: chi square integral 2}(\Cref{sec: proof chi square integral 2}). \\

\Cref{thm: computational lower bound} &
\Cref{sec: proof computational lower bound} &
Prior validity \Cref{lem: valid covariance 3} (\Cref{sec: proof valid covariance 3}); Low-degree quantity control \Cref{lem: low degree integral} (\Cref{sec: proof low degree integral}). \\

\Cref{thm: reduction} &
\Cref{sec: proof reduction} &
Polynomial-time reduction \Cref{lem: reduction}. \\
\hline
\end{tabular}
\end{table}

\section{Proof of the upper bounds}\label{sec: proof upper}
In this section, we prove the upper bounds in \Cref{thm: hypothesis} and \Cref{thm: hypothesis design}.
Define $\mathcal{A}_0$ to be the event such that the following holds:
\begin{equation}\label{eq: event A0}
1.1\hat{\sigma}>\sigma>0.9\hat{\sigma},\|\hat{\beta}-\beta\|_1\le  c_\beta \sigma k_u\sqrt{\frac{\log p}{n}},\|\hat{\beta}-\beta\|_2\le  C_\beta\sigma \sqrt{\frac{k_u\log p}{n}}.
\end{equation}
For any prescribed error probability
$\bar\alpha\in(0,1)$, Conditions~\ref{cdt: linear estimator} and
\ref{cdt: linear variance} allow $n$ to be taken large enough so that
\[
\inf_{\theta\in\Theta(k_u)}
\mathbb P_\theta\{\mathcal A_0\}
\ge 1-\bar\alpha.
\]

Throughout this supplement, we write \(\Omega=\Sigma^{-1}\) for the precision matrix and \(z_\alpha\) for the \(\alpha\)-quantile of the standard normal distribution.
We use the convention
\(\operatorname{sign}(x)=1\) for \(x\ge0\) and
\(\operatorname{sign}(x)=-1\) for \(x<0\).

\subsection{Auxiliary lemmas for the proof of upper bounds}\label{sec: aux upper}

\begin{lemma}[{\cite[Lemma 1]{laurent2000adaptive}}]\label{lem: chi square tail}
    Let $Z\sim \chi^2(n)$. For any $x>0$, we have
    $$\bbP\left(Z-n\ge  2\sqrt{nx}+2x\right)\le  e^{-x},$$
    and
    $$\bbP\left(Z-n\le  -2\sqrt{nx}\right)\le  e^{-x}.$$
\end{lemma}

We introduce the following definitions.
The sub-Gaussian norm of a random variable \(U\) is defined as
\[
\|U\|_{\psi_2}
    := \sup_{q \ge 1} \frac{1}{\sqrt{q}}
        \bigl( \mathbb{E}|U|^{q} \bigr)^{1/q},
\]
and the sub-Gaussian norm of a random vector \(U \in \mathbb{R}^p\) is defined as
\[
\|U\|_{\psi_2}
    := \sup_{v \in S^{p-1}} \|\langle v, U \rangle\|_{\psi_2},
\]
where \(S^{p-1}\) denotes the unit sphere in \(\mathbb{R}^p\).

Similarly, the sub-exponential norm of a random variable \(U\) is defined as
\[
\|U\|_{\psi_1}
    := \sup_{q \ge 1} \frac{1}{q}
        \bigl( \mathbb{E}|U|^{q} \bigr)^{1/q},
\]
and the sub-exponential norm of a random vector \(U \in \mathbb{R}^p\) is defined as
\[
\|U\|_{\psi_1}
    := \sup_{v \in S^{p-1}} \|\langle v, U \rangle\|_{\psi_1}.
\]

\begin{lemma}[{\cite[Remark 5.18]{vershynin2010introduction}}]\label{lem: sub exp center}
    For a sub-exponential random variable \(U\), we have
    $$\|U-\bbE U\|_{\psi_1} \le  2 \|U\|_{\psi_1}.$$
\end{lemma}
\begin{lemma}[{\cite[Proposition 5.16]{vershynin2010introduction}}]\label{lem: sub exp sum}
    Let $X_1,\ldots,X_N$ be independent centered sub-exponential random variables, and
\[
    K = \max_{i} \|X_i\|_{\psi_1}.
\]
Then for every $a = (a_1,\ldots,a_N) \in \mathbb{R}^N$ and every $t \ge 0$, we have
\[
\mathbb{P}\!\left( \left| \sum_{i=1}^N a_i X_i \right| \ge t \right)
    \le 2 \exp\!\left[ -c_0 \min\!\left( \frac{t^2}{K^2 \|a\|_2^2},\; \frac{t}{K \|a\|_\infty} \right) \right],
\]
where $c_0 > 0$ is an absolute constant.

\end{lemma}

\begin{lemma}\label{lem: product of subGaussian}
Let $X$ and $Y$ be sub-Gaussian random variables. Then $XY$ is sub-exponential. Moreover,
$$
\|X Y\|_{\psi_1} \le  2\|X\|_{\psi_2}\|Y\|_{\psi_2}.
$$
\end{lemma}
\begin{proof}
Fix any $q\ge 1$. By H\"older's inequality,
$$
\begin{aligned}
\left(\mathbb{E}|XY|^q\right)^{1/q}
&\le
\left(\mathbb{E}|X|^{2q}\right)^{1/(2q)}
\left(\mathbb{E}|Y|^{2q}\right)^{1/(2q)}\\
&\le
\bigl(\|X\|_{\psi_2}\sqrt{2q}\bigr)\bigl(\|Y\|_{\psi_2}\sqrt{2q}\bigr)
=
2q\,\|X\|_{\psi_2}\,\|Y\|_{\psi_2}.
\end{aligned}
$$
Dividing by $q$ and taking the supremum over $q\ge 1$ yields the desired inequality.
\end{proof}

\begin{lemma}[Mixed confidence interval]\label{lem:generic-mixed-ci}
Let $\mathcal P\subseteq \Theta(k_u)$ and let
$\xi\in\mathbb R^p$ satisfy
$|\xi_1|\ge \cdots \ge |\xi_p|$.
Let $\alpha_1,\alpha_2\in(0,1)$ satisfy $\alpha_1+\alpha_2<1$.
Assume that, for every fixed $v\in\mathbb R^p$, we have the following two intervals with nonnegative random radii
\[
\begin{aligned}
    \CI_{\rm db}(v)
    &=
    [\hat h_{\rm db}(v)-r_{\rm db}(v),\,
      \hat h_{\rm db}(v)+r_{\rm db}(v)],\\
    \CI_{\rm pi}(v)
    &=
    [v^\top\hat\beta-r_{\rm pi}(v),\,
      v^\top\hat\beta+r_{\rm pi}(v)],
\end{aligned}
\]
and that they satisfy the following high-probability coverage and radius bounds uniformly over $\mathcal P$:
\[
    \inf_{\theta\in\mathcal P}
    \mathbb P_\theta
    \left(
        v^\top\beta\in\CI_{\rm db}(v),\
        r_{\rm db}(v)\le R_{\rm db}(v)
    \right)
    \ge 1-\alpha_1
\]
and
\[
    \inf_{\theta\in\mathcal P}
    \mathbb P_\theta
    \left(
        v^\top\beta\in\CI_{\rm pi}(v),\
        r_{\rm pi}(v)\le R_{\rm pi}(v)
    \right)
    \ge 1-\alpha_2,
\]
where $R_{\rm db}(v)$ and $R_{\rm pi}(v)$ are nonnegative non-random radius
envelopes.  The debiased envelope $R_{\rm db}(v)$ may depend on $v$, $\sigma$,
and fixed model constants, but not on the data or on the parameter $\theta\in\mathcal P$.
The plug-in envelope is given by
$$\qquad
    R_{\rm pi}(v)
    =
    C_{\rm pi}^+\sigma\|v\|_\infty k_u\sqrt{\frac{\log p}{n}},
$$
where $C_{\rm pi}^+>0$ is a fixed constant that may depend on fixed estimator constants such as $c_{\beta}$ but not on $n,p,k_u,v$, or on the parameter $\theta\in\mathcal P$.

For a fixed deterministic $0\le m\le p$, define
\[
    \xi^{(1)}_m=(\xi_1,\ldots,\xi_m,0,\ldots,0)^\top,
    \qquad
    \xi^{(2)}_m=\xi-\xi^{(1)}_m ,
\]
with the convention $\xi_{p+1}=0$, and set
\[
    \CI_m(\xi)
    =
    \CI_{\rm db}(\xi^{(1)}_m)
    +
    \CI_{\rm pi}(\xi^{(2)}_m),
\]
where the sum denotes the Minkowski sum of two intervals.  Then
$\CI_m(\xi)$ is a $(1-\alpha_1-\alpha_2)$-level confidence interval for
$\xi^\top\beta$ over $\mathcal P$.  Moreover, uniformly over
$\mathcal P$, with probability at least $1-\alpha_1-\alpha_2$ its radius is
bounded by
\[
    R_m(\xi)
    =
    R_{\rm db}(\xi^{(1)}_m)
    +
    C_{\rm pi}^+\sigma |\xi_{m+1}| k_u\sqrt{\frac{\log p}{n}} .
\]
Consequently, the test
\[
    \psi_m
    =
    \ind\{t_0\notin \CI_m(\xi)\}
\]
has Type-I error at most $\alpha_1+\alpha_2$ over
$\mathcal P\cap\Theta(k_u;\xi,t_0)$ and has power at least
$1-\alpha_1-\alpha_2$ at every $\theta\in\mathcal P$ satisfying
$|\xi^\top\beta-t_0|>2R_m(\xi)$.
\end{lemma}

\begin{proof}
Let $E_m$ be the intersection of the two events  in the two assumed bounds, with
$v=\xi_m^{(1)}$ in the debiased bound and $v=\xi_m^{(2)}$ in the plug-in bound.
In other words, $E_m$ is the event that
\[
        (\xi_m^{(1)})^\top\beta\in\CI_{\rm db}(\xi_m^{(1)}),\
        r_{\rm db}(\xi_m^{(1)})\le R_{\rm db}(\xi_m^{(1)}),
        (\xi_m^{(2)})^\top\beta\in\CI_{\rm pi}(\xi_m^{(2)}),\
        r_{\rm pi}(\xi_m^{(2)})\le R_{\rm pi}(\xi_m^{(2)}).
\]
The union bound gives $\mathbb P_\theta(E_m)\ge 1-\alpha_1-\alpha_2$ uniformly
over $\mathcal P$.  On $E_m$,
\[
    \xi^\top\beta
    =
    (\xi_m^{(1)})^\top\beta+(\xi_m^{(2)})^\top\beta
    \in
    \CI_{\rm db}(\xi_m^{(1)})
    +
    \CI_{\rm pi}(\xi_m^{(2)})
\]
and both radius bounds also hold.  Therefore, the radius of the interval sum is
at most
\[
    R_{\rm db}(\xi^{(1)}_m)
    +
    R_{\rm pi}(\xi^{(2)}_m)
    =
    R_{\rm db}(\xi^{(1)}_m)
    +
    C_{\rm pi}^+\sigma\|\xi^{(2)}_m\|_\infty k_u\sqrt{\frac{\log p}{n}},
\]
which equals $R_m(\xi)$ because
$\|\xi^{(2)}_m\|_\infty=|\xi_{m+1}|$.
If $\xi^\top\beta=t_0$, coverage implies $t_0\in\CI_m(\xi)$; hence the
Type-I error is at most $\alpha_1+\alpha_2$.
If $|\xi^\top\beta-t_0|>2R_m(\xi)$, then under $E_m$, the coverage and radius bound both
hold and they imply that $t_0\notin \CI_m(\xi)$. Therefore, the power is at least $\mathbb P_\theta(E_m)\ge 1-\alpha_1-\alpha_2$.
\end{proof}

\begin{lemma}[Plug-in confidence interval]\label{lem:generic-plugin-ci}
Let $\mathcal P\subseteq\Theta(k_u)$ and fix $v\in\mathbb R^p$.  Suppose that
$(\hat\beta,\hat\sigma)$ is computed from $N$ observations.
Assume that, for a
prescribed $\bar\alpha\in(0,1)$, there exists an event $\mathcal E_{\rm pi}$ such
that
\[
\inf_{\theta\in\mathcal P}
\mathbb P_\theta(\mathcal E_{\rm pi})
\ge 1-\bar\alpha,
\]
and on $\mathcal E_{\rm pi}$, the following bounds hold:
\[
1.1\hat\sigma>\sigma>0.9\hat\sigma,
\qquad
\|\hat\beta-\beta\|_1
\le
c_\beta\sigma k_u\sqrt{\frac{\log p}{N}} .
\]
Choose constants $C_{\rm pi}\ge 1.1c_\beta$ and
$C_{\rm pi}^+\ge C_{\rm pi}/0.9$, and define
\[
\CI_{\rm pi,N}(v)
=
[v^\top\hat\beta-r_{\rm pi,N}(v),\,
  v^\top\hat\beta+r_{\rm pi,N}(v)],
\qquad
r_{\rm pi,N}(v)
=
C_{\rm pi}\hat\sigma\|v\|_\infty k_u\sqrt{\frac{\log p}{N}} .
\]
Then
\[
\inf_{\theta\in\mathcal P}
\mathbb P_\theta
\left(
v^\top\beta\in\CI_{\rm pi,N}(v),\
r_{\rm pi,N}(v)
\le
C_{\rm pi}^+\sigma\|v\|_\infty k_u\sqrt{\frac{\log p}{N}}
\right)
\ge 1-\bar\alpha .
\]
In particular, if $N\asymp n$, the same statement gives the plug-in
assumption in \Cref{lem:generic-mixed-ci} after enlarging $C_{\rm pi}^+$ by a
fixed factor.
\end{lemma}

\begin{proof}
On $\mathcal E_{\rm pi}$,
\[
|v^\top\hat\beta-v^\top\beta|
\le
\|v\|_\infty\|\hat\beta-\beta\|_1
\le
c_\beta\sigma\|v\|_\infty k_u\sqrt{\frac{\log p}{N}}
\le
C_{\rm pi}\hat\sigma\|v\|_\infty k_u\sqrt{\frac{\log p}{N}},
\]
where the last inequality uses $\sigma<1.1\hat\sigma$ and
$C_{\rm pi}\ge1.1c_\beta$.  Thus
$v^\top\beta\in\CI_{\rm pi,N}(v)$ on $\mathcal E_{\rm pi}$.  The same
event also gives $\hat\sigma<\sigma/0.9$, and therefore
\[
r_{\rm pi,N}(v)
\le
\frac{C_{\rm pi}}{0.9}\sigma\|v\|_\infty k_u\sqrt{\frac{\log p}{N}}
\le
C_{\rm pi}^+\sigma\|v\|_\infty k_u\sqrt{\frac{\log p}{N}}.
\]
Taking probabilities and using the assumed bound for $\mathbb{P}_\theta(\mathcal E_{\rm pi})$
proves the claim.
\end{proof}

\subsection{Proof of the upper bound in Theorem~\ref{thm: hypothesis}}\label{sec: proof upper unknown design}

\begin{proof}
Let $\alpha_\star=\min\{\alpha,\eta\}$.
As discussed in \Cref{sec: upper}, to establish the upper bound in
\Cref{thm: hypothesis} it suffices to show that for any loading vector
$\xi \in \mathbb{R}^p$ one can construct:

\begin{itemize}
    \item a $(1-\alpha_\star/2)$-level \emph{plug-in} confidence interval with length
    \[
        O\!\left(
        \sigma \|\xi\|_\infty\, k_u \sqrt{\frac{\log p}{n}}
        \right),
    \]
    \item a $(1-\alpha_\star/2)$-level \emph{debiased} confidence interval with length
    \[
        O\!\left(
        \sigma \|\xi\|_2
        \left(
        \frac{1}{\sqrt{n}} + \frac{k_u \log p}{n}
        \right)
        \right).
    \]
\end{itemize}

Indeed, once these two ingredients are available, we decompose $\xi$ at a cutoff
$m \in [p]$ as $\xi=\xi^{(1)}+\xi^{(2)}$, where
\[
\xi^{(1)}=(\xi_1,\ldots,\xi_m,0,\ldots,0),
\qquad
\xi^{(2)}=(0,\ldots,0,\xi_{m+1},\ldots,\xi_p).
\]
We then construct a $(1-\alpha_\star/2)$-level debiased confidence interval for
$(\xi^{(1)})^\top\beta$ and a $(1-\alpha_\star/2)$-level plug-in confidence interval
for $(\xi^{(2)})^\top\beta$, and combine them to obtain a valid
$(1-\alpha_\star)$-level confidence interval for $\xi^\top\beta$.
The resulting interval length is of order
\[
\sigma\left(
\|\xi^{(1)}\|_2\Bigl(\frac{1}{\sqrt{n}}+\frac{k_u\log p}{n}\Bigr)
+
\|\xi^{(2)}\|_\infty\, k_u\sqrt{\frac{\log p}{n}}
\right).
\]
Using the condition $\sigma\le  M_2$ in \Cref{eq: parameter space}, we obtain an upper bound on the adaptive separation distance for any given $m$.
Since such an upper bound holds for all $m\in [p]$, minimizing over $m$ yields the desired upper bound in \Cref{thm: hypothesis}.

\bigskip

The construction of the plug-in confidence interval is an application of \Cref{lem:generic-plugin-ci} with
$N=n$, $\mathcal P=\Theta(k_u)$, $v=\xi$, and
$\mathcal E_{\rm pi}=\mathcal A_0$.
Therefore, we focus on the construction of the debiased confidence interval.

Let $\widetilde{\mathcal{A}}$ be the event that the vector
$u=\Omega\xi$ is a feasible point for the optimization problem
\eqref{eq: projection vector}. The following lemma guarantees that $\widetilde{\mathcal{A}}$ happens with probability close to 1.

\begin{lemma}\label{lem: feasibility of optimization}
For any $\alpha\in(0,1)$, there exists a constant $C_\xi>0$ such that
$\widetilde{\mathcal{A}}$ holds with probability at least $1-\alpha/24$ for all
sufficiently large $n$.
\end{lemma}

Recall the debiased estimator $\hat L_{\mathrm{db}}(Z;\xi)$ defined in
\eqref{eq: debiased center}. Its estimation error admits the decomposition
\[
\hat L_{\mathrm{db}}(Z;\xi)-\xi^\top\beta
=
\underbrace{\frac{1}{n}\hat{u}^\top X^\top\varepsilon}_{\mathrm{I}}
+
\underbrace{(\xi-\hat{\Sigma}\hat{u})^\top(\hat{\beta}-\beta)}_{\mathrm{II}},
\]
where $\hat{\Sigma}=n^{-1}X^\top X$ and
$\varepsilon=Y-X\beta\sim\mathcal{N}(0,\sigma^2\bfI_n)$ is independent of $X$ and
$\hat{u}$.

\textit{Bounding $\mathrm{I}$.}
Define the event
\[
\mathcal{A}_1
=
\left\{
\left|\frac{1}{n}\hat{u}^\top X^\top\varepsilon\right|
\le
\frac{\sigma}{\sqrt{n}}
\sqrt{\hat{u}^\top\hat{\Sigma}\hat{u}}\,
z_{1-\alpha/8}
\right\}.
\]
Conditional on $X$ and $\hat{u}$,
\[
\frac{1}{n}\hat{u}^\top X^\top\varepsilon
\;\Big|\; X
\sim
\mathcal{N}\!\left(0,
\frac{\sigma^2}{n}\hat{u}^\top\hat{\Sigma}\hat{u}
\right),
\]
so $\mathcal{A}_1$ holds with probability at least $1-\alpha/4$.

We next control $\hat{u}^\top\hat{\Sigma}\hat{u}$.
Let
\[
\mathcal{A}_2
=
\left\{
u^\top\hat{\Sigma}u
\le
1.1\,M_1^2\|\xi\|_2^2
\right\}
\cap \widetilde{\mathcal{A}}.
\]
For $u=\Omega\xi$, we have
\[
\frac{n}{\xi^\top\Omega\xi}\,u^\top\hat{\Sigma}u
\stackrel{d}{=}\chi^2(n).
\]
By Lemma~\ref{lem: chi square tail},
there exists $n_1$ such that for all $n\ge  n_1$,
the following holds with probability at least $1-\alpha/24$:
\begin{equation}\label{eq: bound oracle quadform revised}
u^\top\hat{\Sigma}u
\le
1.1\,\xi^\top\Omega\xi
\le
1.1\,M_1^2\|\xi\|_2^2,
\end{equation}
where the second inequality follows from the eigenvalue bounds in
\Cref{eq: parameter space}.
By \Cref{lem: feasibility of optimization}, there exists $n_2$ such that for $n\ge n_2$, $\widetilde{\mathcal A}$ holds with probability at least $1-\al/24$.
Therefore, $\mathcal{A}_2$ holds with probability at least $1-\alpha/12$.

Since $u$ is feasible on $\widetilde{\mathcal{A}}$, the definition of $\hat{u}$ implies that on the event $\mathcal{A}_2$, we have
\[
\hat{u}^\top\hat{\Sigma}\hat{u}
\le
u^\top\hat{\Sigma}u.
\]
Therefore,
\[
\mathcal{A}_2
\subseteq
\left\{
\hat{u}^\top\hat{\Sigma}\hat{u}
\le
1.1\,M_1^2\|\xi\|_2^2
\right\}
\cap \widetilde{\mathcal{A}}.
\]

\textit{Bounding $\mathrm{II}$.}
Recall the event $\mathcal{A}_0$ in \Cref{eq: event A0}, which holds with probability at least $1-\al/6$.
On $\mathcal{A}_0\cap\widetilde{\mathcal{A}}$, by the definition of $\hat{u}$,
\[
|(\xi-\hat{\Sigma}\hat{u})^\top(\hat{\beta}-\beta)|
\le
\|\hat{\Sigma}\hat{u}-\xi\|_\infty\|\hat{\beta}-\beta\|_1
\le
\sigma c_\beta C_\xi\|\xi\|_2\frac{k_u\log p}{n}.
\]

\bigskip
\textbf{Synthesis.}
Combining the above bounds, for sufficiently large $n$,
the event $\mathcal{A}_0\cap\mathcal{A}_1\cap\mathcal{A}_2$ holds with probability
at least $1-\alpha/2$, and on this event,
\begin{align*}
\left| \hat L_{\mathrm{db}}(Z;\xi)-\xi^\top\beta \right|
&\le
\frac{\sigma}{\sqrt{n}}
\sqrt{\hat{u}^\top\hat{\Sigma}\hat{u}}\,
z_{1-\alpha/8}
+
\sigma c_\beta C_\xi\|\xi\|_2\frac{k_u\log p}{n} \\
&\lesssim
\sigma\|\xi\|_2
\left(
\frac{1}{\sqrt{n}}+\frac{k_u\log p}{n}
\right).
\end{align*}
Replacing $\sigma$ by $1.1\hat{\sigma}$ yields a valid confidence interval with
the stated length, completing the proof.
More explicitly, for any fixed loading vector $v$, let $\hat u_v$ and
$\hat L_{\rm db}(Z;v)$ denote the quantities in
\eqref{eq: projection vector} and \eqref{eq: debiased center} with $\xi$
replaced by $v$.  For any prescribed component error probability
$\bar\alpha\in(0,1)$, define
\[
\CI_{\rm db}(v)
=
[\hat L_{\rm db}(Z;v)-r_{\rm db}(v),\,
 \hat L_{\rm db}(Z;v)+r_{\rm db}(v)]
\]
with
\[
r_{\rm db}(v)
=
1.1\hat\sigma
\left\{
\frac{\sqrt{\hat u_v^\top\hat\Sigma\hat u_v}}{\sqrt n}
z_{1-\bar\alpha/8}
+
c_\beta C_\xi\|v\|_2\frac{k_u\log p}{n}
\right\}.
\]
The preceding argument, applied with $\alpha=\bar\alpha$, gives
\[
\inf_{\theta\in\Theta(k_u)}
\mathbb P_\theta
\left\{
v^\top\beta\in\CI_{\rm db}(v),\
r_{\rm db}(v)\le R_{\rm db}(v)
\right\}
\ge 1-\bar\alpha ,
\]
where one may take
\[
R_{\rm db}(v)
=
C_{\rm db}\sigma\|v\|_2
\left(
\frac{1}{\sqrt n}
+
\frac{k_u\log p}{n}
\right)
\]
for a fixed constant $C_{\rm db}>0$ depending only on the fixed model constants
and on $\bar\alpha$.  This verifies the debiased assumption required by
\Cref{lem:generic-mixed-ci}.
\end{proof}

\subsection{Proof of the upper bound in Theorem~\ref{thm: hypothesis design}}\label{sec: proof upper design}
\begin{proof}
The argument follows the same general strategy as the proof of the upper bound under unknown design in \Cref{sec: proof upper unknown design}. The key difference is that, when the design covariance matrix $\Sigma=\Sigma_0$ is known, one can construct a debiased confidence interval for an arbitrary loading vector $\xi\in\bbR^p$ with confidence level $1-\alpha/2$ and length of order $\sigma\|\xi\|_2/\sqrt{n}$.

Following \cite{cai2017confidence}, we employ a data-splitting strategy. Randomly split the sample into two independent halves $Z^{(1)}=(X^{(1)},Y^{(1)})$ and $Z^{(2)}=(X^{(2)},Y^{(2)})$ of sizes $n_1$ and $n_2$, respectively. Without loss of generality, assume $n$ is even and $n_1=n_2=n/2$.

Using the first half of the data $Z^{(1)}$, we compute the lasso estimator $\hat{\beta}$ and the noise level estimator $\hat{\sigma}$ in \Cref{cdt: linear estimator,cdt: linear variance}.
Consequently, the event
    \[
\widetilde{\mathcal A}_0
=
\left\{
1.1\hat\sigma>\sigma>0.9\hat\sigma,\;
\|\hat\beta-\beta\|_1
\le c_\beta\sigma k_u\sqrt{\frac{\log p}{n_1}},\;
\|\hat\beta-\beta\|_2
\le C_\beta\sigma\sqrt{\frac{k_u\log p}{n_1}}
\right\}
\]
holds with probability approaching 1, because $n_1=n/2$ and Conditions~\ref{cdt: linear estimator} and
\ref{cdt: linear variance} give the same high-probability guarantee, after changing only fixed constants.
The result for plug-in intervals in
\Cref{lem:generic-plugin-ci} is then applied with
$\mathcal E_{\rm pi}=\widetilde{\mathcal A}_0$ and $N=n_1$.
Since the estimators are based on the first half of the data $Z^{(1)}$, the event $\widetilde{\mathcal A}_0$ is independent of the second half of data $Z^{(2)}$.

In the following, we condition on $Z^{(1)}$ and assume
$\widetilde{\mathcal A}_0$ happens.

We construct the debiased estimator as follows:
\begin{equation}\label{eq: bias estimator known design}
    \hat L_{0}(Z;\xi)=\xi^\top \hat{\beta}+\frac{1}{n_2}\xi^\top \Sigma_0^{-1}\left(X^{(2)}\right)^\top \left(Y^{(2)} - X^{(2)} \hat{\beta}\right)
\end{equation}
Its estimation error admits the decomposition
    \begin{equation}\label{eq: error decomp known design}
        \hat L_{0}(Z;\xi)-\xi^\top \beta=
        \underbrace{\left(\xi-\frac{1}{n_2}\left(X^{(2)}\right)^\top X^{(2)} \Sigma_0^{-1}\xi\right)^\top (\hat{\beta}-\beta)}_{\mathrm{I}} + \underbrace{\frac{1}{n_2}\xi^\top \Sigma_0^{-1}\left(X^{(2)}\right)^{\top} \varepsilon^{(2)}}_{\mathrm{II}},
    \end{equation}
    where $\varepsilon^{(2)}=Y^{(2)}-X^{(2)}\beta\sim \mathcal{N}(0, \sigma^2 \bfI_{n_2})$.

\textbf{Bounding $\mathrm{I}$ in \eqref{eq: error decomp known design}}: Let
    $$q_{i}^\prime=\xi^\top \Sigma_0^{-1}X_{i\cdot}^{(2)}\left(X_{i\cdot}^{(2)}\right)^\top (\hat{\beta}-\beta).$$
By \Cref{lem: product of subGaussian},
$$
\begin{aligned}
 \|q_{i}^\prime\|_{\psi_1} & \le  2 \left\|\xi^{\top} \Sigma_0^{-1} X_{i \cdot}^{(2)}\right\|_{\psi_2}
 \left\|\left(X_{i \cdot}^{(2)}\right)^{\top}(\hat{\beta}-\beta)\right\|_{\psi_2} \\
 & \le  2\| \Sigma_0^{-1}\xi\| \|(\hat{\beta}-\beta)\| \left\| X_{i \cdot}^{(2)}\right\|_{\psi_2}^2
 \le  2 M_1^2 \|\xi\|_2\|\hat{\beta}-\beta\|_2,
\end{aligned}
$$
which suggests $q_{i}^\prime$ is sub-exponential.
Consequently, there is a constant $C_1>0$ such that conditional on $\widetilde{\mathcal A}_0$, it holds that
    $$\|q_{i}^\prime -\mathbb{E}q_{i}^\prime\|_{\psi_1} \le  C_1\sigma \|\xi\|_2\sqrt{\frac{k_u\log p}{n_1}}.$$
Furthermore, it holds that $\bbE q_{i}^\prime=\xi^\top (\hat{\beta}-\beta)$.

For any $c>0$, applying Lemma~\ref{lem: sub exp sum} to $n_2^{-1}\sum_{i}(q_{i}^\prime-\mathbb{E}q_{i}^\prime)$  with $t=c\sigma\|\xi\|_2\sqrt{\frac{k_u\log p}{ n_1 n_2 }}>0$, we have
    \begin{align*}
        \bbP\left(\bigg|\left(\xi-\frac{1}{n_2}\left(X^{(2)}\right)^\top X^{(2)} \Sigma_0^{-1} \xi\right)^\top (\hat{\beta}-\beta)\bigg|\ge  t\right)
        \le  2\exp\left(-c_0 \min\left(\frac{c^2}{C_1^2}, \frac{c \sqrt{n_2}}{C_1}\right)\right).
    \end{align*}
We fix a value of $c$ large enough such that $2\exp(-c_0 c^2/C_1^2)\le  \al/12$.

Since $n_1 = n_2 = n/2$ under the data splitting and $n \gtrsim k_u\log p$, we have $n_1 \gtrsim k_u\log p$. Hence,
\[
t=c\sigma\|\xi\|_2\sqrt{\frac{k_u\log p}{n_1n_2}}
=\frac{c\sigma\|\xi\|_2}{\sqrt{n_2}}\sqrt{\frac{k_u\log p}{n_1}}
\le  \frac{C_2\sigma\|\xi\|_2}{\sqrt{n_2}},
\]
where $C_2:=c\sup_n\sqrt{\frac{k_u\log p}{n_1}}$ is a finite constant.
Therefore, the event
\begin{align*}
\mathcal{A}_3
=\Bigg\{
\bigg|\Big(\xi-\frac{1}{n_2}(X^{(2)})^\top X^{(2)} \Sigma_0^{-1}\xi\Big)^\top (\hat{\beta}-\beta)\bigg|
\le  C_2\|\xi\|_2\frac{\sigma}{\sqrt{n_2}}
\Bigg\}
\end{align*}
holds with probability at least $1-\alpha/12$ when $n_2\ge  c^2/C_1^2$.

\textbf{Bounding $\mathrm{II}$ in \eqref{eq: error decomp known design}}: Since $\xi^\top \Sigma_0^{-1}X_{i\cdot}^{(2)}\mid \varepsilon^{(2)}\stackrel{{\rm i.i.d.}}{\sim }{\mathcal{N}}(0,\xi^\top \Sigma_0^{-1}\xi),i=1,\cdots,n_2$, we have
    $$\frac{1}{n_2}\xi^\top \Sigma_0^{-1}\left(X^{(2)}\right)^{\top} \varepsilon^{(2)}\mid \varepsilon^{(2)} \stackrel{{\rm i.i.d.}}{\sim} \mathcal{N}\left(0, \frac{\xi^\top \Sigma_0^{-1}\xi}{n_2^2}\|\varepsilon^{(2)}\|_2^2 \right).$$
    Note that $\xi^\top \Sigma_0^{-1}\xi\le  M_1\|\xi\|_2^2$ by the assumption on the eigenvalues of $\Sigma_0$ and $\|\varepsilon^{(2)}\|_2^2/\sigma^2\sim \chi^2(n_2)$. By Lemma~\ref{lem: chi square tail}, we have, with probability at least $1-\al/12$, that
    $$\mathcal{A}_4=\left\{\bigg|\frac{1}{n_2}\xi^\top \Sigma_0^{-1}\left(X^{(2)}\right)^{\top} \varepsilon^{(2)}\bigg|\le  C_3\|\xi\|_2\frac{\sigma}{\sqrt{n_2}}\right\}$$
    holds for some constant $C_3>0$ and $n$ sufficiently large.

\textbf{Synthesis.}
The event $\widetilde{\mathcal A}_0\cap\mathcal{A}_3\cap\mathcal{A}_4$ holds with probability at least
$1-\alpha/2$, and on this event,
\[
|\hat L_{0}(Z;\xi)-\xi^\top\beta|
\le
(C_2+C_3)\|\xi\|_2\frac{\sigma}{\sqrt{n_2}}
\le
1.1(C_2+C_3)\|\xi\|_2\frac{\hat{\sigma}}{\sqrt{n_2}}.
\]
This yields a valid $(1-\alpha/2)$-level debiased confidence interval of length
$O(\sigma\|\xi\|_2/\sqrt{n})$, completing the proof for the debiased confidence interval.
The final step makes use of \Cref{lem:generic-mixed-ci} in the same way as in \Cref{sec: proof upper unknown design}, so we omit the details.

\end{proof}

\subsection{Proof of the upper bound in Theorem~\ref{thm: sparse spiked}}\label{sec: proof upper np hard}

\begin{proof}
We follow the argument at the beginning of \Cref{sec: proof upper design} using the same balanced data split with $n_1=n_2=n/2$ and the same plug-in interval result.
To apply \Cref{lem:generic-mixed-ci}, the only difference is a new debiased confidence interval.

In particular, we need to prove that for a general
loading vector $\xi\in\mathbb{R}^p$, we can construct a $(1-\alpha/2)$-level debiased confidence interval with length of order
\begin{equation}\label{eq:upper-spike-debias-rate}
    \sigma\left(
\frac{\|\xi\|_2}{\sqrt{n}}
+
\sqrt{\sum_{j\le k_u}\xi_j^2}\,
\frac{k_u\log p}{n}
\right).
\end{equation}

Using the first half of the sample $Z^{(1)}$, we compute the lasso estimator
$\hat{\beta}$, the noise level estimator $\hat{\sigma}$, and an estimator
$\hat{\Omega}$ of the precision matrix $\Omega=\Sigma^{-1}$ such that the event
\[
\begin{aligned}
\widetilde{\mathcal A}_0
=
\Big\{&
1.1\hat{\sigma}>\sigma>0.9\hat{\sigma},\;
\|\hat{\beta}-\beta\|_1
\le c_\beta\sigma k_u\sqrt{\tfrac{\log p}{n_1}},\;
\|\hat{\beta}-\beta\|_2
\le C_\beta\sigma\sqrt{\tfrac{k_u\log p}{n_1}},\\
&
\|\hat{\Omega}-\Omega\|_2
\le C_\Omega\sqrt{\tfrac{k_u\log p}{n_1}},\;
\hat{\Omega}\in\Pi_0(k_u,p)
\Big\}
\end{aligned}
\]
holds with probability tending to one.

\begin{lemma}\label{lem: sparse spiked}
If $\Sigma\in\Pi_0(k_u,p)$ and $n\ge c\,k_u\log p$ for a sufficiently large constant $c>0$,
then there exists an estimator $\hat{\Omega}$ such that
\[
\|\hat{\Omega}-\Omega\|_2
\le
C_\Omega\sqrt{\frac{k_u\log p}{n}},
\]
with probability tending to one, where $C_\Omega$ is a constant.
Moreover, $\hat{\Omega}$ coincides with $\bfI_p$
outside an index set of cardinality at most $k_u$.
\end{lemma}
\Cref{lem: sparse spiked} is proven in \Cref{sec: proof sparse spiked}.

Using $\hat{\Omega}$ and the second half of the sample $Z^{(2)}$, define
\begin{equation}\label{eq: debiased center np hard}
\hat L_0(Z;\xi)
=
\xi^\top\hat{\beta}
+
\frac{1}{n_2}\xi^\top\hat{\Omega}
(X^{(2)})^\top
\bigl(Y^{(2)}-X^{(2)}\hat{\beta}\bigr).
\end{equation}
Its error admits the decomposition
\begin{equation}\label{eq: error decomp np hard}
\begin{aligned}
\hat L_0(Z;\xi)-\xi^\top\beta
=&\;
\underbrace{
\Big(\xi-\tfrac{1}{n_2}(X^{(2)})^\top X^{(2)}\Omega\xi\Big)^\top
(\hat{\beta}-\beta)
}_{\mathrm{I}}
+
\underbrace{
\tfrac{1}{n_2}\xi^\top\hat{\Omega}(X^{(2)})^\top\varepsilon^{(2)}
}_{\mathrm{II}}\\
&-
\underbrace{
\tfrac{1}{n_2}\xi^\top(\hat{\Omega}-\Omega)
(X^{(2)})^\top X^{(2)}(\hat{\beta}-\beta)
}_{\mathrm{III}},
\end{aligned}
\end{equation}
where $\varepsilon^{(2)}=Y^{(2)}-X^{(2)}\beta\sim\mathcal{N}(0,\sigma^2 I_{n_2})$.

Terms $\mathrm{I}$ and $\mathrm{II}$ are controlled as in
\Cref{sec: proof upper design}, with high probability. The only difference for
term $\mathrm{II}$ is that \(\Omega=\Sigma_0^{-1}\) is replaced by
\(\widehat\Omega\). This replacement does not affect the argument, since the
analysis only requires spectral-norm control of the precision matrix. Indeed, on
the event \(\widetilde{\mathcal A}_0\), we have
\[
    \|\widehat\Omega-\Omega\|_2
    \lesssim \sqrt{\frac{k_u\log p}{n}},
\]
which implies that \(\|\widehat\Omega\|_2\) is uniformly controlled. Hence the
same proof as in \Cref{sec: proof upper design} applies, yielding a bound of
order
\(
    {\sigma\|\xi\|_2}/{\sqrt n}.
\)

For term $\mathrm{III}$, write
\[
q_i^{\prime\prime}
=
\xi^\top(\hat{\Omega}-\Omega)
X_{i\cdot}^{(2)}(X_{i\cdot}^{(2)})^\top(\hat{\beta}-\beta),
\qquad i=1,\ldots,n_2.
\]
Then
\[
\mathbb E q_i^{\prime\prime}
=
\xi^\top(\hat{\Omega}-\Omega)\Sigma(\hat{\beta}-\beta)
\le
\|\xi^\top(\hat{\Omega}-\Omega)\|_2\,
\|\hat{\beta}-\beta\|_2.
\]
Moreover,
\[
\|q_i^{\prime\prime}\|_{\psi_1}
\le
2\|\xi^\top(\hat{\Omega}-\Omega)X_{i\cdot}^{(2)}\|_{\psi_2}
\|(X_{i\cdot}^{(2)})^\top(\hat{\beta}-\beta)\|_{\psi_2}
\lesssim
\|\xi^\top(\hat{\Omega}-\Omega)\|_2
\|\hat{\beta}-\beta\|_2.
\]

On $\widetilde{\mathcal A}_0$,
$\|\hat{\beta}-\beta\|_2\lesssim
\sigma\sqrt{\tfrac{k_u\log p}{n}}$.
Since $\hat{\Omega}-\Omega$ is supported on at most $2k_u$ rows and columns,
\[
\|\xi^\top(\hat{\Omega}-\Omega)\|_2
\le
\sqrt{\sum_{j\le 2k_u}\xi_j^2}\,
\|\hat{\Omega}-\Omega\|_2
\lesssim
\sqrt{\sum_{j\le k_u}\xi_j^2}
\sqrt{\tfrac{k_u\log p}{n}}.
\]
Hence
\[
\|q_i^{\prime\prime}\|_{\psi_1}
\lesssim
\sigma\sqrt{\sum_{j\le k_u}\xi_j^2}\,
\frac{k_u\log p}{n}.
\]

Applying Lemma~\ref{lem: sub exp sum} to
$n_2^{-1}\sum_i(q_i^{\prime\prime}-\mathbb E q_i^{\prime\prime})$ with
$t=c\sigma\sqrt{\sum_{j\le k_u}\xi_j^2}\frac{k_u\log p}{n}$ for sufficiently
large $c$, we obtain
\[
\left|
\frac{1}{n_2}\xi^\top(\hat{\Omega}-\Omega)
(X^{(2)})^\top X^{(2)}(\hat{\beta}-\beta)
\right|
\lesssim
\sigma\sqrt{\sum_{j\le k_u}\xi_j^2}\,
\frac{k_u\log p}{n}
\]
with high probability.

Combining the bounds for $\mathrm{I}$-$\mathrm{III}$ yields the claimed rate in \Cref{eq:upper-spike-debias-rate}.

\bigskip
We now complete the proof of the upper bound claimed in \Cref{thm: sparse spiked} by applying \Cref{lem:generic-mixed-ci} with
\[
\mathcal P_u
:=
\Bigl\{
\theta=(\beta,\Sigma,\sigma)\in\Theta(k_u):
\Sigma\in\Pi_0(k_u,p)
\Bigr\}.
\]
The condition on the plug-in interval follows from
\Cref{lem:generic-plugin-ci} with \(N=n_1\).
The preceding debiased construction, with \(v\) in place of \(\xi\), verifies the condition on
the debiased interval over \(\mathcal P_u\).
For \(v=\xi_m^{(1)}\), the debiased
radius is bounded by
\[
R_{\rm db}^{\rm spike}(\xi_m^{(1)})
\lesssim
\sigma\left\{
\frac{(\sum_{j\le m}\xi_j^2)^{1/2}}{\sqrt n}
+
\nu_2\frac{k_u\log p}{n}
\right\},
\]
since
\(\sum_{j\le k_u}(\xi_m^{(1)})_j^2\le \sum_{j\le k_u}\xi_j^2=\nu_2^2\).
Moreover,
\[
\mathcal P_u\cap\Theta(k_u;\xi,t_0)
=
\Theta^{\text{spike}}(k_u;\xi,t_0),
\qquad
\Theta_{\pm\tau}^{\text{spike}}(k;\xi,t_0)\subseteq \mathcal P_u
\]
for \(1\le k\le k_u\). Thus \Cref{lem:generic-mixed-ci} gives, for every fixed
deterministic \(m\),

\[
\tau_{\rm adap}^{\rm spike}(k_u,k;\xi)
\lesssim
\frac{\left(\sum_{j\le m}\xi_j^2\right)^{1/2}}{\sqrt n}
+
|\xi_{m+1}|k_u\sqrt{\frac{\log p}{n}}
+
\nu_2\frac{k_u\log p}{n},
\]
where we have used $\sigma\le  M_2$ for $\theta\in \Theta(k)$.
Optimizing over $m$ and using \Cref{prop: nu1-expression} gives
\[
\inf_{0\le m\le p}
\left\{
    \frac{\left(\sum_{j\le m}\xi_j^2\right)^{1/2}}{\sqrt n}
    +
    |\xi_{m+1}|k_u\sqrt{\frac{\log p}{n}}
\right\}
\lesssim_{\log}
\frac{\nu_1}{\sqrt n}.
\]
Thus, we have
\[
    \tau_{\rm adap}^{\rm spike}(k_u,k;\xi)
    \lesssim_{\log}
     \frac{\nu_1}{\sqrt n}
    +
     \nu_2\frac{k_u\log p}{n}.
\]
This completes the proof.
\end{proof}

\section{Proof of the lower bounds}\label{sec: proof lower}
In this section, we present the proofs of the lower bounds stated in \Cref{thm: hypothesis,thm: computational lower bound,thm: reduction,thm: hypothesis design}.
Our proof strategy largely follows the framework developed in \cite{cai2019optimal,bradic2022testability}. Specifically, for a given candidate parameter under the alternative hypothesis, we construct a prior over the null space and show that the corresponding $\chi^2$-divergences between their induced mixture distribution and that of the alternative parameter are sufficiently small. This implies that the power of any valid test must necessarily be limited, thereby establishing the desired lower bound.

We summarize below several key technical tools used in the proof. For a probability measure $\pi$ over the parameter space $\Theta(k_u)$, we denote
$$\bbP^n_\pi=\int \bbP^n_\theta\rmd \pi(\theta).$$

For two probability measures \(P, Q\) defined on the same measurable space \((\mathcal{X}, \mathcal{U})\), we define:
\[
\mathrm{TV}(P,Q) \;=\; \sup_{B \in \mathcal{U}} \big| P(B) - Q(B) \big| ,
\]
the \emph{total variation distance} and
\[
\chi^2(P \,\|\, Q) \;=\; \int \!\left( \frac{dP}{dQ} - 1 \right)^{\!2} dQ ,
\]
the \emph{chi-square divergence}.
\begin{lemma}\label{lem: power inequality}
    For any test $\psi_*$ and a probability measure $\pi$ over the parameter space $\Theta(k_u)$, we have
    $$\left|\bbE_{\pi}\psi_* -\bbE_{\theta_*}\psi_*\right|\le  {\rm TV}(\bbP^n_\pi, \bbP^n_{\theta_*})\le  \frac{1}{2}\sqrt{\chi^2(\bbP^n_\pi \,\|\, \bbP^n_{\theta_*})}.$$
\end{lemma}

\subsection{Proof of the lower bound in Theorem \ref{thm: hypothesis}}\label{sec: proof lower ultra sparse}
There are two quantities in the lower bound of Theorem \ref{thm: hypothesis}. The quantity involving $\nu_1$ appears in the lower bound of \Cref{thm: hypothesis design} and will be proved in \Cref{sec: proof design lower}. In this section, we focus on proving the lower bound involving $\nu_2$.
\begin{proof}
Since \(k_u\to\infty\), we assume $k_u\ge  4$.

Following the definition of $\tau_{\mathrm{adap}}(k_u, k; \xi, t_0)$ in \eqref{eq: adaptive separation}, we aim to show that for any constant $c > 0$, there exists a constant $c' > 0$ such that for $\tau = c' \nu_2 k_u \tfrac{\log p}{n}$ and for any test $\psi$, it holds that
\begin{equation}\label{eq: lower dense lower bound}
    \bbE_{\theta_*}\psi
    \le
    \sup_{\theta \in \Theta(k_u; \xi, t_0)} \bbE_{\theta}\psi + c,
\end{equation}
where $\theta_* = (\beta_*, \Sigma_*, \sigma_*) \in \Theta_{\pm\tau}(k; \xi, t_0)$ is a fixed alternative point to be specified later.

According to \Cref{lem: power inequality}, it suffices to construct a point $\theta_* = (\beta_*, \Sigma_*, \sigma_*) \in \Theta_{\pm\tau}(k; \xi, t_0)$ and a prior distribution $\pi$ over $\Theta(k_u; \xi, t_0)$ such that
\begin{equation}\label{eq: simpe target for lower bound}
\chi^2(\bbP^n_\pi \,\|\, \bbP^n_{\theta_*}) \le  4c^2 .
\end{equation}

To facilitate the construction, we introduce a transformation of the parameter space.
For any $k \ge 1$ and any $t_0,\tau$, choose an index
$j_0\in \operatorname{supp}(\xi)$ and set
\[
\beta_0 = \frac{t_0-\tau}{\xi_{j_0}} e_{j_0} \in \mathbb{R}^p,
\]
with the convention \(\beta_0=0\) when \(t_0-\tau=0\) and $e_{j_0}$ is the $j_0$-th standard basis.
Then
\(\|\beta_0\|_0\le 1\le k\) and \(\xi^\top\beta_0=t_0-\tau\).
Define the transition mapping $\varphi$ on the parameter space as
\begin{equation}\label{eq: lower bound transform on theta}
    \varphi(\beta, \Sigma, \sigma) = (\beta - \beta_0, \Sigma, \sigma).
\end{equation}
Its inverse is given by $\varphi^{-1}(\tilde{\beta}, \Sigma, \sigma) = (\tilde{\beta} + \beta_0, \Sigma, \sigma).$
We have the following properties:
\begin{itemize}
\item
When changing the parameter from $\theta$ to
$\varphi(\theta)$,
the linear functional changes from $L(\beta;\xi)$ to $L(\beta;\xi)-t_0+\tau$.

\item When changing the parameter from $\theta$ to
$\varphi(\theta)$,
the induced transformation on the data and noise is
\[
(Y,X,\varepsilon)\mapsto (Y-X\beta_0,X,\varepsilon),
\]
which is bijective.
Therefore $\varphi$ preserves the total variation and $\chi^2$-divergence between the induced distributions.
\item
The pre-image of the point \((0,\Sigma,\sigma)\) is  \((\beta_0,\Sigma,\sigma)\) and belongs to
\(\Theta_{\pm\tau}(k;\xi,t_0)\).
\item Since $k_u\ge  4$, the pre-image of \(\Theta(k_u/2;\xi,\tau)\) is contained in
\(\Theta(k_u;\xi,t_0)\).
\end{itemize}

Therefore, it suffices to establish \Cref{eq: simpe target for lower bound} in the
$\varphi$-transformed space with $\beta_* = 0$ and $\pi$ supported on
$\Theta(k_u/2; \xi, \tau)$. We can then apply the inverse transformation $\varphi^{-1}$
to obtain the desired construction in the pre-transformed parameter space.

Let $p_1 = \lfloor k_u/4\rfloor$ and $p_2=p-\lfloor k_u/4\rfloor$.
Let $S_1 = [p_1] $ and $S_2 = [p] \setminus S_1$; $p_1$ and $p_2$ are the sizes of $S_1$ and $S_2$, respectively.
Under the Gaussian design model, $Z_i=(Y_i,X_{i\cdot})\in \bbR^{p+1}$ follows a joint Gaussian distribution with mean $0$. Let $\Sigma^z$ denote the covariance of $Z_i$. Decompose $\Sigma^{z}$ into blocks $\begin{pmatrix} \Sigma_{yy}^{z}& \left(\Sigma_{xy}^{z}\right)^{\top}\\ \Sigma_{xy}^{z}& \Sigma_{xx}^{z} \end{pmatrix},$ where $\Sigma_{yy}^{z}$, $\Sigma_{xx}^{z}$ and $\Sigma_{xy}^{z}$ denote the variance of $y_i$, the variance of $X_{i\cdot}$ and the covariance of $y_i$ and $X_{i\cdot}$, respectively.
We define the function $h : \Sigma^z \rightarrow \theta=\left(\beta,\Sigma,\sigma\right)$  as
\[
h(\Sigma^z)=\left(
\left(\Sigma_{xx}^{z}\right)^{-1}\Sigma_{xy}^{z},
\Sigma_{xx}^{z},
\sqrt{\Sigma_{yy}^{z}-\left(\Sigma_{xy}^{z}\right)^{\top}\left(\Sigma_{xx}^{z}\right)^{-1}\Sigma_{xy}^{z}}
\right).
\]
The function $h$ is bijective and its inverse mapping $h^{-1}: \theta=\left(\beta,\Sigma,\sigma\right) \rightarrow \Sigma^z$ is $$h^{-1}\left(\left(\beta,\Sigma,\sigma\right)\right)=\begin{pmatrix} \beta^{\top}\Sigma\beta+\sigma^2& \beta^{\top}\Sigma\\ \Sigma\beta& \Sigma\end{pmatrix}.$$

\textit{Step 1: Construct the least favorable prior.}\quad
We set the alternative point as
$$\theta_*=(\beta_*=\bfz_{p}, \Sigma_*=\bfI_{p\times p}, \sigma_*=M_2/2)$$ and we have
\begin{equation}\label{eq: alter matrix}
    \Sigma_*^z=h^{-1}(\theta_*)=\left(\begin{array}{c|c|c}
       \sigma_*^2&\bfz_{1\times p_1}& {\bfz}_{1\times p_2}\\ \hline
       \bfz_{p_1\times 1}&\bfI_{p_1\times p_1}&\bfz_{p_1\times p_2}\\ \hline
       \bfz_{p_2\times 1}&\bfz_{p_2\times p_1}&\bfI_{p_2\times p_2}
    \end{array}\right)
\end{equation}
Let $I$ be chosen randomly and uniformly from all subsets of $[p_2]$ with size $p_1$. We consider the random vectors $\alphab_2\in \bbR^{p_2}$ defined coordinate-wise as
\begin{equation}\label{eq: moderate sparse delta}
    (\alphab_2)_j=c_1\sign(\xi_{p_1+j})\sqrt{\frac{\log p}{n}}\ind\left\{j\in I\right\},\quad \forall j\in [p_2],
\end{equation}
where $c_1>0$ is a small positive constant to be specified later.
We further define a fixed vector $\alphab_1\in \bbR^{p_1}$ by
$$(\alphab_1)_j=-\frac{\xi_j}{\sqrt{\sum_{i \in S_1} \xi_i^2}},\quad\forall j\in S_1.$$

Given $\alphab_2$, we can construct the following corresponding covariance matrix:
\begin{equation}\label{eq: null point mat}
    \Sigma^z=
    \left(
    \begin{array}{c|c|c}
        \sigma_*^2 & \bfz_{1\times p_1}&\kappa\alphab_2^\top \\ \hline
        \bfz_{p_1\times 1}&
        \bfI_{p_1\times p_1} & \alphab_1\alphab_2^\top \\
            \hline
        \kappa\alphab_2&  \alphab_2\alphab_1^\top & \bfI_{p_2\times p_2}
    \end{array}
    \right)\stackrel{\triangle}{=}g_1(\alphab_2),
\end{equation}
where $\kappa = \kappa_1(\alphab_2) \in \mathbb{R}$ is a function of $\alphab_2$ defined such that $\Sigma^z$ corresponds to a point in the transformed null space, i.e., $h(\Sigma^z) \in \Theta(k_u/2; \xi, \tau)$.
We will construct the prior in a way such that this $\kappa$ exists a.s.
Furthermore, since the null space imposes a linear constraint on $\beta$, whenever such a $\kappa$ exists, it is uniquely determined since $\tau>0$.

Denote by $\pi_1$ the distribution of $\alphab_2$ defined in \eqref{eq: moderate sparse delta} and $\pi_{}$ the induced prior of $\pi_1$ under the mapping $h\circ g_1$.
The following lemma shows that $\pi_{}$ is supported on $\Theta(k_u/2; \xi, \tau)$, which is the least favorable prior we aim to construct.

\begin{lemma}\label{lem: valid covariance}
If we choose $c_1$ sufficiently small, then there exists a constant $c_2>0$ such that the induced prior $\pi_{}$ is supported on $\Theta(k_u/2;\xi,\tau)$, where $\tau = c_2 \nu_2 \frac{k_u \log p}{n}$, and moreover the associated coefficient $\kappa=\kappa_1(\alphab_2)$ satisfies $0<\kappa\le 1$ $\pi_1$-almost surely.
\end{lemma}
The proof is given in \Cref{sec: proof valid covariance}.

\medskip

\textit{Step 2: Control the $\chi^2$-divergence.}
\quad
For the $\chi^2$-divergence, by Fubini's theorem, we have
\begin{equation}\label{eq: chi square decomp}
\begin{aligned}
    \bbE_{\theta_*}\left(\frac{\rmd \bbP^n_{\pi_{}}}{\rmd \bbP^n_{\theta_*}}-1\right)^2&=\bbE_{\theta_*}\left[\left(\frac{\rmd \bbP^n_{\pi_{}}}{\rmd \bbP^n_{\theta_*}}\right)^2\right]-1
    =\bbE_{(\theta,\tilde{\theta})\sim \pi_{}\otimes \pi_{}}\int_{\bbR^n}\frac{\rmd \bbP^n_{\theta}\rmd \bbP^n_{\tilde{\theta}}}{\rmd \bbP^n_{\theta_*}}-1.
\end{aligned}
\end{equation}

The following lemma provides an upper bound for the integral in \eqref{eq: chi square decomp}, and
its proof is given in \Cref{sec: proof chi square integral}.
\begin{lemma}\label{lem: chi square integral 1}
For any pair of parameters $(\alphab_2,\tilde{\alphab}_2)$ over the support of ${\pi}_1$ constructed in \Cref{lem: valid covariance}, let $\theta = h(g_1( \alphab_2))$ and $\tilde{\theta} = h(g_1( \tilde{\alphab}_2))$.
If $c_1$ is chosen sufficiently small, there exists some constants $c_3>0$ only depending on $\sigma_*$ such that
    \begin{equation*}
        \bbE_{\theta_*}\left(\frac{\rmd \bbP^n_{\pi_{}}}{\rmd \bbP^n_{\theta_*}}-1\right)^2
    \le  \bbE_{(\alphab_2,\tilde{\alphab}_2)\sim \pi_1\otimes \pi_1} \exp\left(c_3 n\alphab_2^\top \tilde{\alphab}_2\right)-1.
    \end{equation*}
\end{lemma}

The following lemma is useful in controlling the right-hand side of the above equation in \Cref{lem: chi square integral 1}, and its proof is given in \Cref{sec: proof hypergeometric}.
\begin{lemma}\label{lem: hypergeometric}
    Let $J$ be a $\mathrm{Hypergeometric}(p, k, k)$ variable with
\[
\mathbb{P}(J = j) = \frac{\binom{k}{j} \binom{p - k}{k - j}}{\binom{p}{k}}.
\]
If $k\le  p^\gamma$ for some constant $\gamma\in [0,1/2)$, then for constant $c\in (0,1-2\gamma)$, we have
\begin{equation}\label{eq: hypergeo bound}
    \lim_{p\to\infty}\mathbb{E}\big[\exp(c\log p \cdot J)\big]=1.
\end{equation}
\end{lemma}
Let $I$ and $\tilde{I}$ denote the supports of $\alphab_2$ and $\tilde{\alphab}_2$, respectively.
By the construction of $\alphab_2$ in \eqref{eq: moderate sparse delta}, both $I$ and $\tilde{I}$ are independently and uniformly sampled from all subsets of $[p_2]$ of size $p_1$.
Consequently, the intersection size $|I \cap \tilde{I}|$ follows a $\mathrm{Hypergeometric}(p_2, p_1, p_1)$ distribution.
Since
\[
\alphab_2^\top \tilde{\alphab}_2
= c_1^{2} \frac{\log p}{n} \, |I \cap \tilde{I}|,
\]
we may apply \Cref{lem: hypergeometric}.
Recall the sparsity condition $p_1 \lesssim k_u \le p_2^\gamma$ for some $\gamma \in (0, 1/2)$ from \Cref{cdt: sparsity assumption}.
\Cref{lem: chi square integral 1,lem: hypergeometric} together imply that
\begin{align*}
    \bbE_{\theta_*}\!\left(\frac{\rmd \bbP_{\pi_{}}^n}{\rmd \bbP_{\theta_*}^n} - 1\right)^{2}
    &\le
        \bbE \exp\!\left( c_3 c_1^{2} \log p \, |I \cap \tilde{I}| \right) - 1
      = o(1),
\end{align*}
where $c_3$ is the constant from \Cref{lem: chi square integral 1} and $c_1$ is chosen sufficiently small such that $c_3c_1^2<1-2\gamma$.
\end{proof}

\subsection{Proof of the lower bound in Theorem \ref{thm: hypothesis design}}\label{sec: proof design lower}
\begin{proof}

It suffices to prove the lower bound for the identity covariance case. Indeed,
if \(\Sigma_0\) is diagonal and satisfies the eigenvalue condition in
\eqref{eq: parameter space}, then a coordinatewise rescaling transforms the
design covariance to \(\mathbf I_p\). Specifically, with
\(\widetilde X=X\Sigma_0^{-1/2}\), \(\widetilde\beta=\Sigma_0^{1/2}\beta\), and
\(\widetilde\xi=\Sigma_0^{-1/2}\xi\), we have
\[
    \xi^\top\beta=\widetilde\xi^\top\widetilde\beta,
    \qquad
    \|\widetilde\beta\|_0=\|\beta\|_0,
\]
and the bounded eigenvalues of \(\Sigma_0\) imply that the loading-profile
quantities for \(\widetilde\xi\) and \(\xi\) are equivalent up to constants.
Therefore, the diagonal known-covariance case reduces to the identity
covariance case up to constant factors. We focus below on \(\Sigma_0=\mathbf I_p\).

We follow the same argument in \Cref{sec: proof lower ultra sparse} and consider the transition mapping $\varphi$ defined in \Cref{eq: lower bound transform on theta}.
In the transformed alternative space, we fix
$\theta_* = (\mathbf{0}_{p\times 1}, \mathbf{I}_{p\times p}, \sigma_*)$ with the
corresponding matrix $\Sigma_*^{z}$ defined in \eqref{eq: alter matrix}.
Suppose that we can construct a prior $\pi$ satisfying
\[
\pi\!\left( \Theta(k_u/2; \xi, \tau) \right) \ge 1 - \frac{c}{2},
\qquad\text{and}\qquad
\chi^{2}\!\left( \mathbb{P}_{\pi}^{n} \,\|\, \mathbb{P}_{\theta_*}^{n} \right)
\le c^{2},
\]
where $\tau=c\nu_1/\sqrt{n}$ for some constant $c > 0$ to be specified later.
Define $\tilde{\pi}$ as the restriction of $\pi$ to a subset of $\Theta(k_u/2;\xi,\tau)$.
Then, we have
\begin{equation}\label{eq: restriction TV}
\begin{aligned}
    \mathrm{TV}\!\left(\mathbb{P}_{\tilde{\pi}}^{n},\, \mathbb{P}_{\theta_*}^{n}\right)
    &\le
    \mathrm{TV}\!\left(\mathbb{P}_{\pi}^{n},\, \mathbb{P}_{\theta_*}^{n}\right)
    +
    \mathrm{TV}\!\left(\mathbb{P}_{\pi}^{n},\, \mathbb{P}_{\tilde{\pi}}^{n}\right) \\
    &\le
    \mathrm{TV}(\tilde{\pi}, \pi)
    +
    \frac{1}{2}\sqrt{
        \chi^{2}\!\left( \mathbb{P}_{\pi}^{n} \,\|\, \mathbb{P}_{\theta_*}^{n} \right)
    } \\
    &\le c,
\end{aligned}
\end{equation}
and the lower bound can be established by \Cref{lem: power inequality} and the properties of $\varphi$.

In the following, we focus on the construction of $\pi$.
We first assume that
$\nu_1 \ge C_4 |\xi_1|$ for some sufficiently large constant $C_4>0$.
The complementary case $\nu_1 < C_4 |\xi_1|$ is technically simpler;
we therefore defer its analysis to the end of this section.

\textit{Step 1: Construct the least favorable prior.}\quad
We define \(\pi_3\) to be the probability measure of $\alphab \in \bbR^p$ constructed as follows.
Let $k_\xi=\|\xi\|_0$ and let $\alphab_j = 0$ for all $j \notin [k_\xi]$.
For each $j \in [k_\xi]$, we specify the coordinates independently as
$$
\begin{aligned}
\alphab_j &= \frac{c_5}{\sqrt{n}}\, b^{(1)}_j \gamma^{(1)}_j,
\end{aligned}
$$
where $b^{(1)}_j={\rm Ber}(q^{(1)}_j)$,
$$q^{(1)}_j=c_4\cdot\frac{ |\xi_j|\exp(-\lambda^2/\xi_j^2)}{\sqrt{\sum_{i=1}^p \xi_i^2\exp(-\lambda^2/\xi_i^2)}},\quad\text{and}\quad \gamma^{(1)}_j=\left\{\begin{array}{lc}
   {\rm sign}(\xi_j)& j\le  j_1,\\
   {\lambda}/{\xi_j}& j> j_1,
\end{array}\right.$$
where $c_4,c_5 > 0$ are some constants to be specified later, $\lambda$ is defined in \eqref{eq: equation sol}, and $j_1=\max\left\{j\in [p]:|\xi_j|\ge  \lambda\right\}$.

Given $\alphab$, we construct the covariance matrix as
\begin{equation}\label{eq: null point mat 2}
    \Sigma^z=
    \left(
    \begin{array}{c|c}
        \sigma_*^2 & \kappa\alphab^\top \\ \hline
        \kappa\alphab& \bfI_{p\times p}
    \end{array}
    \right)\stackrel{\triangle}{=}g_2(\alphab),
\end{equation}
where $\kappa$ is a value to be chosen so that either $\Sigma^z$ is the identity or the above
matrix corresponds to a parameter in the transformed null space.
Concretely, for any $\tau>0$, define
\[
\mathcal{G}_\tau=\{ \alphab\in \mathbb{R}^p : \xi^\top \alphab \ge  \tau, ~\|\alphab\|_0\le  k_u/2,
~\|\alphab\|_2^2\le  (\xi^\top \alphab)^2\sigma_*^2/(2\tau^2)\}.
\]
Given the value of $\tau>0$, we set
\begin{equation}\label{eq: identity cov kappa choice}
  \kappa=  \begin{cases}
        \frac{\tau}{\xi^\top \alphab}& \quad \alphab\in \mathcal{G}_\tau, \\
        0 & \quad \text{otherwise}.
    \end{cases}
\end{equation}
Then $\kappa\in [0,1]$.
Furthermore, when $\pi_3(\mathcal{G}_\tau)>0$, denote the restriction of $\pi_3$ on $\mathcal{G}_\tau$ by $\pi_3(\cdot\mid G_\tau)$.

Define
$\pi_4$ and $\widetilde{\pi}_4$ as the pushforward of $\pi_3$ and $\pi_3(\cdot\mid G_\tau)$, respectively, under the map \(h\circ g_2\).

The next lemma states that for the case $\nu_1 \ge C_4 |\xi_1|$, there exists a choice of the constants \(c_4,c_5>0\) and a corresponding \(c_6\) such that $\mathcal{G}_\tau$ holds with a probability close to 1 and $\widetilde{\pi}_4$ is supported on the transformed null space \(\Theta(k_u/2;\xi,c_6\nu_1/\sqrt{n})\).
\begin{lemma}\label{lem: valid covariance 2}
Consider $\alphab\sim \pi_3$.
For any constants $c>0$ and $c_5\in (0,1)$, one can choose the constants $c_4$ sufficiently small, $C_4$ sufficiently large, and $c_6$ such that for \(\tau=c_6\nu_1/\sqrt n\), if $\nu_1\ge  C_4|\xi_1|$ and $k_u\ge  C_4$, then
\[
\pi_3(G_\tau)\ge 1-c/2,
\]
and $\widetilde{\pi}_4$ is supported on \(\Theta(k_u/2;\xi,\tau)\).

\end{lemma}

\textit{Step 2: Control the $\chi^2$-divergence.}\quad
The null prior is given by $\widetilde{\pi}_4$.
Since total variation distance cannot
increase under a measurable pushforward, we have
$$
\mathrm{TV}(\widetilde{\pi}_4, \pi_4) \le   \mathrm{TV}\left(\pi_3, \pi_3(\cdot\mid \mathbf{G}_\tau) \right) = 1 - \pi_3(\mathbf{G}_\tau) \le c/2,
$$
where the last inequality follows from \Cref{lem: valid covariance 2}.
By \Cref{eq: restriction TV}, it remains to show that the $\chi^2$-divergence between the mixture distribution induced by $\pi_4$ and that of $\theta_*$ is small. Similar to \Cref{lem: chi square integral 1}, we have the following result.
\begin{lemma}\label{lem: chi square integral 2}
    For any pair of $\alphab$ and $\tilde{\alphab}$ sampled independently from ${\pi}_3$, let $\theta = h(g_2( \alphab))$ and $\tilde{\theta} = h(g_2( \tilde{\alphab}))$. If $c_4$ is chosen sufficiently small, there exists some constant $c_7>0$ only depending on $\sigma_*$ such that
    \begin{equation*}
        \bbE_{\theta_*}\left(\frac{\rmd \bbP^n_{\pi_4}}{\rmd \bbP^n_{\theta_*}}-1\right)^2
    \le  \bbE_{(\alphab,\tilde{\alphab})\sim \pi_3\otimes \pi_3} \exp\left(c_7 n\alphab^\top \tilde{\alphab}\right)-1.
    \end{equation*}
\end{lemma}
For $(\alphab,\tilde{\alphab})\sim \pi_3^{\otimes 2}$, we have
$$
\forall j\in [k_\xi], \qquad \alphab_j\tilde{\alphab}_j=\left\{
    \begin{array}{cc}
        c_5^2(\gamma^{(1)}_j)^2/n & \text{w.p. } (q^{(1)}_j)^2,\\
        0 & \text{w.p. } 1-(q^{(1)}_j)^2,\\
    \end{array}
\right.$$
and $\alphab_j\tilde{\alphab}_j=0$ for all $j>k_\xi$. Therefore, we have
\begin{align*}
    \bbE_{(\alphab,\tilde{\alphab})\sim \pi_3\otimes \pi_3} \exp\left(c_7 n\alphab^\top \tilde{\alphab}\right)&=\prod_{j=1}^p \bbE_{(\alphab_j,\tilde{\alphab}_j)\sim \pi_3\otimes \pi_3} \exp\left(c_7 n\alphab_j \tilde{\alphab}_j\right)\\
    &=\prod_{j=1}^{k_\xi} \left[1+(q^{(1)}_j)^2\left(\exp\left(c_7 c_5^2(\gamma^{(1)}_j)^2\right)-1\right)\right]\\
    &\le  \exp\left(\sum_{j=1}^{k_\xi} (q^{(1)}_j)^2\left(\exp\left(c_7 c_5^2(\gamma^{(1)}_j)^2\right)-1\right)\right).
\end{align*}
By the definition of $q^{(1)}_j$ and $\gamma^{(1)}_j$, we have
\begin{align*}
    &\sum_{j=1}^{k_\xi} (q^{(1)}_j)^2\left(\exp\left(c_7c_5^2 (\gamma^{(1)}_j)^2\right)-1\right)\\
    \le  &c_4^2\frac{ \sum_{j\le  j_1} \xi_j^2\exp(-2\lambda^2/\xi_j^2)\left(\exp\left(c_7 c_5^2\right)-1\right)+\sum_{j>j_1}\xi_j^2\exp(-2\lambda^2/\xi_j^2)\left(\exp(c_7 c_5^2\lambda^2/\xi_j^2)-1\right)}{\sum_{i=1}^{k_\xi} \xi_i^2\exp(-\lambda^2/\xi_i^2)}.
\end{align*}
If we choose $c_5$ sufficiently small such that $c_7 c_5^2<\ln 2$, then we have
\begin{align*}
    \exp(-2\lambda^2/\xi_j^2)\left(\exp\left(c_7 c_5^2\right)-1\right)&\le  \exp(-\lambda^2/\xi_j^2),\quad \forall j\le  j_1,\\
    \exp(-2\lambda^2/\xi_j^2)\left(\exp\left(c_7 c_5^2\lambda^2/\xi_j^2\right)-1\right)&\le  \exp(-\lambda^2/\xi_j^2),\quad \forall j>j_1.
\end{align*}
Therefore, we have
$$\bbE_{(\alphab,\tilde{\alphab})\sim \pi_3\otimes \pi_3} \exp\left(c_7 n\alphab^\top \tilde{\alphab}\right)\le  \exp\left(c_4^2\right).$$
By choosing $c_4$ sufficiently small such that $c_4^2 \in (0, \ln(1+c^2))$, we have $\chi^{2}\!\left( \mathbb{P}_{\pi}^{n} \,\|\, \mathbb{P}_{\theta_*}^{n} \right)
\le c^{2}$.

\bigskip
In the remaining case where $\nu_1 < C_4 |\xi_1|$, we only need to prove that $\tau_{\mathrm{adap}}(k_u; \xi, t_0) \gtrsim |\xi_1|/\sqrt{n}$. In this case, we choose the point mass prior at $\theta_o$ with
$$\Sigma^z_o=\left(\begin{array}{c|c}
    \sigma_*^2& \frac{\kappa_0}{\sqrt{n}}\sign(\xi_1){\bf e}_1^\top \\ \hline
    \frac{\kappa_0}{\sqrt{n}} \sign(\xi_1){\bf e}_1& \bfI_{p\times p}
\end{array}\right),$$
where ${\bf e}_1$ is the first standard basis vector, i.e.,
$$\left({\bf e}_1\right)_j=\left\{
    \begin{array}{cc}
        1,&j=1\\
        0,&j\neq 1
    \end{array}
\right.,$$
and $\kappa_0>0$ is some constant specified later. It is easy to verify that the above matrix corresponds to a point in the transformed null space with $\tau=\kappa_0|\xi_1|/\sqrt{n}$. By an argument similar to that in \Cref{lem: chi square integral 2}, we have
$$\chi^2(\bbP_{\theta_o}^n\|\bbP_{\theta_*}^n) = \exp(\kappa_0^2)-1.$$
Therefore we finish the proof by choosing the constant $\kappa_0$ sufficiently small.
\end{proof}

\subsection{Proof of Theorem~\ref{thm: computational lower bound}}\label{sec: proof computational lower bound}
\begin{proof}
Without loss of generality we assume $8\le  2k_u<k_{\rm eff}$ and $$\sum_{k_u<j\le  k_{\rm eff}}\xi_j^2\ge  \sum_{j\le  k_u}\xi_j^2.$$
In the complementary case, the low-degree lower bound in \Cref{thm: computational lower bound} is an immediate consequence of the statistical lower bound in \Cref{thm: hypothesis}, so no additional proof is needed. Since \(D\lesssim p\) and the standing dimensional assumptions
\(\sqrt n/\log p\lesssim k_u \lesssim p^\gamma\) imply
\(\log n\lesssim \log p\), we have
\(
    \log(6npD)\lesssim \log p .
\)
Consequently,
\(
    D\log(6npD)\lesssim D\log p .
\)
On the other hand, by the definition of \(k_{\mathrm{eff}}\),
\[
    k_{\mathrm{eff}}
    \le
    \frac{k_u^2}{D\log p}.
\]
Therefore,
\[
    \frac{k_u^2}{k_{\mathrm{eff}}}
    \gtrsim
    D\log p
    \gtrsim
    D\log(6npD).
\]
Thus the exponent appearing in the Chernoff bound for the event
\(\mathcal C^c\) is at least of order \(D\log(6npD)\), which is sufficient for
the subsequent union bound over degree-\(D\) polynomial terms.

The proof is closely analogous to the argument in \Cref{sec: proof lower ultra sparse}. We again take the alternative point
\(
\theta_* = (\mathbf{0}_{p\times 1}, \mathbf{I}_{p\times p}, \sigma_*)
\)
together with the associated matrix $\Sigma_*^{z}$ defined in \eqref{eq: alter matrix}.
Our strategy is to construct a prior $\pi$ that places overwhelming probability mass on the set
\(
\Theta(\lfloor k_u/2\rfloor; \xi, \tau)
\)
with
\(
\tau = c\, \nu_3\, k_u \frac{\log p}{n}
\)
for a constant $c>0$ to be chosen later.
We then show that the low-degree likelihood ratio satisfies
\(
\mathrm{LD}(D) = 1 + o(1)
\)
for
\(
\mathbb{Q}_1 = \mathbb{P}_{\pi}^{n}
\) and \(
\mathbb{Q}_0 = \mathbb{P}_{\theta_*}^{n}.
\)
The conclusion of the theorem then follows immediately from \Cref{prop:low-degree-hardness}.

\textit{Step 1: Construct the least favorable prior.}\quad
Let $S_3 = [ k_{\rm eff}] $, $S_4 = [p] \setminus [k_{\rm eff}]$, and $S_5=[k_{\rm eff}]\setminus [k_u]$. Denote by $p_3$, $p_4$, and $p_5$ the sizes of $S_3$, $S_4$, and $S_5$, respectively; that is, $p_3 = k_{\rm eff}$, $p_4=p-k_{\rm eff}$, and $p_5=k_{\rm eff}-k_u$.
Let $I$ be chosen randomly and uniformly from all subsets of $[p_4]$ with size $\lfloor k_u/4\rfloor$.
We consider the random vectors $\alphab_1\in \bbR^{p_4},\alphab_2\in \bbR^{p_3}$ defined as
\begin{equation}\label{eq: moderate sparse delta 2}
  \left\{  \begin{aligned}
        (\alphab_1)_j&=c_8\sign(\xi_{p_3+j})\sqrt{\frac{\log p}{n}}\ind\left\{j\in I\right\},\quad \forall j\in [p_4], \\
        (\alphab_2)_j&=-\frac{\sqrt{p_5}}{k_u}\sign(\xi_j) b^{(2)}_j,\quad \forall j\in S_3,
    \end{aligned}\right.
\end{equation}
for a small enough constant $c_8>0$ to be specified later, where $b^{(2)}_j$ are independent Bernoulli variables with parameters $$q^{(2)}_j=\frac{|\xi_j|}{8\sqrt{\sum_{i\in S_5}\xi_i^2}}\cdot \frac{k_u}{\sqrt{p_5}}\ind\left\{j\in S_5\right\},\quad \forall j\in S_3.$$
Specifically, we have $\sum_{i\in S_5}\xi_i^2\ge  \sum_{i\le  k_u}\xi_i^2\ge  k_u\xi_{k_u}^2$ and $p_5>k_u$ by our assumption.
Therefore, we have $$q^{(2)}_j\le  \frac{1}{8}\sqrt{\frac{k_u}{p_5}}< \frac{1}{8}, \qquad \forall j\in S_5$$
and $q^{(2)}_j=0$ for all $j\in S_3\setminus S_5$. This verifies that the above construction of $q_j^{(2)}$ is well-defined.

Given a pair $(\alphab_1, \alphab_2)$, we can construct the following corresponding covariance matrix:
\begin{equation}\label{eq: null point mat 3}
    \Sigma^z=
    \left(
    \begin{array}{c|c|c}
        \sigma_*^2 & \bfz_{1\times p_3}&\kappa\alphab_1^\top \\ \hline
        \bfz_{p_3\times 1}&
        \bfI_{p_3\times p_3} & \alphab_2\alphab_1^\top \\
            \hline
        \kappa\alphab_1&  \alphab_1\alphab_2^\top & \bfI_{p_4\times p_4}
    \end{array}
    \right)\stackrel{\triangle}{=}g_3(\alphab_1,\alphab_2),
\end{equation}
where $\kappa = \kappa_3(\alphab_1,\alphab_2) \in \mathbb{R}$ is defined such that the above matrix corresponds to a point in the transformed null space, i.e., $h(\Sigma^z) \in \Theta(k_u/2; \xi, \tau)$ for some nonzero $\tau$ specified later. We additionally require $\kappa$ to be bounded by $1$. If no such value exists, we set $\kappa = 0$. This definition is well-posed since the null space imposes a linear constraint on $\beta$, and whenever such a $\kappa$ exists, it is uniquely determined.

Denote by $\pi_5$ the joint distribution of $(\alphab_1,\alphab_2)$ defined in \eqref{eq: moderate sparse delta 2} and $\pi_6$ the induced prior of $\pi_5$ under the mapping $h\circ g_3$.
The following lemma shows that $\pi_6$ is mostly supported on the null space.
\begin{lemma}\label{lem: valid covariance 3}
If we choose $c_8$ sufficiently small, then there exists a constant $c_9>0$ such that
$\pi_6$ assigns probability $1-o(1)$ to the transformed null space $\Theta(\lfloor k_u/2\rfloor;\xi,c_ 9\nu_3\frac{k_u\log p}{n})$ and $\kappa_3\in (0,1]$.
\end{lemma}
The proof is given in \Cref{sec: proof valid covariance 3}.

The low-degree comparison must be made with a prior supported on the null
space.
For the mapping
\[
    \theta(\alphab_1,\alphab_2)
    :=
    h\circ g_3(\alphab_1,\alphab_2),
\]
define the validity event
\[
    \mathcal A_n
    =
    \left\{
        (\alphab_1,\alphab_2):
        \theta(\alphab_1,\alphab_2)
        \in
        \Theta\left(k_u;\xi,c_9\nu_3\frac{k_u\log p}{n}\right)
        \ \text{and}\ 0<\kappa_3(\alphab_1,\alphab_2)\le 1
    \right\}.
\]
By \Cref{lem: valid covariance 3}, the exception probability rate $\varepsilon_n:=\pi_5(\mathcal A_n^c)$ satisfies $\varepsilon_n=o(1)$.
Let $\pi_7$ be the restricted probability measure of $\pi_5$ on $\mathcal A_n$, i.e., the conditional distribution of $(\alphab_1,\alphab_2)$ given
$\mathcal A_n$. Define
\[
    \pi_8
    :=
    \theta_{\#}\pi_7
    =
    (h\circ g_3)_{\#}\pi_7 .
\]
Then $\pi_8$ is supported on
\[
    \Theta\left(k_u;\xi,c_9\nu_3\frac{k_u\log p}{n}\right)
\]
and is the least favorable prior we aim to construct.

We also verify here that the same construction is supported on the sparse signed-spiked
covariance class used in \Cref{sec: np hard}.  On the validity event
\(\mathcal A_n\), the \(X\)-covariance block of \(g_3(\alphab_1,\alphab_2)\) is
\[
\Sigma(\alphab_1,\alphab_2)
=
\begin{pmatrix}
\mathbf I_{p_3} & \alphab_2\alphab_1^\top\\
\alphab_1\alphab_2^\top & \mathbf I_{p_4}
\end{pmatrix}
=
\mathbf I_p+
\begin{pmatrix}
0 & \alphab_2\alphab_1^\top\\
\alphab_1\alphab_2^\top & 0
\end{pmatrix}.
\]
The perturbation has rank at most two and is supported on
\(\supp(\alphab_2)\cup(k_{\rm eff}+\supp(\alphab_1))\).  The construction gives
\(\|\alphab_1\|_0=\lfloor k_u/4\rfloor\), and the proof of
\Cref{lem: valid covariance 3} gives \(\|\alphab_2\|_0\le k_u/4\) with
probability \(1-o(1)\).  Conditioning additionally on this latter event changes
the restricted prior only by a \(1+o(1)\) factor, by the same comparison as in
\eqref{eq:restriction-measure-comparison}.  Hence we may take the null prior
\(\pi_8\) to be supported on
\[
\Theta^{\mathrm{spike}}\!\left(k_u;\xi,c_9\nu_3\frac{k_u\log p}{n}\right).
\]
The alternative point \(\theta_*=(\mathbf 0,\mathbf I_p,\sigma_*)\) also has
covariance in \(\Pi_0(k,p)\).  Therefore the low-degree likelihood-ratio bound
proved below applies without change to the sparse signed-spiked null and alternative
spaces.

Moreover, for any nonnegative measurable function \(F\),
\begin{equation}\label{eq:restriction-measure-comparison}
\begin{aligned}
\mathbb E_{\pi_7^{\otimes2}}F
&=
\frac{
\mathbb E_{\pi_5^{\otimes2}}
\left[
F\,
\ind({\mathcal A_n})(\alphab_1,\alphab_2)
\ind({\mathcal A_n})(\tilde\alphab_1,\tilde\alphab_2)
\right]
}{(1-\varepsilon_n)^2} \\
&\le
\frac{1}{(1-\varepsilon_n)^2}
\mathbb E_{\pi_5^{\otimes2}}F
=
(1+o(1))\mathbb E_{\pi_5^{\otimes2}}F .
\end{aligned}
\end{equation}
This comparison will be used below to evaluate bounds under the simpler
unrestricted latent prior \(\pi_5\).

\textit{Step 2: Control $\mathrm{LD}(D)$.} \quad Set
\[
    \mathbb Q_0=\mathbb P^n_{\theta_*},
    \qquad
    \mathbb Q_1=\mathbb P^n_{\pi_8}
    =
    \int \mathbb P^n_{\theta(\alphab_1,\alphab_2)}
        \,\pi_7(d\alphab_1,d\alphab_2).
\]
For
\[
    L_\theta=\frac{d\mathbb P^n_\theta}{d\mathbb P^n_{\theta_*}},
\]
the likelihood ratio of $\mathbb Q_1$ with respect to $\mathbb Q_0$ is
\[
    L
    =
    \frac{d\mathbb Q_1}{d\mathbb Q_0}
    =
    \mathbb E_{(\alphab_1,\alphab_2)\sim\pi_7}
    L_{\theta(\alphab_1,\alphab_2)} .
\]
By linearity of the orthogonal projection onto polynomials of degree at most
$D$,
\[
    L^{\le D}
    =
    \mathbb E_{(\alphab_1,\alphab_2)\sim\pi_7}
    L_{\theta(\alphab_1,\alphab_2)}^{\le D}.
\]
Therefore,
\[
\begin{aligned}
    \mathrm{LD}(D)
    &=
    \|L^{\le D}\|_{L^2(\mathbb Q_0)}^2                                      \\
    &=
    \mathbb E_{(\alphab_1,\alphab_2),(\tilde\alphab_1,\tilde\alphab_2)
        \sim\pi_7^{\otimes 2}}
    \mathbb E_{\theta_*}
    \left[
        L_{\theta(\alphab_1,\alphab_2)}^{\le D}
        L_{\theta(\tilde\alphab_1,\tilde\alphab_2)}^{\le D}
    \right].
\end{aligned}
\]

The expectation above is with respect to the restricted latent prior
\(\pi_7^{\otimes2}\).  Hence both parameter pairs lie in the validity event
\(\mathcal A_n\), so the Gaussian covariance matrices are well defined and
\Cref{lem: low degree integral} applies. The following lemma provides some properties for the $\bbE_{\theta_*}(L_\theta^{\le D}L_{\tilde{\theta}}^{\le D})$.

\begin{lemma}\label{lem: low degree integral}
Let \((\alphab_1,\alphab_2)\) and
\((\tilde{\alphab}_1,\tilde{\alphab}_2)\) belong to the validity event
\(\mathcal A_n\), and define
\[
    \theta=h(g_3(\alphab_1,\alphab_2)),
    \qquad
    \tilde\theta=h(g_3(\tilde{\alphab}_1,\tilde{\alphab}_2)).
\]
Then
\[
    \bbE_{\theta_*}\!\left(
        L_\theta^{\le D}L_{\tilde\theta}^{\le D}
    \right)
    \le
    \bbE_{\theta_*}\!\left(
        L_\theta L_{\tilde\theta}
    \right),
\]
and
\[
    \bbE_{\theta_*}\!
        \left[L_\theta^{\le D}\right]^2
    \le
    9(6npD)^{4D}.
\]
\end{lemma}

The proof of \Cref{lem: low degree integral} is provided in \Cref{sec: proof low degree integral}.

We next split the low-degree second moment according to whether the two latent covariance perturbations have unusually large overlap. Write
\[
\mathcal{C}
    = \left\{
        (\alphab_1, \alphab_2, \tilde{\alphab}_1, \tilde{\alphab}_2)
        : \alphab_2^{\top} \tilde{\alphab}_2 \le  C_5
      \right\},
\]
where $C_5$ is an absolute constant to be specified later.
Its corresponding
event on the parameter space is the pushforward event
\[
\widetilde{\mathcal C}
=
\left\{
    \bigl(
        \theta(\alphab_1,\alphab_2),
        \theta(\tilde\alphab_1,\tilde\alphab_2)
    \bigr):
    (\alphab_1,\alphab_2,\tilde\alphab_1,\tilde\alphab_2)\in\mathcal C
\right\},
\]
where we recall that
\(
    \theta(\alphab_1,\alphab_2):=h(g_3(\alphab_1,\alphab_2)).
\)
Since \(\pi_8=\theta_\#\pi_7\), expectations over
\(\pi_8^{\otimes2}\) can be evaluated through the latent prior
\(\pi_7^{\otimes2}\). In particular, for any nonnegative measurable function
\(F\),
\[
\begin{aligned}
&\mathbb E_{\pi_8^{\otimes2}}
\left[
    F(\theta,\tilde\theta)
  \ind(\widetilde{\mathcal C})(\theta,\tilde\theta)
\right] \\
&\qquad =
\mathbb E_{\pi_7^{\otimes2}}
\left[
    F\!\left(
        \theta(\alphab_1,\alphab_2),
        \theta(\tilde\alphab_1,\tilde\alphab_2)
    \right)
   \ind(\mathcal C)
    (\alphab_1,\alphab_2,\tilde\alphab_1,\tilde\alphab_2)
\right].
\end{aligned}
\]
Thus the likelihood-ratio terms are naturally written on the parameter space,
whereas the probability of \(\mathcal C\) and \(\mathcal C^c\) is controlled on
the latent space.

Define
\[
    T_{\widetilde{\mathcal C}}
    :=
    \mathbb E_{\pi_8^{\otimes2}}
    \left[
      \ind(\widetilde{\mathcal C})
        \mathbb E_{\theta_*}
        \left(
            L_\theta^{\le D}L_{\tilde\theta}^{\le D}
        \right)
    \right],
\]
and
\[
    T_{\widetilde{\mathcal C}^c}
    :=
    \mathbb E_{\pi_8^{\otimes2}}
    \left[
        \ind(\widetilde{\mathcal C}^c)
        \mathbb E_{\theta_*}
        \left(
            L_\theta^{\le D}L_{\tilde\theta}^{\le D}
        \right)
    \right].
\]
Then
\[
    \mathrm{LD}(D)
    =
    T_{\widetilde{\mathcal C}}
    +
    T_{\widetilde{\mathcal C}^c}.
\]

\textbf{For the term $T_{\widetilde{\mathcal C}^c}$:}
By Cauchy's inequality and \Cref{lem: low degree integral},
\[
\begin{aligned}
    T_{\widetilde{\mathcal C}^c}
    &\le
    \mathbb E_{\pi_8^{\otimes2}}
    \left[
        \ind(\widetilde{\mathcal C}^c)
        \left\{
            \mathbb E_{\theta_*}\left(L_\theta^{\le D}\right)^2
            \mathbb E_{\theta_*}\left(L_{\tilde\theta}^{\le D}\right)^2
        \right\}^{1/2}
    \right] \\
    &\le
9(6npD)^{4D}\,
\mathbb P_{\pi_7^{\otimes2}}(\mathcal C^c) \\
&\le
9(6npD)^{4D}(1+o(1))\,
\mathbb P_{\pi_5^{\otimes2}}(\mathcal C^c),
\end{aligned}
\]
where the last inequality follows from
\eqref{eq:restriction-measure-comparison}.
Let
\[
    J
    :=
    \sum_{j\in S_5} b_j^{(2)}\tilde b_j^{(2)} .
\]
Since $q_j^{(2)}=0$ for $j\in S_3\setminus S_5$, this is also the sum over
$S_3$.  By the definition of $\alphab_2$,
\[
    \alphab_2^\top\tilde\alphab_2
    =
    \frac{p_5}{k_u^2}J .
\]
The variables $b_j^{(2)}\tilde b_j^{(2)}$ are independent Bernoulli random
variables with success probabilities $(q_j^{(2)})^2$, and hence
\[
    \mu_1
    :=
    \mathbb E_{\pi_5^{\otimes 2}}J
    =
    \sum_{j\in S_5}(q_j^{(2)})^2
    =
    \frac{k_u^2}{64p_5}
    \ge
    \frac{k_u^2}{64k_{\rm eff}}.
\]
Therefore,
\[
    \mathcal C^c
    =
    \left\{
        \alphab_2^\top\tilde\alphab_2>C_5
    \right\}
    =
    \left\{
        J>C_5\frac{k_u^2}{p_5}
    \right\}
    =
    \left\{
        J>64C_5\mu_1
    \right\}.
\]
Thus $\mathcal C^c$ is an upper-tail event.  For $a>1$, let
\[
    \psi(a)=a\log a-a+1 .
\]
The Chernoff bound for sums of independent Bernoulli random variables gives
\[
    \mathbb P_{\pi_5^{\otimes 2}}(J\ge a\mu_1)
    \le
    \exp\{-\mu_1\psi(a)\}.
\]
Taking $a=64C_5$ and using the comparison between $\pi_7$ and $\pi_5$ in \Cref{eq:restriction-measure-comparison}, we have
\[
\begin{aligned}
    \mathbb P_{\pi_7^{\otimes 2}}(\mathcal C^c)
    &\le
    (1-\varepsilon_n)^{-2}
    \mathbb P_{\pi_5^{\otimes 2}}(\mathcal C^c)                              \\
    &\le
    (1+o(1))
    \exp\{-\mu_1\psi(64C_5)\}.
\end{aligned}
\]
Consequently,
\[
   T_{\widetilde{\mathcal C}^c}
    \le
    9(1+o(1))
    \exp\left\{
        4D\log(6npD)
        -
        \mu_1\psi(64C_5)
    \right\}.
\]
The growth condition above gives
\[
    \mu_1\ge c_\mu D\log(6npD)
\]
for some absolute constant $c_\mu>0$.  Choose \(C_5\) large enough so that
\(c_\mu\psi(64C_5)>8\). Then
\[
    T_{\widetilde{\mathcal C}^c}
    \le
    9(1+o(1))
    \exp\{-4D\log(6npD)\}
    =
    o(1).
\]
We now fix such a value of \(C_5\). The constant \(c_8\) in the prior
construction will be chosen sufficiently small below, after \(C_5\) is fixed.

\textbf{For the term $T_{\widetilde{\mathcal C}}$:}
On this event, the overlap between the two \(\alphab_2\)-perturbations is bounded by \(C_5\), so the projected likelihood-ratio inner product can be compared with the full likelihood-ratio inner product. By the first inequality in \Cref{lem: low degree integral}, for every parameter pair in the support of
\(\pi_8^{\otimes2}\),
\[
    \mathbb E_{\theta_*}
    \left(
        L_\theta^{\le D}L_{\tilde\theta}^{\le D}
    \right)
    \le
    \mathbb E_{\theta_*}
    \left(
        L_\theta L_{\tilde\theta}
    \right).
\]
Hence
\[
    T_{\widetilde{\mathcal C}}\le \mathbb E_{\pi_8^{\otimes2}}\left[
        \ind(\widetilde{\mathcal C})
        \mathbb E_{\theta_*}
        \left(
            L_\theta L_{\tilde\theta}
        \right)
    \right].
\]
Using the pushforward representation, the right-hand side equals
\[
    \mathbb E_{\pi_7^{\otimes2}}
    \left[
        \ind({\mathcal C})
        \mathbb E_{\theta_*}
        \left(
            L_{\theta(\alphab_1,\alphab_2)}
            L_{\theta(\tilde\alphab_1,\tilde\alphab_2)}
        \right)
    \right].
\]
As in \Cref{lem: chi square integral 1}, for
$\theta=\theta(\alphab_1,\alphab_2)$ and
$\tilde\theta=\theta(\tilde\alphab_1,\tilde\alphab_2)$,
\[
    \mathbb E_{\theta_*}
    \left(
        L_\theta L_{\tilde\theta}
    \right)
    =
    \left[
        1-
        \left(
            \frac{\kappa\tilde\kappa}{\sigma_*^2}
            +
            \alphab_2^\top\tilde\alphab_2
        \right)
        \alphab_1^\top\tilde\alphab_1
    \right]^{-n}.
\]
On $\mathcal C$ and on $\mathcal A_n$, we have
\[
    0<\kappa,\tilde\kappa\le 1,
    \qquad
    \alphab_2^\top\tilde\alphab_2\le C_5 .
\]
Therefore, with
\[
    A_5:=C_5+\sigma_*^{-2},
\]
we obtain
\[
    T_{\mathcal C}
    \le
    \mathbb E_{\pi_7^{\otimes 2}}
    \left[
        \left(
            1-A_5\alphab_1^\top\tilde\alphab_1
        \right)^{-n}
    \right],
\]
where the indicator is dropped because the integrand is nonnegative. Using \eqref{eq:restriction-measure-comparison},
\[
    T_{\mathcal C}
    \le
    (1-\varepsilon_n)^{-2}
    \mathbb E_{\pi_5^{\otimes 2}}
    \left[
        \left(
            1-A_5\alphab_1^\top\tilde\alphab_1
        \right)^{-n}
    \right].
\]
This is where $\varepsilon_n=o(1)$ is needed.

Let
\[
    s_1=\lfloor k_u/4\rfloor,
    \qquad
    H=|I\cap \tilde I|.
\]
Under $\pi_5^{\otimes 2}$,
\[
    H\sim {\rm Hypergeometric}(p_4,s_1,s_1),
\]
and by the construction of $\alphab_1$,
\[
    \alphab_1^\top\tilde\alphab_1
    =
    c_8^2\frac{\log p}{n}H.
\]
Choose $c_8>0$ sufficiently small, after $C_5$ has been fixed, so that
\[
    A_5c_8^2\frac{\log p}{n}s_1\le \frac12
\]
and
\[
    2A_5c_8^2<1-2\gamma',
\]
where $\gamma'<1/2$ satisfies $s_1\le p_4^{\gamma'}$ for all sufficiently large
$n$.  Such a $\gamma'$ exists because $k_u\le p^\gamma$ with $\gamma<1/2$ and
$p_4=p-k_{\rm eff}\sim p$.
Then
\[
    0\le A_5\alphab_1^\top\tilde\alphab_1\le \frac12,
\]
and $(1-x)^{-1}\le \exp(2x)$ for $0\le x\le 1/2$ gives
\[
\begin{aligned}
    \left(
        1-A_5\alphab_1^\top\tilde\alphab_1
    \right)^{-n}
    &\le
    \exp\left(
        2nA_5c_8^2\frac{\log p}{n}H
    \right)                                                               \\
    &=
    \exp\left(
        2A_5c_8^2\log p\cdot H
    \right).
\end{aligned}
\]
By \Cref{lem: hypergeometric}, since $2A_5c_8^2<1-2\gamma'$,
\[
    \mathbb E_{\pi_5^{\otimes 2}}
    \exp\left(
        2A_5c_8^2\log p\cdot H
    \right)
    =
    1+o(1).
\]
Therefore,
\[
    T_{\widetilde{\mathcal C}}
    \le
    (1-\varepsilon_n)^{-2}(1+o(1))
    =
    1+o(1).
\]
Combining this bound with $T_{\mathcal C^c}=o(1)$ gives
\[
    \mathrm{LD}(D)
    =
    T_{\mathcal C}+T_{\mathcal C^c}
    \le
    1+o(1).
\]
Since the degree-$D$ polynomial space contains the constant function,
\[
    \mathrm{LD}(D)\ge \|1\|_{L^2(\mathbb Q_0)}^2=1.
\]
Consequently,
\[
    \mathrm{LD}(D)=1+o(1).
\]
By \Cref{prop:low-degree-hardness}, no degree-$D$ polynomial weakly separates
$\mathbb Q_1=\mathbb P^n_{\pi_8}$ from
$\mathbb Q_0=\mathbb P^n_{\theta_*}$.  Since $\pi_8$ is supported on the null
space and $\theta_*$ lies in the alternative space at separation
\[
    c_9\nu_3\frac{k_u\log p}{n},
\]
the desired low-degree lower bound follows.

\end{proof}

\subsection{Proof of Theorem~\ref{thm: reduction}}\label{sec: proof reduction}
\begin{proof}
The proof is based on constructing a polynomial-time reduction from a sparse
canonical correlation analysis (SCCA) detection problem to the linear hypothesis
testing problem considered in this paper. For a
generic detection problem $\mathcal{P}$, we use
$\mathcal{L}_{H_0}(X)$ and $\mathcal{L}_{H_1}(X)$ to denote the distributions of
the instance $X$ under the null $H_0$ and the alternative $H_1$, respectively.
The reduction principle we rely on is summarized in the following lemma, which
is adapted from \cite[Lemma~3.1]{brennan2018reducibility}.

\begin{lemma}\label{lem: reduction}
Let $\mathcal{P}$ and $\mathcal{P}'$ be detection problems with hypotheses
$(H_0, H_1)$ and $(H_0', H_1')$, and let $X$ and $Y$ be instances of
$\mathcal{P}$ and $\mathcal{P}'$, respectively.  Suppose there exists a
polynomial-time computable map $\phi$ and a prior $\pi$ on $H_1^\prime$ such that
\[
d_{\mathrm{TV}}\!\left(\mathcal{L}_{H_0}(\phi(X)), \mathcal{L}_{H_0'}(Y)\right)
\;+\;\inf
d_{\mathrm{TV}}\!\left(
\mathcal{L}_{H_1}(\phi(X)),
\int_{H_1^\prime}\mathcal{L}_{\bbP^\prime}(Y)\rmd \pi(\bbP^\prime)
\right)
\;\le\; \delta .
\]
If there exists a polynomial-time algorithm solving $\mathcal{P}'$ with
Type~I+II error at most $\epsilon$, then there exists a polynomial-time
algorithm solving $\mathcal{P}$ with Type~I+II error at most $\epsilon+\delta$.
\end{lemma}
The proof of \Cref{lem: reduction} follows directly from the definition of total
variation distance and is therefore omitted.

For notational convenience, we set
\(
p_6 = k_\xi,\; p_7 = p-k_\xi,
\)
and
\(
k_* = \lfloor k_u/4 \rfloor.
\)
Fix \(\sigma_*=M_2/2\), so
\((\mathbf{0}_{p\times1},\mathbf{I}_{p\times p},\sigma_*)\in
\Theta(\lfloor k_u/2\rfloor)\) by \eqref{eq: parameter space}.
Let \(\tau_{\rm red}\) be the exact value of \(\xi^\top\beta\) produced by
the \(h\)-transformation under the alternative covariance in
\eqref{eq: reduction alternative}; it is computed in
\eqref{eq: reduction tau red}.  Since \(\xi_1=1\), set
\[
    \beta_0=(t_0-\tau_{\rm red})e_1,
    \qquad
    \varphi(\beta,\Sigma,\sigma)=(\beta-\beta_0,\Sigma,\sigma).
\]
This is the transition argument from \Cref{sec: proof lower ultra sparse},
with \(j_0=1\) and \(\tau\) replaced by \(\tau_{\rm red}\).  The associated
data transformation is \((Y,X)\mapsto(Y-X\beta_0,X)\), and its inverse sends a
translated sample \(Z'=(Y',X)\) to \((Y'+X\beta_0,X)\).
For the present reduction we need two consequences.  First, the inverse image
of \((\mathbf 0,\mathbf I,\sigma_*)\) is
\((\beta_0,\mathbf I,\sigma_*)\), whose linear functional is
\(t_0-\tau_{\rm red}\).  Once \(c\) is chosen so that
\(\tau\le\tau_{\rm red}\), this inverse image belongs to the LT alternative
class \(\Theta_{\pm\tau}(k;\xi,t_0)\).  Second, if
\(\tilde\theta\in\Theta(\lfloor k_u/2\rfloor;\xi,\tau_{\rm red})\), then
\[
\xi^\top(\tilde\beta+\beta_0)
=
\tau_{\rm red}+t_0-\tau_{\rm red}
=t_0,
\]
and the sparsity increases by at most one, so
\(\varphi^{-1}(\tilde\theta)\in\Theta(k_u;\xi,t_0)\).
Thus the SCCA null is sent to an LT alternative point, whereas the SCCA
alternative is sent to the LT null class.  If \(\psi\) solves
\({\rm LT}(n,k_u,k,k_\xi,p,\tau)\), the SCCA decision uses
\[
1-\psi(Y'+X\beta_0,X).
\]
It remains to construct an exact reduction from SCCA to
\begin{equation}\label{eq: reduction target}
H_0^{\rm tr}:\theta=(\mathbf0,\mathbf I,\sigma_*),
\qquad
H_1^{\rm tr}:\theta\in
\Theta(\lfloor k_u/2\rfloor;\xi,\tau_{\rm red}).
\end{equation}
The prior on \(H_1^{\rm tr}\) is induced by the random supports of
\(\widetilde{\alphab}_1\) and \(\widetilde{\alphab}_2\) in
\eqref{eq: reduction alternative}; the construction below gives exact law
matching under both hypotheses, hence the total-variation slack is
\(\delta=0\).

We now specialize to the regime
\(
\tau \asymp \rho_n^2 k_u / \sqrt{k_\xi}
\)
for some $\rho_n \le 1$.
Consider the Gaussian testing problem with $V_1\in\mathbb{R}, V_2\in\mathbb{R}^{p_6}$, and $V_3\in\mathbb{R}^{p_7}$:
\begin{equation}\label{eq: reduction alternative}
\begin{aligned}
& H_0:\;
\begin{pmatrix}
V_1\\ V_2\\ V_3
\end{pmatrix}
&\sim
\mathcal{N}\!\left(
{\bfz},
\left(\begin{array}{c|c|c} \sigma_*^2 & {{\bf0}}_{1\times p_6} & {\bf0}_{1\times p_7} \\ \hline {{\bf0}}_{p_6\times 1} & {\bfI}_{p_6\times p_6} & {\bf0}_{p_6\times p_7} \\ \hline {\bf0}_{p_7\times 1} & {\bf0}_{p_7\times p_6} & {\bfI}_{p_7\times p_7} \end{array}\right)\right)\\
& \text{vs.}
\\[0.5em]
& H_1:\;
\begin{pmatrix}
V_1\\ V_2\\ V_3
\end{pmatrix}
&\sim \renewcommand{\arraystretch}{1.4}
\mathcal{N}\!\left(
{\bfz},
\left(\begin{array}{c|c|c}
\sigma_*^2 & \mathbf{0}_{1\times p_6} & \widetilde{\alphab}_2^\top \\
\hline
\mathbf{0}_{p_6\times 1} & \mathbf{I}_{p_6} & \widetilde{\alphab}_1\widetilde{\alphab}_2^\top \\
\hline
\widetilde{\alphab}_2 & \widetilde{\alphab}_2\widetilde{\alphab}_1^\top & \mathbf{I}_{p_7}
\end{array}\right)
\right),
\end{aligned}
\end{equation}
where $\widetilde{\alphab}_1\in\mathbb{R}^{p_6}$ and
$\widetilde{\alphab}_2\in\mathbb{R}^{p_7}$ are $k_*$-sparse vectors with supports
drawn uniformly at random from all subsets of $[p_6]$ and $[p_7]$ of size $k_*$,
respectively.
Their nonzero entries are defined by
\[
(\widetilde{\alphab}_1)_j = -\frac{c_{10}}{\sigma_*}\frac{\sqrt{p_6}}{k_*},
\quad j\in\supp(\widetilde{\alphab}_1),
\qquad
(\widetilde{\alphab}_2)_j
=
\frac{\sigma_*\rho_n}{\sqrt{2p_6}},
\quad j\in\supp(\widetilde{\alphab}_2),
\]
for a sufficiently small $c_{10}>0$.

Under the mapping $h$, the induced parameter
$\theta=(\beta,\Sigma,\sigma)$ corresponding to the covariance matrix in the alternative $H_1$ of
\eqref{eq: reduction alternative} admits the explicit expressions
\begin{align*}
\beta
&=
\begin{pmatrix}
\mathbf{I}_{p_6} & \widetilde{\alphab}_1\widetilde{\alphab}_2^\top \\
\widetilde{\alphab}_2\widetilde{\alphab}_1^\top & \mathbf{I}_{p_7}
\end{pmatrix}^{-1}
\begin{pmatrix}
\mathbf{0}_{p_6} \\
\widetilde{\alphab}_2
\end{pmatrix}
=
\begin{pmatrix}
-\dfrac{\|\widetilde{\alphab}_2\|_2^2}
{1-\|\widetilde{\alphab}_1\|_2^2\|\widetilde{\alphab}_2\|_2^2}
\,\widetilde{\alphab}_1
\\[0.6em]
\dfrac{1}
{1-\|\widetilde{\alphab}_1\|_2^2\|\widetilde{\alphab}_2\|_2^2}
\,\widetilde{\alphab}_2
\end{pmatrix},\\
\Sigma
&=
\begin{pmatrix}
\mathbf{I}_{p_6} & \widetilde{\alphab}_1\widetilde{\alphab}_2^\top \\
\widetilde{\alphab}_2\widetilde{\alphab}_1^\top & \mathbf{I}_{p_7}
\end{pmatrix},
\qquad
\sigma
=
\sqrt{
\sigma_*^2
-
\frac{\|\widetilde{\alphab}_2\|_2^2}
{1-\|\widetilde{\alphab}_1\|_2^2\|\widetilde{\alphab}_2\|_2^2}
}.
\end{align*}

It is immediate that $\beta$ is $2k_*$-sparse, and hence at most
$\lfloor k_u/2\rfloor$-sparse.
By choosing $c_{10}$ sufficiently small, $\Sigma$ and $\sigma$ are well defined and
satisfy the constraints in the parameter space \eqref{eq: parameter space}.
Moreover, since
\(
\|\widetilde{\alphab}_1\|_2^2 = c_{10}^2p_6/(k_*\sigma_*^2)
\)
and
\(
\|\widetilde{\alphab}_2\|_2^2 = \sigma_*^2\rho_n^2 k_*/(2p_6),
\)
we obtain

{
\begin{equation}\label{eq: reduction tau red}
\tau_{\rm red}
:=
\xi^\top\beta
=
\frac{ \frac{c_{10}}{\sigma_*} \sqrt{p_6}\,\|\widetilde{\alphab}_2\|_2^2}
{1-\|\widetilde{\alphab}_1\|_2^2\|\widetilde{\alphab}_2\|_2^2}
=
\frac{c_{10} \sigma_*}{2-c_{10}^2\rho_n^2}\,
\rho_n^2\frac{k_*}{\sqrt{p_6}}.
\end{equation}
Since \(p_6=k_\xi\), \(k_*=\lfloor k_u/4\rfloor\), and \(\rho_n<1/2\), the
denominator in \eqref{eq: reduction tau red} is at most \(2\), and
\(k_*\ge k_u/8\) for all sufficiently large \(n\).  Hence
\[
\tau_{\rm red}
\ge
\frac{c_{10}\sigma_*}{16}\,
\rho_n^2\frac{k_u}{\sqrt{k_\xi}}.
\]
After fixing \(c_{10}\in(0,1)\) sufficiently small for the covariance and noise
constraints above, choose the theorem constant \(c\le c_{10}\sigma_*/16\).  Then
the theorem's separation level
\(\tau=c\rho_n^2 k_u/\sqrt{k_\xi}\) satisfies
\(\tau\le \tau_{\rm red}\), as required in the translated testing problem
\eqref{eq: reduction target}.
}

It remains to construct a polynomial-time mapping from
${\rm SCCA}(2n,k_*,p_6,p_7,\rho_n)$
to the testing problem in \eqref{eq: reduction alternative}.
Since the observations are i.i.d.\ under both models, it suffices to describe
the transformation at the level of a single pair of observations.
Specifically, we describe how two independent samples
$(U_1,U_2)$ and $(U_1',U_2')$ from the SCCA model are mapped to a single
observation $(V_1,V_2,V_3)$.

The mapping is defined as follows.
\begin{itemize}
    \item Compute $W_1 = \sum_{j=1}^{p_6}(U_1')_j/\sqrt{p_6}\in\mathbb{R}$.
    \item Set
    \[
    V_1 = \sigma_* W_1,\qquad
    V_2 = -c_{10} U_1 + \sqrt{1-c_{10}^2}\,V_*,\qquad
    V_3 = (U_2+U_2')/\sqrt{2},
    \]
    where $V_* \sim \mathcal{N}(0,\mathbf{I}_{p_6})$ is independent of all other
    random variables.
\end{itemize}

Then under the null, since $U_1,U_1^\prime\sim \mathcal{N}(0,\mathbf{I}_{p_6})$ and $U_2,U_2^\prime\sim \mathcal{N}(0,\mathbf{I}_{p_7})$ are independent, we have
\[\begin{pmatrix}
V_1\\ V_2\\ V_3
\end{pmatrix}\sim
\mathcal{N}\!\left(
0,
\left(\begin{array}{c|c|c}
    \sigma_*^2 & {{\bf0}}_{1\times p_6} & {\bf0}_{1\times p_7} \\ \hline
    {{\bf0}}_{p_6\times 1} & {\bfI}_{p_6\times p_6} & {\bf0}_{p_6\times p_7} \\ \hline
    {\bf0}_{p_7\times 1} & {\bf0}_{p_7\times p_6} & {\bfI}_{p_7\times p_7} \end{array}\right)\right).
\]
Under the alternative, we first have:
$$\begin{pmatrix}
    U_1\\U_2
\end{pmatrix}\sim
\mathcal{N}\!\left(
    0,
    \begin{pmatrix}
    \mathbf{I}_{p_6}
    &
    \rho_n \alphab_1 \alphab_2^\top
    \\[0.2em]
    \rho_n \alphab_2 \alphab_1^\top
    &
    \mathbf{I}_{p_7}
    \end{pmatrix}
    \right),\qquad
\begin{pmatrix}
    W_1\\U_2^\prime
\end{pmatrix} \sim
\mathcal{N}\!\left(
    0,
    \begin{pmatrix}
    1
    &
    \rho_n\sqrt{\frac{ k_*}{p_6}} \alphab_2^\top
    \\[0.2em]
    \rho_n\sqrt{\frac{ k_*}{p_6}} \alphab_2
    &
    \mathbf{I}_{p_7}
    \end{pmatrix}
    \right)$$
are independent, where we have used the fact that $\alphab_1$ has exactly $k_*$ nonzero entries each equal to $1/\sqrt{k_*}$ under the alternative of SCCA. Therefore, we have
\[
\renewcommand{\arraystretch}{1.2}
\begin{pmatrix}
V_1\\ V_2\\ V_3
\end{pmatrix}\sim \mathcal{N}\!\left(
{\bfz},
\left(\begin{array}{c|c|c}
\sigma_*^2 & \mathbf{0}_{1\times p_6} & \rho_n\sigma_*\sqrt{\frac{k_*}{2p_6}}{\alphab}_2^\top \\
\hline
\mathbf{0}_{p_6\times 1} & \mathbf{I}_{p_6} & -\frac{c_{10}}{\sqrt{2}}\rho_n{\alphab}_1{\alphab}_2^\top \\
\hline
\rho_n\sigma_*\sqrt{\frac{k_*}{2p_6}}{\alphab}_2 & -\frac{c_{10}}{\sqrt{2}}\rho_n{\alphab}_2{\alphab}_1^\top & \mathbf{I}_{p_7}
\end{array}\right)
\right).
\]
Consequently, $(V_1,V_2,V_3)$ follows the alternative distribution in
\eqref{eq: reduction alternative}, completing the reduction.
\end{proof}

\section{Proof of Additional Technical Results for the Upper Bounds}\label{sec: proof lemma upper}

\subsection{Proof of Lemma~\ref{lem: feasibility of optimization}}\label{sec: proof feasibility of optimization}
\begin{proof}[Proof of \Cref{lem: feasibility of optimization}]
    For $u=\Omega\xi$, define
$$
d=\hat{\Sigma}u-\xi=\frac{1}{n}\sum_{i=1}^n X_{i\cdot}X_{i\cdot}^{\top}\Omega\xi-\xi\in\mathbb{R}^p,
$$
with coordinates
$$
d_j=\frac{1}{n}\sum_{i=1}^n X_{ij}X_{i\cdot}^{\top}\Omega\xi-\xi_j,\qquad j\in[p].
$$
 For any $j\in[p]$ and $i\in [n]$, define $f_{ij}=X_{ij}X_{i\cdot}^{\top}\Omega\xi$. Then $\mathbb{E}f_{ij}=\xi_j$, and $f_{ij}$ is sub-exponential.
Indeed, by \Cref{lem: product of subGaussian},
$$
\begin{aligned}
\|f_{ij}\|_{\psi_1}
=\|X_{ij}\,X_{i\cdot}^{\top}\Omega\xi\|_{\psi_1}
&\le 2\,\|X_{ij}\|_{\psi_2}\,\|X_{i\cdot}^{\top}\Omega\xi\|_{\psi_2}.
\end{aligned}
$$
Since $X_{ij}\sim \mathcal{N}(0, \Sigma_{jj})$ and $X_{i\cdot}^{\top}\Omega\xi\sim \mathcal{N}(0, \xi^\top\Omega\xi)$, we have $\|X_{ij}\|_{\psi_2}\le  \sqrt{M_1}$ and $\|X_{i\cdot}^{\top}\Omega\xi\|_{\psi_2}\le  \sqrt{M_1}\|\xi\|$ where $M_1$ is the constant in the assumption on the eigenvalues of $\Sigma$ (see \Cref{eq: parameter space}).

Hence, there exists a constant $C'>0$ such that
$\|f_{ij}\|_{\psi_1}\le C'\|\xi\|_2$.
Furthermore, Lemma~\ref{lem: sub exp center} implies that
\[
\| f_{ij} - \mathbb{E} f_{ij} \|_{\psi_1}
    \le 2 \| f_{ij} \|_{\psi_1}
    \le 2C' \|\xi\|_2 .
\]
We pick a constant $c$ sufficiently large such that $2^{2-c^2c_0}<\al/24$, where  $c_0$ is the constant in Lemma~\ref{lem: sub exp sum}.
For any $j\in [p]$, we apply Lemma~\ref{lem: sub exp sum} with $d_j= n^{-1}\sum_{i=1}^n f_{ij}$ and
\(
t = 2c C' \|\xi\|_2 \sqrt{\tfrac{\log p}{n}}
\)
to conclude that for sufficiently large $n$, it holds that
\[
\mathbb{P}\!\left( |d_j| \ge 2c C' \|\xi\|_2 \sqrt{\frac{\log p}{n}} \right)
    \le 2 \exp\!\left( - c_0 \min\!\left( c^2 \log p,\; c \sqrt{n \log p} \right) \right)
    \le 2 p^{-c^2 c_0},
\]
where the last inequality holds as long as $n>c^2 \log p$, which is guaranteed because
\( n / \log p \gtrsim k_u \to \infty \).
Taking a union bound over all \(j\), we obtain
\[
\mathbb{P}\!\left( \|d\|_\infty \ge 2c C' \|\xi\|_2 \sqrt{\frac{\log p}{n}} \right)
    \le 2 p^{\,1 - c^2 c_0}\le  2^{2-c^2c_0}\le  \al/24.
\]
Choosing $C_\xi=2c C'$ completes the proof.
\end{proof}

\subsection{Proof of Lemma~\ref{lem: sparse spiked}}\label{sec: proof sparse spiked}
\begin{proof}
Before proceeding, we record two concentration tools used repeatedly in the proof.
\begin{lemma}[{\cite[Proposition D.1]{ma2013sparse}}]\label{lem: wishart op}
Let $Y$ be an $n\times k$ matrix with i.i.d.\ $\mathcal{N}(0,1)$ entries.
For any $t>0$,
\[
\mathbb{P}\!\left(
\left\|
\frac{1}{n}Y^\top Y - \bfI_k
\right\|_2
\le
2\left(\sqrt{\frac{k}{n}}+t\right)
+
\left(\sqrt{\frac{k}{n}}+t\right)^2
\right)
\;\ge\;
1-2e^{-nt^2/2}.
\]
\end{lemma}
\begin{lemma}[{\cite{davidson2001local}}]\label{lem: davidson szarek}
    Let A be an $N \times n$ matrix whose entries are independent standard normal random variables. Then for every $t \ge 0$, with probability at least $1 - 2 \exp(-t^2/2)$, one has
    $$\|A\|_2\le  \sqrt{N}+\sqrt{n}+t$$
\end{lemma}
We now prove \Cref{lem: sparse spiked}.
Let
\[
\Sigma = V\Lambda V^\top + \bfI_p,
\qquad
\Lambda=\operatorname{diag}(\lambda_1,\ldots,\lambda_r)\in\mathbb{R}^{r\times r},
\]
and let $B_*=\supp(V)\subseteq[p]$ denote the row support of $V$.
Note that $k:=|B_*|\le  k_u$ for a known upper bound $k_u$. Define the sample covariance matrix based on the
first half of the data by
\begin{equation}\label{eq: sample cov 1}
\hat{\Sigma}^{(1)}=\frac{1}{n_1}\bigl(X^{(1)}\bigr)^\top X^{(1)}.
\end{equation}

For any symmetric matrix $\bfA=(a_{ij})\in\mathbb{R}^{p\times p}$ and any index set $B\subseteq[p]$, define $\Gamma_B(\bfA)$ to be the $p\times p$ matrix whose $B\times B$ principal submatrix equals $\bfA_{BB}$, whose diagonal entries indexed
by $B^c$ are equal to one, and whose remaining off-diagonal entries are zero:
\begin{equation}\label{eq:AB}
(\Gamma_B(\bfA))_{ij}
=
a_{ij}\,\ind\{i\in B,\, j\in B\}
+
\ind\{i=j\in B^c\}.
\end{equation}
Equivalently, after permuting indices so that $B$ appears first,
\[
\Gamma_B(\bfA)
=
\begin{bmatrix}
\bfA_{BB} & \bfz\\
\bfz & \bfI_{B^cB^c}
\end{bmatrix}.
\]
For two index sets $I,J\subseteq[p]$, we write $\bfA_{IJ}$ for the
$|I|\times|J|$ submatrix of $\bfA$ with rows indexed by $I$ and columns indexed by
$J$.

Following \cite{cai2015optimal}, define the candidate support class
\begin{equation}\label{eq:supp-set}
\begin{aligned}
\bbB_{k_u}
=
\Biggl\{
&B\subseteq [p]:~
|B|\le  k_u,
~\text{and for all }D\subseteq B^c\text{ with }|D|\le k_u,\\
&
\|\Gamma_D(\hat{\Sigma}^{(1)})-\bfI_p\|_2
\le
2\left(\sqrt{\frac{|D|}{n_1}}+\sqrt{\frac{\gamma_*|D|\log p}{n_1}}\right)
+\left(\sqrt{\frac{|D|}{n_1}}+\sqrt{\frac{\gamma_*|D|\log p}{n_1}}\right)^2,\\
&\|\hat{\Sigma}^{(1)}_{DB}\|_2
\le
\sqrt{\|\Gamma_B(\hat{\Sigma}^{(1)})\|_2}\biggl(\sqrt{\frac{|D|}{n_1}}+\sqrt{\frac{|B|}{n_1}}+\sqrt{\frac{\gamma_* |D|\log p}{n_1}}\biggr)
\Biggr\}.
\end{aligned}
\end{equation}
(Here $\Gamma_D(\hat{\Sigma}^{(1)})$ and $\Gamma_B(\hat{\Sigma}^{(1)})$ are interpreted in the
sense of \eqref{eq:AB}.)

Define ${\cal E}_1=\{B_*\in\bbB_{k_u}\}$.
\begin{lemma}\label{lem: set A nonempty}
Suppose $\Sigma\in \Pi_0(k,p)$ and $\hat{\Sigma}^{(1)}$ is constructed from
$n_1$ i.i.d.\ samples drawn from $\mathcal{N}(0,\Sigma)$ as in
\eqref{eq: sample cov 1}.
If
\(
n_1\ge  c k_u \log p
\)
for some sufficiently large constant $c>0$,
then for $\gamma_*\ge  3$ and $p\ge  2$, we have
\[
\mathbb{P}\bigl({\cal E}_1\bigr)\ge  1-8 p^{ 1-\gamma_*/2 }.
\]
In particular, on \({\cal E}_1\), the class \(\bbB_{k_u}\) is nonempty.
\end{lemma}

By \Cref{lem: set A nonempty},
\(\mathbb{P}({\cal E}_1)\ge 1-8p^{1-\gamma_*/2}\).
On the event \({\cal E}_1\), the class \(\bbB_{k_u}\) is nonempty and we choose any \(\widehat B\in\bbB_{k_u}\).
We then construct
\begin{equation}\label{eq: spike estimators}
\hat{\Sigma}_{\text{spike}}
:=
\Gamma_{\widehat{B}}(\hat{\Sigma}^{(1)})\ind\left\{{\cal E}_1\right\}+{\bfI}_p\ind\left\{{\cal E}_1^c\right\},
\qquad
\widehat{\Omega}
:=
\hat{\Sigma}_{\text{spike}}^{-1}.
\end{equation}

By triangle inequality,
\begin{align}\label{eq: triangle inequality matrix}
\bigl\|\hat{\Sigma}_{{\text{spike}}}-\Sigma\bigr\|_2 \ind\left\{{\cal E}_1\right\}
&\le
\underbrace{\,\bigl\|\Gamma_{\widehat{B}}(\hat{\Sigma}^{(1)})-\Gamma_{B_*}(\hat{\Sigma}^{(1)})\bigr\|_2\ind\left\{{\cal E}_1\right\}}_{\mathrm{I}}
+
\underbrace{\,\bigl\|\Gamma_{B_*}(\hat{\Sigma}^{(1)})-\Sigma\bigr\|_2\ind\left\{{\cal E}_1\right\}}_{\mathrm{II}}
\end{align}

It remains to control $\|\widehat{\Omega}-\Omega\|_2$ on ${\cal E}_1$.

\textit{Bounding $\mathrm{I}$ in \eqref{eq: triangle inequality matrix}.} Let
\[
G=B_*\cap \widehat{B},\qquad
M=B_*\cap \widehat{B}^{\,c},\qquad
O=B_*^{\,c}\cap \widehat{B}.
\label{eq:decomp-GMO}
\]
These correspond to correctly identified, missing, and overly identified coordinates, respectively.
Writing the matrix in the block order \((G,M,O)\), the difference
\(\Gamma_{\widehat{B}}(\hat{\Sigma}^{(1)})-\Gamma_{B_*}(\hat{\Sigma}^{(1)})\) can be decomposed as a sum of four
block-sparse matrices,
\[
\Gamma_{\widehat{B}}(\hat{\Sigma}^{(1)})-\Gamma_{B_*}(\hat{\Sigma}^{(1)})
=
A_M + A_O + A_{GM} + A_{GO},
\]
where
\[
A_M=
\begin{pmatrix}
0&0&0\\
0&\bfI_{MM}-\hat{\Sigma}^{(1)}_{MM}&0\\
0&0&0
\end{pmatrix},\quad
A_O=
\begin{pmatrix}
0&0&0\\
0&0&0\\
0&0&\hat{\Sigma}^{(1)}_{OO}-\bfI_{OO}
\end{pmatrix},
\]
and
\[
A_{GM}=
\begin{pmatrix}
0&-\hat{\Sigma}^{(1)}_{GM}&0\\
-\hat{\Sigma}^{(1)}_{MG}&0&0\\
0&0&0
\end{pmatrix},\quad
A_{GO}=
\begin{pmatrix}
0&0&\hat{\Sigma}^{(1)}_{GO}\\
0&0&0\\
\hat{\Sigma}^{(1)}_{OG}&0&0
\end{pmatrix}.
\]
By the triangle inequality for the operator norm,
\[
\bigl\|\Gamma_{\widehat{B}}(\hat{\Sigma}^{(1)})-\Gamma_{B_*}(\hat{\Sigma}^{(1)})\bigr\|_2
\le
\|A_M\|_2+\|A_O\|_2+\|A_{GM}\|_2+\|A_{GO}\|_2.
\]
Moreover, since \(A_M\) and \(A_O\) are block-diagonal, we have
\[
\|A_M\|_2=\|\bfI_{MM}-\hat{\Sigma}^{(1)}_{MM}\|_2,
\qquad
\|A_O\|_2=\|\hat{\Sigma}^{(1)}_{OO}-\bfI_{OO}\|_2.
\]
For the off-diagonal blocks, applying the triangle inequality again yields
\[
\left\|
\begin{pmatrix}
0 & X\\
X^\top & 0
\end{pmatrix}
\right\|_2
\le
\left\|
\begin{pmatrix}
0 & X\\
0 & 0
\end{pmatrix}
\right\|_2
+
\left\|
\begin{pmatrix}
0 & 0\\
X^\top & 0
\end{pmatrix}
\right\|_2
=
\|X\|_2+\|X^\top\|_2
=
2\|X\|_2,
\]
and thus \(\|A_{GM}\|_2\le 2\|\hat{\Sigma}^{(1)}_{GM}\|_2\) and
\(\|A_{GO}\|_2\le 2\|\hat{\Sigma}^{(1)}_{GO}\|_2\).

Combining these bounds gives
\begin{align}\label{eq: decomp variance error}
\|\Gamma_{\widehat{B}}(\hat{\Sigma}^{(1)})-\Gamma_{B_*}(\hat{\Sigma}^{(1)})\|_2
&\le
\|\bfI_{MM}-\hat{\Sigma}^{(1)}_{MM}\|_2
+\|\hat{\Sigma}^{(1)}_{OO}-\bfI_{OO}\|_2
+2\|\hat{\Sigma}^{(1)}_{GM}\|_2
+2\|\hat{\Sigma}^{(1)}_{GO}\|_2.
\end{align}

We now define the event
\[
{\cal E}_2=\left\{
\left\|\Gamma_{\widehat{B}}(\hat{\Sigma}^{(1)})
-\Gamma_{B_*}(\hat{\Sigma}^{(1)})\right\|_2
\lesssim C_{E,1}\sqrt{\frac{k_u\log p}{n_1}}
\right\}.
\]
The following analysis of the right-hand side of
\eqref{eq: decomp variance error} establishes that, for a sufficiently large
constant \(C_{E,1}\), the event \({\cal E}_2\) occurs with probability \(1-o(1)\).

\begin{enumerate}
    \item Since $M\subseteq \widehat{B}^{\,c}$ and $|M|\le k_u$, the defining property of $\widehat{B}\in\bbB_{k_u}$ yields, for every such subset $D\subseteq \widehat{B}^{\,c}$ with $|D|\le k_u$,
\[
\|\hat\Sigma^{(1)}_{D}-\bfI_p\|_2\le 2\left(\sqrt{\frac{|D|}{n_1}}+\sqrt{\frac{\gamma_*|D|\log p}{n_1}}\right)
+\left(\sqrt{\frac{|D|}{n_1}}+\sqrt{\frac{\gamma_*|D|\log p}{n_1}}\right)^2,
\]
where, since we have $n_1\ge  c k_u \log p$ for some sufficiently large constant $c$, the leading rate in the right-hand side is $\sqrt{(\gamma_*|D|\log p)/{n_1}}$.

Applying this with $D=M$ gives
\begin{align}\label{eq: leading I in covariance}
\|\hat\Sigma^{(1)}_{MM}-\bfI_{MM}\|_2
=
\|\hat\Sigma^{(1)}_{M}-\bfI_p\|_2
\lesssim \sqrt{\frac{k_u\log p}{n_1}}.
\end{align}

\item
For any fixed $D\subseteq B_*^{\,c}$ with $|D|\le k_u$,
the submatrix $X^{(1)}_{\cdot D}$ has i.i.d.\ $\mathcal N(0,\bfI_{|D|})$ columns, hence
\[
\hat\Sigma^{(1)}_{DD}=\frac{1}{n_1}(X^{(1)}_{\cdot D})^\top X^{(1)}_{\cdot D}
=
\frac{1}{n_1}Z^\top Z,
\]
for $Z\in\mathbb R^{n_1\times |D|}$ with i.i.d.\ $\mathcal N(0,1)$ entries.
Therefore, by similar calculation in \Cref{sec: proof A nonempty}, with probability at least $1-4(ep)^{1-\gamma_*/2}$, we have
\[
\|\hat\Sigma^{(1)}_{DD}-\bfI_{DD}\|_2\le  2\left(\sqrt{\frac{|D|}{n_1}}+\sqrt{\frac{\gamma_*|D|\log p}{n_1}}\right)
+\left(\sqrt{\frac{|D|}{n_1}}+\sqrt{\frac{\gamma_*|D|\log p}{n_1}}\right)^2,
\]
for all $D\subseteq B_*^c$ and $|D|\le  k_u$.
We take $D=O$ and conclude that the leading term on the right-hand side of the above display is of order $\sqrt{k_u\log p/n_1}$.
Therefore, with probability at least $1-4(ep)^{1-\gamma_*/2}$, it holds that
\begin{align}\label{eq: leading II in covariance}
\|\hat\Sigma^{(1)}_{OO}-\bfI_{OO}\|_2
\lesssim \sqrt{\frac{k_u\log p}{n_1}}.
\end{align}

\item
Since $M\subseteq \widehat{B}^{\,c}$ and $G\subseteq \widehat{B}$, we have $\hat\Sigma^{(1)}_{GM}$ as a submatrix of $\hat\Sigma^{(1)}_{\widehat{B}M}$.
By monotonicity of the operator norm under taking submatrices,
\(
\|\hat\Sigma^{(1)}_{GM}\|_2
\le
\|\hat\Sigma^{(1)}_{\widehat{B}M}\|_2
\).
The defining property of $\widehat{B}\in\bbB_{k_u}$ states that for every $D\subseteq \widehat{B}^{\,c}$ with $|D|\le k_u$,
\[
\|\hat\Sigma^{(1)}_{D\widehat{B}}\|_2
\le
\sqrt{\|\Gamma_{\widehat{B}}(\hat{\Sigma}^{(1)})\|_2}\biggl(\sqrt{\frac{|D|}{n_1}}+\sqrt{\frac{|\widehat{B}|}{n_1}}+\sqrt{\frac{\gamma_* |D|\log p}{n_1}}\biggr).
\]
Applying this with $D=M$ gives
\begin{align*}
\|\hat\Sigma^{(1)}_{GM}\|_2
\le
\|\hat\Sigma^{(1)}_{\widehat{B}M}\|_2
\le
\sqrt{\|\hat{\Sigma}^{(1)}_{\hat{B}\hat{B}}\|_2}\biggl(\sqrt{\frac{|M|}{n_1}}+\sqrt{\frac{|\widehat{B}|}{n_1}}+\sqrt{\frac{\gamma_* |M|\log p}{n_1}}\biggr).
\end{align*}

Moreover, we have
$$\max_{|B|\le k_u}\|\hat{\Sigma}_{BB}^{(1)}\|_2\le \max_{|B|\le k_u}\|\Sigma_{BB}\|\cdot \|\Sigma_{BB}^{-1/2}\hat{\Sigma}_{BB}^{(1)}\Sigma_{BB}^{-1/2}\|_2\le  M_1\max_{|B|\le k_u }\|\Sigma_{BB}^{-1/2}\hat{\Sigma}_{BB}^{(1)}\Sigma_{BB}^{-1/2}\|_2,$$
where we have used the eigenvalue condition $\lambda_{\max}(\Sigma)\le  M_1$.
Therefore, Lemma~\ref{lem: wishart op} with $k=k_u$, together with a union bound, gives (similarly to \Cref{sec: proof A nonempty}) that
$$\max_{|B|\le k_u}\|\hat{\Sigma}_{BB}^{(1)}\|_2\le  M_1\left(1+\sqrt{\frac{k_u}{n_1}}+\sqrt{\frac{\gamma_* k_u\log p}{n_1}}\right)=O(1)$$
with probability $1-4p^{1-\gamma_*/2}$. Therefore, we have
\begin{align}\label{eq: leading III in covariance}
\|\hat\Sigma^{(1)}_{GM}\|_2\lesssim \sqrt{\frac{k_u\log p}{n_1}}.
\end{align}

\item

We next control \(\|\hat{\Sigma}^{(1)}_{GO}\|_2\).
On $\mathcal{E}_1$, since
\[
G\subseteq B_*,
\qquad
O\subseteq B_*^c,
\qquad
|O|\le |\widehat B|\le k_u,
\]
the defining property of \(B_*\in\bbB_{k_u}\) in \Cref{eq:supp-set}, applied with \(D=O\), gives
\[
\|\hat\Sigma^{(1)}_{O B_*}\|_2
\le
\sqrt{\|\Gamma_{B_*}(\hat{\Sigma}^{(1)})\|_2}
\biggl(
\sqrt{\frac{|O|}{n_1}}
+\sqrt{\frac{|B_*|}{n_1}}
+\sqrt{\frac{\gamma_* |O|\log p}{n_1}}
\biggr).
\]
Moreover,
\[
\|\hat{\Sigma}^{(1)}_{GO}\|_2
=
\|\hat{\Sigma}^{(1)}_{OG}\|_2
\le
\|\hat{\Sigma}^{(1)}_{O B_*}\|_2.
\]
Using the same argument with Lemma~\ref{lem: wishart op} in the bound on $\|\hat{\Sigma}_{GM}^{(1)}\|_2$,
it holds with probability $1-4p^{1-\gamma_*/2}$ that
\[
\|\Gamma_{B_*}(\hat{\Sigma}^{(1)})\|_2=O(1).
\]
On the intersection of both events, we have
\begin{equation}\label{eq: leading IV in covariance}
\|\hat{\Sigma}^{(1)}_{GO}\|_2
\lesssim \sqrt{\frac{k_u\log p}{n_1}}.
\end{equation}

\end{enumerate}

\bigskip

Combining \eqref{eq: decomp variance error} with
\eqref{eq: leading I in covariance}, \eqref{eq: leading II in covariance},
\eqref{eq: leading III in covariance}, and \eqref{eq: leading IV in covariance} yields that with probability $1-20 p^{1-\gamma_*/2}$, we have
\begin{align}\label{eq: bound I in covariance}
\|\Gamma_{\widehat{B}}(\hat{\Sigma}^{(1)})-\Gamma_{B_*}(\hat{\Sigma}^{(1)})\|_2
\lesssim \sqrt{\frac{k_u\log p}{n_1}}.
\end{align}

\textit{Bounding $\mathrm{II}$ in \eqref{eq: triangle inequality matrix}.} Note that
\[\|\Gamma_{B_*}(\hat{\Sigma}^{(1)})-\Sigma\|_2=\|\hat{\Sigma}^{(1)}_{B_*B_*}-\Sigma_{B_*B_*}\|_2\le  \|\Sigma_{B_*B_*}\|_2\cdot \|\Sigma_{B_*B_*}^{-1/2}\hat{\Sigma}_{B_*B_*}^{(1)}\Sigma_{B_*B_*}^{-1/2}-\bfI_{|B_*|}\|_2,\]
where we have $\|\Sigma_{B_*B_*}\|_2\le  M_1$ by the eigenvalue condition in \Cref{eq: parameter space} and Lemma~\ref{lem: wishart op} with $k=|B_*|$ implies that with probability at least $1-4p^{1-\gamma_*/2}$, we have
$$\|\|\Sigma_{B_*B_*}^{-1/2}\hat{\Sigma}_{B_*B_*}^{(1)}\Sigma_{B_*B_*}^{-1/2}-\bfI_{|B_*|}\|_2\lesssim \sqrt{\frac{k_u\log p}{n_1}}.$$
We denote this event by \({\cal E}_3\).

Combining the above bounds for $\mathrm{I}$ and $\mathrm{II}$ in \eqref{eq: triangle inequality matrix}, we have arrived at the spectral norm of the estimation error of $\hat{\Sigma}_{\text{spike}}$:
\begin{align}\label{eq: final covariance error}
\bigl\|\hat{\Sigma}_{\text{spike}}-\Sigma\bigr\|_2
\lesssim \sqrt{\frac{k_u\log p}{n_1}}
\end{align}
conditioning on ${\cal E}_1\cap {\cal E}_2\cap {\cal E}_3$.

If $n_1\ge c k_u\log p$ for a sufficiently large constant $c>0$, then
\eqref{eq: final covariance error} implies
$\|\hat{\Sigma}_{\text{spike}}-\Sigma\|_2\le (2M_1)^{-1}$ with high probability.
On ${\cal E}_1\cap {\cal E}_2\cap {\cal E}_3$, using
\[
\widehat{\Omega}-\Omega
=
\hat{\Sigma}_{\text{spike}}^{-1}-\Sigma^{-1}
=
\hat{\Sigma}_{\text{spike}}^{-1}\,(\Sigma-\hat{\Sigma}_{\text{spike}})\,\Sigma^{-1},
\]
we obtain
\[
\|\widehat{\Omega}-\Omega\|_2
\le
\|\widehat{\Omega}\|_2\,\|\Omega\|_2\,\|\hat{\Sigma}_{\text{spike}}-\Sigma\|_2.
\]
By Weyl's inequality and $\lambda_{\min}(\Sigma)\ge 1/M_1$,
\(
\lambda_{\min}(\hat{\Sigma}_{\text{spike}})
\ge
\lambda_{\min}(\Sigma)-\|\hat{\Sigma}_{\text{spike}}-\Sigma\|_2
\ge
1/M_1-\|\hat{\Sigma}_{\text{spike}}-\Sigma\|_2,
\)
hence
\[
\|\widehat{\Omega}\|_2
=
\frac{1}{\lambda_{\min}(\hat{\Sigma}_{\text{spike}})}
\le
\frac{1}{1/M_1-\|\hat{\Sigma}_{\text{spike}}-\Sigma\|_2}.
\]
Combining the displays yields
\[
\|\widehat{\Omega}-\Omega\|_2
\le
\frac{M_1\|\hat{\Sigma}_{\text{spike}}-\Sigma\|_2}{1/M_1-\|\hat{\Sigma}_{\text{spike}}-\Sigma\|_2}
\lesssim
\sqrt{\frac{k_u\log p}{n_1}}.
\]
Since $n_1\asymp n$, this gives the stated rate
$\|\widehat{\Omega}-\Omega\|_2\lesssim \sqrt{(k_u\log p)/n}$.
Moreover, by construction $\hat{\Sigma}_{\text{spike}}$ equals $\bfI_p$ outside $\widehat{B}$ and
$|\widehat{B}|\le k_u$. The proof is completed by choosing $\gamma_*>2$.

\end{proof}

\subsection{Proof of Lemma~\ref{lem: set A nonempty}}\label{sec: proof A nonempty}
\begin{proof}
Recall that $\Sigma\in\Pi_0(k,p)$ with $\Sigma=\bfI_p+V\Lambda V^\top$ and
$B_*=\supp(V)$ satisfying $|B_*|=k\le k_u$.
Let $\hat{\Sigma}^{(1)}$ be defined in \eqref{eq: sample cov 1}.
To prove that $\bbB_{k_u}\neq\varnothing$ with high probability, it suffices to show that
$B_*\in\bbB_{k_u}$ with high probability.

From the definition of $\bbB_{k_u}$ in \eqref{eq:supp-set},
$\bbP(B_*\notin \bbB_{k_u})$ can be bounded by the sum of
\begin{equation}\label{eq:eq: prob B notin B revised-ondiag}
    \bbP\Biggl(\exists D\subseteq B_*^c, |D|\le k_u:\
\|\hat{\Sigma}^{(1)}_D-\bfI_p\|_2>
2\left(\sqrt{\frac{|D|}{n_1}}+\sqrt{\frac{\gamma_*|D|\log p}{n_1}}\right)
+\left(\sqrt{\frac{|D|}{n_1}}+\sqrt{\frac{\gamma_*|D|\log p}{n_1}}\right)^2
\Biggr),
\end{equation}
and
\begin{equation}\label{eq:eq: prob B notin B revised-offdiag}
    \bbP\Biggl(\exists D\subseteq B_*^c, |D|\le k_u:\
\|\hat{\Sigma}^{(1)}_{DB_*}\|_2>
\sqrt{\|\Gamma_{B_*}(\hat{\Sigma}^{(1)})\|_2}\biggl(\sqrt{\frac{|D|}{n_1}}+\sqrt{\frac{|B_*|}{n_1}}+\sqrt{\frac{\gamma_* |D|\log p}{n_1}}\biggr)
\Biggr).
\end{equation}

We bound these two terms separately.

For the term in \eqref{eq:eq: prob B notin B revised-ondiag}, since $D\subseteq B_*^c$ and $\Sigma_{DD}=\bfI_{DD}$, we have
$\hat{\Sigma}^{(1)}_{DD}\stackrel{d}{=}n_1^{-1}Y^\top Y$ for a matrix $Y$ with
i.i.d.\ $\mathcal{N}(0,1)$ entries. Therefore, by \Cref{lem: wishart op} and a
union bound,
\begin{align*}
&\bbP\Biggl(\exists D\subseteq B_*^c,\ |D|\le k_u:\
\|\hat{\Sigma}^{(1)}_D-\bfI_p\|_2>
2\left(\sqrt{\frac{|D|}{n_1}}+\sqrt{\frac{\gamma_*|D|\log p}{n_1}}\right)
+\left(\sqrt{\frac{|D|}{n_1}}+\sqrt{\frac{\gamma_*|D|\log p}{n_1}}\right)^2
\Biggr)\\
\le\;&
\sum_{\substack{D\subseteq B_*^c\\ |D|\le k_u}}
\bbP\Biggl(\|\hat{\Sigma}^{(1)}_D-\bfI_p\|_2>
2\left(\sqrt{\frac{|D|}{n_1}}+\sqrt{\frac{\gamma_*|D|\log p}{n_1}}\right)
+\left(\sqrt{\frac{|D|}{n_1}}+\sqrt{\frac{\gamma_*|D|\log p}{n_1}}\right)^2
\Biggr)\\
\le\;&
\sum_{\ell=1}^{k_u}
{
\binom{p-k_u}{\ell}
}
\,2\exp\!\left(-\frac{\gamma_*}{2}\,\ell\log p\right)
\le
2\sum_{\ell=1}^{k_u} p^{\ell (1-\gamma_*/2)}
\le
4 p^{1-\gamma_*/2},
\end{align*}
where the last inequality holds for all $\gamma_*\ge 3$ and $p\ge 2$.

For the term in \eqref{eq:eq: prob B notin B revised-offdiag}, fix $D\subseteq B_*^c$ with $|D|=\ell\le k_u$.
Let $W$ be the left singular vector matrix of $X_{\cdot B_*}^{(1)}$.
Then
\[
\bigl\|\hat{\Sigma}^{(1)}_{DB_*}\bigr\|_2
\le
\frac{1}{n_1}\bigl\|\bigl(X_{\cdot D}^{(1)}\bigr)^\top W\bigr\|_2\,
\bigl\|X_{\cdot B_*}^{(1)}\bigr\|_2
\stackrel{d}{=}
\frac{1}{\sqrt{n_1}}\|Y\|_2\,\sqrt{\bigl\|\Gamma_{B_*}(\hat{\Sigma}^{(1)})\bigr\|_2},
\]
where $Y$ is a $\ell\times |B_*|$ matrix with i.i.d.\ $\mathcal{N}(0,1)$ entries. The last equality in distribution is understood conditionally on
\(X_{\cdot B_*}^{(1)}\), or equivalently conditionally on
\(\Gamma_{B_*}(\hat{\Sigma}^{(1)})\). Indeed, by the block-diagonal structure of
\(\Sigma_*\) over the \(D\)- and \(B_*\)-blocks, \(X_{\cdot D}^{(1)}\) is
independent of \(X_{\cdot B_*}^{(1)}\), and hence is independent of \(W\), since
\(W\) is constructed from \(X_{\cdot B_*}^{(1)}\). Therefore, conditional on
\(W\), the Gaussian matrix
\[
\bigl(X_{\cdot D}^{(1)}\bigr)^\top W
\]
has the same distribution as
\[
\sqrt{n_1}\,Y\,\Gamma_{B_*}(\hat{\Sigma}^{(1)})^{1/2},
\]
with \(Y\) independent of \(\Gamma_{B_*}(\hat{\Sigma}^{(1)})\). Taking spectral
norms and using
\[
\bigl\|X_{\cdot B_*}^{(1)}\bigr\|_2
=
\sqrt{n_1}\sqrt{\bigl\|\Gamma_{B_*}(\hat{\Sigma}^{(1)})\bigr\|_2}
\]
gives the displayed bound.

Applying the Davidson--Szarek bound (\Cref{lem: davidson szarek}) and then taking a union bound over $D$ gives
\begin{align*}
&\bbP\left(\exists D\subseteq B_*^c,\ |D|\le k_u:\ \bigl\|\hat{\Sigma}^{(1)}_{DB_*}\bigr\|_2> \sqrt{\bigl\|\Gamma_{B_*}(\hat{\Sigma}^{(1)})\bigr\|}\biggl(\sqrt{\frac{|D|}{n_1}}+\sqrt{\frac{|B_*|}{n_1}}+\sqrt{\frac{\gamma_* |D|\log p}{n_1}}\biggr)\right)\\
\le& \sum_{\substack{D\subseteq B_*^c\\ |D|\le k_u}} \bbP\left(\bigl\|\hat{\Sigma}^{(1)}_{DB_*}\bigr\|> \sqrt{\bigl\|\Gamma_{B_*}(\hat{\Sigma}^{(1)})\bigr\|}\biggl(\sqrt{\frac{|D|}{n_1}}+\sqrt{\frac{|B_*|}{n_1}}+\sqrt{\frac{\gamma_* |D|\log p}{n_1}}\biggr)\right)\\
\le& \sum_{\substack{D\subseteq B_*^c\\ |D|\le k_u}} \bbP\left(\|Y\|_2>\sqrt{n_1}\biggl(\sqrt{\frac{|D|}{n_1}}+\sqrt{\frac{|B_*|}{n_1}}+\sqrt{\frac{\gamma_* |D|\log p}{n_1}}\biggr)\right)\\
\le& \sum_{\ell=1}^{k_u}\binom{p-k_u}{\ell}\,2\exp\!\left(-\frac{\gamma_*}{2}\,\ell\log p\right) \le 2\sum_{\ell=1}^{k_u} p^{\ell (1- {\gamma_*}/{2})} \le 4 p^{1-\gamma_*/2},
\end{align*}
where the last inequality holds for all $\gamma_*\ge 3$ and $p\ge 2$.

Combining the two bounds yields $\mathbb{P}\bigl({\cal E}_1^c\bigr)=\bbP(B_*\notin\bbB_{k_u})\le 8p^{1-\gamma_*/2}$.
Since \(B_*\in\bbB_{k_u}\) implies \(\bbB_{k_u}\neq\varnothing\), this completes the proof.

\end{proof}

\section{Proof of Additional Technical Results for the Lower Bounds}\label{sec: proof lemma lower}

\subsection{Technical lemmas}
\begin{lemma}[{\cite[Lemma 9]{cai2012optimal}}]\label{lem: Gaussian chi square}
    Let \( g_i \) be the density function of \( {\mathcal{N}}(0, \Sigma_i) \) for \( i = 0, 1, 2 \), respectively. Then
    \[
    \int \frac{g_1 g_2}{g_0} = \left( \det \left( \bfI - \Sigma_0^{-1} (\Sigma_1 - \Sigma_0) \Sigma_0^{-1} (\Sigma_2 - \Sigma_0) \right) \right)^{-1/2}.
    \]
\end{lemma}

\begin{lemma}[{\cite[Lemma 10]{xie2025minimax}}]\label{lem: decreasing}
   The following function $\phi_\beta:\bbR\to\bbR$ is continuous and strictly decreasing:
  $$\phi_\beta(\beta)=\frac{\sum_{j=1}^p|\xi_j| \exp(-\beta/\xi_j^2)}{\sqrt{\sum_{j=1}^p \xi_j^2\exp(-\beta/\xi_j^2)}}.$$
\end{lemma}

The following lemma records a basic comparison between expectations under a probability measure
and under the corresponding conditional measure on an event.

\begin{lemma}\label{lem:restriction_expectation_short}
Let $(\Omega,\mathcal{F},P)$ be a probability space.
Let $Q=P(\,\cdot\,\mid \mathcal{A})$ for some $\mathcal{A}\in\mathcal{F}$ with $P(\mathcal{A})>0$.
Then for any nonnegative measurable $X$,
\[
\mathbb{E}_{Q}[X]
=\frac{1}{P(\mathcal{A})}\mathbb{E}_{P}\!\left[X\mathbf{1}\{\mathcal{A}\}\right]
\le \frac{1}{P(\mathcal{A})}\mathbb{E}_{P}[X].
\]
If additionally $P(\mathcal{A})>1/2$, then
\[
\mathbb{E}_{Q}[X]\le \bigl(1+2P(\mathcal{A}^{c})\bigr)\mathbb{E}_{P}[X].
\]
\end{lemma}

\begin{proof}
The identity follows from the definition of $Q$, and the first inequality uses $X\mathbf{1}\{\mathcal{A}\}\le X$.
For the second inequality, write $P(\mathcal{A})=1-\varepsilon$ with $\varepsilon=P(\mathcal{A}^c)<1/2$ and note that
$1/(1-\varepsilon)\le 1+2\varepsilon$.
\end{proof}

\subsection{Proof of Lemma~\ref{lem: valid covariance}}\label{sec: proof valid covariance}
\begin{proof}

We verify the conditions in the definition of the null space in \eqref{eq: null space} hold with probability 1 under the prior distribution $\pi_1$.

\textit{Eigenvalues Control of $\Sigma$:}
The covariance matrix of $X$ is given by the lower-right $(p\times p)$ block of the covariance matrix of $Z$, i.e. $\Sigma^z$ defined in \eqref{eq: covariance construction}:
\[
\Sigma
= \left(\begin{array}{c|c}
     \bfI_{p_1\times p_1} & \alphab_1\alphab_2^\top \\
     \hline
     \alphab_2\alphab_1^\top & \bfI_{p_2\times p_2}
    \end{array}\right).
\]
The largest and smallest eigenvalues of $\Sigma$ are given by $1\pm \|\alphab_1\|_2\|\alphab_2\|_2$, and the other $p-2$ eigenvalues are all equal to 1.
Specifically, we have $\|\alphab_1\|_2=1$ and
\begin{equation}
\label{eq: valid cov alphab2}
\|\alphab_2\|_2=c_1\sqrt{p_1\log p/n}
\in [3^{-1}c_1\sqrt{k_u \log p/n},  2^{-1}c_1\sqrt{k_u \log p/n}],
\end{equation}
where we have used $p_1=\lfloor k_u/4\rfloor \in [k_u/9, k_u/4]$ for $k_u\ge  4$.
Since $k_u\lesssim n/\log p$ from \Cref{cdt: sparsity assumption}, we can  choose $c_1$ sufficiently small such that $1/M_1\le  \lambda_{\min}(\Sigma)\le  \lambda_{\max}(\Sigma)\le  M_1$ with probability 1.

\textit{Sparsity control of $\beta$:} For the covariance matrix $\Sigma^z$ defined in \eqref{eq: null point mat}, we have
\begin{equation*}
    \beta
    = \left(
        \begin{array}{cc}
            \bfI_{p_1\times p_1}& \alphab_1\alphab_2^\top\\
            \alphab_2\alphab_1^\top & \bfI_{p_2\times p_2}
        \end{array}
	      \right)^{-1} \left(\begin{array}{c}
	        {\bfz}_{p_1\times 1}\\
	        \kappa\alphab_2
	      \end{array}\right),
	\end{equation*}
	which leads to
\begin{equation}\label{eq: beta expression}
 \left\{
 \begin{aligned}
        \beta_{S_1}&=-\frac{\kappa \|\alphab_2\|_2^2}{1-\|\alphab_2\|_2^2\|\alphab_1\|_2^2}\alphab_1,\\
        \beta_{S_2}&=\frac{\kappa}{1-\|\alphab_2\|_2^2\|\alphab_1\|_2^2}\alphab_2.
	    \end{aligned}
	    \right.
	\end{equation}
	Therefore, we have $\|\beta\|_0=\|\beta_{S_1}\|_0+\|\beta_{S_2}\|_0\le  \|\alphab_1\|_0+\|\alphab_2\|_0\le  k_u/4+k_u/4=k_u/2$.

\textit{Control of $\sigma$:}
We verify that $\sigma$ is bounded between $0$ and $M_2$, which implies the constructed $\Sigma^z$ is positive definite.
A direct calculation yields
$$
\begin{aligned}
\sigma^2
& = \sigma_*^2-\beta^\top \Sigma\beta\\
& =\sigma_*^2-\frac{\kappa^2}{1-\|\alphab_2\|_2^2\|\alphab_1\|_2^2}\|\alphab_2\|_2^2\\
&\ge  M_2^2/4 - \frac{2^{-2}\kappa^2 c_1^2 k_u\log p/n }{1-2^{-2}c_1^2 k_u\log p/n}
\end{aligned}
$$
Since $k_u\lesssim n/\log p$, we can choose $c_1$ sufficiently small such that
\begin{equation}\label{eq: temp bound of delta1}
  2^{-2}  c_1^2 k_u\log p/n < \min(1/2, M_2^2/16).
\end{equation}

Therefore, we have
\begin{equation}\label{eq: bound of sigma in lem of valid covariance}
\sigma^2
\ge  M_2^2/4 - \kappa^2 M_2^2/8.
\end{equation}

\textit{Control of $\kappa$:}
Recall that $\kappa$ is defined as the unique solution to the linear equation
\begin{equation}\label{eq: kappa equation}
    \xi^\top \beta= c_2\sqrt{\sum_{j\le  k_u}\xi_j^2}\frac{k_u\log p}{n}.
\end{equation}
Note that
$$
(\xi_{S_1})^\top \beta_{S_1}=\kappa \frac{-\|\alphab_2\|_2^2(\xi_{S_1})^\top \alphab_1}{1-\|\alphab_2\|_2^2\|\alphab_1\|_2^2}=\kappa  \cdot \frac{\|\alphab_2\|_2^2 }{1-\|\alphab_2\|_2^2 } \sqrt{\sum_{j\le  p_1}\xi_j^2}
$$
and $\xi_{S_2}^\top \beta_{S_2}$ has the same sign as $\kappa$, which is positive, we see that the coefficient of $\kappa$ in \Cref{eq: kappa equation} is positive.
Consequently, the solution to the linear equation \eqref{eq: kappa equation} exists and is unique.
From \Cref{eq: valid cov alphab2} and \Cref{eq: temp bound of delta1}, we have
$$
3^{-2}c_1^2 \frac{k_u\log p}{n}\le  \|\alphab_2\|_2^2 \le  2^{-2}c_1^2 \frac{k_u\log p}{n}\le  1/2.
$$
Therefore, the solution satisfies that
$$c_2\sqrt{\sum_{j\le  k_u}\xi_j^2}\frac{k_u\log p}{n} \ge  (\xi_{S_1})^\top \beta_{S_1} \ge  2\kappa \|\alphab_2\|_2^2  \sqrt{\sum_{j\le  p_1}\xi_j^2} \ge  \frac{2}{9} \kappa c_1^2\sqrt{\sum_{j\le  p_1}\xi_j^2}\frac{k_u\log p}{n}.
$$
Note that $\sum_{p_1<j\le  k_u}\xi_j^2$ is a sum of at most $4p_1$ terms and each term is upper bounded by $\xi_{p_1}^2$. Therefore, we have $$\sqrt{\sum_{j\le  k_u}\xi_j^2}= \sqrt{\sum_{j\le  p_1}\xi_j^2+\sum_{p_1<j\le  k_u}\xi_j^2}\le  \sqrt{5\sum_{j\le  p_1}\xi_j^2}.$$
Therefore, the solution $\kappa$ satisfies
$$
\kappa\le  9\sqrt{5}c_2/(2c_1^2).
$$
Consequently, for any given value of $c_1$, we choose $c_2$ sufficiently small such that $9\sqrt{5}c_2/(2c_1^2)\le  1$. It follows that $\kappa\le  1$.
Furthermore, the inequality in \Cref{eq: bound of sigma in lem of valid covariance} guarantees that $\sigma\in (M_2/\sqrt{8}, M_2)$.

\end{proof}

\subsection{Proof of Lemma~\ref{lem: chi square integral 1}}\label{sec: proof chi square integral}
\begin{proof}
For $\Sigma_*^z$ defined in \eqref{eq: alter matrix}
and $\Sigma^z=g_1(\alphab_2),\tilde{\Sigma}^z=g_1(\tilde{\alphab}_2)$ defined in \eqref{eq: null point mat}, we have
$$\left(\Sigma_*^z\right)^{-1}=\left(\begin{array}{c|c|c}
   1/{\sigma_*^2}&\bfz_{1\times p_1}&\bfz_{1\times p_2}\\
   \hline
    \bfz_{p_1\times 1}&\bfI_{p_1\times p_1}&\bfz_{p_1\times p_2}\\
    \hline
    \bfz_{p_2\times 1}&\bfz_{p_2\times p_1}&\bfI_{p_2\times p_2}
\end{array}\right),$$
and
$$\left(\Sigma_*^z\right)^{-1}\left(\Sigma^z-\Sigma_*^z\right)=\left(
    \begin{array}{c|c|c}
        0&\bfz_{1\times p_1}&{\kappa\alphab_2^\top}/{\sigma_*^2}\\
        \hline
        \bfz_{p_1\times 1}&\bfz_{p_1\times p_1}&\alphab_1\alphab_2^\top\\
        \hline
        \kappa\alphab_2&\alphab_2\alphab_1^\top&\bfz_{p_2\times p_2}
    \end{array},
\right)$$
and
$$\left(\Sigma_*^z\right)^{-1}\left(\Sigma^z-\Sigma_*^z\right)\left(\Sigma_*^z\right)^{-1}\left(\tilde{\Sigma}^z-\Sigma_*^z\right)=\left(
    \begin{array}{c|c}
        \Delta&\bfz_{(1+p_1)\times p_2}\\
        \hline
        \bfz_{p_2\times (1+p_1)}&\ds \left[\frac{\kappa\tilde{\kappa}}{\sigma_*^2}+\|\alphab_1\|_2^2\right]\alphab_2\tilde{\alphab}_2^\top
    \end{array}
\right),$$
where the upper left $(1+p_1)\times (1+p_1)$ block matrix is
$$\Delta=\left(\begin{array}{c}
    \kappa \alphab_2^\top/\sigma_*^2\\
    \alphab_1\alphab_2^\top
\end{array}\right)\left(\begin{array}{cc}
    \tilde{\kappa} \tilde{\alphab}_2 \,&\, \tilde{\alphab}_2\alphab_1^\top
\end{array}\right), $$
and the right lower $p_2\times p_2$ block matrix has rank no larger than 1 with the only nonzero eigenvalue (if $\tilde{\alphab}_2^\top\alphab_2\neq 0$)
$$\left[\frac{\kappa\tilde{\kappa}}{\sigma_*^2}+\|\alphab_1\|_2^2\right]\tilde{\alphab}_2^\top\alphab_2.$$
Similarly, $\Delta$ has rank no larger than 1 and its only nonzero eigenvalue is given by
$$\left(\begin{array}{cc}
   \tilde{\kappa} \tilde{\alphab}_2&\tilde{\alphab}_2\alphab_1^\top
\end{array}\right)\left(\begin{array}{c}
    \kappa \alphab_2^\top/\sigma_*^2\\
    \alphab_1\alphab_2^\top
\end{array}\right)=\left[\frac{\kappa\tilde{\kappa}}{\sigma_*^2}+\|\alphab_1\|_2^2\right]\tilde{\alphab}_2^\top\alphab_2. $$
Consequently, the matrix
$$\left(\bfI_{p+1}-\left(\Sigma_*^z\right)^{-1}\left(\Sigma^z-\Sigma_*^z\right)\left(\Sigma_*^z\right)^{-1}\left(\tilde{\Sigma}^z-\Sigma_*^z\right)\right)$$
has two eigenvalues given by $1-\left[\frac{\kappa\tilde{\kappa}}{\sigma_*^2}+\|\alphab_1\|_2^2\right]\tilde{\alphab}_2^\top\alphab_2$ and the rest $p-1$ eigenvalues are all equal to 1; if  $\tilde{\alphab}_2^\top\alphab_2=0$, then all its eigenvalues are 1.
Therefore, with \Cref{lem: Gaussian chi square}, we have
\begin{equation}\label{eq: chi square integral bound}
\begin{aligned}
    &\bbE_{(\theta,\tilde{\theta})\sim \pi_{}\times \pi_{}}\int_{\bbR^n}\frac{\rmd \bbP_{\theta}\rmd \bbP_{\tilde{\theta}}}{\rmd \bbP_{\theta_*}}\\
    =\, &\bbE_{(\alphab_2,\tilde{\alphab}_2)\sim \pi_1\times \pi_1}\left[1-\left(\frac{\kappa\tilde{\kappa}}{\sigma_*^2}+\|\alphab_1\|_2^2\right)\tilde{\alphab}_2^\top\alphab_2\right]^{-n}\\
    \le \,& \bbE_{(\alphab_2,\tilde{\alphab}_2)\sim {\pi}_1\times {\pi}_1}\exp\left(2n\left(\frac{\kappa\tilde{\kappa}}{\sigma_*^2}+\|\alphab_1\|_2^2\right)\tilde{\alphab}_2^\top\alphab_2\right),
\end{aligned}
\end{equation}
where the last inequality follows from the fact that $(1-x)^{-1} \le  \exp(2x)$ for $x \in [0,1/2]$ and based on \eqref{eq: moderate sparse delta} we have $$0\le  \left(\frac{\kappa\tilde{\kappa}}{\sigma_*^2}+\|\alphab_1\|_2^2\right)\tilde{\alphab}_2^\top\alphab_2\le  \left(\frac{\kappa\tilde{\kappa}}{\sigma_*^2}+\|\alphab_1\|_2^2\right)c_1^2\frac{k_u\log p}{n}\le  \frac{1}{2}$$ if $c_1$ is chosen sufficiently small (recall that $k_u\log p/n$ is bounded).
Since we have $\kappa, \tilde{\kappa}\le  1$,  $\sigma_*=M_2/2$, and $\|\alphab_1\|_2^2=1$, we can choose the constant $c_3$ (for example $c_3 = 2 (4/M_2^2 +1)$) to complete the proof.
\end{proof}

\subsection{Proof of Lemma~\ref{lem: hypergeometric}}\label{sec: proof hypergeometric}
\begin{proof}
Since $J$ is non-negative, we only need to show that $\mathbb{E}[\exp (c(\log p) J)]\le  1+o(1)$ as $p$ goes to infinity. Given any constant $c\in (0, 1-2\gamma)$, let $$A_m=p^{cm}\mathbb{P}(J=m)=\exp(c\log p\cdot m)
\frac{\binom{k}{m}\binom{p-k }{k - m}}{\binom{p }{k}}.$$
Then we have $\mathbb{E}[\exp (c(\log p) J)]=\sum_{m=0}^k A_m$.
Notice that for $0 \le m \le k-1$,
\begin{align*}
\frac{A_{m+1}}{A_m}
&= p^{c}\left[\frac{\binom{k}{m+1}}{\binom{k}{m}} \cdot \frac{\binom{p-k}{k-m-1}}{\binom{p-k}{k-m}}\right]\\
&= p^c \left[
    \frac{(k-m)^2}{(m + 1)(p+m+1-2k)}
  \right] \\
&\le p^c \frac{(k-m)^2}{p-2k} \\
&\le p^c \frac{p^{2\gamma}}{p - 2p^\gamma} \\
&= \frac{p^{c + 2\gamma - 1}}{1 - 2p^{\gamma - 1}} =: r_p.
\end{align*}
Since the constant $c+2\gamma-1<0$ and $\gamma-1<-1/2$, $r_p\to 0$. In particular, for sufficiently large $p$, $r_p\le  1/2$ and $A_m\le  A_1 r_p^{m-1}$, which implies
\begin{equation}\label{eq: Ak sum bound}
\sum_{m=0}^{k} A_k
= A_0 + \sum_{m=1}^{k} A_m
\le A_0 + A_1 \sum_{m=0}^{k-1} 2^{-m}
\le A_0 + 2A_1.
\end{equation}

Notice that
\[
A_0
= \frac{\binom{p - k}{k}}{\binom{p }{k}}
= \prod_{j=1}^{k} \frac{p - 2k + j}{p - k+j}
= \prod_{j=1}^{k} \left(1 - \frac{k}{p - k + j}\right).
\]
Hence,
\[
\left(1 - \frac{k}{p - k}\right)^k
\le A_0
\le \left(1 - \frac{k}{p}\right)^k.
\]
Since $k^2/p \le p^{2\gamma - 1} \to 0$, both bounds converge to $1$, and thus $A_0 \to 1$.
Furthermore, the inequality $A_1\le  r_p A_0$ implies that $A_1=o(1)$.
By \eqref{eq: Ak sum bound}, the proof is completed.
\end{proof}
\subsection{Proof of Lemma~\ref{lem: valid covariance 2}}\label{sec: proof valid covariance 2}
\begin{proof}
The proof is similar to that of \Cref{lem: valid covariance}. It remains to verify that the following conditions hold with probability $1-c/2$ under the prior distribution $\pi_3$.

\textit{Eigenvalues control of $\Sigma$:}
For the joint covariance matrix $\Sigma^z$ defined in \eqref{eq: null point mat 2}, the covariance matrix for $X$ is $\Sigma=\bfI_{p\times p}$, so the eigenvalues of $\Sigma$ are always controlled.

\textit{Sparsity control of $\beta$:}
Define
$$\mu\stackrel{\triangle}{=}\sum_{j=1}^{k_\xi} q^{(1)}_j=c_4\frac{\sum_{j=1}^{k_\xi} |\xi_j|\exp(-\lambda^2/\xi_j^2)}{\sqrt{\sum_{j=1}^{k_\xi}\xi_j^2\exp(-\lambda^2/\xi_j^2)}}.$$
Based on \eqref{eq: equation sol}, we have
$$
c_4\le  \mu\le  2^{-1} c_4 k_u,
$$
where the second inequality follows from \Cref{lem: decreasing} and the fact that $\zeta\le  \lambda^2$.

Choosing \(c_4\le  1/2\), the bound
\(\mu\le  c_4 k_u/2\) gives $\mu\le  k_u/4$.
Let $S=\|\alphab\|_0$. \(S\) is a sum of independent Bernoulli variables with
\[
\mathbb E_{\pi_3}S=\mathbb{E}_{\pi_3} \sum_{j=1}^{k_\xi} b^{(1)}_j=\mu \le k_u/4.
\]
Since \(\beta=\kappa\alphab\), we have
$\|\beta\|_0\le  S$.
By Chernoff's inequality for sums of independent Bernoulli variables (see \cite[Theorem 2.3.1]{vershynin2018high}), we have
$$
\begin{aligned}
\mathbb P_{\pi_4}\!\left(\|\beta\|_0\ge  \frac{k_u}{2}\right) \le  & \;
\mathbb P_{\pi_3}(S\ge  k_u/2) \\
\le  &
\exp\left\{-\mu+\frac{k_u}{2}\log\left(\frac{e\mu}{k_u/2}\right)\right\}\\
\le  &
\exp\left\{- \frac{k_u}{4}+\frac{k_u}{2}\log\left(\frac{e}{2}\right)\right\} \\
\le  & \exp\left\{- \frac{\log(4/e)k_u}{8}\right\},
\end{aligned}
$$
where the third inequality follows from the fact that the function $x\mapsto -x+(k_u/2)\log(ex/(k_u/2))$ is increasing on $(0, (k_u/2)]$ and $\mu\le k_u/4$.
Finally, choose \(C_4\) sufficiently large so that
\[
\exp\left(- \frac{\log(4/e)C_4}{8}\right)\le  \frac{c}{6}.
\]
Under the condition that  \(k_u\ge  C_4\), we have
\[
\mathbb P_{\pi_4}\!\left(\|\beta\|_0\ge  \frac{k_u}{2}\right)
\le  \frac{c}{6}.
\]

\textit{Control of $\kappa$:}
Note that $\beta=\kappa \alphab$. To show that the unique solution to $\xi^\top  \beta =\tau$ satisfies $\kappa\le  1$ for $\tau=c_6\nu_1/\sqrt{n}$, we only need to show  $\xi^\top \alphab\ge  c_6\nu_1/\sqrt{n}$ holds. Here $c_6$ is determined by the values of $c_4$ and $c_5$.

For $\alphab\sim \pi_3$, we have
$$\xi^\top \alphab \stackrel{{\rm d}}{=}\sum_{j=1}^{k_\xi}\frac{c_5\gamma^{(1)}_j\xi_j}{\sqrt{n}}b^{(1)}_j,$$
where $b^{(1)}_j\sim {\rm Bernoulli}(q^{(1)}_j)$.
Since $\xi_j\gamma^{(1)}_j= \max(|\xi_j|, \lambda)$ is non-increasing in $j$, we have
$$
\begin{aligned}
{\rm Var}_{\pi_3} [\xi^\top \alphab] & = \frac{c_5^2}{n}\sum_{j=1}^{k_\xi}\xi_j^2(\gamma^{(1)}_j)^2q^{(1)}_j[1-q^{(1)}_j]
 \\
 & \le  \frac{c_5^2\xi_1\gamma_1^{(1)}}{n}\sum_{j=1}^{k_\xi}\xi_j\gamma^{(1)}_jq^{(1)}_j\\
 & =  \frac{c_5}{\sqrt{n}}\max(|\xi_1|, \lambda)\bbE_{\pi_3}\xi^\top \alphab.
\end{aligned}
$$
Using $\xi_j\gamma^{(1)}_j\ge  |\xi_j|$, we have $$\bbE_{\pi_3} \xi^\top \alphab=\frac{c_5}{\sqrt{n}}\sum_{j=1}^{k_\xi}\xi_j\gamma^{(1)}_jq^{(1)}_j\ge   \frac{c_5}{\sqrt{n}} \sum_{j=1}^{k_\xi}|\xi_j|q^{(1)}_j= \frac{c_4c_5}{\sqrt{n}}\sqrt{\sum_{j=1}^{k_\xi}\xi_j^2\exp(-\lambda^2/\xi_j^2)}.$$
Using $\xi_j\gamma^{(1)}_j\ge   \lambda$, we have $$\bbE_{\pi_3} \xi^\top \alphab=\frac{c_5}{\sqrt{n}} \sum_{j=1}^{k_\xi}\xi_j\gamma^{(1)}_jq^{(1)}_j\ge  \frac{c_5}{\sqrt{n}}  \lambda \sum_{j=1}^{k_\xi}q^{(1)}_j.$$
If $\lambda>0$, \eqref{eq: equation sol} implies that $\sum_{j=1}^{k_\xi}q^{(1)}_j= c_4 k_u$ and thus $\bbE_{\pi_3} \xi^\top \alphab\ge  \frac{c_4 c_5 }{\sqrt{n}}  \lambda k_u$.
This inequality also holds when $\lambda=0$.
Consequently, we have
\begin{equation}\label{eq: control mean of LF in lem of valid covariance 2}
    \bbE_{\pi_3}\xi^\top \alphab\ge  \frac{c_4 c_5}{2\sqrt{n}} \nu_1>0.
\end{equation}

By Chebyshev's inequality, we have
\begin{equation}\label{eq: bounding LF in lem of valid covariance 2}
    \begin{aligned}
\mathbb{P}_{\pi_3}\left(\xi^\top \alphab< \frac{1}{2}\bbE_{\pi_3}[\xi^\top \alphab]\right)& \le  \frac{4{\rm Var}_{\pi_3} (\xi^\top \alphab)}{\left[\bbE_{\pi_3}(\xi^\top \alphab)\right]^2}\le  \frac{8}{c_4}\frac{\max(|\xi_1|,\lambda)}{\nu_1}.
\end{aligned}
\end{equation}
We pick $C_4 = 48/(c c_4)$.
Recall the conditions that $\nu_1\ge  C_4 |\xi_1|$, $k_u\ge  C_4$, and $\nu_1\ge  \lambda k_u$, we have $\max(|\xi_1|,\lambda)/\nu_1\le  \max(1/C_4,1/k_u)\le  1/C_4\le  c c_4 /48$.
\Cref{eq: control mean of LF in lem of valid covariance 2} and \Cref{eq: bounding LF in lem of valid covariance 2} together imply that
$$
\mathbb{P}_{\pi_4}\left(\kappa\in (0,1]\right)\ge
\mathbb{P}_{\pi_3}\left(\xi^\top \delta\ge  c_6\nu_1/\sqrt{n}\right)\ge  1- c/6
$$
with $c_6=c_4 c_5 /4$.

\textit{Control of $\sigma$:}
By the construction of $\Sigma^z$, we have $\sigma^2=\sigma_*^2-\beta^{\top} \Sigma \beta=\sigma_*^2-\kappa^2 \alphab^\top \alphab$; our goal is to show that the right-hand side is positive, so that $\Sigma^z$ is positive definite.
In the last part, we have shown that with probability $1-c/6$, the event $\{\xi^\top \delta\ge  c_6\nu_1/\sqrt{n}\}$ (and thus $\kappa\le  1$) holds.
It remains to show that $\|\alphab\|_2$ is smaller than $\sigma_*$ with probability close to 1. Note that
$$\|\alphab\|_2^2\stackrel{{\rm d}}{=}\frac{c_5^2}{n}\sum_{j=1}^{k_\xi}b^{(1)}_j(\gamma^{(1)}_j)^2$$
with $b^{(1)}_j\sim {\rm Bernoulli}(q^{(1)}_j)$.
By the definitions of $q^{(1)}_j$ and $\gamma^{(1)}_j$, we have
\begin{align*}
    &\sum_{j=1}^{k_\xi}(q^{(1)}_j)^2\exp((\gamma^{(1)}_j)^2)\\
    \le  &c_4^2\frac{\sum_{j\le  j_1} \xi_j^2\exp(-2\lambda^2/\xi_j^2) e +\sum_{j>j_1}\xi_j^2\exp(-2\lambda^2/\xi_j^2) \exp(\lambda^2/\xi_j^2)}{\sum_{i=1}^p\xi_i^2\exp(-\lambda^2/\xi_i^2)}\\ \le  &c_4^2 e .
\end{align*}
Therefore, we have $(\gamma^{(1)}_j)^2\le  \log(c_4^2 e / (q^{(1)}_j)^2) $ and
\begin{align*}
    \bbE \|\alphab\|_2^2&=\frac{c_5^2}{n}\sum_{j=1}^{k_\xi}q^{(1)}_j(\gamma^{(1)}_j)^2\le  \frac{c_5^2}{n}\sum_{j=1}^{k_\xi}q^{(1)}_j\log \left(\frac{c_4^2 e}{(q^{(1)}_j)^2}\right).
\end{align*}
Since the function $x\log(c_4^2 e /x^2)$ is concave, by Jensen's inequality, we have
\begin{align*}
    \bbE \|\alphab\|_2^2&\le  \frac{c_5^2}{n}\left(\sum_{j=1}^{k_\xi}q_j^{(1)}\right)\log \left(\frac{ e  c_4^2 k_\xi^2}{\left(\sum_{j=1}^{k_\xi} q_j^{(1)}\right)^2}\right)
\end{align*}
In the previous analysis, we have
$\mu=\sum_{j=1}^{k_\xi}q^{(1)}_j\in [c_4, c_4 k_u]$, so we further have
\begin{align*}
    \bbE \|\alphab\|_2^2\le  2 c_4c_5^2\frac{k_u\log ( e p)}{n}.
\end{align*}
By Markov's inequality,
$$
\mathbb{P}_{\pi_3}\left( \|\alphab\|_2 >  \sigma_*/\sqrt{2} \right) \le  2\sigma_*^{-2} \bbE \|\alphab\|_2^2 \le  16 M_2^{-2} c_4c_5^2\frac{k_u\log ( e p)}{n}
$$
Note that $k_u\log p\lesssim n$ based on \Cref{cdt: sparsity assumption}.
Given $c_5$, we can choose $c_4$ sufficiently small so that the right-hand side on the last inequality is bounded by $c/6$.

Combining the above analyses, we complete the proof.
\end{proof}
\subsection{Proof of Lemma~\ref{lem: chi square integral 2}}\label{sec: proof chi square integral 2}

\begin{proof}
   For $\Sigma_*^z$ defined in \eqref{eq: alter matrix} and $\Sigma^z=g_2(\alphab),\tilde{\Sigma}^z=g_2(\tilde{\alphab})$ defined in \eqref{eq: null point mat 2}, we have $\kappa,\tilde{\kappa}\in [0,1]$ and
   $$\left(\Sigma_*^z\right)^{-1}\left(\Sigma^z-\Sigma_*^z\right)=\left(
    \begin{array}{c|c}
       0& {\kappa\alphab^\top}/{\sigma_*^2}\\
       \hline
    \kappa\alphab& \bfz_{p\times p}
    \end{array}
   \right).$$
    Therefore, we have
    $$\left(\Sigma_*^z\right)^{-1}\left(\Sigma^z-\Sigma_*^z\right)\left(\Sigma_*^z\right)^{-1}\left(\tilde{\Sigma}^z-\Sigma_*^z\right)=\left(
    \begin{array}{c|c}
       \frac{\kappa\tilde{\kappa}}{\sigma_*^2}\alphab^\top \tilde{\alphab}& \bfz_{1\times p}\\
       \hline
    \bfz_{p\times 1}& \frac{\kappa\tilde{\kappa}}{\sigma_*^2}\alphab \tilde{\alphab}^\top
    \end{array}
   \right).$$
    Consequently, the matrix
    $$\left(\bfI_{p+1}-\left(\Sigma_*^z\right)^{-1}\left(\Sigma^z-\Sigma_*^z\right)\left(\Sigma_*^z\right)^{-1}\left(\tilde{\Sigma}^z-\Sigma_*^z\right)\right)$$
    has two non-unit eigenvalues given by $1-\frac{\kappa\tilde{\kappa}}{\sigma_*^2}\alphab^\top \tilde{\alphab}$ and the rest $p-1$ eigenvalues are all equal to 1.

    When $\tilde{\kappa}\kappa=0$, the desired result holds obviously.

When $\tilde{\kappa} \kappa\in (0,1]$, by definition, we have $\tilde{\alphab}, \alphab\in \mathcal{G}_\tau$.
By definition of $\mathcal{G}_\tau$, we have
$$\frac{\tilde{\kappa} \kappa}{\sigma_*^2}\alphab^\top \tilde{\alphab}\le  \frac{1}{2}.$$
Therefore, with \Cref{lem: Gaussian chi square}, we have
    \begin{equation}\label{eq: chi square integral bound 2}
    \begin{aligned}
        &\bbE_{(\theta,\tilde{\theta})\sim \pi_4\times \pi_4}\int_{\bbR^n}\frac{\rmd \bbP_{\theta}\rmd \bbP_{\tilde{\theta}}}{\rmd \bbP_{\theta_*}}\\
        =&\bbE_{(\alphab,\tilde{\alphab})\sim \pi_3\times \pi_3}\left[1-\frac{\kappa\tilde{\kappa}}{\sigma_*^2}\alphab^\top \tilde{\alphab}\right]^{-n}\\
        \le  &\bbE_{(\alphab,\tilde{\alphab})\sim {\pi}_3\times {\pi}_3}\exp\left(\frac{2n\kappa\tilde{\kappa}}{\sigma_*^2}\alphab^\top \tilde{\alphab}\right),
    \end{aligned}
    \end{equation}
    where the last inequality follows from the fact that $(1-x)^{-1} \le  \exp(2x)$ for $x \in [0,1/2]$.
    Since $\kappa\tilde{\kappa}\le1$, we complete the proof with $c_7=2/\sigma_*^2$.

\end{proof}
\subsection{Proof of Lemma~\ref{lem: valid covariance 3}}\label{sec: proof valid covariance 3}
\begin{proof}
    The proof is similar to that of \Cref{lem: valid covariance}. It remains to verify that the following conditions hold with probability arbitrarily close to 1 under the prior distribution $\pi_5$:

    \textit{Sparsity control of $\alphab_2$:} For the random vector $\alphab_2$ defined in \eqref{eq: moderate sparse delta 2}, the sparsity level of $\alphab_2$ is a sum of independent Bernoulli random variables with parameters $q^{(2)}_j$. Let
    $$\mu_2\stackrel{\triangle}{=}\sum_{j\in S_3}q_j^{(2)}=\frac{k_u}{8}\frac{\sum_{j\in S_5}|\xi_j|}{\sqrt{p_5\sum_{j\in S_5}\xi_j^2}}\le  \frac{k_u}{8},$$
    where the inequality follows from Cauchy-Schwarz inequality. Consequently, based on Chernoff's inequality \cite[Theorem 2.3.1]{vershynin2018high}, we have
    \begin{align*}
        \bbP(\|\alphab_2\|_0\ge  \frac{k_u}{4})&\le  \exp\left[-\mu_2+\frac{k_u}{4}\ln\left(\frac{4e\mu_2}{k_u}\right)\right]\\
        &\stackrel{(*)}\le  \exp\left[-\frac{k_u}{8}+\frac{k_u}{4}\ln(\frac{e}{2})\right]\le  \exp\left[- \frac{\ln(4/e) }{8}k_u\right],
    \end{align*}
    where (*) follows from the fact that the function $x\mapsto -x+(k_u/4)\ln x$ is increasing when $x\le  k_u/4$. Therefore, we have $\|\alphab_2\|_0\le  k_u/4$ with high probability.

    \textit{Eigenvalues Control of $\Sigma$:} Similar to \Cref{sec: proof valid covariance}, $\Sigma$ has largest and smallest eigenvalues given by $1\pm \|\alphab_2\|_2\|\alphab_1\|_2$ and the remaining $p-2$ eigenvalues are all equal to 1. Specifically, we have $$\|\alphab_1\|_2=c_8\sqrt{\frac{k_u\log p}{n}},\quad \text{and}\quad \|\alphab_2\|_2=\sqrt{\|\alphab_2\|_0}\frac{\sqrt{p_5}}{k_u}.$$ Since $\|\alphab_2\|_0\le  k_u/4$ with high probability from the analysis above, we have
    $$\|\alphab_1\|_2\|\alphab_2\|_2\le  \frac{c_8}{2} \sqrt{\frac{p_5\log p}{n}}.$$
    Since $p_5\le  k_{\rm eff}\lesssim n/\log p$, by choosing $c_8$ sufficiently small, we have $1/M_1\le  \lambda_{\min}(\Sigma)\le  \lambda_{\max}(\Sigma)\le  M_1$ with high probability.

    \textit{Sparsity control of $\beta$:} Following the analysis in \Cref{sec: proof valid covariance}, we have
    \begin{align*}
        \beta_{S_3}&=-\frac{\kappa \|\alphab_1\|_2^2}{1-\|\alphab_1\|_2^2\|\alphab_2\|_2^2}\alphab_2,\\
        \beta_{S_4}&=\frac{\kappa}{1-\|\alphab_1\|_2^2\|\alphab_2\|_2^2}\alphab_1.
    \end{align*}
    Note that $\|\alphab_1\|_0\le  k_u/4$ and $\|\alphab_2\|_0\le  k_u/4$ with high probability. Therefore, we have $\|\beta\|_0= \|\alphab_2\|_0+\|\alphab_1\|_0\le  k_u/2$ with high probability.

    \textit{Control of $\kappa$:} We just need to show that there exists some constant $c_9>0$ such that
    \begin{equation}\label{eq: kappa equation 2}
        \frac{-\|\alphab_1\|_2^2\xi_{S_3}^\top \alphab_2}{1-\|\alphab_1\|_2^2\|\alphab_2\|_2^2}+\frac{\xi_{S_4}^\top \alphab_1}{1-\|\alphab_1\|_2^2\|\alphab_2\|_2^2}\ge  c_9\nu_3\frac{k_u\log p}{n}
    \end{equation}
    with high probability. Specifically, we have $\xi_{S_4}^\top \alphab_1\ge  0$ based on \eqref{eq: moderate sparse delta 2} and $$\frac{\|\alphab_1\|_2^2}{1-\|\alphab_1\|_2^2\|\alphab_2\|_2^2}\asymp \|\alphab_1\|_2^2=c_8^2\frac{k_u\log p}{n}.$$
    Moreover, we have
    \begin{align*}
        -\xi_{S_3}^\top \alphab_2&=\sum_{j\in S_3}|\xi_j| \frac{\sqrt{p_5}}{k_u}b^{(2)}_j,
    \end{align*}
    where $b^{(2)}_j\sim {\rm Bernoulli}(q^{(2)}_j)$. Therefore, we have
    \begin{align*}
        \bbE(-\xi_{S_3}^\top \alphab_2)&=\frac{\sqrt{p_5}}{k_u}\sum_{j\in S_5}|\xi_j| q^{(2)}_j=\frac{1}{8}\sqrt{\sum_{j\in S_5}\xi_j^2}.\\
        {\rm Var}(-\xi_{S_5}^\top \alphab_2)&=\frac{p_5}{k_u^2}\sum_{j\in S_3}\xi_j^2q^{(2)}_j(1-q^{(2)}_j)\le  \frac{\sqrt{p_5}}{8k_u}\frac{\sum_{j \in S_5}|\xi_j|^3}{\sqrt{\sum_{j\in S_5}\xi_j^2}}.
    \end{align*}
    Note that $p_5\le  k_{\rm eff}\lesssim k_u^2/\log p$ and
    $$\sum_{j\in S_5}|\xi_j|^3\le  |\xi_1|\sum_{j\in S_5}\xi_j^2\le  \left(\sum_{j\in S_5}\xi_j^2\right)^{3/2}.$$
    Therefore, we have $${\rm Var}(-\xi_{S_3}^\top \alphab_2)\lesssim \frac{1}{\log p}\left(\bbE(-\xi_{S_3}^\top \alphab_2)\right)^2.$$
    By Chebyshev's inequality, we have for any constant $c>0$, with probability asymptotically at least $1-c/4$, we have $$-\xi_{S_3}^\top \alphab_2\ge  \frac{1}{16}\sqrt{\sum_{j\in S_5}\xi_j^2}.$$
    Based on the assumption at the beginning of \Cref{sec: proof computational lower bound}, we have
    $$\sum_{j\in S_5}\xi_j^2\ge  \frac{1}{2}\sum_{j\in S_3}\xi_j^2=\frac{1}{2}\nu_3^2$$
    Therefore, we can choose $c_9>0$ in \eqref{eq: kappa equation 2} sufficiently small such that we have $0<\kappa\le  1$ with large probability.

    \textit{Control of $\sigma$:} Finally, we verify that $\sigma$ is bounded between $0$ and $M_2$. Note that $$\sigma^2=\sigma_*^2-\beta^\top \Sigma\beta=\sigma_*^2-\frac{\kappa^2}{1-\|\alphab_1\|_2^2\|\alphab_2\|_2^2}\|\alphab_1\|_2^2.$$ Combining the previous analysis, we can choose $c_8$ sufficiently small such that $\sigma\in (0,M_2)$.
\end{proof}

\subsection{Proof of Lemma~\ref{lem: low degree integral}}\label{sec: proof low degree integral}
\begin{proof}
    There are two statements to prove:
    \begin{equation}\label{eq: statement 1}
        \bbE_{\theta_*}(L_\theta^{\le D}, L_{\tilde{\theta}}^{\le D})\le  \bbE_{\theta_*}(L_\theta, L_{\tilde{\theta}})
    \end{equation}
    and
    \begin{equation}\label{eq: statement 2}
        \bbE_{\theta_*}\left((L_{\theta}^{\le D})\right)^2\le 9(6npD)^{4D}.
    \end{equation}
    Without loss of generality, we scale the problem such that $\sigma_*=1$.
    Then $\bbP_{\theta_*}^n$ is the standard Gaussian distribution on $N={n\times(p+1)}$ dimensions and we denote it by $\mathbb{Q}$. Then the space $L^2(\mathbb{Q})$ admits the orthogonal basis of Hermite polynomials, see \cite{szeg1939orthogonal} for a standard reference. We denote these by $(H_{\alpha})_{\alpha\in\mathbb{N}^{N}}$ (where $0\in\mathbb{N}$ by convention),
where
\[
H_{\alpha}(z)
    = \prod_{i=1}^{N} h_{\alpha_{i}}(z_{i})
\]
for univariate Hermite polynomials $(h_{j})_{j\in\mathbb{N}}$, and
$z \in \mathbb{R}^{N}$.
We adopt the normalization where $\|H_{\alpha}\|_{\mathbb{Q}} = 1$,
which is not usually the standard convention in the literature.
This basis is graded in the sense that for any $D \in \mathbb{N}$,
$(H_{\alpha})_{\alpha\in\mathbb{N}^{N},\,|\alpha|\le D}$
is an orthonormal basis for the polynomials of degree at most $D$,
where $|\alpha| := \sum_{i=1}^{N} \alpha_{i}$.

Regarding \(\bbP_\theta^n\), we will make use of the following fact:
Since \(\theta=h(g_3(\alphab_1,\alphab_2))\), after the scaling
\(\sigma_*=1\), the diagonal entries of the covariance matrix in
\eqref{eq: null point mat 3} are all equal to one. Hence each coordinate
\(Z_i\) has standard normal marginal distribution under \(\bbP_\theta^n\).

\textit{Proof of \eqref{eq: statement 1}:}
Throughout this argument, $\theta$ and $\tilde{\theta}$ are generated by the
restricted prior used in \Cref{sec: proof computational lower bound}, and we
work on the validity event from \Cref{lem: valid covariance 3}; in particular,
their associated coefficients satisfy
\[
    0<\kappa\le 1,
    \qquad
    0<\tilde{\kappa}\le 1.
\]
Expanding in the orthonormal Hermite basis $\{H_{\alpha}\}$, we have
\begin{equation}\label{eq: low degree expansion}
    \begin{aligned}
        \left\langle L^{\le D}_\theta,\, L^{\le D}_{\tilde{\theta}} \right\rangle_{\mathbb{Q}}
        &= \sum_{\alpha\in\mathbb{N}^{N},\, |\alpha|\le D}
            \left\langle L_\theta, H_{\alpha} \right\rangle_{\mathbb{Q}}
            \left\langle L_{\tilde{\theta}}, H_{\alpha} \right\rangle_{\mathbb{Q}} \\
        &= \sum_{\alpha\in\mathbb{N}^{N},\, |\alpha|\le D}
            \bbE_{Z \sim \mathbb{P}_\theta^n}\!\left[H_{\alpha}(Z)\right]\,
            \bbE_{Z \sim \mathbb{P}_{\tilde{\theta}}^n}\!\left[H_{\alpha}(Z)\right],
    \end{aligned}
\end{equation}
where we have used the change-of-measure identity
$\bbE_{\mathbb{Q}}[ L_\theta f ] = \bbE_{\mathbb{P}_\theta^n}[ f ]$.

We claim that for every Hermite multi-index $\alpha$,
\begin{equation}\label{eq: claim low degree}
    \bbE_{Z \sim \mathbb{P}_\theta^n}\!\left[H_{\alpha}(Z)\right]\,
    \bbE_{Z \sim \mathbb{P}_{\tilde{\theta}}^n}\!\left[H_{\alpha}(Z)\right]
    \ge 0 .
\end{equation}

\begin{lemma}[Rank-one Hermite covariance identity]
\label{lem:rank-one-hermite-covariance}
Let $(U,V)$ be a centered Gaussian vector satisfying
\[
    \operatorname{Cov}(U)=I_a,\qquad
    \operatorname{Cov}(V)=I_b,\qquad
    \operatorname{Cov}(U,V)=r c^\top .
\]
For multi-indices $\mu\in\mathbb{N}^{a}$ and $\nu\in\mathbb{N}^{b}$, write
\[
    h_\mu(U)=\prod_{\ell=1}^a h_{\mu_\ell}(U_\ell),
    \qquad
    h_\nu(V)=\prod_{j=1}^b h_{\nu_j}(V_j),
\]
where the univariate Hermite polynomials have the normalization fixed above.
Then
\[
    \bbE\{h_\mu(U)h_\nu(V)\}=0
    \qquad \text{if } |\mu|\ne|\nu|,
\]
and, if $|\mu|=|\nu|=m$, then
\[
    \bbE\{h_\mu(U)h_\nu(V)\}
    =
    \frac{m!}{\sqrt{\mu!\nu!}}\,
    r^\mu c^\nu ,
\]
where
\[
    \mu!=\prod_{\ell=1}^a \mu_\ell!,
    \qquad
    \nu!=\prod_{j=1}^b \nu_j!,
    \qquad
    r^\mu=\prod_{\ell=1}^a r_\ell^{\mu_\ell},
    \qquad
    c^\nu=\prod_{j=1}^b c_j^{\nu_j}.
\]
\end{lemma}

\begin{proof}[Proof of \Cref{lem:rank-one-hermite-covariance}]
Let
\[
    G(s,t)
    =
    \bbE\left[
        \prod_{\ell=1}^a
        \exp\{s_\ell U_\ell-s_\ell^2/2\}
        \prod_{j=1}^b
        \exp\{t_j V_j-t_j^2/2\}
    \right] .
\]
Since \((U,V)\) is centered Gaussian,
\[
\begin{aligned}
G(s,t)
= &\bbE\left[
    \prod_{\ell=1}^a
    \exp\{s_\ell U_\ell-s_\ell^2/2\}
    \prod_{j=1}^b
    \exp\{t_j V_j-t_j^2/2\}
\right] \\
\qquad = & \bbE\left[\exp(s^{\top} U+t^{\top} V)\right] \exp(-\|s\|^2/2-\|t\|^2/2)\\
\qquad = &
\exp(s^\top r c^\top t),
\end{aligned}
\]
where the last equation follows from $\operatorname{Var}\left(s^{\top} U+t^{\top} V\right)=\|s\|^2+\|t\|^2+2 s^{\top} r c^{\top} t$.

Hence
\[
\begin{aligned}
G(s,t)
&=
\sum_{k=0}^\infty
\frac{1}{k!}(s^\top r)^k(c^\top t)^k  \\
&=
\sum_{k=0}^\infty
\sum_{|\mu|=k}
\sum_{|\nu|=k}
\frac{k!}{\mu!\nu!}
r^\mu c^\nu s^\mu t^\nu .
\end{aligned}
\]
Thus the coefficient of \(s^\mu t^\nu\) in \(G(s,t)\) is zero unless
\(|\mu|=|\nu|\). If \(|\mu|=|\nu|=m\), that coefficient is
\[
    \frac{m!}{\mu!\nu!}r^\mu c^\nu .
\]

On the other hand, by the generating function of the normalized Hermite
polynomials,
\[
    \exp\{sx-s^2/2\}
    =
    \sum_{k=0}^\infty h_k(x)\frac{s^k}{\sqrt{k!}},
\]
we have
\[
\begin{aligned}
G(s,t)
&=
\sum_{\mu,\nu}
\bbE\{h_\mu(U)h_\nu(V)\}
\frac{s^\mu t^\nu}{\sqrt{\mu!\nu!}} .
\end{aligned}
\]
Comparing the coefficient of \(s^\mu t^\nu\) gives
\[
    \frac{\bbE\{h_\mu(U)h_\nu(V)\}}{\sqrt{\mu!\nu!}}
    =
    \frac{m!}{\mu!\nu!}r^\mu c^\nu
\]
when \(|\mu|=|\nu|=m\), and gives zero otherwise. Therefore
\[
    \bbE\{h_\mu(U)h_\nu(V)\}
    =
    \frac{m!}{\sqrt{\mu!\nu!}}r^\mu c^\nu
\]
when \(|\mu|=|\nu|=m\), while the expectation is zero when
\(|\mu|\ne |\nu|\). This proves the lemma.
\end{proof}

We apply \Cref{lem:rank-one-hermite-covariance} row by row.  After the
rescaling $\sigma_*=1$, write each observation under the covariance model $g_3$
as
\[
    Z_i=(Y_i,X_{i,S_3},X_{i,S_4}),\qquad i=1,\ldots,n.
\]
Following the coordinate convention used in the construction of $g_3$ in
\Cref{sec: proof computational lower bound}, $\alphab_1\in\mathbb{R}^{p_4}$
is the coefficient vector paired with $X_{i,S_4}$, and
$\alphab_2\in\mathbb{R}^{p_3}$ is the coefficient vector paired with $X_{i,S_3}$.
For $\theta=h(g_3(\alphab_1,\alphab_2))$, the covariance matrix of a single
observation, denoted by
\[
\Sigma_\theta^z=
\begin{pmatrix}
1 & 0 & \kappa\alphab_1^\top\\
0 & I_{p_3} & \alphab_2\alphab_1^\top\\
\kappa\alphab_1 & \alphab_1\alphab_2^\top & I_{p_4}
\end{pmatrix}.
\]
Define
\[
    U_i=(Y_i,X_{i,S_3})\in\mathbb{R}^{1+p_3},
    \qquad
    V_i=X_{i,S_4}\in\mathbb{R}^{p_4}.
\]
Then
\[
    \operatorname{Cov}_\theta(U_i)=I_{1+p_3},
    \qquad
    \operatorname{Cov}_\theta(V_i)=I_{p_4},
    \qquad
    \operatorname{Cov}_\theta(U_i,V_i)=r_\theta c_\theta^\top,
\]
where
\[
    r_\theta=(\kappa,\alphab_2^\top)^\top\in\mathbb{R}^{1+p_3},
    \qquad
    c_\theta=\alphab_1\in\mathbb{R}^{p_4}.
\]
The same representation holds for $\tilde{\theta}$ with
\[
    r_{\tilde{\theta}}=(\tilde{\kappa},\tilde{\alphab}_2^\top)^\top,
    \qquad
    c_{\tilde{\theta}}=\tilde{\alphab}_1.
\]

The signs of the corresponding coordinates of these rank-one factors agree.
Let $\tilde I$ and $\tilde b_j^{(2)}$ denote the support and Bernoulli variables
used to construct $\tilde{\alphab}_1$ and $\tilde{\alphab}_2$.
By the construction in \eqref{eq: moderate sparse delta 2},
\[
    (\alphab_1)_j
    =
    c_8\operatorname{sign}(\xi_{p_3+j})
    \sqrt{\frac{\log p}{n}}\,\mathbf 1\{j\in I\},
\]
and the same formula holds for $\tilde{\alphab}_1$ with $\tilde I$ in place of
$I$.  Hence
\[
    (\alphab_1)_j(\tilde{\alphab}_1)_j\ge 0
    \qquad \text{for every } j.
\]
Similarly,
\[
    (\alphab_2)_j
    =
    -\frac{\sqrt{p_5}}{k_u}\operatorname{sign}(\xi_j)b_j^{(2)}
\]
and the same formula holds for $\tilde{\alphab}_2$, so
\[
    (\alphab_2)_j(\tilde{\alphab}_2)_j\ge 0
    \qquad \text{for every } j.
\]
Together with $0<\kappa,\tilde{\kappa}\le 1$, this gives
\begin{equation}\label{eq: rank-one sign relations}
    (r_\theta)_\ell(r_{\tilde{\theta}})_\ell\ge 0
    \quad\text{for every } \ell,
    \qquad
    (c_\theta)_j(c_{\tilde{\theta}})_j\ge 0
    \quad\text{for every } j.
\end{equation}

Let $\alpha^{(i)}$ be the part of $\alpha$ corresponding to the $i$th
observation, and write
\[
    \alpha^{(i)}=(\mu^{(i)},\nu^{(i)})
\]
according to the coordinates $U_i=(Y_i,X_{i,S_3})$ and $V_i=X_{i,S_4}$, where
$\mu^{(i)}\in\mathbb{N}^{1+p_3}$ and
$\nu^{(i)}\in\mathbb{N}^{p_4}$.  With this notation,
\[
    H_{\alpha^{(i)}}(Z_i)
    =
    h_{\mu^{(i)}}(U_i)h_{\nu^{(i)}}(V_i).
\]
Since the observations are independent across $i=1,\ldots,n$,
\begin{equation}\label{eq:hermite-product-form}
        H_\alpha(Z)
    =
    \prod_{i=1}^n H_{\alpha^{(i)}}(Z_i),
    \qquad
    \bbE_{\theta} H_\alpha(Z)
    =
    \prod_{i=1}^n
    \bbE_{\theta} H_{\alpha^{(i)}}(Z_i),
\end{equation}

and the same factorization holds under $\tilde{\theta}$.
Fix a row $i$.  If $|\mu^{(i)}|\ne|\nu^{(i)}|$, then
\Cref{lem:rank-one-hermite-covariance} gives
\[
    \bbE_\theta H_{\alpha^{(i)}}(Z_i)
    =
    \bbE_{\tilde{\theta}} H_{\alpha^{(i)}}(Z_i)
    =
    0.
\]
If $|\mu^{(i)}|=|\nu^{(i)}|=m_i$, set
\[
    K_i=\frac{m_i!}{\sqrt{\mu^{(i)}!\nu^{(i)}!}}>0.
\]
Then \Cref{lem:rank-one-hermite-covariance} gives
\[
    \bbE_\theta H_{\alpha^{(i)}}(Z_i)
    =
    K_i r_\theta^{\mu^{(i)}}c_\theta^{\nu^{(i)}},
    \qquad
    \bbE_{\tilde{\theta}} H_{\alpha^{(i)}}(Z_i)
    =
    K_i r_{\tilde{\theta}}^{\mu^{(i)}}c_{\tilde{\theta}}^{\nu^{(i)}}.
\]
Therefore, in the nonzero case,
\[
\begin{aligned}
&\bbE_\theta H_{\alpha^{(i)}}(Z_i)\,
\bbE_{\tilde{\theta}} H_{\alpha^{(i)}}(Z_i)\\
&\qquad =
K_i^2
\prod_{\ell=1}^{1+p_3}
\bigl\{(r_\theta)_\ell(r_{\tilde{\theta}})_\ell\bigr\}^{\mu^{(i)}_\ell}
\prod_{j=1}^{p_4}
\bigl\{(c_\theta)_j(c_{\tilde{\theta}})_j\bigr\}^{\nu^{(i)}_j}.
\end{aligned}
\]
The inequality \eqref{eq: rank-one sign relations} and the nonnegative integer
exponents $\mu^{(i)}_\ell,\nu^{(i)}_j$ imply that the last display is
nonnegative.  Thus, in both cases,
\[
    \bbE_\theta H_{\alpha^{(i)}}(Z_i)\,
    \bbE_{\tilde{\theta}} H_{\alpha^{(i)}}(Z_i)
    \ge 0
    \qquad \text{for every } i.
\]
Taking the product over $i=1,\ldots,n$ proves \eqref{eq: claim low degree}.

We now complete the proof of \eqref{eq: statement 1}.  By
\eqref{eq: claim low degree}, every summand in \eqref{eq: low degree expansion}
is nonnegative.  Hence
\[
\begin{aligned}
    \left\langle L^{\le D}_\theta, L^{\le D}_{\tilde{\theta}} \right\rangle_{\mathbb{Q}}
    &=
    \sum_{\alpha\in\mathbb{N}^{N},\,|\alpha|\le D}
    \bbE_{\theta}H_\alpha(Z)\,
    \bbE_{\tilde{\theta}}H_\alpha(Z)\\
    &\le
    \sum_{\alpha\in\mathbb{N}^{N}}
    \bbE_{\theta}H_\alpha(Z)\,
    \bbE_{\tilde{\theta}}H_\alpha(Z).
\end{aligned}
\]
For the valid covariance matrices considered here, the likelihood ratios
$L_\theta$ and $L_{\tilde{\theta}}$ belong to $L^2(\mathbb{Q})$.  Since the
Hermite polynomials form a complete orthonormal basis of $L^2(\mathbb{Q})$, the
right-hand side equals
\[
    \left\langle L_\theta,L_{\tilde{\theta}}\right\rangle_{\mathbb{Q}}
    =
    \bbE_{\theta_*}\!\left(L_\theta L_{\tilde{\theta}}\right).
\]
Therefore,
\[
    \bbE_{\theta_*}\!\left(
        L_\theta^{\le D}L_{\tilde{\theta}}^{\le D}
    \right)
    \le
    \bbE_{\theta_*}\!\left(
        L_\theta L_{\tilde{\theta}}
    \right),
\]
which proves \eqref{eq: statement 1}.

\textit{Proof of \eqref{eq: statement 2}:}
Expanding in the orthonormal basis $\{H_{\alpha}\}$, we have for any $\theta$,
\[
\|L^{\le D}_\theta\|^{2}_{\mathbb{Q}}
    = \sum_{|\alpha|\le D} \langle L_{\theta}, H_{\alpha}\rangle_{\mathbb{Q}}^{2}
    = \sum_{|\alpha|\le D}
        \left(
            \mathbb{E}_{Z\sim \bbP_\theta^n}\, H_{\alpha}(Z)
        \right)^{2}
    \le (N+1)^{D}\,
        \max_{\alpha:\,|\alpha|\le D}
        \left(
            \mathbb{E}_{Z\sim \bbP_\theta^n}\, H_{\alpha}(Z)
        \right)^{2}.
\]

For each $a \in \mathbb{N}$, the function $h_a$ admits the expansion
\[
h_{a}(z)
    = \frac{1}{\sqrt{a!}} \sum_{j=0}^{a} c_{a,j} z^{j},
\]
where the coefficients satisfy $\sum_{j=0}^{a} |c_{a,j}| = T(a)$.
Here, \(T(a)\) is known as the \emph{telephone number}, which counts the number
of involutions on \(a\) elements \cite{banderier2002generating}.
In particular, we have the trivial upper bound $T(a) \le a!$.
This means for any $a \ge 1$ and $q \in [1,\infty)$,
\begin{align*}
\mathbb{E}|h_{a}(z)|^{q}
    &= \mathbb{E}\left|
        \frac{1}{\sqrt{a!}}
        \sum_{j=0}^{a} c_{a,j} z^{j}
      \right|^{q}
    \le \mathbb{E}\left(
        \sqrt{a!} \max_{0 \le j \le a} |z|^{j}
      \right)^{q}
    = (a!)^{q/2} \mathbb{E}\left( \max\{1, |z|^{a}\} \right)^{q}\\
&= (a!)^{q/2}\, \mathbb{E}\max\{1, |z|^{aq}\}
    \le (a!)^{q/2}\,\left(1 + \mathbb{E}|z|^{aq}\right).
\end{align*}
Using the formula for Gaussian moments, and that for all $x \ge 1$,
$\Gamma(x) \le x^{x}$ (see e.g.\ \cite{li2007inequalities}), we have
\begin{equation}\label{eq:hermite-h-q-bound}
\begin{aligned}
\mathbb{E}|h_{a}(z)|^{q}&\le (a!)^{q/2}
    \left(
        1 + \pi^{-1/2} 2^{aq/2} \Gamma\!\left( \frac{aq+1}{2} \right)
    \right)\\
&\le a^{aq/2}
   \left(
      1 + 2^{aq/2}
      \left( \frac{aq+1}{2} \right)^{(aq+1)/2}
   \right)\\
&\le a^{aq/2}\bigl(1 + 2^{aq/2} (aq)^{aq}\bigr)
    \qquad \text{since } \frac{aq+1}{2} \le aq\\
&\le 2 a^{aq/2}\, 2^{aq/2} (aq)^{aq}
\le 2 (2aq)^{2aq}.
\end{aligned}
\end{equation}

Now fix \(\alpha\) with \(|\alpha|\le D\), and set
\(d=|\alpha|=\sum_i\alpha_i\). If \(d=0\), then \(H_\alpha\equiv1\), so
the desired bound is immediate. Hence assume \(d\ge1\).

Using the product form in \Cref{eq:hermite-product-form},
\(h_0\equiv 1\), H\"older's inequality with exponents
\(d/\alpha_i\) for indices \(i\) such that \(\alpha_i>0\),
we obtain
\[
\begin{aligned}
\left|
    \mathbb{E}_{Z\sim\bbP_\theta^n}H_\alpha(Z)
\right|
&\le
\mathbb{E}_{Z\sim\bbP_\theta^n}
    \prod_{i:\alpha_i>0}|h_{\alpha_i}(Z_i)|  \\
&\le
\prod_{i:\alpha_i>0}
\left(
    \mathbb{E}_{Z\sim\bbP_\theta^n}
    |h_{\alpha_i}(Z_i)|^{d/\alpha_i}
\right)^{\alpha_i/d},
\end{aligned}
\]
because
\(\sum_{i:\alpha_i>0}\alpha_i/d=1\).

For each \(i\) with $\alpha_i>0$,  set
\(a=\alpha_i\) and \(q=d/\alpha_i\).
Then \(a\ge1\), \(q\ge1\), and
\(aq=d\).
Applying \Cref{eq:hermite-h-q-bound} to the standard
normal variable \(Z_i\) yields
\[
\mathbb{E}_{Z\sim\bbP_\theta^n}|h_{\alpha_i}(Z_i)|^{d/\alpha_i}
\le
2(2d)^{2d}.
\]
Consequently,
\[
\left|
    \mathbb{E}_{Z\sim\bbP_\theta^n}H_\alpha(Z)
\right|
\le
\prod_{i:\alpha_i>0}
\left(2(2d)^{2d}\right)^{\alpha_i/d}
=
\left(2(2d)^{2d}\right)^{\sum_{i:\alpha_i>0}\alpha_i/d}
=
2(2d)^{2d}.
\]
Since \(d=|\alpha|\le D\), this gives
\[
\left|
    \mathbb{E}_{Z\sim\bbP_\theta^n}H_\alpha(Z)
\right|
\le 2(2D)^{2D}.
\]

Finally, using the bound
\(
N+1 = n(p+1)+1 \le 3np\),
we have
\[
\|L^{\le D}_{\theta}\|_{Q}^{2}
    \le (N+1)^{D} \bigl[2(2D)^{2D}\bigr]^2
    \le 4 (3np)^{D} (2D)^{4D}
    \le 4 (6npD)^{4D},
\]
which completes the proof.

\end{proof}

\section{Additional Results}\label{sec: additional results}

\subsection{Low-degree polynomial framework}\label{sec: low-degree framwork}

In this section, we briefly recall the low-degree polynomial method, which provides a widely used
formal framework for studying computational limits in high-dimensional inference problems.
The key idea is to analyze algorithms that can be represented (or well-approximated) by evaluating
polynomials of bounded degree in the observed data; we refer readers to the survey \cite{wein2025computational} for broader background and further details.

\textit{Setup.}
Let $\{\bbP^n_{\pi_1}\}_{n\ge1}$ and $\{\bbP^n_{\pi_2}\}_{n\ge1}$ be two sequences of probability
measures on an observation space $(\Omega_n,\mathcal F_n)$, corresponding to the null and alternative
(or to two priors supported on two parameter spaces). Let ${\cal Z}\in\Omega_n$ denote the observed data.
For $D\in\mathbb N$, write $\mathbb{R}[{\cal Z}]_{\le D}$ for the space of multivariate polynomials in
the coordinates of ${\cal Z}$ of total degree at most $D$. As in the main text, the degree $D=D_n$ is
allowed to grow with $n$, and with a slight abuse of notation we view a \emph{polynomial} as a sequence
$f=(f_n)_{n\ge1}$ with $f_n\in\mathbb{R}[{\cal Z}]_{\le D_n}$.

\textit{Weak and strong separation.}
The low-degree method quantifies the ability of degree-bounded polynomials to distinguish
$\bbP^n_{\pi_1}$ and $\bbP^n_{\pi_2}$ through signal-to-noise separation notions.
We recall weak separation in Definition~\ref{def: separation} and state the corresponding strong notion.

\begin{definition}[Strong separation]\label{def: strong-separation}
We say that $f\in \mathbb{R}[{\cal Z}]_{\le D}$ \emph{strongly separates}
$\bbP^n_{\pi_1}$ and $\bbP^n_{\pi_2}$ if, as $n \to \infty$,
\[
\sqrt{
\max\!\left\{
\Var_{\bbP^n_{\pi_1}}(f({\cal Z})),
\Var_{\bbP^n_{\pi_2}}(f({\cal Z}))
\right\}
}
= o\!\left(
\left|
\mathbb{E}_{\bbP^n_{\pi_1}}[f({\cal Z})]
-
\mathbb{E}_{\bbP^n_{\pi_2}}[f({\cal Z})]
\right|
\right).
\]
\end{definition}

If $f$ weakly separates, then thresholding $f({\cal Z})$ at an appropriate level yields a test
with nontrivial advantage (weak detection). If $f$ strongly separates, thresholding yields a test
whose sum of type-I and type-II errors tends to zero (strong detection). We do not reproduce
these standard reductions here.

\textit{Low-degree likelihood ratio norm.}
Set $\mathbb{Q}_1:=\bbP^n_{\pi_1}$ and $\mathbb{Q}_2:=\bbP^n_{\pi_2}$, and assume
$\mathbb{Q}_1$ is absolutely continuous with respect to $\mathbb{Q}_2$.
Define the likelihood ratio
\[
L \;=\; \frac{\rmd\mathbb{Q}_1}{\rmd\mathbb{Q}_2}.
\]
Endow $L^2(\mathbb{Q}_2)$ with inner product
$\langle f,g\rangle := \mathbb{E}_{\mathbb{Q}_2}[f({\cal Z})g({\cal Z})]$.
Let $L^{\le D}$ denote the $L^2(\mathbb{Q}_2)$-orthogonal projection of $L$ onto the polynomial
subspace $\mathbb{R}[{\cal Z}]_{\le D}$. The associated low-degree quantity is
\begin{equation}\label{eq:LD-def-app}
\mathrm{LD}(D)
\;:=\;
\|L^{\le D}\|_{L^2(\mathbb{Q}_2)}^2
\;=\;
\mathbb{E}_{\mathbb{Q}_2}\!\left[(L^{\le D}({\cal Z}))^2\right].
\end{equation}
By construction, $\mathrm{LD}(D)\ge 1$ since $1\in\mathbb{R}[{\cal Z}]_{\le D}$ and
$\mathbb{E}_{\mathbb{Q}_2}[L]=1$.

\textit{Interpreting $\mathrm{LD}(D)$: weak vs strong low-degree indistinguishability.}
The following implication formalizes the meaning of the two regimes
$\mathrm{LD}(D)\to 1$ versus $\mathrm{LD}(D)=O(1)$; we refer to
\cite{bandeira2022franz,schramm2022computational} and the aforementioned survey for proofs and refinements.

\begin{proposition}[Low-degree obstruction to separation]\label{prop:ld-weak-strong}
Let $\mathbb{Q}_1=\bbP^n_{\pi_1}$ and $\mathbb{Q}_2=\bbP^n_{\pi_2}$ with $\mathbb{Q}_1\ll \mathbb{Q}_2$.
Fix $D=D_n$.
\begin{enumerate}
\item If $\mathrm{LD}(D)=1+o(1)$ as $n\to\infty$, then no degree-$D$ polynomial weakly separates
$\mathbb{Q}_1$ and $\mathbb{Q}_2$.
\item If $\mathrm{LD}(D)=O(1)$ as $n\to\infty$, then no degree-$D$ polynomial strongly separates
$\mathbb{Q}_1$ and $\mathbb{Q}_2$.
\end{enumerate}
\end{proposition}

\textit{Remark.}
Statement (i) asserts that, when $\mathrm{LD}(D)=1+o(1)$, every degree-$D$ polynomial test is
asymptotically no better than the trivial constant statistic in the sense of weak separation.
Statement (ii) is stronger in that it rules out vanishing total error: boundedness of $\mathrm{LD}(D)$
permits at most a constant-factor signal-to-noise ratio and therefore precludes strong separation.
For additional perspectives and related connections (e.g.\ to sum-of-squares lower bounds and
pseudo-calibration), we refer readers to \cite{hopkins2018statistical} and the survey references above.

\subsection{Restricted eigenvalue condition for Gaussian designs}
\label{sec: restricted eigenvalue}

We briefly justify the restricted eigenvalue condition used after
\eqref{eq: restricted eigenvalue} under Gaussian random designs, and explain how
it implies the estimator bounds in Condition~\ref{cdt: linear estimator}. Recall
that
\[
\kappa(X,k,\alpha_0)
=
\min_{\substack{J_0\subseteq\{1,\ldots,p\}\\ |J_0|\le k}}
\;
\min_{\substack{\delta\neq 0\\
\|\delta_{J_0^c}\|_1\le \alpha_0\|\delta_{J_0}\|_1}}
\frac{\|X\delta\|_2}{\sqrt n\,\|\delta_{J_0}\|_2}.
\]
Direct verification of this condition is computationally difficult
\citep{bandeira2013certifying,tillmann2013computational}. For Gaussian random
designs, however, it follows from standard uniform concentration results. In
particular, \cite{raskutti2010restricted} show that if
\(X\in\mathbb R^{n\times p}\) has i.i.d.\ \(N(0,\Sigma)\) rows and
\(\rho(\Sigma)=\sqrt{\max_{1\le j\le p}\Sigma_{jj}}\), then with probability at
least \(1-c'\exp(-cn)\),
\[
    \frac{\|Xv\|_2}{\sqrt n}
    \ge
    \frac14\|\Sigma^{1/2}v\|_2
    -
    9\rho(\Sigma)\sqrt{\frac{\log p}{n}}\|v\|_1,
    \qquad \text{for all } v\in\mathbb R^p .
\]

Suppose further that
\[
    M_1^{-1}\le \lambda_{\min}(\Sigma)
    \le \lambda_{\max}(\Sigma)\le M_1 .
\]
For any vector \(\delta\) in the cone
\(\|\delta_{J_0^c}\|_1\le \alpha_0\|\delta_{J_0}\|_1\) with
\(|J_0|\le k_u\), we have
\[
    \|\delta\|_1
    \le
    (1+\alpha_0)\|\delta_{J_0}\|_1
    \le
    (1+\alpha_0)\sqrt{k_u}\|\delta_{J_0}\|_2 .
\]
Therefore, on the event above,
\[
\frac{\|X\delta\|_2}{\sqrt n\,\|\delta_{J_0}\|_2}
\ge
\frac{1}{4\sqrt{M_1}}
-
9\sqrt{M_1}(1+\alpha_0)
\sqrt{\frac{k_u\log p}{n}} .
\]
Consequently, for each fixed \(\alpha_0>0\), there exists a constant
\(c_{\rm RE}>0\), depending only on \(M_1\) and \(\alpha_0\), such that if
\[
    \frac{k_u\log p}{n}\le c_{\rm RE},
\]
then
\[
    \kappa(X,k_u,\alpha_0)
    \ge
    \frac{1}{8\sqrt{M_1}}
\]
with probability at least \(1-c'\exp(-cn)\).

We next explain how this lower bound yields the constants \(c_\beta\) and
\(C_\beta\) in Condition~\ref{cdt: linear estimator}. Consider the Lasso
estimator
\[
    \hat\beta
    \in
    \argmin_{b\in\mathbb R^p}
    \left\{
        \frac{1}{2n}\|Y-Xb\|_2^2+\lambda\|b\|_1
    \right\},
    \qquad
    \lambda=A\sigma\sqrt{\frac{\log p}{n}},
\]
where \(A>0\) is sufficiently large. On the event
\[
    \frac{1}{n}\left\|X^\top\varepsilon\right\|_\infty
    \le  \frac{\lambda}{2},
\]
the usual basic inequality implies that
\(\Delta=\hat\beta-\beta\) satisfies the cone condition
\[
    \|\Delta_{S^c}\|_1\le  3\|\Delta_S\|_1,
    \qquad
    S=\supp(\beta),\quad |S|\le  k_u .
\]
Thus, if \(\kappa(X,k_u,3)>0\), the restricted eigenvalue condition in
\eqref{eq: restricted eigenvalue} gives
\[
    \frac{\|X\Delta\|_2}{\sqrt n}
    \ge
    \kappa(X,k_u,3)\|\Delta_S\|_2 .
\]
Combining this with the standard Lasso oracle inequality yields
\[
    \|\hat\beta-\beta\|_1
    \le
    \frac{C A}{\kappa^2(X,k_u,3)}
    \sigma k_u\sqrt{\frac{\log p}{n}},
\]
and
\[
    \|\hat\beta-\beta\|_2
    \le
    \frac{C A}{\kappa^2(X,k_u,3)}
    \sigma\sqrt{\frac{k_u\log p}{n}},
\]
where \(C>0\) is a universal constant. Hence, if
\[
    \kappa(X,k_u,3)\ge  \kappa_0>0
\]
with probability approaching one, then Condition~\ref{cdt: linear estimator}
holds with admissible constants
\[
    c_\beta=\frac{C A}{\kappa_0^2},
    \qquad
    C_\beta=\frac{C A}{\kappa_0^2}.
\]
In particular, under Gaussian designs with bounded population eigenvalues, the
previous concentration argument gives such a lower bound with
\(\kappa_0=1/(8\sqrt{M_1})\), provided \(k_u\log p/n\le c_{\rm RE}\). Thus, the
restricted eigenvalue argument gives fixed constants \(c_\beta\) and
\(C_\beta\) depending only on \(A\), \(M_1\), and the cone constant.

\Cref{cdt: linear estimator,cdt: linear variance} are satisfied by the scaled Lasso under Gaussian designs provided
$$
k_u \log p / n \le  c_{\mathrm{RE}}
$$
where $c_{\mathrm{RE}}>0$ depends only on $M_1$ and the cone constant.
This is a small-constant version of the sparsity scaling $k_u \lesssim n / \log p$.

The scaled Lasso satisfies analogous bounds with
\(\lambda\asymp \hat\sigma\sqrt{\log p/n}\), and additionally yields a
consistent estimator of \(\sigma^2\). Hence, under the same restricted
eigenvalue scaling, the scaled Lasso verifies both
Conditions~\ref{cdt: linear estimator} and~\ref{cdt: linear variance}.

\subsection{Performance of some test statistics for ${\rm SCCA}(n,s,p_1,p_2,\lambda)$}\label{sec: test statistics}

In this section we justify the performance of several test statistics for
${\rm SCCA}(n,s,p_1,p_2,\lambda)$ defined in \eqref{eq: sparse cca detection}.
We focus on the regime
\[
s \lesssim p_1 \lesssim n \ll \sqrt{p_2},
\qquad
\text{with } s,p_1,p_2,n\to\infty,
\]
which differs from the classical sparse submatrix detection scaling where typically $p_1\asymp p_2$
(e.g.\ \cite{10.3150/12-BEJ470,ma2015computational}).
As discussed in \Cref{sec: reduction from scca}, it suffices to analyze statistics based on the
sample cross-covariance matrix
\begin{equation}\label{eq:Rhatscca}
\widehat R
\;:=\;
\frac{1}{2n}\sum_{i=1}^{2n} U_{1}^{(i)}\bigl(U_{2}^{(i)}\bigr)^\top
\;\in\;
\mathbb{R}^{p_1\times p_2},
\end{equation}
where under $H_0$ the pairs $\{(U_1^{(i)},U_2^{(i)})\}_{i=1}^{2n}$ are i.i.d.\ from
$N(0,I_{p_1})\otimes N(0,I_{p_2})$.

Under $H_1$, conditional on $(\alphab_1,\alphab_2)$, we have
\begin{equation}\label{eq:meanRhat}
\mathbb E[\widehat R \mid \alphab_1,\alphab_2] \;=\; \lambda \alphab_1\alphab_2^\top.
\end{equation}
Since $\alphab_1,\alphab_2$ are $s$-sparse and flat-on-support, each nonzero entry of
$\lambda \alphab_1\alphab_2^\top$ equals $\lambda/s$, and the signal is supported on an unknown
$s\times s$ submatrix.

\textit{Scan test.} The scan statistic is given by
\begin{equation}\label{eq:Tscan-scca}
T_{\text{scan}}
\;=\;
\max_{\substack{S_1 \subseteq [p_1],\, |S_1| = s\\
S_2 \subseteq [p_2],\, |S_2| = s}}
\frac{1}{s^2}\sum_{i \in S_1} \sum_{j \in S_2} \widehat R_{ij}.
\end{equation}
The test rejects $H_0$ when $T_{\text{scan}}$ exceeds a threshold $\tau_{\text{scan}}$.

Fix $(S_1,S_2)$ with $|S_1|=|S_2|=s$ and define the averaged submatrix sum
\[
Z(S_1,S_2)
\;:=\;
\frac{1}{s^2}\sum_{i \in S_1}\sum_{j \in S_2}\widehat R_{ij}
\;=\;
\frac{1}{2n}\sum_{\ell=1}^{2n}\underbrace{\left(\frac{1}{s}\sum_{i\in S_1}U_{1,i}^{(\ell)}\right)
\left(\frac{1}{s}\sum_{j\in S_2}U_{2,j}^{(\ell)}\right)}_{=: \;A_\ell(S_1)\,B_\ell(S_2)}.
\]
Under $H_0$, conditional on $(S_1,S_2)$, the random variables
\[
A_\ell(S_1)=\frac{1}{s}\sum_{i\in S_1}U_{1,i}^{(\ell)},\qquad
B_\ell(S_2)=\frac{1}{s}\sum_{j\in S_2}U_{2,j}^{(\ell)}
\]
are independent Gaussians with
\[
A_\ell(S_1)\sim N\!\left(0,\frac{1}{s}\right),
\qquad
B_\ell(S_2)\sim N\!\left(0,\frac{1}{s}\right),
\]
hence $A_\ell(S_1)B_\ell(S_2)$ is sub-exponential with scale $\asymp 1/s$.
Applying Bernstein's inequality yields: for all $t\ge 0$,
\begin{equation}\label{eq:Z-tail}
\bbP_{H_0}\!\left(|Z(S_1,S_2)|\ge t\right)
\le
2\exp\!\left(-c\,n\,\min\{s^2 t^2,\, s t\}\right).
\end{equation}
In particular, for $t\le 1/s$,
\begin{equation}\label{eq:Z-subg}
\bbP_{H_0}\!\left(|Z(S_1,S_2)|\ge t\right)
\le
2\exp(-c n s^2 t^2).
\end{equation}

Now take a union bound over all $(S_1,S_2)$:
the number of candidates is
\[
N_{\text{scan}}
=
\binom{p_1}{s}\binom{p_2}{s},
\qquad
\log N_{\text{scan}}
\le
s\log\!\left(\frac{e p_1}{s}\right)+s\log\!\left(\frac{e p_2}{s}\right).
\]
Setting
\begin{equation}\label{eq:tauscan}
\tau_{\text{scan}}
=
C\sqrt{\frac{\log N_{\text{scan}}}{n s^2}}
\;\asymp\;
\sqrt{\frac{\log N_{\text{scan}}}{n s^2}},
\end{equation}
and assuming $\tau_{\text{scan}}\le 1/s$ (which holds in the regime of interest),
\eqref{eq:Z-subg} implies
\begin{equation}\label{eq:type1-scan}
\bbP_{H_0}\!\left(T_{\text{scan}}>\tau_{\text{scan}}\right)
\le
2N_{\text{scan}}\exp(-c n s^2 \tau_{\text{scan}}^2)
\le
2\exp\!\left(\log N_{\text{scan}}-c C^2\log N_{\text{scan}}\right)
\to 0
\end{equation}
for $C$ large enough. Therefore, we can choose the rejection threshold $\tau_{\text{scan}}$ as in \eqref{eq:tauscan} to control the type-I error.

Under $H_1$, let $S_1^\star=\mathrm{supp}(\alphab_1)$ and $S_2^\star=\mathrm{supp}(\alphab_2)$.
By \eqref{eq:meanRhat}, the corresponding submatrix average satisfies
\[
\bbE_{H_1}\!\left[ Z(S_1^\star,S_2^\star)\mid \alphab_1,\alphab_2\right]
=
\frac{1}{s^2}\sum_{i\in S_1^\star}\sum_{j\in S_2^\star}\frac{\lambda}{s}
=
\frac{\lambda}{s}.
\]
Moreover, the concentration bound \eqref{eq:Z-subg} continues to hold under $H_1$ up to
absolute-constant changes in $c$ (since the model remains Gaussian with bounded covariance operator norm).
Thus, if
\begin{equation}\label{eq:scan-power-cond}
\frac{\lambda}{s}
\;\ge\;
2\tau_{\text{scan}}
\;\asymp\;
\sqrt{\frac{\log N_{\text{scan}}}{n s^2}},
\end{equation}
then with probability $1-o(1)$ we have
$Z(S_1^\star,S_2^\star)\ge \tau_{\text{scan}}$, hence $T_{\text{scan}}\ge \tau_{\text{scan}}$ and the scan test rejects.
Equivalently, the scan test is powerful whenever
\begin{equation}\label{eq:lambda-scan-rate}
\lambda
\;\gtrsim\;
\sqrt{\frac{\log N_{\text{scan}}}{n}}
\;\asymp\;
\sqrt{\frac{s\log(p_1/s)+s\log(p_2/s)}{n}}
\;\asymp\;
\sqrt{\frac{s\log p_2}{n}},
\end{equation}
where the last simplification uses $p_2\gg p_1\gtrsim s$ so that $\log(p_2/s)$ dominates.
This matches the information-theoretically optimal benchmark, but computing $T_{\text{scan}}$ is combinatorial.

\textit{Entrywise maximum statistic.} Consider
\begin{equation}\label{eq:Tinf-scca}
T_{\infty}:=\max_{i\in[p_1],\,j\in[p_2]}\widehat R_{ij}.
\end{equation}
Under $H_0$, for each $(i,j)$, we have
\[
\widehat R_{ij}=\frac{1}{2n}\sum_{\ell=1}^{2n} W_\ell^{(ij)},
\qquad
W_\ell^{(ij)}:=U_{1,i}^{(\ell)}U_{2,j}^{(\ell)}.
\]
where $U_{1,i}^{(\ell)}\sim N(0,1)$ and $U_{2,j}^{(\ell)}\sim N(0,1)$ are independent, so
$W_\ell^{(ij)}$ is a product of independent standard normals and is sub-exponential.
Consequently, by Bernstein's inequality for averages of i.i.d.\ sub-exponential variables,
there exist absolute constants $c,C>0$ such that for all $t\ge 0$,
\begin{equation}\label{eq:Rij-tail-scca}
\mathbb P_{0}\!\left(\left|\widehat R_{ij}\right|\ge t\right)
\le
2\exp\!\left(-c\,n\,\min\{t^2,t\}\right).
\end{equation}
In particular, for $0\le t\le 1$,
\begin{equation}\label{eq:Rij-subg-small-scca}
\mathbb P_{0}\!\left(\left|\widehat R_{ij}\right|\ge t\right)
\le
2\exp(-c n t^2),
\end{equation}
which matches sub-Gaussian behavior at the relevant scale $t\asymp \sqrt{(\log p_2)/n}$.

Now take a union bound over all $p_1p_2$ entries,
\begin{equation}\label{eq:Tinf-null}
T_{\infty}
=
O_{\bbP_{H_0}}\!\left(\sqrt{\frac{\log(p_1p_2)}{n}}\right)
=
O_{\bbP_{H_0}}\!\left(\sqrt{\frac{\log p_2}{n}}\right),
\end{equation}
using $p_2\gg p_1$.
Under $H_1$, on the signal support $(S_1^\star\times S_2^\star)$,
\[
\bbE_{H_1}[\widehat R_{ij}\mid \alphab_1,\alphab_2]=\lambda/s,
\qquad (i,j)\in S_1^\star\times S_2^\star.
\]
Therefore, if
\begin{equation}\label{eq:lambda-entrywise}
\frac{\lambda}{s}
\;\gg\;
\sqrt{\frac{\log(p_1p_2)}{n}},
\end{equation}
then $T_\infty$ exceeds any null-calibrated threshold with probability $1-o(1)$, and the test is powerful.
Equivalently, the entrywise maximum test requires
\begin{equation}\label{eq:lambda-rate-entrywise}
\lambda \;\gtrsim\; s\sqrt{\frac{\log p_2}{n}},
\end{equation}
which is strictly weaker (i.e., needs larger $\lambda$) than \eqref{eq:lambda-scan-rate} for $s\to\infty$.

\textit{Max-column statistic.} Define the max-column statistic
\begin{equation}\label{eq:Tmaxcol-scca}
T_{\text{max-col}}
:=
\max_{j\in[p_2]}\frac{1}{s}\sum_{i=1}^{p_1}\widehat R_{ij}.
\end{equation}

Under $H_0$, for fixed $j$, write
\[
C_j:=\frac{1}{s}\sum_{i=1}^{p_1}\widehat R_{ij}
=
\frac{1}{2n}\sum_{\ell=1}^{2n}
\left(\frac{1}{s}\sum_{i=1}^{p_1}U_{1,i}^{(\ell)}\right)U_{2,j}^{(\ell)}.
\]
Under $H_0$, $\sum_{i=1}^{p_1}U_{1,i}^{(\ell)}\sim N(0,p_1)$ and is independent of $U_{2,j}^{(\ell)}\sim N(0,1)$,
so the summand is a product of independent Gaussians with standard deviations $\sqrt{p_1}/s$ and $1$.
Hence $C_j$ is an average of i.i.d.\ sub-exponential variables with scale $\asymp \sqrt{p_1}/s$.
Bernstein's inequality yields, for all $t\ge 0$,
\begin{equation}\label{eq:Cj-tail}
\bbP_{H_0}\!\left(|C_j|\ge t\right)
\le
2\exp\!\left(-c\,n\,\min\left\{\frac{s^2 t^2}{p_1},\,\frac{s t}{\sqrt{p_1}}\right\}\right).
\end{equation}
In particular, for $t\le \sqrt{p_1}/s$,
\begin{equation}\label{eq:Cj-subg}
\bbP_{H_0}\!\left(|C_j|\ge t\right)
\le
2\exp\!\left(-c n \frac{s^2 t^2}{p_1}\right).
\end{equation}
Taking a union bound over $j\in[p_2]$ and choosing
\begin{equation}\label{eq:taumaxcol}
\tau_{\text{max-col}}
=
C\sqrt{\frac{p_1\log p_2}{n s^2}},
\end{equation}
we obtain $\bbP_{H_0}(T_{\text{max-col}}>\tau_{\text{max-col}})\to 0$ for $C$ large enough.

Under $H_1$, let $S_1^\star=\mathrm{supp}(\alphab_1)$ and $S_2^\star=\mathrm{supp}(\alphab_2)$.
For any $j\in S_2^\star$, using \eqref{eq:meanRhat},
\[
\bbE_{H_1}[C_j\mid \alphab_1,\alphab_2]
=
\frac{1}{s}\sum_{i\in S_1^\star}\frac{\lambda}{s}
=
\frac{\lambda}{s}.
\]
Thus, if $\lambda/s \ge 2\tau_{\text{max-col}}$, then with probability $1-o(1)$
we have $T_{\text{max-col}}\ge \tau_{\text{max-col}}$ and the max-column test is powerful.
Equivalently, the max-column statistic requires
\begin{equation}\label{eq:lambda-rate-maxcol}
\lambda
\;\gtrsim\;
\sqrt{\frac{p_1\log p_2}{n}}.
\end{equation}

\subsubsection*{Max-row statistic}

Define the max-row statistic
\begin{equation}\label{eq:Tmaxrow-scca}
T_{\text{max-row}}
:=
\max_{i\in[p_1]}\frac{1}{s}\sum_{j=1}^{p_2}\widehat R_{ij}.
\end{equation}
This statistic is polynomial-time computable.

Under $H_0$, for a fixed $i$, write
\[
R_i
:=
\frac{1}{s}\sum_{j=1}^{p_2}\widehat R_{ij}
=
\frac{1}{2n}\sum_{\ell=1}^{2n}
U_{1,i}^{(\ell)}\left(\frac{1}{s}\sum_{j=1}^{p_2}U_{2,j}^{(\ell)}\right).
\]
Under $H_0$, $U_{1,i}^{(\ell)}\sim N(0,1)$ and
$\sum_{j=1}^{p_2}U_{2,j}^{(\ell)}\sim N(0,p_2)$ are independent, hence the summand is a product
of independent Gaussians with standard deviations $1$ and $\sqrt{p_2}/s$.
Therefore $R_i$ is an average of i.i.d.\ sub-exponential variables with scale $\asymp \sqrt{p_2}/s$.
By Bernstein's inequality, there exist absolute constants $c,C>0$ such that for all $t\ge 0$,
\begin{equation}\label{eq:Ri-tail}
\mathbb P_0\!\left(|R_i|\ge t\right)
\le
2\exp\!\left(-c\,n\,\min\left\{\frac{s^2 t^2}{p_2},\,\frac{s t}{\sqrt{p_2}}\right\}\right).
\end{equation}
In particular, for $t\le \sqrt{p_2}/s$,
\begin{equation}\label{eq:Ri-subg}
\mathbb P_0\!\left(|R_i|\ge t\right)
\le
2\exp\!\left(-c n \frac{s^2 t^2}{p_2}\right).
\end{equation}
Taking a union bound over $i\in[p_1]$ and choosing the threshold
\begin{equation}\label{eq:taumaxrow}
\tau_{\text{max-row}}
=
C\sqrt{\frac{p_2\log p_1}{n s^2}},
\end{equation}
we obtain $\mathbb P_0(T_{\text{max-row}}>\tau_{\text{max-row}})\to 0$ for $C$ large enough.

Under $H_1$, let $S_1^\star=\mathrm{supp}(\alphab_1)$ and $S_2^\star=\mathrm{supp}(\alphab_2)$.
For any $i\in S_1^\star$, using \eqref{eq:meanRhat},
\[
\mathbb E_1[R_i\mid \alphab_1,\alphab_2]
=
\frac{1}{s}\sum_{j\in S_2^\star}\frac{\lambda}{s}
=
\frac{\lambda}{s}.
\]
Thus, if $\lambda/s \ge 2\tau_{\text{max-row}}$, then with probability $1-o(1)$
$T_{\text{max-row}}\ge \tau_{\text{max-row}}$ and the max-row test is powerful.
Equivalently, the max-row statistic requires
\begin{equation}\label{eq:lambda-rate-maxrow}
\lambda
\;\gtrsim\;
\sqrt{\frac{p_2\log p_1}{n}}.
\end{equation}

\textit{Global sum statistic.} Finally, consider the global sum
\begin{equation}\label{eq:Tsum-scca}
T_{\text{sum}}:=
\frac{1}{p_1p_2}\sum_{i=1}^{p_1}\sum_{j=1}^{p_2}\widehat R_{ij}.
\end{equation}
Under $H_0$,
\[
T_{\text{sum}}
=
\frac{1}{2n}\sum_{\ell=1}^{2n}
\left(\frac{1}{p_1}\sum_{i=1}^{p_1}U_{1,i}^{(\ell)}\right)
\left(\frac{1}{p_2}\sum_{j=1}^{p_2}U_{2,j}^{(\ell)}\right),
\]
where the two parentheses are independent Gaussians with variances $1/p_1$ and $1/p_2$.
Hence $T_{\text{sum}}$ concentrates at scale $\sqrt{1/(n p_1 p_2)}$ (sub-exponential tails analogously).

Under $H_1$, the mean shift equals
\[
\bbE_{H_1}[T_{\text{sum}}\mid \alphab_1,\alphab_2]
=
\frac{1}{p_1p_2}\sum_{i\in S_1^\star}\sum_{j\in S_2^\star}\frac{\lambda}{s}
=
\frac{\lambda s}{p_1p_2},
\]
which is heavily diluted for sparse alternatives. Balancing the signal mean against the null
standard deviation suggests that power requires
\begin{equation}\label{eq:lambda-rate-sum}
\frac{\lambda s}{p_1p_2}
\;\gtrsim\;
\sqrt{\frac{1}{n p_1 p_2}}
\qquad\Longleftrightarrow\qquad
\lambda
\;\gtrsim\;
\sqrt{\frac{p_1p_2}{n}}\cdot \frac{1}{s},
\end{equation}
which is far worse than \eqref{eq:lambda-scan-rate} in the sparse regime $s\ll \sqrt{p_1p_2}$.

\textit{Summary of thresholds.}
In the regime $s\lesssim p_1\lesssim n\ll p_2$, the above calculations yield the detection boundary for all test statistics:
\begin{align*}
\lambda_{\text{scan}}
\asymp
\sqrt{\frac{s\log p_2}{n}},
\quad
\lambda_{\infty}
\asymp
&s\sqrt{\frac{\log p_2}{n}},
\quad
\lambda_{\text{max-col}}
\asymp
\sqrt{\frac{p_1\log p_2}{n}}\\
\lambda_{\text{max-row}}\asymp\sqrt{\frac{p_2\log p_1}{n}}
,&\quad \lambda_{\text{sum}}\asymp \sqrt{\frac{p_1 p_2}{ns^2}}.
\end{align*}

\section{Loading-profile examples and consequences}
\label{sec:loading-examples}

This appendix works out several concrete loading profiles to illustrate the scope of
the profile-based theory developed in the main text. The calculations below serve two
purposes. First, they translate the abstract quantities
\(H(\cdot~;~\xi)\), \(\nu_1\), and \(\nu_2\) into explicit separation
rates. Second, they identify loading profiles that are not covered by existing
regular-loading or exact polynomial-decay theories.

The regular-loading example in
\Cref{app:regular-loading} recovers the phase diagram in \Cref{fig: tony loading}
and highlights the intermediate moderately sparse range \(k_u\ll K\ll k_u^2\).
\Cref{app:dense-nonregular-loading} gives dense nonregular profiles for which the
adaptive separation distance can still be determined. \Cref{app:multiscale-computational-gap}
constructs a multiscale profile for which the low-degree rate exceeds the statistical rate by a polynomial
factor, which illustrates a statistical--computational gap in sparse signed-spiked models.
\Cref{app:random-subweibull-loading} is related to the case where the loading vector is a random test point; it treats random loadings by conditioning on the realized vector and evaluating the resulting profile quantities.

Throughout this appendix, we assume \Cref{cdt: linear estimator,cdt: linear variance,cdt: sparsity assumption} hold.

\subsection{Regular loading vectors}
\label{app:regular-loading}

We derive the rates displayed in \Cref{fig: tony loading}
for loading vectors satisfying the regular loading condition~\eqref{eq: tony loading vector}.
Throughout this subsection, let
\[
    K=\|\xi\|_0,
    \qquad
    a=\|\xi\|_\infty .
\]
Since the coordinates of \(\xi\) are ordered in decreasing absolute value, the regular loading
condition implies that, for some constant \(\bar c>0\),
\[
    a/\bar c\le |\xi_j|\le a,\qquad 1\le j\le K,
    \qquad
    \xi_j=0,\qquad j>K.
\]
All constants below may depend on \(\bar c\), but not on \(n,p,k_u,K\), or \(a\).

We first state several elementary consequences of the regular loading condition.

\begin{lemma}
\label{lem:regular-loading-profile}
Suppose that \(\xi\) satisfies the regular loading condition~\eqref{eq: tony loading vector}. Then,
for all \(t\ge 1\),
\[
    H(t~;~\xi)
    =
    \left(\sum_{j\le \lceil t\rceil}\xi_j^2\right)^{1/2}
    \asymp
    a\sqrt{t\wedge K},
\]
and thus
\(
    \nu_2
    =
    H(k_u~;~\xi)
    \asymp
    a\sqrt{k_u\wedge K}\).
Moreover, the quantity \(\nu_1\) satisfies
\[
    \nu_1
    \asymp
    \begin{cases}
        a\sqrt K, & K\lesssim k_u^2,\\[4pt]
        a k_u\left\{1+\sqrt{\log\left(eK/k_u^2\right)}\right\},
        & K\gtrsim k_u^2 .
    \end{cases}
\]
\end{lemma}

\begin{proof}
The first two displays follow immediately from the regular loading condition. We prove the assertion
for \(\nu_1\).

Let \(x_j=|\xi_j|\), and define
\[
    F(z)
    =
    \frac{\sum_{j=1}^{K} x_j \exp(-z/x_j^2)}
    {\left(\sum_{j=1}^{K} x_j^2 \exp(-z/x_j^2)\right)^{1/2}},
    \qquad z\in\bbR .
\]
By the definition of \(\lambda\) in~\eqref{eq: equation sol}, \(\lambda=\sqrt{\zeta_+}\), where
\(\zeta\) is the unique solution to \(F(\zeta)=k_u/2\).
By \ref{eq: tony loading vector}, we have
\(F(0)\asymp \sqrt K\).

For \(z\ge0\), by Cauchy--Schwarz inequality, we have
\begin{equation}\label{eq: regular example CS}
    F(z)
    \le
    \left(\sum_{j=1}^{K}\exp(-z/x_j^2)\right)^{1/2}
    \le
    \sqrt K \exp(-z/(2a^2)).
\end{equation}

\textbf{Setting 1: \(K\lesssim k_u^2\)}.
If $F(0)\leq k_u/2$,  then \(\zeta\le 0\) and \(\lambda=0\).

If $F(0)>k_u/2$, we can prove that
    \(\zeta\) is bounded above by a constant multiple of \(a^2\).
    Indeed, by \Cref{eq: regular example CS} and $x_j\leq a$, we have $F(ta^2)\leq \sqrt{K}\exp(-t/2)$ for any $t>0$. Since $k_u/(2\sqrt{K})\asymp \frac{k_u}{2F(0)}\lesssim 1$ we can choose a constant $t>0$ such that $\exp(-t/2)\leq k_u/(2\sqrt{K})$.

    In both cases, we have \(\lambda\lesssim a\),  and thus $\exp(-\lambda^2/\xi_j^2)$ is bounded by some constant for all $j\leq K$.
    Therefore,
\[
    \left(\sum_{j=1}^K \xi_j^2 \exp(-\lambda^2/\xi_j^2)\right)^{1/2}
    \asymp a\sqrt K .
\]
Moreover, we have \(\lambda k_u=0\) in the first case and
\(\lambda k_u\lesssim a k_u\lesssim a\sqrt K\) in the second case. Therefore
\(\nu_1\asymp a\sqrt K\).

\textbf{Setting 2: \(K\gtrsim k_u^2\)}.
If \(K\asymp k_u^2\), the preceding argument gives
\(\lambda\lesssim a\), and the second term in \(\nu_1\) is of order \(ak_u\); hence
\(\nu_1\asymp ak_u\), which agrees with the displayed bound because
\(\log(eK/k_u^2)\asymp 1\). It remains to consider the case in which \(K/k_u^2\) is larger than
a sufficiently large constant. In this case the solution is positive and we show that
\[
    \lambda^2 \asymp a^2\log\left(eK/k_u^2\right).
\]
Since \(x_j\ge a/\bar c\) for \(1\le j\le K\),
\begin{equation}\label{eq: regular example F lower}
        F(z)
    \ge
    \frac{(a/\bar c)K\exp(-\bar c^2 z/a^2)}
    {a\sqrt K}
    =
    \bar c^{-1}\sqrt K \exp(-\bar c^2 z/a^2).
\end{equation}

Combining this with \Cref{eq: regular example CS}, we see that the solution of \(F(z)=k_u/2\) satisfies
\[
    \zeta \asymp a^2\log\left(eK/k_u^2\right),
\]
and therefore
\[
    \lambda \asymp a\sqrt{\log\left(eK/k_u^2\right)} .
\]

It remains to evaluate the second term in \(\nu_1\). Let
\[
    W_\lambda
    =
    \sum_{j=1}^K \exp(-\lambda^2/x_j^2).
\]
Since $a/\bar{c}\leq x_j\leq a$, we can follow the proofs for \Cref{eq: regular example CS,eq: regular example F lower} to see that
\[
    F(\lambda^2)
    \asymp
    W_\lambda^{1/2}.
\]
Since \(F(\lambda^2)=k_u/2\), we have \(W_\lambda\asymp k_u^2\). Consequently,
\[
    \left(\sum_{j=1}^K \xi_j^2\exp(-\lambda^2/\xi_j^2)\right)^{1/2}
    \asymp
    a W_\lambda^{1/2}
    \asymp
    a k_u .
\]
Combining this with the bound for \(\lambda\) gives
\[
    \nu_1
    =
    \lambda k_u+
    \left(\sum_{j=1}^K \xi_j^2\exp(-\lambda^2/\xi_j^2)\right)^{1/2}
    \asymp
    a k_u\left\{1+\sqrt{\log\left(eK/k_u^2\right)}\right\}.
\]
This proves the lemma.
\end{proof}

We now derive the rates in the ultra-sparse and moderately sparse regimes.

\begin{proposition}
\label{prop:regular-loading-rates}
Suppose that \(\xi\) satisfies the regular loading condition~\eqref{eq: tony loading vector}.
Let \(K=\|\xi\|_0\) and \(a=\|\xi\|_\infty\).

\begin{enumerate}
\item
Suppose \(k_u\lesssim \sqrt n/\log p\). If $K\lesssim k_u^2$, then
\[
    \tau_{\mathrm{adap}}(k_u,k~;~\xi)
    \asymp
    \frac{a\sqrt K}{\sqrt n}.
\]
If \(K\gtrsim k_u^2\), then
\[
    \tau_{\mathrm{adap}}(k_u,k~;~\xi)
    \asymp_{\log}
    a k_u\sqrt{\frac{\log p}{n}}.
\]
Moreover, if \(K/k_u^2\ge p^c\) for some constant \(c>0\), then the logarithmic equivalence
can be strengthened to
\[
    \tau_{\mathrm{adap}}(k_u,k~;~\xi)
    \asymp
    a k_u\sqrt{\frac{\log p}{n}} .
\]

\item Suppose \(k_u\gg \sqrt n/\log p\). If $K\lesssim k_u$, then
\[
    \tau_{\mathrm{adap}}(k_u,k~;~\xi)
    \asymp
    a\sqrt K\,\frac{k_u\log p}{n}.
\]
If \(K\gtrsim k_u^2\), then
\[
    \tau_{\mathrm{adap}}(k_u,k~;~\xi)
    \asymp_{\log}
    a k_u\sqrt{\frac{\log p}{n}}.
\]
Moreover, if \(K/k_u^2\ge p^c\) for some constant \(c>0\), then
\[
    \tau_{\mathrm{adap}}(k_u,k~;~\xi)
    \asymp
    a k_u\sqrt{\frac{\log p}{n}} .
\]
In the intermediate regime \(k_u\ll K\ll k_u^2\), the general bounds in
\Cref{thm: hypothesis,prop: upper bound equivalent} give
\[
    a\left\{
        \frac{\sqrt K}{\sqrt n}
        +
        \frac{k_u^{3/2}\log p}{n}
    \right\}
    \lesssim
    \tau_{\mathrm{adap}}(k_u,k~;~\xi)
    \lesssim
    a\sqrt{K\wedge \frac n{\log p}}\,
    \frac{k_u\log p}{n}.
\]
If, in addition, \(K\le n/\log p\), and both \(k_u\log p/\sqrt n\) and \(K/k_u\)
diverge by polynomial factors, then these available upper and lower bounds do not match up to
logarithmic factors.
\end{enumerate}
All the statements above hold uniformly over \(k\le k_u\).
\end{proposition}

\begin{proof}
We repeatedly use \Cref{thm: hypothesis,prop: upper bound equivalent}, together with
\Cref{lem:regular-loading-profile}.

\textbf{1. Ultra-sparse regime \(k_u\lesssim \sqrt n/\log p\).}

By
\Cref{prop: upper bound equivalent},
\[
    \tau_{\mathrm{adap}}(k_u,k~;~\xi)\lesssim \frac{1}{\sqrt n}H(k_u^2\log p~;~\xi).
\]
If \(K\lesssim k_u^2\), then \(H(k_u^2\log p~;~\xi)\asymp a\sqrt K\), and the lower bound
\(\nu_1/\sqrt n\) satisfies \(\nu_1/\sqrt n\asymp a\sqrt K/\sqrt n\) by
\Cref{lem:regular-loading-profile}. Hence
\[
    \tau_{\mathrm{adap}}(k_u,k~;~\xi)
    \asymp
    \frac{a\sqrt K}{\sqrt n}.
\]

If \(K\gtrsim k_u^2\), \Cref{lem:regular-loading-profile} gives
\[
    \nu_1
    \gtrsim
    a k_u .
\]
Therefore, \Cref{cor: nu1-sharp-ultrasparse} gives
\[
    \tau_{\mathrm{adap}}(k_u,k~;~\xi)
    \asymp_{\log}
    a k_u\sqrt{\frac{\log p}{n}}.
\]
If  \(K/k_u^2\ge p^c\) for $c>0$, then \(H(k_u^2\log p~;~\xi)\asymp a\sqrt{k_u^2\log p}\). Furthermore, \Cref{lem:regular-loading-profile} gives
\[
    \nu_1
    \asymp
    a k_u\sqrt{\log p},
\]
which implies that the lower bound $\nu_1/\sqrt{n}$ matches the upper bound $a k_u\sqrt{\frac{\log p}{n}}$ up to constants.

\textbf{2. Moderately sparse regime \(k_u\gg \sqrt n/\log p\).}

In this case,
\Cref{prop: upper bound equivalent} yields the upper bound
\begin{equation}\label{eq:regular-moderate-sparse-upper}
    \tau_{\mathrm{adap}}(k_u,k~;~\xi)
    \lesssim
    \frac{k_u\log p}{n}H(n/\log p~;~\xi)
    \asymp
    a\sqrt{K\wedge \frac n{\log p}}\,
    \frac{k_u\log p}{n}.
\end{equation}
The lower bound in \Cref{thm: hypothesis} is
\[
    \tau_{\mathrm{adap}}(k_u,k~;~\xi)
    \gtrsim
    \frac{\nu_1}{\sqrt n}
    \vee
    \nu_2\frac{k_u\log p}{n}.
\]

If \(K\lesssim k_u\), then
\[
    H(n/\log p~;~\xi)\asymp H(k_u~;~\xi)\asymp a\sqrt K.
\]
Recall that $\nu_2=H(k_u;\xi)$. We see that both the upper bound and the lower bound \(\nu_2 k_u\log p/n\) are at the scale of
\[
    a\sqrt K\,\frac{k_u\log p}{n}.
\]
Thus
\[
    \tau_{\mathrm{adap}}(k_u,k~;~\xi)
    \asymp
    a\sqrt K\,\frac{k_u\log p}{n}.
\]

\smallskip

If \(K\gtrsim k_u^2\), then the upper bound in \Cref{eq:regular-moderate-sparse-upper} satisfies
\[
    \tau_{\mathrm{adap}}(k_u,k~;~\xi)
    \lesssim
    a\sqrt{\frac n{\log p}}\frac{k_u\log p}{n}
    =
    a k_u\sqrt{\frac{\log p}{n}}.
\]
By \Cref{lem:regular-loading-profile}, the lower bound involving \(\nu_1\) gives at least
\[
    \frac{\nu_1}{\sqrt n}
    \gtrsim
    \frac{a k_u}{\sqrt n}.
\]
Hence the upper and lower bounds match up to logarithmic factors.

\smallskip

Furthermore, if \(K/k_u^2\ge p^c\) for some $c>0$, then
\Cref{lem:regular-loading-profile} gives
\[
    \frac{\nu_1}{\sqrt n}
    \asymp
    a k_u\sqrt{\frac{\log p}{n}},
\]
which matches the upper bound up to constants and thus \[
    \tau_{\mathrm{adap}}(k_u,k~;~\xi)
    \asymp
    a k_u\sqrt{\frac{\log p}{n}}.
\]

\smallskip

Finally, suppose \(k_u\ll K\ll k_u^2\). Then $\nu_2=H(k_u;\xi)\asymp a\sqrt{k_u}$, while \Cref{lem:regular-loading-profile} gives
\[
    \nu_1\asymp a\sqrt K.
\]
The lower bound in \Cref{thm: hypothesis} therefore gives
\[
    \tau_{\mathrm{adap}}(k_u,k~;~\xi)
    \gtrsim
    \frac{a\sqrt K}{\sqrt n}
    \vee
    a\sqrt{k_u}\frac{k_u\log p}{n}
    \asymp
    a\left\{
        \frac{\sqrt K}{\sqrt n}
        +
        \frac{k_u^{3/2}\log p}{n}
    \right\},
\]
where the last equivalence uses that the maximum and the sum are the same up to a universal
constant.
If \(K\le n/\log p\),
the upper bound in \Cref{eq:regular-moderate-sparse-upper} reduces to
\[
    \tau_{\mathrm{adap}}(k_u,k~;~\xi)
    \lesssim
    a\sqrt{K\wedge \frac n{\log p}}\,
    \frac{k_u\log p}{n} \asymp a\sqrt K\,\frac{k_u\log p}{n}.
\]

The ratios of the upper bound to the two lower-bound terms are
\[
    \frac{a\sqrt K\, k_u\log p/n}{a\sqrt K/\sqrt n}
    =
    \frac{k_u\log p}{\sqrt n},
    \qquad
    \frac{a\sqrt K\, k_u\log p/n}{a k_u^{3/2}\log p/n}
    =
    \sqrt{\frac{K}{k_u}}.
\]
Thus the additional polynomial-divergence assumptions in the proposition make the upper bound
polynomially larger than each available lower-bound term.

This proves the intermediate-regime display and completes the proof.
\end{proof}

\subsection{Dense nonregular loading profiles}
\label{app:dense-nonregular-loading}

This subsection contains two dense profiles that violate the regular loading condition.
Since \Cref{cor: nu1-sharp-ultrasparse} has already determined the adaptive separation distance $\nu_1/\sqrt{n}$ up to logarithmic factors, we will focus on the moderately sparse regime in this subsection.
Let
\[
    N=\left\lfloor \frac{n}{\log p}\right\rfloor.
\]

The common purpose of the two examples is to demonstrate that the regular loading condition
\eqref{eq: tony loading vector} is sufficient but not necessary for determining the adaptive
separation rate through \(\nu_1/\sqrt n\) in the moderately sparse regime.

In both examples, the leading coordinates at the
scale \(N\) carry dense, nearly flat energy even though the full support of \(\xi\) is
nonregular.
The examples differ in how this nonregularity appears: the first has a nearly
flat leading block and allows a small nonzero tail, while the second lets the coordinate
magnitudes vary logarithmically across the full support.

We again assume $\{|\xi_j|\}$ are sorted in the decreasing order.

\medskip
\noindent\textbf{Example 1: nearly flat leading block with a small tail.}

For the first example, suppose that there exist constants \(c_0,C_0,C_1,c_1>0\), a scale \(a>0\), and an integer \(K\)
such that $N\le K\ll p$, $K\ge  k_u^2 p^{c_1}$, $k_u\leq p^{c_2}$ with $c_2<c_1$, and the loading vector satisfies
\begin{equation}\label{eq:nearly-flat-leading-block}
\begin{cases}
    c_0a\le |\xi_j|\le C_0a,
    \quad
   \text{ for } 1\le j\le K,  \\
   \xi_j\neq 0, \quad  \text{ for } j > K, \\
    \sum_{j>K}\xi_j^2\le C_1Ka^2.
\end{cases}
\end{equation}
In other words, the leading $K$ elements are of the same order while  the remaining coordinates do not dominate the total energy.

We note that the loading vector fails the regular loading condition \Cref{eq: tony loading vector}.
Since $0<\xi_p^2\leq (p-K)^{-1}\sum_{j>K}\xi_j^2\lesssim a^2 K/(p-K)$, we have
\(|\xi_p|/a\to0\) while \(|\xi_p|>0\). Then
\[
    \frac{\max_{j\in\operatorname{supp}(\xi)}|\xi_j|}
    {\min_{j\in\operatorname{supp}(\xi)}|\xi_j|}
    \to\infty .
\]

Furthermore, \Cref{eq:nearly-flat-leading-block} excludes the exact polynomially decaying profile considered in
\cite{cai2019optimal}, since they assume $\xi_j\asymp j^{-\alpha}$ with fixed
\(\alpha>0\). This is because the ratio between its first and $K$th coordinates is $K^\alpha$ and does not meet \Cref{eq:nearly-flat-leading-block}.

Thus, in this nonregular dense example, the regular-loading optimality theory of \cite{cai2017confidence} and the
polynomial-decay analysis of \cite{cai2019optimal} do not apply.

\begin{proposition}
\label{prop:dense-nonregular-loading}
Suppose \(k_u\gg \sqrt n/\log p\) and the loading vector satisfies \Cref{eq:nearly-flat-leading-block}.
Then, uniformly over \(k\le k_u\),
\[
    \tau_{\mathrm{adap}}(k_u,k~;~\xi)
    \asymp
    a k_u\sqrt{\frac{\log p}{n}}.
\]
\end{proposition}

\begin{proof}
Since \(N\le K\) and the first \(K\) coordinates are all of order \(a\),
\[
    H(N~;~\xi)
    =
    \left(\sum_{j\le N}\xi_j^2\right)^{1/2}
    \asymp
    a\sqrt N.
\]
Because \(k_u\gg \sqrt n/\log p\), \Cref{prop: upper bound equivalent} gives
\[
    \tau_{\mathrm{adap}}(k_u,k~;~\xi)
    \lesssim
    \frac{k_u\log p}{n}H(N~;~\xi)
    \asymp
    \frac{k_u\log p}{n}a\sqrt{\frac n{\log p}}
    =
    a k_u\sqrt{\frac{\log p}{n}}.
\]

It remains to show that the \(\nu_1\)-term gives the matching lower bound. Let \(x_j=|\xi_j|\),
and define
\[
    F(z)
    =
    \frac{\sum_{j=1}^{k_\xi}x_j\exp(-z/x_j^2)}
    {\left(\sum_{j=1}^{k_\xi}x_j^2\exp(-z/x_j^2)\right)^{1/2}},
    \qquad z\in\bbR .
\]
By \eqref{eq: equation sol}, \(\lambda=\sqrt{\zeta_+}\), where \(F(\zeta)=k_u/2\).

We claim that \(\lambda\gtrsim a\sqrt{\log p}\). Let \(z=c_\star a^2\log p\), where
\(c_\star>0\) is a sufficiently small constant.
For $j\leq K$, we have $x_j\exp(-z/x_j^2)\geq c_0 a \exp(-c_0^2 z/a^2)$.
Combining this with $\sum_{j>K}\xi_j^2\le C_1Ka^2$, we have
\[
    F(c_\star a^2\log p)
    \gtrsim
    \frac{Ka\exp(-c_0^2 z/a^2)}
    {a\sqrt K}
    =
    \sqrt K \exp(-c_0^2 c_\star\log p).
\]
Since \(K/k_u^2\ge p^{c_1}\), we can choose \(c_\star>0\) sufficiently small such that $c_1-c_0^2c_\star>c_2$.
Then for sufficiently large $p$, we have
\[
    F(c_\star a^2\log p) \geq \exp( c_2 \log p) \ge k_u.
\]
As \(F\) is decreasing and \(F(\zeta)=k_u/2\), this implies
\[
    \lambda^2=\zeta_+\gtrsim a^2\log p.
\]
Consequently,
\[
    \nu_1
    =
    \lambda k_u+
    \left(\sum_{j=1}^{k_\xi}\xi_j^2\exp(-\lambda^2/\xi_j^2)\right)^{1/2}
    \ge
    \lambda k_u
    \gtrsim
    a k_u\sqrt{\log p}.
\]
The lower bound in \Cref{thm: hypothesis} therefore yields
\[
    \tau_{\mathrm{adap}}(k_u,k~;~\xi)
    \gtrsim
    \frac{\nu_1}{\sqrt n}
    \gtrsim
    a k_u\sqrt{\frac{\log p}{n}}.
\]
This matches the upper bound and proves the proposition.
\end{proof}

This example reveals that, even if a small nonzero tail makes the loading vector nonregular, the adaptive separation
distance is the same as for a dense regular loading vector because the leading \(N\) coordinates
already contain dense, nearly flat energy.

\medskip
\noindent\textbf{Example 2: logarithmically varying dense loading vectors.}

The second example places the nonregularity across the full support rather than in a small tail. It
gives a dense loading profile that violates the regular loading condition only through a logarithmic
variation across coordinates.

Let \(b>0\) be fixed and suppose
\[
    |\xi_j|=a\{\log(e+j)\}^{-b},
    \qquad
    1\le j\le p,
\]
where \(a>0\). Then
\[
    \frac{|\xi_1|}{|\xi_p|}
    \asymp
    (\log p)^b\to\infty,
\]
so the regular loading condition fails.
This logarithmically varying profile is also outside the exact polynomial-decay class in
\cite{cai2019optimal}. For instance, we can check the ratio of the $j$th and $2j$th coordinates for $j=\lfloor \sqrt{p}\rfloor$:
\[
    \frac{|\xi_j|}{|\xi_{2j}|}
    =
    \left\{\frac{\log(e+2j)}{\log(e+j)}\right\}^{b}
    \to 1, \text{ as } p\to \infty,
\]
whereas a polynomial-decay profile \(j^{-\alpha}\) with \(\alpha>0\) has a constant dyadic ratio
\(2^\alpha>1\).
Hence neither the regular-loading condition used in \cite{cai2017confidence} nor
the exact polynomial-decay condition treated in \cite{cai2019optimal} covers this example.

\begin{proposition}
\label{prop:slowly-varying-dense-loading}
Assume that
$N\to\infty$, $\log N\asymp \log p$,
and that \(k_u\le Cp^\gamma\) for some \(\gamma<1/2\).
Suppose \(k_u\gg \sqrt n/\log p\). Then, uniformly over
\(k\le k_u\),
\[
    \tau_{\mathrm{adap}}(k_u,k~;~\xi)
    \asymp
    a k_u\sqrt{\frac{\log p}{n}}\frac{1}{(\log p)^b}.
\]
\end{proposition}

\begin{proof}
We first compute the top-\(r\) norm for $r$ sufficiently large.

A simple comparison gives
\[
    \sum_{j\le r}\{\log(e+j)\}^{-2b}
    \asymp
    \frac{r}{(\log r)^{2b}}.
\]
Indeed, the lower bound follows by summing over \(r/2<j\le r\). For the upper bound, split the
sum at \(\lfloor\sqrt r\rfloor\):
\[
    \sum_{j\le r}\{\log(e+j)\}^{-2b}
    \le
    \sqrt r
    +
    \sum_{\sqrt r<j\le r}
    \left(\frac12\log r\right)^{-2b}
    \lesssim
    \frac{r}{(\log r)^{2b}},
\]
where the last step uses \((\log r)^{2b}=o(\sqrt r)\).
Thus
\[
    H(r~;~\xi)\asymp a\frac{\sqrt r}{(\log r)^b}.
\]
In the moderately sparse regime, \Cref{prop: upper bound equivalent} gives
\[
    \tau_{\mathrm{adap}}(k_u,k~;~\xi)
    \lesssim
    \frac{k_u\log p}{n}H(N~;~\xi).
\]
Using the condition that \(\log N\asymp\log p\) and the definition of \(N=n/\log p\) (up to integer rounding),
we have
\[
    \frac{k_u\log p}{n}H(N~;~\xi)
    \asymp
    \frac{k_u\log p}{n}
    \cdot
    a\frac{\sqrt{n/\log p}}{(\log N)^b}
    \asymp
    a k_u\sqrt{\frac{\log p}{n}}\frac{1}{(\log p)^b}.
\]

We now prove the matching lower bound through \(\nu_1\). Let \(x_j=|\xi_j|\).
We have
\begin{equation}\label{eq:log-varying-example-numerator}
    x_j\asymp a_p:=a(\log p)^{-b}, ~~ p/2\le j\le p,
\end{equation}
and
\begin{equation}\label{eq:log-varying-example-denominator}
    \sum_{j=1}^p \xi_j^2 = H^2(p;\xi)
    \asymp
    a_p^2 \cdot p .
\end{equation}
As in the proof of \Cref{prop:dense-nonregular-loading}, define
\[
    F(z)
    =
    \frac{\sum_{j=1}^{p}x_j\exp(-z/x_j^2)}
    {\left(\sum_{j=1}^{p}x_j^2\exp(-z/x_j^2)\right)^{1/2}}.
\]
Let \(z=c_\star a_p^2\log p\), where \(c_\star>0\) is sufficiently small.
Using \Cref{eq:log-varying-example-numerator} for the block
\(p/2\le j\le p\) in the numerator and \Cref{eq:log-varying-example-denominator} for the total energy bound  in the denominator, we have
\[
    F(z)
    \gtrsim
    \sqrt p\,\exp(-Cc_\star\log p),
\]
where $C>0$ is a constant.
Since \(k_u\le Cp^\gamma\) with \(\gamma<1/2\), choosing \(c_\star\) sufficiently small gives
\[
    F(z)\gtrsim k_u.
\]
Therefore \(\lambda^2\gtrsim a_p^2\log p\), and hence
\[
    \lambda\gtrsim a(\log p)^{-b}\sqrt{\log p}.
\]
It follows that
\[
    \nu_1\ge \lambda k_u
    \gtrsim
    a k_u\frac{\sqrt{\log p}}{(\log p)^b}.
\]
\Cref{thm: hypothesis} then gives
\[
    \tau_{\mathrm{adap}}(k_u,k~;~\xi)
    \gtrsim
    \frac{\nu_1}{\sqrt n}
    \gtrsim
    a k_u\sqrt{\frac{\log p}{n}}\frac{1}{(\log p)^b}.
\]
This matches the upper bound.
\end{proof}

There is an interesting equivalence between this example and a dense flat loading profile.
Although the coordinate ratio diverges as \((\log p)^b\) in the current example, the loading vector behaves the same as a dense loading vector if we replace the common scale \(a\) by the effective scale
\(a(\log p)^{-b}\) on the critical dense block.

\subsection{A multiscale example for the statistical--computational gap}
\label{app:multiscale-computational-gap}

This example shows that the statistical--computational gap in the moderately sparse regime
can occur for loading vectors outside the regular-loading phase diagram in \Cref{sec: examples}.
The construction in \eqref{eq:multiscale-loading} puts equal \(\ell_2\)-energy on \(L\) blocks
whose sizes increase with the block index.  The example is interpreted together with
\Cref{thm: computational lower bound,thm: sparse spiked}: \Cref{thm: computational lower bound}
gives a low-degree lower bound for the general unknown-covariance problem, whereas
\Cref{thm: sparse spiked} gives a smaller adaptive separation distance under the
sparse signed-spiked covariance restriction.

Fix a degree \(D\) satisfying the assumptions of \Cref{thm: computational lower bound}, and define
\[
    k_{\mathrm{eff}}
    =
    \left\lfloor
        \frac n{\log p}\wedge \frac{k_u^2}{D\log p}
    \right\rfloor .
\]
Assume that
\[
    k_u\gg \frac{\sqrt n}{\log p}.
\]
Let \(L=L_n\to\infty\) be an integer such that
\begin{equation}\label{eq:multiscale-L-conditions}
    L^3\le c_0 k_u,
    \qquad
    k_uL^3\le c_0 k_{\mathrm{eff}},
\end{equation}
where \(c_0>0\) is a sufficiently small absolute constant. Define block sizes
\begin{equation}\label{eq:multiscale-block-sizes}
    m_\ell=\lceil k_u\ell^2\rceil,
    \qquad
    \ell=1,\ldots,L,
\end{equation}
and cumulative indices
\[
    M_0=0,
    \qquad
    M_\ell=\sum_{r=1}^{\ell}m_r,
    \qquad
    \ell=1,\ldots,L.
\]
For a scale \(a>0\), define the loading vector by
\begin{equation}\label{eq:multiscale-loading}
    |\xi_j|
    =
    \frac{a}{\sqrt{m_\ell}},
    \qquad
    M_{\ell-1}<j\le M_\ell,
    \quad
    \ell=1,\ldots,L,
    \qquad
    \xi_j=0 \quad \text{for } j>M_L .
\end{equation}
Since \(m_\ell\) is increasing in \(\ell\), the coordinates in \eqref{eq:multiscale-loading}
are ordered in decreasing absolute value. Moreover, \eqref{eq:multiscale-block-sizes} and
\eqref{eq:multiscale-L-conditions} give
\begin{equation}\label{eq:multiscale-support-size}
    M_L
    =
    \sum_{\ell=1}^L m_\ell
    \asymp
    k_u\sum_{\ell=1}^L\ell^2
    \asymp
    k_uL^3
    \le
    k_{\mathrm{eff}}
    \le
    \frac n{\log p}.
\end{equation}
The loading vector in \eqref{eq:multiscale-loading} is not regular when \(L\to\infty\), because
the ratio between the largest and smallest nonzero coordinates is
\[
    \frac{a/\sqrt{m_1}}{a/\sqrt{m_L}}
    =
    \sqrt{\frac{m_L}{m_1}}
    \asymp L.
\]
It is also not an exact polynomially decaying loading vector in the sense of
\cite{cai2019optimal}:
the profile in \eqref{eq:multiscale-loading} has flat
blocks whose lengths diverge. In particular, the first block contains
\(m_1=\lceil k_u\rceil\) coordinates with magnitude \(a/\sqrt{m_1}\). In contrast,
an exact polynomial-decay profile \(a_0j^{-\alpha}\), with \(\alpha>0\), would
satisfy \(|\xi_1|/|\xi_{m_1}|=m_1^\alpha\to\infty\). Thus the construction in
\eqref{eq:multiscale-loading} is not an exact deterministic polynomial-decay
profile. Consequently, the example is not covered directly by either the
regular-loading theory of \cite{cai2017confidence} or the exact polynomial-decay
analysis of \cite{cai2019optimal}.

\begin{proposition}
For the multiscale loading vector in \eqref{eq:multiscale-loading}, the following relations hold:
\[
    H(k_{\mathrm{eff}}~;~\xi)
    \asymp
    H(n/\log p~;~\xi)
    \asymp
    a\sqrt L,
    \qquad
    \nu_2\asymp a,
    \qquad
    \nu_1\asymp a\sqrt L.
\]
Consequently, the upper bound in \Cref{prop: upper bound equivalent} is of order
\begin{equation}\label{eq:multiscale-computational-scale}
    a\sqrt L\,\frac{k_u\log p}{n},
\end{equation}
and \Cref{thm: computational lower bound} gives a low-degree lower bound of the order in
\eqref{eq:multiscale-computational-scale}.
By contrast, under the sparse signed-spiked covariance model of \Cref{thm: sparse spiked}, we have
\[
    \tau^{\text{spike}}_{\text{adap}}(k_u,k~;~\xi)
    \asymp_{\log}
    \frac{a\sqrt L}{\sqrt n}
    +
    a\frac{k_u\log p}{n}.
\]

If $\sqrt L\wedge \frac{k_u\log p}{\sqrt n}$
diverges polynomially in $p$, the low-degree scale exceeds the sparse signed-spiked rate by a
polynomial factor.
\end{proposition}

\begin{proof}
For each block \(\ell\), \eqref{eq:multiscale-loading} gives constant \(\ell_2\)-energy:
\begin{equation}\label{eq:multiscale-block-energy}
    \sum_{j=M_{\ell-1}+1}^{M_\ell}\xi_j^2
    =
    m_\ell\frac{a^2}{m_\ell}
    =
    a^2.
\end{equation}
By \eqref{eq:multiscale-support-size},
the support of \(\xi\) is a subset of the first
\(k_{\mathrm{eff}}\) coordinates and also a subset of the first \(n/\log p\) coordinates.
Therefore, the definition
of \(H\) and the block-energy identity \eqref{eq:multiscale-block-energy}
give
\[
    H(k_{\mathrm{eff}}~;~\xi)^2
    =
    H(n/\log p~;~\xi)^2
    =
    \sum_{\ell=1}^L a^2
    =
    La^2,
\]
which implies
\begin{equation}\label{eq:multiscale-H-result}
    H(k_{\mathrm{eff}}~;~\xi)
    \asymp
    H(n/\log p~;~\xi)
    \asymp
    a\sqrt L.
\end{equation}

The block with \(\ell=1\) has size \(m_1=\lceil k_u\rceil\) and coordinate magnitude
\(a/\sqrt{m_1}\). Therefore, we have
\begin{equation}\label{eq:multiscale-nu2}
    \nu_2
    =
    H(k_u~;~\xi)
    =
    \left(m_1\frac{a^2}{m_1}\right)^{1/2}
    =
    a.
\end{equation}

It remains to compute \(\nu_1\). Let \(x_j=|\xi_j|\), and define
\[
    F(z)
    =
    \frac{\sum_j x_j\exp(-z/x_j^2)}
    {\left(\sum_j x_j^2\exp(-z/x_j^2)\right)^{1/2}}.
\]
This \(F\) is the left-hand side of \eqref{eq: equation sol} with \(z\) in place of
\(\zeta\). At \(z=0\),
\[
    F(0)
    =
    \frac{\sum_j x_j}{\left(\sum_jx_j^2\right)^{1/2}}.
\]
For block \(\ell\), \eqref{eq:multiscale-block-sizes} and \eqref{eq:multiscale-loading} imply
\[
    \sum_{j=M_{\ell-1}+1}^{M_\ell}x_j
    =
    m_\ell\frac{a}{\sqrt{m_\ell}}
    =
    a\sqrt{m_\ell}
    \asymp
    a\sqrt{k_u}\,\ell .
\]
Summing this blockwise identity over \(\ell=1,\ldots,L\) gives
\[
    \sum_jx_j
    \asymp
    a\sqrt{k_u}\sum_{\ell=1}^L \ell
    \asymp
    a\sqrt{k_u}L^2.
\]
The same block-energy identity \eqref{eq:multiscale-block-energy} gives
\begin{equation}\label{eq:multiscale-total-l2}
    \left(\sum_jx_j^2\right)^{1/2}
    =
    a\sqrt L.
\end{equation}
Combining the formula for \(F(0)\), the \(\ell_1\) computation, and
\eqref{eq:multiscale-total-l2}, we obtain
\[
    F(0)\asymp \sqrt{k_u}L^{3/2}.
\]
By choosing \(c_0>0\) sufficiently small in \eqref{eq:multiscale-L-conditions}, this gives
\[
    F(0)\le k_u/2.
\]
Since \(F\) is decreasing by \Cref{lem: decreasing} and the defining equation
\eqref{eq: equation sol} is \(F(\zeta)=k_u/2\), the inequality \(F(0)\le k_u/2\) implies
\(\zeta\le0\). Therefore \(\lambda=\sqrt{\zeta_+}=0\). Substituting \(\lambda=0\) into
\eqref{eq: nu1 definition} and using \eqref{eq:multiscale-total-l2} gives
\begin{equation}\label{eq:multiscale-nu1}
    \nu_1
    =
    \left(\sum_j\xi_j^2\right)^{1/2}
    =
    a\sqrt L.
\end{equation}
Equations \eqref{eq:multiscale-H-result}, \eqref{eq:multiscale-nu2}, and
\eqref{eq:multiscale-nu1} prove the three estimates stated at the beginning of the proposition.

By the moderately sparse assumption \(k_u\gg \sqrt n/\log p\), the moderately sparse branch of
\Cref{prop: upper bound equivalent} applies. Combining that branch with
\eqref{eq:multiscale-H-result} gives the computationally feasible upper bound
\[
    \tau_{\mathrm{adap}}(k_u,k~;~\xi)
    \lesssim
    H(n/\log p~;~\xi)\frac{k_u\log p}{n}
    \asymp
    a\sqrt L\frac{k_u\log p}{n}.
\]
Similarly, \Cref{thm: computational lower bound} and \eqref{eq:multiscale-H-result} give a
low-degree lower bound at separation
\[
    H(k_{\mathrm{eff}}~;~\xi)\frac{k_u\log p}{n}
    \asymp
    a\sqrt L\frac{k_u\log p}{n}.
\]
Thus the computationally feasible upper bound and the low-degree scale both have order
\(a\sqrt L\,k_u\log p/n\).

Finally, \Cref{thm: sparse spiked}, \eqref{eq:multiscale-nu1}, and \eqref{eq:multiscale-nu2}
give, under the sparse signed-spiked covariance model,
\[
    \tau^{\text{spike}}_{\text{adap}}(k_u,k~;~\xi)
    \asymp_{\log}
    \frac{\nu_1}{\sqrt n}
    +
    \nu_2\frac{k_u\log p}{n}
    \asymp_{\log}
    \frac{a\sqrt L}{\sqrt n}
    +
    a\frac{k_u\log p}{n}.
\]
Dividing this sparse signed-spiked rate by the low-degree scale
\(H(k_{\mathrm{eff}}~;~\xi)k_u\log p/n\), and using
\eqref{eq:multiscale-H-result}, gives
\[
    \frac{
        \tau^{\text{spike}}_{\text{adap}}(k_u,k~;~\xi)
    }{
        H(k_{\mathrm{eff}}~;~\xi)k_u\log p/n
    }
    \asymp_{\log}
    \frac{
        a\sqrt L/\sqrt n+a k_u\log p/n
    }{
        a\sqrt L\,k_u\log p/n
    }
    =
    \frac{1}{\sqrt L}+\frac{\sqrt n}{k_u\log p}.
\]
If
\(\min(\sqrt L,~ k_u\log p/\sqrt n)\) diverges polynomially in $p$, then the displayed ratio is
polynomially small, which proves the claimed separation.
\end{proof}

\subsection{I.i.d. sub-Weibull random-predictor loadings}
\label{app:random-subweibull-loading}

This subsection treats random-predictor loadings, where \(\xi\) is the covariate vector of a new
test point. For this example, the additional randomness is the draw of \(\xi\). After the test point
is observed, we condition on its realized value, relabel coordinates so that
\[
    |\xi_1|\ge |\xi_2|\ge\cdots\ge |\xi_p|,
\]
and evaluate the deterministic quantities \(H(\cdot~;~\xi)\), \(\nu_1\), and \(\nu_2\) from
\eqref{eq: H definition}, \eqref{eq: nu1 definition}, and \eqref{eq: nu2 definition}. Thus the
probability statements below concern the draw of \(\xi\). Throughout
\Cref{app:random-subweibull-loading}, let
\[
    N=\left\lfloor \frac n{\log p}\right\rfloor .
\]

Let \(W_1,\ldots,W_p\) be i.i.d. nonnegative random variables, and let \(\xi\) be the decreasing
rearrangement of \((W_1,\ldots,W_p)\). Assume that \(W\) has a two-sided sub-Weibull tail: for some
constants \(q>0\), \(0<c_1,c_2,C_1,C_2<\infty\), and \(t_0>0\),
\begin{equation}\label{eq:subweibull-tail}
    c_1\exp(-C_1 t^q)
    \le
    \mathbb P(W\ge t)
    \le
    C_2\exp(-c_2 t^q),
    \qquad t\ge t_0.
\end{equation}
This tail condition includes absolute Gaussian coordinates \((q=2)\), absolute sub-exponential
coordinates \((q=1)\), and other light-tailed predictors, up to constants in the exponent. Assume
\begin{equation}\label{eq:subweibull-scaling}
    1\ll k_u\le N,
    \qquad
    N=o(p),
    \qquad
    k_u\le p^\gamma
    \quad\text{for some fixed }\gamma<1/2.
\end{equation}

With probability tending to one, the realized loading vector lies outside the two deterministic
classes most closely related to this example.
\begin{itemize}
    \item It fails the regular-loading condition of
\cite{cai2017confidence}.
The lower bound in \eqref{eq:subweibull-tail} implies there exists some constant $c_1>0$ such that
\(\max_{j\le p}W_j\ge c_1 (\log p)^{1/q}\) with probability tending to 1.
Furthermore, it also implies that there exists a fixed $M$ such that the interval
\([t_0,M]\) receives positive probability mass. Therefore, with probability tending to 1, there is at least one nonzero coordinate with magnitude at most
\(M\). When both events happen, we have
\[
    \frac{\max_{j\in\supp(\xi)}|\xi_j|}
    {\min_{j\in\supp(\xi)}|\xi_j|}
    \gtrsim (\log p)^{1/q},
\]
which is unbounded.
\item
It is also not the exact deterministic polynomial-decay profile treated in \cite{cai2019optimal}.
The order-statistic comparison in \Cref{eq:subweibull-order-stat-event} in the proof of \Cref{prop:random-subweibull-loading} shows that for
\(j\to\infty\) with \(2j\le N\), the ratio between $|\xi_j|$ and $|\xi_{2j}|$ is of order
\[
    \frac{\{\log(ep/j)\}^{1/q}}{\{\log(ep/(2j))\}^{1/q}}
    \to 1,
\]
whereas \(a_0j^{-\alpha}\) with \(\alpha>0\) has dyadic ratio \(2^\alpha\).
\end{itemize}
Therefore, the regular-loading
theory of \cite{cai2017confidence} and the exact polynomial-decay analysis of
\cite{cai2019optimal} do not apply directly to this random-predictor loading.

\begin{proposition}
\label{prop:random-subweibull-loading}
Under \eqref{eq:subweibull-tail} and \eqref{eq:subweibull-scaling}, with probability tending to
one, we have
\begin{equation}\label{eq:subweibull-H-rate}
    H(r~;~\xi)
    =
    \left(\sum_{j\le r}\xi_j^2\right)^{1/2}
    \asymp
    \sqrt r\,\left\{\log\left(\frac{ep}{r}\right)\right\}^{1/q},
\end{equation}
uniformly for \(1\le r\le N\).
Moreover, we have
\begin{equation}\label{eq:subweibull-nu1-rate}
    \nu_1
    \asymp
    k_u\left\{\log\left(\frac{ep}{k_u^2}\right)\right\}^{1/2+1/q}.
\end{equation}
Consequently, in the moderately sparse regime \(k_u\gg \sqrt n/\log p\), we have
\begin{equation}\label{eq:subweibull-upper-rate}
    \tau_{\mathrm{adap}}(k_u,k~;~\xi)
    \lesssim
    k_u\sqrt{\frac{\log p}{n}}
    \left\{\log\left(\frac{ep}{N}\right)\right\}^{1/q},
\end{equation}
while \Cref{thm: hypothesis} gives
\begin{equation}\label{eq:subweibull-lower-rate}
    \tau_{\mathrm{adap}}(k_u,k~;~\xi)
    \gtrsim
    \frac{k_u}{\sqrt n}
    \left\{\log\left(\frac{ep}{k_u^2}\right)\right\}^{1/2+1/q}.
\end{equation}
In the regime where it holds that
\begin{equation}\label{eq:subweibull-matched-cond}
        \log(ep/N)\asymp \log(ep/k_u^2)\asymp \log p,
\end{equation}
we have
\begin{equation}\label{eq:subweibull-matched-rate}
    \tau_{\mathrm{adap}}(k_u,k~;~\xi)
    \asymp
    k_u\frac{(\log p)^{1/2+1/q}}{\sqrt n}.
\end{equation}
\end{proposition}

For the moderately sparse regime $\sqrt{n}/\log p \ll k_u\lesssim n/\log p$, \Cref{eq:subweibull-matched-cond} holds if $n=O(p^{1-\varepsilon})$ for some $\varepsilon>0$ and \Cref{cdt: sparsity assumption} holds.

For Gaussian coordinates \(q=2\), \eqref{eq:subweibull-matched-rate} becomes
\[
    \tau_{\mathrm{adap}}(k_u,k~;~\xi)
    \asymp
    k_u\frac{\log p}{\sqrt n}
\]
under polynomial aspect-ratio scaling.
For sub-exponential coordinates \(q=1\),
\eqref{eq:subweibull-matched-rate} becomes \(\tau_{\mathrm{adap}}(k_u,k~;~\xi)
    \asymp k_u(\log p)^{3/2}/\sqrt n\).
Formally, bounded predictors correspond to the limiting case \(q=\infty\), and the displayed
rate then reduces to \(\tau_{\mathrm{adap}}(k_u,k~;~\xi)
    \asymp k_u\sqrt{\log p/n}\), which is the same as the rate for dense regular loadings.

\subsubsection{Proof of Proposition~\ref{prop:random-subweibull-loading}}
The proof of \Cref{prop:random-subweibull-loading} makes use of two auxiliary facts:
a high-probability order-statistic comparison and a deterministic logarithmic-sum comparison.

\begin{lemma}[Order-statistic envelope]
\label{lem:subweibull-order-statistics}
Under \eqref{eq:subweibull-tail} and \eqref{eq:subweibull-scaling}, there exist constants
\(0<c<C<\infty\), depending only on the constants in \eqref{eq:subweibull-tail}, such that, with
probability tending to one,
\begin{equation}\label{eq:subweibull-order-stat-event}
    c\left\{\log\left(\frac{ep}{j}\right)\right\}^{1/q}
    \le
    |\xi_j|
    \le
    C\left\{\log\left(\frac{ep}{j}\right)\right\}^{1/q},
    \qquad
    1\le j\le N.
\end{equation}
\end{lemma}

\begin{lemma}[Logarithmic sum comparison]
\label{lem:subweibull-log-sum}
Suppose \(N=o(p)\). Then, for every fixed \(q>0\), the following holds uniformly for \(1\le r\le N\):
\[
    \sum_{j\le r}
    \left\{\log\left(\frac{ep}{j}\right)\right\}^{2/q}
    \asymp
    r\left\{\log\left(\frac{ep}{r}\right)\right\}^{2/q}.
\]
\end{lemma}

By \Cref{lem:subweibull-order-statistics}, with probability tending to one,
\eqref{eq:subweibull-order-stat-event} holds uniformly for \(1\le j\le N\).
We work on this event in the following.

Using \eqref{eq:subweibull-order-stat-event}, uniformly for \(1\le r\le N\),
\[
    H(r~;~\xi)^2
    \asymp
    \sum_{j\le r}
    \left\{\log\left(\frac{ep}{j}\right)\right\}^{2/q}.
\]
By \Cref{lem:subweibull-log-sum}, the sum on the right satisfies
\[
    \sum_{j\le r}
    \left\{\log\left(\frac{ep}{j}\right)\right\}^{2/q}
    \asymp
    r\left\{\log\left(\frac{ep}{r}\right)\right\}^{2/q},
    \qquad 1\le r\le N.
\]
Taking square roots proves \eqref{eq:subweibull-H-rate}.

We next compute \(\nu_1\). As in the preceding examples, let \(F_\xi(z)\) denote the left-hand side of \eqref{eq: equation sol}: \[ F_\xi(z) = \frac{\sum_{j=1}^{k_\xi}|\xi_j| \exp(-z/\xi_j^2)} {\left(\sum_{j=1}^{k_\xi}\xi_j^2 \exp(-z/\xi_j^2)\right)^{1/2}} . \] For the tail estimates below, it is more convenient to work with the squared version on the \(t=\sqrt z\) scale. For \(t>0\), define \begin{equation}\label{eq:subweibull-D-definition} D_\xi(t) = F_\xi(t^2)^2 = \frac{ \left(\sum_{j=1}^{k_\xi}|\xi_j|e^{-t^2/\xi_j^2}\right)^2 }{ \sum_{j=1}^{k_\xi}\xi_j^2 e^{-t^2/\xi_j^2} }. \end{equation} The equation \eqref{eq: equation sol} is equivalent to \[ D_\xi(\lambda)=k_u^2/4 \text{ if } \lambda>0. \]

We state the following result but defer its proof to the end of the section.

\begin{lemma}[Localization of the solution defining \(\nu_1\)]
\label{lem:subweibull-D-localization}
Assume \eqref{eq:subweibull-tail} and \eqref{eq:subweibull-scaling}. Let
\[
    L=\log\left(\frac{ep}{k_u^2}\right),
    \qquad
    s=L^{(q+2)/(2q)}.
\]
The following properties hold with probability tending to one:
\begin{enumerate}
    \item There exist constants \(0<c_-<C_+<\infty\)
such that
\begin{equation}\label{eq:subweibull-upper-crossing}
    D_\xi(C_+s)\le \frac{k_u^2}{8},
\end{equation}
and
\begin{equation}\label{eq:subweibull-lower-crossing}
    D_\xi(c_-s)\ge \frac{k_u^2}{2}.
\end{equation}
Consequently, the solution \(\lambda\) in \eqref{eq: equation sol} is positive and satisfies
\begin{equation}\label{eq:subweibull-lambda-localization}
    c_-s\le \lambda\le C_+s.
\end{equation}
\item It holds that
\begin{equation}\label{eq:subweibull-S2-at-lambda}
    \left(\sum_{j=1}^{k_\xi}\xi_j^2\exp(-\lambda^2/\xi_j^2)\right)^{1/2}
    \lesssim
    k_u L^{1/q}.
\end{equation}
\end{enumerate}
\end{lemma}

Let
\[
    S_2(t)
    =
    \sum_{j=1}^{k_\xi}\xi_j^2\exp(-t^2/\xi_j^2).
\]

By \Cref{lem:subweibull-D-localization}, we have
\[
    \lambda
    \asymp
    \left\{\log\left(\frac{ep}{k_u^2}\right)\right\}^{1/2+1/q},
\]
and
\[
    S_2(\lambda)^{1/2}
    \lesssim
    k_u
    \left\{\log\left(\frac{ep}{k_u^2}\right)\right\}^{1/q}.
\]
Therefore, by the definition of \(\nu_1\) in \eqref{eq: nu1 definition},
\[
\begin{aligned}
    \nu_1
    &=
    k_u\lambda+S_2(\lambda)^{1/2}\\
    &\asymp
    k_u
    \left\{\log\left(\frac{ep}{k_u^2}\right)\right\}^{1/2+1/q}.
\end{aligned}
\]
This proves \eqref{eq:subweibull-nu1-rate}.

In the moderately sparse regime, \Cref{prop: upper bound equivalent} and
\(N=\lfloor n/\log p\rfloor\) give
\[
    \tau_{\mathrm{adap}}(k_u,k~;~\xi)
    \lesssim
    \frac{k_u\log p}{n}H(N~;~\xi).
\]
Using \eqref{eq:subweibull-H-rate} with \(r=N\), the above upper bound becomes
\[
\begin{aligned}
    \tau_{\mathrm{adap}}(k_u,k~;~\xi)
    &\lesssim
    \frac{k_u\log p}{n}
    \sqrt{N}
    \left\{\log\left(\frac{ep}{N}\right)\right\}^{1/q}\\
    &\asymp
    k_u\sqrt{\frac{\log p}{n}}
    \left\{\log\left(\frac{ep}{N}\right)\right\}^{1/q},
\end{aligned}
\]
which proves \eqref{eq:subweibull-upper-rate}.

The lower bound in \eqref{eq:subweibull-lower-rate} follows from the
\(\nu_1/\sqrt n\) term in \Cref{thm: hypothesis} and
\eqref{eq:subweibull-nu1-rate}.

If
\[
    \log(ep/N)\asymp \log(ep/k_u^2)\asymp\log p,
\]
then the upper bound \eqref{eq:subweibull-upper-rate} and the lower bound
\eqref{eq:subweibull-lower-rate} have the common order
\[
    k_u\frac{(\log p)^{1/2+1/q}}{\sqrt n}
\]
up to logarithmic factors. This proves \eqref{eq:subweibull-matched-rate} and completes the
proof of \Cref{prop:random-subweibull-loading}.

\subsubsection{Proofs for Lemma~\ref{lem:subweibull-order-statistics} and Lemma~\ref{lem:subweibull-log-sum}}

\begin{proof}[Proof of \Cref{lem:subweibull-order-statistics}]
For \(1\le j\le N\), put
\[
    u_j=\left\{\log\left(\frac{ep}{j}\right)\right\}^{1/q}.
\]
Since \(N=o(p)\), we have \(u_N\to\infty\). Hence all thresholds used below exceed \(t_0\)
uniformly over \(1\le j\le N\), for all sufficiently large \(p\).

We first prove the upper bound. Choose \(A_+>0\) so large that
\[
    a_+ := c_2A_+^q>1.
\]
For each \(j\le N\), define the exceedance count
\[
    X_j^+
    =
    \sum_{i=1}^p \mathbf 1\{W_i>A_+u_j\}.
\]
If \(|\xi_j|>A_+u_j\), then \(X_j^+\ge j\). Since \(X_j^+\) is binomial with mean
\(\mu_j^+=p\mathbb P(W>A_+u_j)\), the upper tail bound in \eqref{eq:subweibull-tail} gives
\[
    \mu_j^+
    \le
    C_2p\exp(-c_2A_+^q u_j^q)
    =
    C_2p\left(\frac{ep}{j}\right)^{-a_+}.
\]
Therefore, uniformly for \(1\le j\le N\), we have
\[
    \frac{e\mu_j^+}{j}
    \le
    K_+\left(\frac{N}{p}\right)^{a_+-1}
    =:\rho_p,
\]
where \(K_+<\infty\) is a constant and \(\rho_p\to0\). Recall the standard binomial tail bound
$$
\mathbb P(X\ge j)\le(e\mu/j)^j$$
where $X$ is a binomial random variable with mean \(\mu\).
Together with the fact that \(\rho_p<1\) for all sufficiently large \(p\), the standard binomial tail bound implies that
\[
    \mathbb P\!\left(\exists\,1\le j\le N:\ |\xi_j|>A_+u_j\right)
    \le
    \sum_{j=1}^N \mathbb P(X_j^+\ge j)
    \le
    \sum_{j=1}^N \rho_p^j
    =
    o(1).
\]
Thus, with probability tending to one, \(|\xi_j|\le A_+u_j\) for all \(1\le j\le N\).

We next prove the lower bound. Choose \(A_->0\) so small that
\[
    a_- :=C_1A_-^q<1.
\]
For each \(j\le N\), define
\[
    X_j^-
    =
    \sum_{i=1}^p \mathbf 1\{W_i\ge A_-u_j\}.
\]
If \(X_j^-\ge j\), then \(|\xi_j|\ge A_-u_j\).
Denote the mean as \(\mu_j^-=p\mathbb P(W\ge A_-u_j)\).
The lower tail bound in \eqref{eq:subweibull-tail} gives  the bound
\[
    \mu_j^-
    \ge
    c_1p\exp(-C_1A_-^q u_j^q)
    =
    c_1p\left(\frac{ep}{j}\right)^{-a_-}.
\]
Consequently, denoting \(K_- = c_1e^{-a_-}>0\), we have \(\mu_j^-\ge K_-p^{1-a_-}\).
Moreover,
\[
    \frac{\mu_j^-}{j}
    \ge
    \left(\frac pj\right)^{1-a_-}
    \ge
    K_- \left(\frac pN\right)^{1-a_-}
    \to\infty,
\]
uniformly for \(1\le j\le N\).
Consequently, for all sufficiently large \(p\), we have
\(j\le \mu_j^-/2\) uniformly over \(1\le j\le N\). Chernoff's lower-tail bound therefore gives
\[
    \mathbb P(X_j^-<j)
    \le
    \mathbb P(X_j^-<\mu_j^-/2)
    \le
    \exp(-\mu_j^-/8).
\]
Using \(N=o(p)\) and  \(\mu_j^-\ge K_-p^{1-a_-}\), we obtain
\[
    \mathbb P\!\left(\exists\,1\le j\le N:\ |\xi_j|<A_-u_j\right)
    \le
    \sum_{j=1}^N \mathbb P(X_j^-<j)
    \le
    N\exp(-K_-p^{1-a_-}/8)
    =
    o(1).
\]
Combining the upper and lower events proves \eqref{eq:subweibull-order-stat-event} with
\(c=A_-\) and \(C=A_+\).
\end{proof}

\bigskip

\begin{proof}[Proof of \Cref{lem:subweibull-log-sum}]
The lower bound is immediate when \(r=1\). When \(r\ge2\), the indices \(r/2<j\le r\)
contribute at least a constant multiple of \(r\) terms, and each such term is at least
\(\{\log(ep/r)\}^{2/q}\). Hence the sum is bounded below by a constant multiple of
\(r\{\log(ep/r)\}^{2/q}\). This gives the lower bound.

For the upper bound, write \(L_r=\log(ep/r)\). Since \(r\le N=o(p)\), we have \(L_r\to\infty\)
uniformly over \(1\le r\le N\). The function \(x\mapsto \{\log(ep/x)\}^{2/q}\) is decreasing on \((0,r]\), and hence
\[
    \sum_{j\le r}\left\{\log\left(\frac{ep}{j}\right)\right\}^{2/q}
    \le
    \int_0^r
    \left\{\log\left(\frac{ep}{x}\right)\right\}^{2/q}\,dx .
\]
The integral is finite. With the change of variables
\(y=\log(ep/x)\),
\[
    \int_0^r
    \left\{\log\left(\frac{ep}{x}\right)\right\}^{2/q}\,dx
    =
    ep\int_{L_r}^{\infty} y^{2/q}e^{-y}\,dy .
\]
Putting \(y=L_r+u\), we obtain
\[
    \int_{L_r}^{\infty} y^{2/q}e^{-y}\,dy
    =
    e^{-L_r}L_r^{2/q}
    \int_0^\infty \left(1+\frac{u}{L_r}\right)^{2/q}e^{-u}\,du
    \lesssim
    e^{-L_r}L_r^{2/q},
\]
where the last inequality holds uniformly for large \(p\), because \(L_r\to\infty\) and
\(\int_0^\infty(1+u)^{2/q}e^{-u}\,du<\infty\). Since \(ep\,e^{-L_r}=r\), the upper bound is
\(O(rL_r^{2/q})\). This proves the lemma.
\end{proof}

\bigskip
\subsubsection{Preliminary results for proving Lemma~\ref{lem:subweibull-D-localization}}
To prove \Cref{lem:subweibull-D-localization}, we need two more results.

\begin{lemma}[A deterministic geometric saddle comparison]
\label{lem:subweibull-geometric-saddle}
Fix \(q>0\), \(B>1\), and constants \(a,b>0\). For \(\ell=0,1,2\), there exist
constants \(0<c<C<\infty\), depending only on \(q,B,a,b,\ell\), such that for all \(t\ge1\),
\begin{equation}\label{eq:subweibull-geometric-saddle}
    \sum_{m\in\mathbb Z}
    B^{m\ell}
    \exp\left\{-a\frac{t^2}{B^{2m}}-bB^{mq}\right\}
    \le
    C\,x(t)^\ell \exp\{-c\Phi(t)\},
\end{equation}
where
\begin{equation}\label{eq:subweibull-x-phi-def}
    x(t)=t^{2/(q+2)},\qquad
    \Phi(t)=x(t)^q=t^{2q/(q+2)}.
\end{equation}
Moreover, if \(m(t)\) is any integer satisfying \(B^{m(t)}\asymp x(t)\), then
\begin{equation}\label{eq:subweibull-geometric-saddle-lower}
    B^{m(t)\ell}
    \exp\left\{-a\frac{t^2}{B^{2m(t)}}-bB^{m(t)q}\right\}
    \ge
    c\,x(t)^\ell \exp\{-C\Phi(t)\}.
\end{equation}
\end{lemma}

\begin{lemma}[Grid-count event for sub-Weibull coordinates]
\label{lem:subweibull-grid-count}
Under \eqref{eq:subweibull-tail}, there exist constants
\(B>1\), \(A>0\), \(C_M<\infty\), and \(0<c<C<\infty\) such that the following event has
probability tending to one:
\begin{enumerate}
\item
\begin{equation}\label{eq:subweibull-max-bound}
    \max_{j\le p}|\xi_j|\le C_M(\log p)^{1/q}.
\end{equation}
\item
For every grid interval \(I_m=[B^m,B^{m+1})\) intersecting
\([t_0,C_M(\log p)^{1/q}]\),
\begin{equation}\label{eq:subweibull-grid-upper-count}
    \#\{j:|\xi_j|\in I_m\}
    \le
    (\log p)^A p\exp(-cB^{mq}).
\end{equation}
\item
Whenever \(p\exp(-CB^{mq})\ge(\log p)^A\),
\begin{equation}\label{eq:subweibull-grid-lower-count}
    \#\{j:|\xi_j|\in I_m\}
    \ge
    (\log p)^{-A}p\exp(-CB^{mq}).
\end{equation}
\end{enumerate}
\end{lemma}

\begin{proof}[Proof of \Cref{lem:subweibull-geometric-saddle}]
Set \(x=x(t)\). We have \(t^2=x^{q+2}\).

We first prove the upper bound
\eqref{eq:subweibull-geometric-saddle}.

Choose \(m_0\in\mathbb Z\) such that
\[
    B^{m_0}\le x < B^{m_0+1}.
\]
For \(m=m_0+k\), we have
\[
    B^m = B^{m_0}B^k = x\theta B^k,
    \qquad
    \theta:=\frac{B^{m_0}}{x}\in[B^{-1},1].
\]
It follows that
\[
\begin{aligned}
    \frac{t^2}{B^{2m}}
    &=
    \frac{x^{q+2}}{x^2\theta^2B^{2k}}
    =
    x^q\theta^{-2}B^{-2k},\\
    B^{mq}
    &=
    x^q\theta^qB^{qk}.
\end{aligned}
\]
Since \(\theta\in[B^{-1},1]\), there are constants \(c_1,C_1>0\), depending only on
\(B,a,b,q\), such that
\begin{equation}\label{eq:geom-saddle-exponent-lower}
    a\frac{t^2}{B^{2m}}+bB^{mq}
    \ge
    c_1x^q\left(B^{-2k}+B^{qk}\right).
\end{equation}
Also, because \(\theta\le1\), and $\ell\in \{0,1,2\}$,
\[
    B^{m\ell}
    =
    x^\ell \theta^\ell B^{k\ell}
    \le
    x^\ell B^{k\ell}.
\]
Therefore
\[
\begin{aligned}
    &\sum_{m\in\mathbb Z}
    B^{m\ell}
    \exp\left\{-a\frac{t^2}{B^{2m}}-bB^{mq}\right\}  \\
    &\qquad\le
    x^\ell
    \sum_{k\in\mathbb Z}
    B^{k\ell}
    \exp\left\{-c_1x^q\left(B^{-2k}+B^{qk}\right)\right\}.
\end{aligned}
\]
It remains to show that the last sum is at most a constant multiple of
\(\exp(-c x^q)\).

We split the sum into \(k\ge0\) and \(k<0\).

For \(k\ge0\), the term with \(B^{qk}\) in the exponent dominates. Since \(t\ge1\), we have
\(x^q\ge1\). Because \(\ell\) is fixed, the polynomial prefactor \(B^{k\ell}\) can be
absorbed into the exponential: there exists \(c_2\in(0,c_1)\) such that, for all \(k\ge0\),
\[
    B^{k\ell}\exp\{-c_1x^qB^{qk}\}
    \le
    \exp\{-c_2x^qB^{qk}\}.
\]
Hence
\[
    \sum_{k\ge0}
    B^{k\ell}
    \exp\left\{-c_1x^q\left(B^{-2k}+B^{qk}\right)\right\}
    \le
    \sum_{k\ge0}\exp\{-c_2x^qB^{qk}\}.
\]
Since \(B^{qk}\) grows geometrically, the last sum is bounded by a constant multiple of
its first term. Thus, for some constants \(C_2,c_3>0\),
\[
    \sum_{k\ge0}\exp\{-c_2x^qB^{qk}\}
    \le
    C_2\exp(-c_3x^q).
\]

For \(k<0\), write \(h=-k\ge1\). Then \(B^{-2k}=B^{2h}\), and $B^{k\ell}\leq 1$. We have
\[
\begin{aligned}
    &\sum_{k<0}
    B^{k\ell}
    \exp\left\{-c_1x^q\left(B^{-2k}+B^{qk}\right)\right\} \\
    &\qquad\le
    \sum_{h\ge1}
    \exp\{-c_1x^qB^{2h}\}
    \le
    C_3\exp(-c_4x^q),
\end{aligned}
\]
again because \(B^{2h}\) grows geometrically.

Combining the bounds for \(k\ge0\) and \(k<0\), we obtain
\[
    \sum_{m\in\mathbb Z}
    B^{m\ell}
    \exp\left\{-a\frac{t^2}{B^{2m}}-bB^{mq}\right\}
    \le
    C x^\ell \exp(-c x^q)
    =
    Cx(t)^\ell\exp\{-c\Phi(t)\}.
\]
This proves \eqref{eq:subweibull-geometric-saddle}.

We now prove the lower bound \eqref{eq:subweibull-geometric-saddle-lower}.
Let $m(t)=\lceil \frac{\log(x(t))}{\log B}\rceil$. Then
\(m(t)\) satisfies that \(B^{m(t)}\asymp x(t)\); that is,
there are constants \(0<c_5<C_5<\infty\)
such that
\[
    c_5x\le B^{m(t)}\le C_5x.
\]
Hence
\[
    B^{m(t)\ell}\asymp x^\ell.
\]
Moreover,
\[
    \frac{t^2}{B^{2m(t)}}+B^{m(t)q}
    \asymp
    \frac{x^{q+2}}{x^2}+x^q
    \asymp
    x^q
    =
    \Phi(t).
\]
Therefore the single summand at \(m=m(t)\) satisfies
\[
    B^{m(t)\ell}
    \exp\left\{-a\frac{t^2}{B^{2m(t)}}-bB^{m(t)q}\right\}
    \ge
    c x^\ell\exp(-C x^q),
\]
which is exactly \eqref{eq:subweibull-geometric-saddle-lower}.
\end{proof}

\bigskip

\begin{proof}[Proof of \Cref{lem:subweibull-grid-count}]
Choose \(C_M\) sufficiently large. By the upper tail bound in \eqref{eq:subweibull-tail},
\[
    \mathbb P\!\left(\max_{j\le p} W_j>C_M(\log p)^{1/q}\right)
    \le
    pC_2\exp(-c_2C_M^q\log p)
    =
    C_2p^{1-c_2C_M^q}
    =
    o(1),
\]
which proves \eqref{eq:subweibull-max-bound}.

Next choose \(B>1\) large enough that
\[
    c_2B^q>C_1
\]
and
\[
    \sup_{u\ge t_0}
    \exp\{C_1u^q-c_2B^qu^q\}
    <
    \frac{c_1}{2C_2}.
\]
Then, for all \(u\ge t_0\),
\begin{equation}\label{eq:subweibull-annulus-mass}
\begin{aligned}
    \mathbb P(u\le W\le Bu)
    &=
    \mathbb P(W\ge u)-\mathbb P(W>Bu)\\
    &\ge
    c_1\exp(-C_1u^q)-C_2\exp(-c_2B^q u^q)\\
    &\ge
    \frac{c_1}{2}\exp(-C_1u^q).
\end{aligned}
\end{equation}
Thus each grid interval \(I_m=[B^m,B^{m+1})\) with \(B^m\ge t_0\) has probability at least
a constant multiple of \(\exp(-CB^{mq})\), while the upper tail bound gives
\begin{equation}
    \label{eq:subweibull-grid-mean-upper}
    \mathbb P(W\in I_m)\le \mathbb P(W\ge B^m)\le C_2\exp(-c_2B^{mq}).
\end{equation}

We now justify the simultaneous count bounds. Let
\[
    \mathcal M_p
    =
    \{m\in\mathbb Z: I_m=[B^m,B^{m+1}) \text{ intersects }
    [t_0,C_M(\log p)^{1/q}]\}.
\]
Then \(|\mathcal M_p|=O(\log\log p)\). For \(m\in\mathcal M_p\), define
\[
    N_m
    =
    \#\{j:|\xi_j|\in I_m\}
    =
    \sum_{j=1}^p\mathbf 1\{W_j\in I_m\}.
\]
Thus \(N_m\) is binomial with mean $\mu_m:=
    p\,\mathbb P(W\in I_m)$.
For the finitely many boundary intervals with \(B^m<t_0\), the bounds below can be absorbed
by changing constants. Hence we focus on intervals with \(B^m\ge t_0\).

Choose a constant \(c>0\) smaller than \(c_2\), and also small enough that for all \(m\in\mathcal M_p\) and all sufficiently large \(p\), it holds that
\[
    p\exp(-cB^{mq})\ge p^{1/2}.
\]
This is possible because
\(B^{mq}\lesssim \log p\) on \(\mathcal M_p\).

For any $A>0$ to be determined, set
\[
    T_m
    =
    (\log p)^A p\exp(-cB^{mq}).
\]
By \eqref{eq:subweibull-grid-mean-upper},
\[
    \frac{T_m}{e\mu_m}
    \ge
    C(\log p)^A
    \exp\{(c_2-c)B^{mq}\}
    \ge
    C(\log p)^A.
\]
By taking \(A\) large enough, we can ensure that \(T_m\ge e\mu_m\) for all \(m\in\mathcal M_p\).
The binomial Chernoff bound yields
\[
    \mathbb P(N_m\ge T_m)
    \le
    \left(\frac{e\mu_m}{T_m}\right)^{T_m}.
\]
Therefore, for some constant $c'$ and sufficiently large $p$, we have
\[
    \mathbb P(N_m\ge T_m)
    \le
    \exp\{-c'T_m\log\log p\}
    \le
    \exp\{-c'p^{1/2}\log\log p\}.
\]
Taking a union bound over \(\mathcal M_p\), we obtain
\[
    \mathbb P\left(
    \exists m\in\mathcal M_p \text{ such that }
    N_m> T_m
    \right)
    \le o(1),
\]
since $|\mathcal M_p|=O(\log\log p)$.
This proves the simultaneous upper count bound \eqref{eq:subweibull-grid-upper-count}.

We next prove the simultaneous lower count bound. By \eqref{eq:subweibull-annulus-mass},
there exist constants \(c_\ell,C_\ell>0\) such that
\begin{equation}\label{eq:subweibull-grid-mean-lower}
    \mu_m
    =
    p\,\mathbb P(W\in I_m)
    \ge
    c_\ell p\exp(-C_\ell B^{mq})
\end{equation}
for all $m\in \mathcal M_p$. Choose \(C\ge C_\ell\).
For any $A>0$ to be determined, consider only those \(m\in\mathcal M_p\)
for which
\[
    p\exp(-CB^{mq})\ge(\log p)^A.
\]
Set
\[
    L_m
    =
    (\log p)^{-A}p\exp(-CB^{mq}).
\]
By \eqref{eq:subweibull-grid-mean-lower},
\[
    \frac{L_m}{\mu_m}
    \le
    c_\ell^{-1}(\log p)^{-A}
    \exp\{-(C-C_\ell)B^{mq}\}
    \le
    c_\ell^{-1}(\log p)^{-A}.
\]
Taking \(A\) sufficiently large gives \(L_m\le \mu_m/2\) uniformly over these intervals.
Hence Chernoff's lower-tail bound implies
\[
    \mathbb P(N_m<L_m)
    \le
    \mathbb P(N_m<\mu_m/2)
    \le
    \exp(-\mu_m/8).
\]
Moreover, since \(C\ge C_\ell\) and
\(p\exp(-CB^{mq})\ge(\log p)^A\), we have
\[
    \mu_m
    \ge
    c_\ell p\exp(-C_\ell B^{mq})
    \ge
    c_\ell p\exp(-CB^{mq})
    \ge
    c_\ell(\log p)^A.
\]
Therefore, for some constant $c>0$, we have
\[
    \mathbb P(N_m<L_m)
    \le
    \exp\{-c(\log p)^A\}.
\]
A union bound over \(O(\log\log p)\) intervals gives
\[
    \mathbb P\left(
    \exists m\in\mathcal M_p:
    p\exp(-CB^{mq})\ge(\log p)^A
    \text{ and }
    N_m< L_m
    \right)
    =o(1).
\]
This proves the simultaneous lower count bound \eqref{eq:subweibull-grid-lower-count}.

\end{proof}

\subsubsection{Proof of Lemma~\ref{lem:subweibull-D-localization}}

To prove \Cref{lem:subweibull-D-localization},
we work on the event in \Cref{lem:subweibull-grid-count}. For \(t>0\), define
\[
    S_\ell(t)
    =
    \sum_{j=1}^{k_\xi}
    |\xi_j|^\ell \exp(-t^2/\xi_j^2),
    \qquad \ell=0,1,2.
\]
Then
\[
    D_\xi(t)=\frac{S_1(t)^2}{S_2(t)}.
\]

First we prove \eqref{eq:subweibull-upper-crossing}. By Cauchy's inequality,
\[
    S_1(t)^2\le S_0(t)S_2(t),
\]
and hence
\begin{equation}\label{eq:subweibull-D-by-S0}
    D_\xi(t)\le S_0(t).
\end{equation}
Set \(t_+=C_+s\). Coordinates with $|\xi_j|\le t_0$ contribute at most
\(p\exp(-t_+^2/t_0^2)\) to \(S_0(t_+)\). For the grid intervals above \(t_0\) as defined in \Cref{lem:subweibull-grid-count},
\Cref{eq:subweibull-grid-upper-count} implies that
\begin{equation}\label{eq:subweibull-S1-upper-tminus}
    \begin{aligned}
    S_0(t_+)
    &\le
    p\exp(-t_+^2/t_0^2)
    +
    \sum_m
    \#\{j:|\xi_j|\in I_m\}
    \exp(-t_+^2/B^{2m+2})\\
    &\le
    p\exp(-t_+^2/t_0^2)
    +
    (\log p)^A p
    \sum_m
    \exp\{-t_+^2/B^{2m+2}-cB^{mq}\}.
\end{aligned}
\end{equation}

By \Cref{lem:subweibull-geometric-saddle},
\begin{equation}\label{eq:subweibull-S0-upper}
    S_0(t_+)
    \lesssim_{\log}
    p\exp\{-c_0 t_+^{2q/(q+2)}\}
    =
    p\exp\{-c_0 C_+^{2q/(q+2)}L\},
\end{equation}
for a constant \(c_0>0\). Choose \(C_+\) so large that
\[
    a_+:=c_0 C_+^{2q/(q+2)}>1.
\]
Then
\[
    \frac{p\exp(-a_+L)}{k_u^2}
    =
    e^{-a_+}
    \left(\frac{k_u^2}{p}\right)^{a_+-1}
    \to0,
\]
because \(k_u\le p^\gamma\) with \(\gamma<1/2\). The logarithmic factor in
\eqref{eq:subweibull-S0-upper} is also negligible compared with the resulting polynomial
decay. Combining this with \eqref{eq:subweibull-D-by-S0} proves
\eqref{eq:subweibull-upper-crossing}.

We next prove \eqref{eq:subweibull-lower-crossing}. Set \(t_-=c_-s\) and
\[
    x_-=x(t_-)=t_-^{2/(q+2)}.
\]
Choose a grid interval \(I_m=[B^m,B^{m+1})\) such that \(B^m\asymp x_-\). Since
\(x_-\asymp L^{1/q}\to\infty\), this interval lies above \(t_0\) for large \(p\).
Also,
\[
    B^{mq}\asymp x_-^q
    =
    t_-^{2q/(q+2)}
    =
    c_-^{2q/(q+2)}L.
\]
Taking \(c_->0\) sufficiently small, we have
\[
    p\exp(-CB^{mq})
    \ge
    (\log p)^A,
\]
so the lower count bound \eqref{eq:subweibull-grid-lower-count} applies to this interval.
Write $N_m=\#\{j:|\xi_j|\in I_m\}$.
We have
\begin{equation*}
    N_m
    \ge
    (\log p)^{-A}p\exp(-CB^{mq}).
\end{equation*}
For every coordinate with \(|\xi_j|\in I_m=[B^m,B^{m+1})\),
\[
    |\xi_j|\ge B^m,
    \qquad
    \exp(-t_-^2/\xi_j^2)\ge \exp(-t_-^2/B^{2m}).
\]
Therefore
\begin{align}
    S_1(t_-)
    &=
    \sum_{j=1}^{k_\xi}|\xi_j|\exp(-t_-^2/\xi_j^2) \notag\\
    &\ge
    \sum_{j:\,|\xi_j|\in I_m}
    |\xi_j|\exp(-t_-^2/\xi_j^2) \notag\\
    &\ge
    N_m B^m\exp(-t_-^2/B^{2m}) \notag\\
    &\ge
    (\log p)^{-A}pB^m
    \exp\{-t_-^2/B^{2m}-CB^{mq}\}.
    \label{eq:subweibull-S1-selected-interval}
\end{align}
The choice \(B^m\asymp x_-\) implies, after changing constants only by factors
depending on \(B\),
\[
    B^m\gtrsim x_-,
    \qquad
    \frac{t_-^2}{B^{2m}}\lesssim \frac{t_-^2}{x_-^2},
    \qquad
    B^{mq}\lesssim x_-^q .
\]
Substituting these three comparisons into
\eqref{eq:subweibull-S1-selected-interval} gives
\[
    S_1(t_-)
    \gtrsim_{\log}
    p\,x_-\,
    \exp\{-C_0(t_-^2/x_-^2+x_-^q)\}.
    \]
Finally, since \(x_-=t_-^{2/(q+2)}\), we have $\frac{t_-^2}{x_-^2}
    =
    t_-^{2q/(q+2)}$ and $
    x_-^q
    =
    t_-^{2q/(q+2)}$.
After increasing \(C_0\) if necessary, we have
\begin{equation}\label{eq:subweibull-S1-lower}
    S_1(t_-)
    \gtrsim_{\log}
    p\,x_-\,
    \exp\{-C_0t_-^{2q/(q+2)}\}.
\end{equation}

\medskip

For \(S_2(t_-)\), we now have \(\ell=2\) and we can still use the same argument in \Cref{eq:subweibull-S1-upper-tminus,eq:subweibull-S0-upper} together with \Cref{lem:subweibull-geometric-saddle} to get
\begin{equation}\label{eq:subweibull-S2-upper-tminus}
\begin{aligned}
    S_2(t_-)
    &\le
    p t_0^2\exp(-t_-^2/t_0^2)
    +
    (\log p)^A p
    \sum_m
    B^{2m+2}\exp\{-t_-^2/B^{2m+2}-cB^{mq}\}\\
    &\lesssim_{\log}
    p\,x_-^2\,\exp\{-c_0't_-^{2q/(q+2)}\},
\end{aligned}
\end{equation}
where $c_0'$ is a positive constant.
Combining \eqref{eq:subweibull-S1-lower} and \eqref{eq:subweibull-S2-upper-tminus}, we get
\begin{equation}\label{eq:subweibull-D-lower-general}
    D_\xi(t_-)
    =
    \frac{S_1(t_-)^2}{S_2(t_-)}
    \gtrsim_{\log}
    p\exp\{-C_1't_-^{2q/(q+2)}\}.
\end{equation}
Since
\[
    t_-^{2q/(q+2)}
    =
    c_-^{2q/(q+2)}L,
\]
we can choose \(c_->0\) sufficiently small so that
\[
    a_-:=C_1'c_-^{2q/(q+2)}<1.
\]
Then
\[
    \frac{p\exp(-a_-L)}{k_u^2}
    =
    e^{-a_-}
    \left(\frac{p}{k_u^2}\right)^{1-a_-}
    \to\infty,
\]
again because \(k_u\le p^\gamma\) with \(\gamma<1/2\). This dominates the logarithmic
loss in \eqref{eq:subweibull-D-lower-general}, proving
\eqref{eq:subweibull-lower-crossing}.

Since \Cref{lem: decreasing} shows that \(F_\xi(z)\) is nonincreasing in \(z\), the
function \(D_\xi(t)=F_\xi(t^2)^2\) is nonincreasing in \(t>0\). The two crossing inequalities
\eqref{eq:subweibull-upper-crossing} and \eqref{eq:subweibull-lower-crossing} therefore
imply \eqref{eq:subweibull-lambda-localization}.

It remains to prove \eqref{eq:subweibull-S2-at-lambda}. By \eqref{eq:subweibull-max-bound},
\[
    \|\xi\|_\infty
    \le
    C_M(\log p)^{1/q}
    \asymp
    L^{1/q},
\]
where the last comparison follows from \(k_u\le p^\gamma\) with \(\gamma<1/2\), which
implies \(L=\log(ep/k_u^2)\asymp\log p\). Since
\[
    S_2(\lambda)
    =
    \sum_j \xi_j^2e^{-\lambda^2/\xi_j^2}
    \le
    \|\xi\|_\infty
    \sum_j |\xi_j|e^{-\lambda^2/\xi_j^2}
    =
    \|\xi\|_\infty S_1(\lambda),
\]
and since \(D_\xi(\lambda)=S_1(\lambda)^2/S_2(\lambda)=k_u^2/4\), we have
\[
    S_2(\lambda)
    \le
    \|\xi\|_\infty \{D_\xi(\lambda)S_2(\lambda)\}^{1/2}.
\]
Therefore
\[
    S_2(\lambda)^{1/2}
    \le
    \|\xi\|_\infty D_\xi(\lambda)^{1/2}
    \lesssim
    k_u L^{1/q},
\]
which proves \eqref{eq:subweibull-S2-at-lambda}.

\end{document}